\documentclass[generic,11pt,preprint]{imsart} %[aihp][aos,ba, generic] 
\usepackage{amsmath,amssymb,amsfonts,amsthm} %,xcolor,bm}
\RequirePackage[numbers]{natbib}
\RequirePackage[colorlinks,citecolor=blue,urlcolor=blue]{hyperref}

\usepackage{fancyhdr}
\allowdisplaybreaks
\startlocaldefs
\numberwithin{equation}{section}
\newtheorem{theorem}{Theorem}[section]
\newtheorem{lemma}{Lemma}[section]
\newtheorem{corollary}{Corollary}[section]
\newtheorem{proposition}{Proposition}[section]
\theoremstyle{definition}
\newtheorem{remark}{Remark}[section]

\DeclareMathOperator*{\argmax}{\arg\!\max}
\DeclareMathOperator*{\argmin}{\arg\!\min}
\DeclareMathOperator{\rank}{rank}
\endlocaldefs

\begin{document}
\begin{frontmatter}
\title{General framework for projection structures}
\runtitle{General framework for projection structures}

\begin{aug}
\author{\fnms{Eduard} \snm{Belitser}$^1$} %\thanksref{addr1}}%\ead[label=e1]{e.n.belitser@vu.nl}}
\and 
\author{\fnms{Nurzhan} \snm{Nurushev}$^{2,}$\thanksref{t}}%\thanksref{addr2} \ead[label=e2]{n.nurushev@uva.nl}}
\address{$^1$VU Amsterdam and\;  $^2$Rabobank} 
%\address[addr1]{Department of Mathematics,VU Amsterdam} %, \printead{e1}}
%\address[addr2]{Rabobank, Utrecht} %, \printead{e2}}
\thankstext{t}{The main part of the paper was done when the second author was affiliated 
with VU Amsterdam and University of Amsterdam.}
\runauthor{Belitser, E. and Nurushev, N.}
\affiliation{VU Amsterdam and Rabobank} 
\end{aug}

\begin{abstract}
In the first part,
we develop a {\em general framework for projection structures} and study several inference 
problems %on the unknown parameter 
within this framework. 
We propose procedures based on \emph{data dependent measures} (DDM) and make 
connections with {\em empirical Bayes} and {\em penalization} methods.
%by using {\em empirical Bayes} and {\em penalization} methods. 
%For the inference problems of interest to be feasible, the high dimensional parameter is assumed to have 
%a structure from a family of possible structures (e.g., a sparse vector or a biclustered matrix), called 
%by the unifying term {\em projection structures
%We use the predictive loss to compare projection procedures and 
The main inference problem is the \emph{uncertainty quantification} (UQ), but on the way we solve %as well
the {\em estimation}, {\em DDM-contraction} problems, and a weak version of the \emph{structure recovery} 
problem. The  approach is {\em local} in that the quality of the inference procedures is measured by 
the local quantity, the \emph{oracle rate}, which is the best trade-off between the approximation 
error by a projection structure and the complexity of that approximating projection structure.
Like in statistical learning settings, we develop \emph{distribution-free} theory as 
no particular model is imposed, we only assume certain mild condition on the stochastic 
part of the projection predictor. 
%The approach is also \emph{robust} in that the stochastic part of the general framework  
%is assumed to satisfy only certain mild condition, the errors may be non-iid with unknown distribution.  
We introduce the \emph{excessive bias restriction} (EBR) under which we establish the local %(oracle) 
confidence optimality of the constructed confidence ball.  

The proposed general framework unifies a very broad class of high-dimensional models 
and structures, interesting and important on their own right. %including graphical/network models, 
In the second part, we apply the developed theory and demonstrate how the general results 
deliver a whole avenue of local and global minimax results (many new ones, some known results 
from the literature are improved) for particular models and structures as consequences, 
including {\em white noise model} and {\em density estimation} with {\em smoothness} structure, 
{\em linear regression} and {\em dictionary learning} with {\em sparsity} structures, 
{\em biclustering} and {\em stochastic block models} with {\em clustering} structure, 
{\em covariance matrix} estimation with {\em banding} and sparsity structures, 
and many others. Various adaptive minimax results over various scales follow also from our local results. 
\end{abstract}


% history:
% \received{\smonth{1} \syear{0000}}
\tableofcontents
\end{frontmatter}
%\sloppy

\part{\Large Theory}
 
\section{Introduction}
\sloppy
%In this subsection the model \eqref{model} in case $\mathcal{Y}=\Theta\subseteq\mathbb{R}^{\infty}$ 
%is replaced by the same model but with $\mathcal{Y}=\mathbb{R}^{N}$, because  any element 
%$x\in\Theta\subseteq\mathbb{R}^{\infty}$ can be approximated by some element $x'\in\mathbb{R}^N$ 
%for some big enough $N>0$.
Suppose we observe  a random element $(Y,X)\in(\mathcal{Y}\times\mathcal{X})$:
\begin{align*}
%\label{model0}
Y \sim \mathbb{P}_\theta=\mathbb{P}_{\theta,X}, \;\; \theta\in\Theta\subseteq\mathcal{Y}, \quad \text{such that} 
\quad \mathbb{E}_\theta Y=\theta(X)=\theta,  
\end{align*}
where $\mathbb{P}_\theta$ %=\mathrm{P}^{(\sigma,N)}_\theta$,
is the probability measure of $Y$ %from the model (\ref{model}) ,  
($\mathbb{E}_\theta$ is the corresponding expectation) depending on
an unknown high-dimensional parameter of interest $\theta$.
%and $\mathrm{P}_{I^*}$ is the projection operator onto some linear 
%subspace $\mathbb{L}_{I^*}\subseteq \mathcal{Y}$.
By default, $\Theta=\mathcal{Y}=\mathbb{R}^N$, $\mathcal{X}=\mathbb{R}^{d_X}$  for  
``big'' $N,d_X\in\mathbb{N}$ (with the usual norm $\|\cdot \|$), unless stated otherwise.
In some particular models (see Part II), 
$\mathcal{Y},\Theta\subseteq \mathbb{R}^\infty$ 
can be infinite dimensional and $\Theta$ can be a proper subset of $\mathcal{Y}$, 
e.g., $\mathcal{Y}=\mathbb{R}^\infty$ and $\Theta=\ell_2$.
Those models can also be reduced to the high-dimensional case by assuming 
that any $\theta\in\Theta\subseteq\mathbb{R}^{\infty}$ can be arbitrarily well approximated 
by $\bar{\theta}\in\mathbb{R}^N$ for sufficiently large $N\in\mathbb{N}$. 
%For a moment, one can simplistically think of $X$ is a vector of covariates and $Y$ as the vector of responses.

%$\mathcal{I}=\cup_{s\in\mathcal{S}} \mathcal{I}_{s}$, which is further sliced into layers  $\mathcal{I}_{s}$ of 
%structures of equal complexities with parameter $s\in\mathcal{S}$ characterizing the complexity of 
%all structures from the layer $\mathcal{I}_{s}$. The layers  $\mathcal{I}_{s}$, $s\in\mathcal{S}$, 
%will be assumed to satisfy certain condition; see Condition (A2). 
Let $\sigma \xi = Y - \mathbb{E}_\theta Y$ (any $Y$ is its expectation plus zero mean ``noise''), 
$\sigma>0$ be the  ``noise intensity'', then
%$\mbox{Var}(\xi_{i})\le C_\xi$, 
%$\xi=(\xi_i, i \in[N])$. %with $\mathcal{N}=[N]=\{1,\ldots, N\}$ %or $\mathcal{N}=\mathbb{N}$. 
%Besides, $\mathbb{E}_\theta \xi_i=0$. 
%For any high-dimensional $\theta\in\Theta$ there exist  some linear subspace 
%$\mathbb{L}_I\subseteq \mathcal{Y}$ and projection operator  $\mathrm{P}_{I}$ onto $\mathbb{L}_I$ 
%such that    $\theta =\mathrm{P}_{I} \theta$, where   the structure $I$ is in some abstract finite space 
%$\mathcal{I}_{s}$, which is further indexed by $s\in\mathcal{S}$. We assume that  finite space  
%$\mathcal{I}_{s}$, $s\in\mathcal{S}$, satisfies certain condition; see Condition (A2). 
\begin{align}
\label{model}
Y=\theta(X)+\sigma \xi =\theta+\sigma \xi  %=\mathrm{P}_{I^*} \upsilon+ \sigma \xi 
\quad \text{with}\quad \mathbb{E}_\theta \xi=0.
\end{align}
The (known) parameter $\sigma$ is introduced to accommodate certain asymptotic regimes
where $\sigma \to 0$ reflects an information increase. In some particular models, some extra 
information can be converted into a smaller noise intensity $\sigma$. For example, suppose 
we originally observed $X_{ij}$'s with $\mathrm{E} X_{ij}=\theta_i$ and $\mathrm{Var}(X_{ij}) =1$, 
such that $(X_{ij},j\in[m])$ are independent for each $i\in[n]$. By taking 
$X_i=\frac{1}{m}\sum_{j=1}^{m} X_{ij}$, we obtain \eqref{model} with $\sigma^2=m^{-1}$.
In what follows, we derive {\em non-asymptotic} results, which imply asymptotic assertions if needed. 
Possible asymptotic regimes are: high-dimensional setup $N \to \infty$ 
(the leading case in the literature for high-dimensional models $\mathcal{Y}=\mathbb{R}^N$), 
decreasing noise level $\sigma \to 0$, or their combination, e.g., $\sigma = N^{-1/2}$ and $N\to \infty$.

Useful inference %in high-dimensional models 
is not possible without some (approximate) \emph{structure} %on the parameter of interest, 
in the data, the basic idea is to reduce the ``effective'' dimensionality of the high-dimensional 
$\theta$ in \eqref{model}. The most popular structural assumptions are \emph{smoothness}, 
\emph{sparsity} and  \emph{clustering}. These structures and many others can be represented 
via appropriate families of linear spaces $\mathbb{L}_I\subseteq \mathcal{Y}$, $I\in \mathcal{I}$. 
%``original'' $\upsilon\in\Theta$ and we are back to the inference problem on $\theta$.
%We will interchangeably use the same notation $\theta$ to denote the  true structured parameter 
%$\mathrm{P}_{I^*}\theta$ and unstructured parameter $\theta$.  It should be clear from the context 
%which notation is meant in each expression.
Precisely, we introduce a finite (or countable) family $\mathcal{I}$ of possible structures $I$ 
and an associated  family of linear subspaces $\{\mathbb{L}_I, I\in \mathcal{I}\}$ 
of $\mathcal{Y}$, which express these structures. This in turn determines the family of corresponding 
projection operators $\{\mathrm{P}_I, I\in \mathcal{I}\}$ onto linear subspaces 
$\{\mathbb{L}_I, I\in \mathcal{I}\}$. The true $\theta$ is ``approximately structured'' according 
to the family $\mathcal{I}$ if $\|\theta -\mathrm{P}_{I^*} \theta\|^2
=\min_{I\in\mathcal{I}} \|\theta -\mathrm{P}_I \theta\|^2$ is close to zero. If 
$\|\theta -\mathrm{P}_{I^*} \theta\|^2=0$, the true $\theta$ happens to be exactly structured, 
i.e., $\theta\in\mathbb{L}_{I^*}$ and $I^*$ has the meaning of the ``true structure'' of the true $\theta$.
The family of structures $\mathcal{I}=\mathcal{I}(X)$ may depend on $X$,
(e.g., in the linear regression model from %Section \ref{sec_multiple_linear_regression} of 
Part II). We skip the dependence of $\mathcal{I}=\mathcal{I}(X)$ (and other quantities) on 
$X$ in further notation.

The general goal is to make inference on the  parameter $\theta$ based on the data. % $(Y,X)$:
For that, we propose a \emph{data dependent measure} (DDM) $\hat{\pi}(\vartheta|Y)$ on $\theta$ 
(we will use the variable $\vartheta$ in DDMs, to distinguish it from the true parameter $\theta$)
and use it to construct an estimator $\hat{\theta}$ and a structure selector $\hat{I}$, also 
making connections with \emph{empirical Bayes} and \emph{penalization} methods. 
%As we construct the DDM $\pi(\vartheta|Y)$, 
In the Bayesian literature, the quality of Bayesian procedures is 
characterized by the posterior contraction rate: ``good'' posteriors should concentrate 
around the truth.  
DDM is an extension of the notion of posterior, as DDM does not have to result from a prior. 
By analogy with posteriors in Bayesian analysis, an accompanying problem of interest is 
therefore the \emph{contraction of the DDM} $\hat{\pi}(\vartheta|Y)$ 
to the ``true''  $\theta$ from the %frequentist 
perspective of the ``true'' measure $\mathbb{P}_\theta$,
the actual distribution of the data, %from \eqref{model}, 
which is unknown.

Despite the rapidly growing number 
of papers about particular high-dimensional and nonparametric models and structures 
(cf.\ \cite{%Donoho&Johnstone&Hoch&Stern:1992, Donoho&Johnstone:1994b, 
Belitser&Levit:1995, Birge&Massart:2001, Belitser&Ghosal:2003, %Bunea&Tsybakov&Wegkamp:2007, 
Baraud:2004, %vanderVaart&vanZanten:2008, 
Babenko&Belitser:2010, %Caietal:2010, 
%Lounici&Pontil&vandeGeer&Tsybakov:2011, Raskutti&Wainwright&Yu:2011, 
Rigollet&Tsybakov:2011, Castillo&vanderVaart:2012, Nickl&vandeGeer:2013, 
Szabo&vanderVaart&vanZanten:2013, %Tsybakov:2014, 
Martin&Walker:2014, vanderPas&Kleijn&vanderVaart:2014, Belitser&Nurushev:2015, 
%Bellec&Tsybakov:2015, Castilloetal:2015, %Chatterjeeetal:2015, Gao&Lu&Zhou:2015, 
%Hoffmann&Rousseau&Schmidt-Hieber:2015, %Klopp:2015, 
Belitser:2016, Belitser&Nurushev:2017, %Kirichenko&vanZanten:2017, %Klopp&Tsybakov&Verzelen&:2017, 
Martin&etal:2017, Bellec:2018}), there are %very 
few approaches in both frequentist and Bayesian literature that can deal with general classes
of high-dimensional and nonparametric models: general posterior contraction rate results are studied 
in \cite{Ghosal&Ghos&Vaart:2000, vanderVaart&vanZanten:2008, Gao&vanderVaart&Zhou:2015, Han:2017},  
general frameworks for estimation in \cite{Gao&vanderVaart&Zhou:2015, Klopp&Lu&Tsybakov&Zhou&:2017}.
We should %especially 
highlight the paper \cite{Gao&vanderVaart&Zhou:2015} which provided us 
with important insights for certain aspects of the present study (although our approach is very different).  
However, all estimation (and posterior contraction) results do not reveal  how far the optimal estimator 
(posterior) is from the ``true'' $\theta$. It is of great importance to quantify this uncertainty, 
which we cast into the problem of constructing optimal confidence sets for 
relevan quantities. %$\theta$.  

\section{The scope of the paper} %of this paper.}
The main contributions of this paper are:
1) we develop a \emph{general abstract framework of projection structures}, bringing to 
culmination the research path followed by the papers 
\cite{Babenko&Belitser:2010, Belitser:2016, Belitser&Nurushev:2017, 
Belitser&Nurushev:2018, Belitser&Nurushev:2015, Belitser&Ghosal:2018};
2) within this general abstract framework, we solve the following inference problems for $\theta$:  \emph{estimation}, \emph{DDM contraction}, (weak) \emph{structure recovery}, 
and \emph{uncertainty quantification} (UQ); %by constructing an \emph{optimal confidence set};
3) we derive \emph{local} results in the \emph{refined} formulation and 
in the \emph{distribution-free} setting;
4) we deal with the \emph{deceptiveness phenomenon in UQ} by introducing
the \emph{excessive bias restriction} (EBR) in the general framework.

%This  basically bringing all of them (and many others) under one general umbrella 
%as all the models and structures considered 
%in the above mentioned previous papers  are 
%particular cases of the general abstract framework of projection structures developed in this paper.
%
%marks culmination
%basically the generalizes the results of the  
%structure of the paper follows that of 
%We consider all the above mentioned inference problems within the 
%\emph{general framework of projection structures}.
%
%and Part II basically generalize the results of  of several 
%earlier papers of the authors \cite{Babenko&Belitser:2010, Belitser:2016, Belitser&Nurushev:2017, 
%Belitser&Nurushev:2018, Belitser&Nurushev:2015, Belitser&Ghosal:2018},

The keywords summarizing the main novel features of our approach in this paper are 
therefore {\em general framework}, {\em distribution-free}, {\em local}, {\em refined}, 
{\em EBR}. Below we explain these features in some more detail.

\subsection{General framework}
\label{subsec_general_framework}
We develop a \emph{general framework of projection structures} 
and study above mentioned inference problems within this framework. 
%The main inference problem is the \emph{uncertainty quantification}, but on the way we solve the 
%\emph{estimation} problem, \emph{DDM contraction} problem and (a weak version of) 
%\emph{structure recovery} problem as well.
 As the proposed general framework unifies a broad class of models with various 
structures (including graphical/network models), interesting and important on their own right,
the general framework results deliver a whole avenue of results (many new ones, some are 
known in the literature, some are improved) for particular models and structures as consequences.
%Besides, specializing the general framework results and notions to particular 
%models of interest %This opens up 
%delivers a whole avenue of possible (known in the literature and new ones) results and 
%new interesting interpretations and insights in various important settings,  
There are numerous examples of models and structures falling into our general framework.
In Part II, we apply our general methodology to the following cases of model/structure:
\begin{itemize}
%1) 
\item
signal+noise model with smoothness structure;
%2) 
\item
smooth function on a graph; % (Laplacian graph) with smoothness structure;
%3) 
\item
density estimation with smoothness structure;
%4) 
\item
regression under wavelet basis (smoothness+sparsity structure);
%5) 
\item
signal+noise model with sparsity structure;
%6) 
\item
signal+noise model with \emph{clustering} (or, \emph{multi-level sparsity}) structure; 
%(this structure is considered for the first time);
%7) 
\item
signal+noise with shape structure: isotonic, unimodal and convex
regressions;
%8) 
\item
linear regression with sparsity structure;
%9) 
\item
linear regression with shape structure: aggregation;
%10) 
\item
matrix+noise with smoothness structure: banded covariance matrix;
%11) 
\item
matrix+noise with sparsity structure: sparse covariance matrix;
%12) 
\item
matrix+noise with sparsity structure;
%13)
\item 
matrix+noise with clustering structure: biclustering model;
%14)
\item 
matrix linear regression with group sparsity;
%15)
\item 
matrix linear regression with group clustering (multi-task learning);
%16)
\item 
matrix linear regression with mixture structure;
%17)
\item 
matrix linear regression with unknown design: dictionary learning.
%%18)$ precision matrix with sparsity structure; %19) matrix completion.
\end{itemize}

%1) signal+noise model with smoothness structure;
%2) signal+noise model under wavelet basis (minimax results for Besov balls);
%3) signal+noise model with sparsity structure;
%4) signal+noise model with \emph{clustering} (or, \emph{multi-level sparsity}) structure; 
%5) noisy function on a large graph (Laplacian graph) with smoothness structure;
%6) density estimation with smoothness structure;
%7) biclustering model (also for stochastic block model and graphon classes);
%8) linear regression with sparsity structure;
%9) linear regression with  group sparsity;
%10) linear regression with group clustering;
%11) linear regression with mixture structure;
%12) dictionary learning; 
%13) aggregation in nonparametric regression;
%14) isotonic, unimodal and convex regressions; 
%15) mean matrix with submatrix sparsity;
%16) covariance matrix with banding structure; 
%17) covariance matrix with sparsity structure. 
%%$13)$ precision matrix with sparsity structure; 
%%13) matrix completion.

We also demonstrate how the local results imply many global results for corresponding scales. 
For example, local results for the signal+noise model with smoothness structure imply 
the minimax results for the Sobolev ellipsoids and hyperrectangles, analytic and tail classes;
local results for the signal+noise model with sparsity structure imply the minimax 
results for nearly black vector $\ell_0$, weak $\ell_q$-balls and Besov scales; etc.

For the above listed examples, almost all the results on uncertainty quantification problem are new, 
many known results on estimation and DDM (posterior) contraction are improved, as our results 
are local and hold in the refined formulation and distribution-free setting. 
For example, some obtained local rates lead to improved versions of some global ones from the literature. 
Some considered structures (like \emph{multi-level sparsity} or \emph{clustering} structure) 
are new and studied for the first time. The results on the weak structure recovery are new, note that 
this weak version of the structure recovery result holds without any extra condition. 
By assuming stronger conditions one can further strengthen weak structure recovery results 
to obtain stronger versions. 
%We only did this for the sparsity structure in the high-dimensional linear regression model. 

We emphasize that the scope of our approach extends further than these specific cases. 
In fact, the results are readily obtained for any particular models and structures 
that fall into the proposed general framework.

\subsection{Distribution-free setting}
Like in \emph{statistical learning} settings, we develop \emph{distribution-free} theory 
(or, \emph{robust} theory) meaning that 
we do not assume any specific form of the underlying measure $\mathbb{P}_{\theta,X}$ of $Y$.
%In other words, we do not impose any specific model: the distribution of $\xi$ is not known and 
%can depend on $\theta$ (often we suppress this dependence in notation), the coordinates 
%$\xi_i$'s of $\xi$ do not have to be iid, even not independent. 
Clearly, a non-void theory is impossible with no condition at all, so we only assume certain 
 condition on the stochastic part $\xi$ of the observed $Y$ in \eqref{model}.
In fact, it is not really a condition, but rather a description of our 
prior knowledge of how ``bad'' the projected ``noise'' $\mathrm{P}_I\xi$ is.
%We also pursue \emph{robust} inference in the sense that 
%The distribution of $\xi$ is assumed to satisfy only certain mild condition; 
For some fixed $\alpha>0$, define the quantity $d_I(\theta)=\log\big(\mathbb{E}_\theta 
\exp\big\{\alpha \|\mathrm{P}_{I}\xi\|^2\big\}\big)$, $\theta\in\Theta$, $I\in\mathcal{I}$, 
which is always defined (possibly as infinity). In a way, one can think of $d_I(\theta)$ as 
the \emph{statistical dimension} of structure $I$ at point $\theta$. Then the actual condition
is that $\sup_{\theta\in\Theta}d_I(\theta) \le d_I$ for some known $d_I$, $I \in\mathcal{I}$;
see Condition \eqref{cond_nonnormal} in Section \ref{Preliminaries}.

\subsection{Local approach}
Commonly in the literature, the quality of estimators and posteriors is measured by 
global asymptotic quantities, such as minimax estimation rate 
$r^2(\Theta_\beta)=\inf_{\tilde{\theta}}\sup_{\theta\in\Theta_\beta} 
\mathbb{E}_\theta \|\tilde{\theta}-\theta\|^2$ with respect to some scale 
$\{\Theta_\beta,\, \beta \in\mathcal{B}\}$, $\Theta_\beta \subseteq \Theta$, indexed 
by $\beta\in\mathcal{B}$, e.g., smoothness or sparsity.  For example, a typical asymptotic 
global minimax adaptive (i.e., knowledge of $\beta$ is not used) results for an estimator 
$\hat{\theta}$ and contraction rate for a posterior $\pi(\vartheta|Y)$ would be of the form: 
for sufficiently large $C_1,C_2$,
\begin{align}
\label{minimax}
\sup_{\theta\in\Theta_\beta} \mathrm{E}_\theta \| \hat{\theta}- \theta\|^2 \le 
C_1r^2(\Theta_\beta),% +C_2 \sigma^2 , 
\quad \sup_{\theta\in\Theta_\beta} \mathrm{E}_\theta \pi\big(\|\vartheta - \theta\|^2 \ge 
C_2 r^2(\Theta_\beta) |Y\big) \to 0, %\quad \Theta_\beta \subseteq \Theta, \quad \beta\in\mathcal{B},
\end{align}
as $N\to \infty$ or $\sigma\to 0$.

In this paper, we  pursue the \emph{local} approach for all the inference problems:   
instead of global $r^2(\Theta_\beta)$, the quality of the procedures is measured by 
the \emph{local} quantity, the \emph{local (oracle) rate} $r^2(\theta)=\min_{I\in\mathcal{I}} 
r^2(I,\theta)$, the best rate over the family of rates
$\{r^2(I,\theta),I\in\mathcal{I}\}$. Informally, $r^2(I, \theta)$ is the sum of %best trade-off between 
the approximation error by the projection structure $I$ and the complexity of that approximating 
projection structure. This means that, in a way, the local rate $r^2(\theta)$ expresses the main 
statistical paradigm of trading-off the model fit against the model complexity. 
The exact definitions are given in Section \ref{main_results}. 
%The proposed general framework provides a generic way 
%the local describe in general terms the main statistical paradigm of trading-off 
%of the model fit against the model complexity. ???

The local results are more powerful and flexible than global in that we do not need 
to consider any specific scale $\{\Theta_\beta,\, \beta \in\mathcal{B}\}$, local results 
essentially mean that our approach automatically extracts as much structure 
%(according to the family of structures $\mathcal{I}$, once $\mathcal{I}$ is chosen)
as there is in the underlying $\theta$. 
In a way, it is the \emph{local rate} $r^2(\theta)$ that measures 
the amount of structure in $\theta$: the smaller $r^2(\theta)$, the more structured $\theta$.
%It is this quantity that we use in characterizing the performance of our procedures.
%We will make this notion precise later.
%The quality of posterior is characterized by the the posterior contraction rate.
%In this paper we allow this to be a local quantity, i.e., depending on the true $\theta$, while 
%usually in the literature on Bayesian nonparametrics it is a global quantity 
%related to the minimax estimation rates over certain classes. 
Importantly, the local results imply a whole panorama of global minimax adaptive results 
\emph{over various scales at once}: it suffices to verify that 
$r^2(\theta) \le C  r^2(\Theta_\beta)$ for all $\theta \in \Theta_\beta$, $\beta \in \mathcal{B}$;
see examples in Part II.

\subsection{Refined formulation of the results}
Besides being local, our approach is also {\em refined}.  
For example, in this paper we derive local DDM-contraction and estimation 
results for the DDM $\hat{\pi}(\vartheta |Y)$ and the estimator $\hat{\theta}$, respectively,  
in the following refined \emph{non-asymptotic exponential probability bound} formulation: 
\begin{align}
\label{refined1}
& \sup_{\theta\in\Theta}\mathbb{E}_\theta\hat{\pi}(\|\vartheta -\theta\|^2 \ge M_0 r^2(\theta)+
M \sigma^2|Y) \le H_0e^{-m_0 M},\\ 
\label{refined2}
& \sup_{\theta\in\Theta}\mathbb{P}_\theta\big( \|\hat{\theta}-\theta\|^2\ge M_1 r^2(\theta)+M\sigma^2\big)
\le H_1 e^{-m_1 M},
\end{align}
for some fixed $M_0, H_0, m_0, M_1, H_1, m_1>0$ and arbitrary $M \ge 0$. %, uniformly in $\theta\in\Theta$.
Besides, we derive the local results on (weak) structure recovery and, most importantly, 
two versions of UQ (see the exact statements in Section \ref{main_results}), 
also as non-asymptotic exponential probability bounds. These refined formulations provide rather
sharp characterizations of the quality of the DDM $\hat{\pi}(\vartheta|Y)$ and the estimator $\hat{\theta}$
(finer than, e.g., traditional oracle estimation inequalities in expectation or asymptotic claims for posterior contraction, like \eqref{minimax}), allowing subtle analysis for various asymptotic regimes. 
These results, besides being ingredients for the uncertainty 
quantification problem, are of interest and importance on its own as they 
establish the local (oracle) optimality of our DDM and estimator 
in this refined formulation. As we have mentioned already, the local results imply in turn 
the corresponding global minimax adaptive results, also in the refined formulation.

\subsection{Uncertainty quantification (UQ)  and deceptiveness phenomenon in UQ}
%One of the prime goals in this paper is to construct confidence sets with optimal properties 
%for model \eqref{model}.
%For the usual norm $\| \cdot \|$ in $\Theta$, a random ball in $\Theta$ is 
%$B(\hat{\theta},\hat{r})=\{\theta\in\Theta: \|\hat{\theta}-\theta\| \le \hat{r}\}$,
%where the center $\hat{\theta}=\hat{\theta}(Y):\mathcal{Y} \mapsto \Theta$ and  radius 
%$\hat{r}=\hat{r}(Y): \mathcal{Y} \mapsto \mathbb{R}_+ =[0,+\infty]$ are
%measurable functions of the data $Y$. 
One of our main goals is to construct confidence sets for $\theta$ with optimal properties.
It is realized by many authors that this problem is more delicate than estimation, 
the main issue in UQ is that it suffers from the so called \emph{deceptiveness phenomenon}. 
This is explained in detail in \cite{Belitser&Nurushev:2015} for the sparsity structure, 
here we shortly outline this issue for general projection structures. 

First introduce the optimality framework for uncertainty quantification. 
Let $B(\theta_0,r)=\{\theta\in\Theta: \|\theta-\theta_0\| \le r\}$ denote 
the ball with center $\theta_0$ and radius $r\in\mathbb{R}_+ =[0,+\infty]$.
We measure the size of a confidence set by the smallest radius of 
a ball containing this set, hence it suffices to consider confidence balls. 
Let the center $\hat{\theta}=\hat{\theta}(Y):\mathcal{Y}\times\mathcal{X} \mapsto \Theta$ and radius 
$\hat{r}=\hat{r}(Y): \mathcal{Y}\times\mathcal{X} \mapsto \mathbb{R}_+$ be
measurable functions of the data.
The goal is to construct such a confidence ball $B(\hat{\theta},C\hat{r})$
that for any $\alpha_1,\alpha_2\in(0,1]$
and some functional $R(\theta)=R_{\sigma,N}(\theta)$,
$R:\Theta\mapsto\mathbb{R}_+$, there exist $C, c > 0$ such that
\begin{align}
\label{defconfball}
\sup_{\theta\in\Theta_0}\mathbb{P}_\theta\big(\theta\notin 
B(\hat{\theta},C\hat{r})\big)\le\alpha_1,\quad \sup_{\theta\in\Theta_1}
\mathbb{P}_\theta\big(\hat{r}\ge c R(\theta)\big)\le\alpha_2,
\end{align}
for some $\Theta_0,\Theta_1\subseteq\Theta$. We call the first expression 
in \eqref{defconfball} by \emph{coverage relation} and the second by \emph{size relation}.
The quantity $R(\theta)$, called \emph{radial rate}, describes the effective 
radius of the confidence ball. % $B(\hat{\theta},C\hat{r})$.
%from the perspective of the ``true'' $\mathrm{P}$-measure. 
%There are several optimality aspects in the framework \eqref{defconfball}: 
%the coverage, the radial rate and the subsets $\Theta_0$ and $\Theta_1$. 
%There may be various combinations of $r(\theta)$ and $\Theta_0$ for which 
%\eqref{defconfball} holds.
%Typically in the literature, global radial rates for specific smoothness (or sparsity) structures are studied. 
It is desirable to find the smallest  $R(\theta)$ and the biggest $\Theta_0,\Theta_1$, 
for which \eqref{defconfball} holds. These are contrary requirements, 
and we can trade them off against each other in different ways, leading to different optimality frameworks.

For example, the global (minimax adaptive) version of 
\eqref{defconfball} for a scale $\{\Theta_\beta,\, \beta \in\mathcal{B}\}$ 
(indexed by %a structural parameter 
$\beta\in\mathcal{B}$, e.g., sparsity or smoothness ) 
would be obtained by taking %$\Theta_0=\Theta_1=\Theta_\beta$ and 
the global radial rate $R(\theta)=r(\Theta_\beta)$,  $\theta\in\Theta_\beta$,
$\beta\in\mathcal{B}$, where $r(\Theta_\beta)$ is the 
minimax estimation rate over the sets $\Theta_\beta$. 
The traditional (global) optimality framework commonly pursued in the literature 
(in earlier papers on the topic) was to insist on $\Theta_0=\Theta_\beta$
%(in case $\mathcal{Y}=\Theta\subseteq\mathbb{R}^{\infty}$ we would like to have 
%$\Theta_0=\mathbb{R}^N$ for some big enough $N\in\mathbb{N}$) 
in \eqref{defconfball}.
This means that one considers only those confidence sets that satisfy 
the coverage property uniformly over $\Theta_\beta$, 
in some papers such sets are called ``honest''. 
Then one tries to find a ``honest'' confidence set with 
the fastest radial rate $R(\theta)$ and the biggest set $\Theta_1$, preferably 
$\Theta_1\supseteq\Theta_\beta$. 
%However, pursuing such an optimality framework often leads to
%discarding many good procedures and optimality of uninteresting ones.
%Many ``good'' confidence sets are  not ``honest'', therefore cannot be optimal and 
%effectively excluded from the consideration. 
However, according to the negative results of \cite{Li:1989, Baraud:2004,Nickl&vandeGeer:2013}, 
%and \cite{Cai&Low:2004} (formulated for the high-dimensional setting $\Theta=\mathbb{R}^N$),
insisting on $\Theta_0=\Theta_\beta$  in the coverage relation leads 
necessarily to the extra term $\sigma N^{1/4}$ in the expression for the effective radial rate 
$R(\theta)=r(\Theta_\beta)+\sigma N^{1/4}$. 
%whereas ideally we want the size relation to hold with the effective radial rate $r(\Theta_\beta)$. 
%This fact has also been observed by \cite{Nickl&vandeGeer:2013} for the
%case of linear regression with two sparsity classes.
The optimal rate $r(\Theta_\beta)$ is hence impossible to attain for those 
(sparsities or smoothness) $\beta\in\mathcal{B}$ for which $r(\Theta_\beta)\ll \sigma N^{1/4}$.
To summarize, in general the overall uniform coverage 
and optimal size properties cannot hold together and  it is necessary to 
sacrifice at least one of these. %preferably as little as possible.
%The fact that there are deceptive parameters $\Theta_{\rm dec}\not = \varnothing$ for which it is %impossible to attain either the coverage or the optimal size uniformly over the whole space 
This is the core of the so called \emph{deceptiveness phenomenon} in UQ,
%in the uncertainty quantification problem,
%From the perspective of the second approach, This deceptiveness phenomenon 
which is well understood only for sparsity and smoothness structures; see 
\cite{Cai&Low:2004, Robins&vanderVaart:2006, Nickl&vandeGeer:2013, Bull&Nickl:2013, 
Szabo&vanderVaart&vanZanten:2015, Belitser:2016,  vanderPas&Szabo&vanderVaart:2017, Belitser&Nurushev:2015} and further references therein.
%\cite{Szabo&vanderVaart&vanZanten:2015}, 
%\cite{vanderPas&Szabo&vanderVaart:2016}. 
%In \cite{Szabo&vanderVaart&vanZanten:2015} such parameters are called
%``inconvenient truths'' and give an implicit construction of  a $\theta'\in\Theta'$. 
%Examples of non-deceptive parameters are the set of \emph{self-similar} parameters 
%$\Theta_0=\Theta_{\rm ss}$  %introduced by  \cite{Picard&Tribouley:2000} and 
%studied by \cite{Bull:2012}, \cite{Bull&Nickl:2013}, %\cite{Nickl&Szabo:2014}, 
%\cite{Szabo&vanderVaart&vanZanten:2015}, and the set of \emph{polished tail parameters} 
%$\Theta_0=\Theta_{\rm pt}$ considered by \cite{Szabo&vanderVaart&vanZanten:2015}. 
%In all the above mentioned papers 

In this paper, we allow the radial rate $R(\theta)$ depend on the ``true'' $\theta$. 
The proposed UQ-framework is thus local, in contrast with global minimax 
frameworks commonly used in the literature on the UQ.
%our approach is local (and hence genuinely adaptive) as the radial rate 
%$R(\theta)$ is a function of the ``true'' parameter $\theta$.
Local results, delivering also (global) adaptive 
minimax results for  smoothness and sparsity structures, are obtained in 
\cite{Belitser:2016, Belitser&Nurushev:2015}.

But the deceptiveness phenomenon manifests itself also in the local setting.
Indeed, the results of \cite{Li:1989, Baraud:2004, Nickl&vandeGeer:2013} %and \cite{Cai&Low:2004} 
(formulated for the high-dimensional setting $\Theta=\mathbb{R}^N$) basically claim 
that the radial rate $R(\theta)$ cannot be of a faster order than $\sigma N^{1/4}$ for every 
$\theta\in\Theta=\mathbb{R}^N$ and is at least of the order $\sigma N^{1/2}$ for some $\theta$.  
This means that, in the situations when the targeted optimal local size $r(\theta)$
can be of a smaller order than 
$\sigma N^{1/4}$ for some $\theta$'s (which is typically the case, e.g., 
for smoothness and sparsity structures), this optimal size cannot be attained 
in the size relation uniformly over $\Theta$ and necessarily $R(\theta) \gg r(\theta)$ 
for some $\theta\in\Theta'$. Thus, insisting on $\Theta_0=\Theta$ implies that either 
the radial rate $R(\theta)$ or the set  $\Theta_1$ in the size relation has 
to be sacrificed: $\Theta_1=\Theta$ but $R(\theta) \gg r(\theta)$ 
for $\theta\in \Theta'$, or $R(\theta)=r(\theta)$ but $\Theta_1=\Theta\backslash\Theta'$.
Another, seemingly more reasonable approach to optimality developed recently 
in the literature is to sacrifice in the set $\Theta_0=\Theta\backslash\Theta_{\rm dec}$ 
by removing a preferably small portion of  ``deceptive parameters'' $\Theta_{\rm dec}$ from 
$\Theta$ in the coverage property, so that  the size property would then hold with 
$R(\theta)=r(\theta)$ uniformly over $\Theta_1=\Theta$.

In this paper, we construct a confidence ball by using the proposed DDM $\hat{\pi}(\vartheta|Y)$. 
Since we want the size of our confidence sets to be of the order of the oracle rate $r(\theta)$, 
this comes with the price that the coverage property can hold
uniformly only over some set of parameters satisfying the so called \emph{excessive bias restriction} 
(EBR) $\Theta_0=\Theta_{\rm eb}\subseteq\Theta$. 
The main result consists in establishing the optimality \eqref{defconfball} in the refined formulation, 
with $\Theta_0=\Theta_{\rm eb}$, $\Theta_1=\Theta$ and the local radial rate 
$R(\theta)=r(\theta)$. 
%%Recall that in order to have the targeted optimal radial rate $r(\theta)$ in 
%%the size relation of the optimality framework \eqref{defconfball}, the overall uniformity in
%%the coverage property of \eqref{defconfball} must be sacrificed: 
%%$\Theta_0=\Theta_{\rm eb}\subseteq\Theta$, where 
%%$\Theta_{\rm eb}$ is a set of parameters satisfying the so called \emph{excessive bias restriction} (EBR). 
%The set $\Theta_{\rm eb}$ is of the same type for all models and structures 
%coming from the general framework \eqref{model}:
% $\Theta_{\rm eb}=\Theta_{\rm eb}(t)= \{ \theta \in \Theta:
%b(\theta) \le t V(\theta)\}$, where $b(\theta)$ and $V(\theta)$ are the 
%approximation and complexity (or ``bias'' and ``variance'')
%parts of the squared oracle rate $r^2(\theta)=b(\theta)+V(\theta)$ for the corresponding 
%particular model and structure.
It turns out  that the EBR leads to a new \emph{EBR-scale} $\{\Theta_{\rm eb}(t), t \ge 0\}$, 
which gives a slicing of the entire space: $\Theta= \cup_{t\ge 0} \Theta_{\rm eb}(t)$.
This slicing 
is very suitable for UQ and provides a new perspective at the deceptiveness 
issue within the general abstract framework: basically, each parameter $\theta$ is 
deceptive (or non deceptive) to some extent.
%and what matters is the amount of its deceptiveness (or non-deceptiveness). 
It is the parameter $t$ that measures the deceptiveness in $\Theta_{\rm eb}(t)$ 
%The price for handling deceptive parameters is 
and affects %determines the effective amount of inflating of 
the size of the confidence ball needed to provide a guaranteed high coverage uniformly over 
$\Theta_{\rm eb}(t)$.

%Indeed, there is no payment in terms of removing deceptive parameters from the parameter space 
%$\mathcal{Y}$ in the coverage relation, and the size relation holds also uniformly over $\Theta_1=\mathcal{Y}$.

%{\em white noise model} and {\em density estimation} with {\em smoothness} structure, {\em linear regression} 
%and {\em dictionary learning} with {\em sparsity} structures, {\em biclustering} and {\em stochastic block models} 
%with {\em clustering} structure, {\em covariance matrix} estimation with sparsity and {\em banding} structures, 
%and many others. 

%Although the original motivation of the EBR condition was to remove the deceptive parameters, 
%it turned out to be a very useful notion in the context of uncertainty quantification. 
%In effect, the EBR condition gives rise  to a \emph{new biclustering EBR-scale} which gives 
%the slicing of the entire space that is very suitable for uncertainty quantification for the biclustering 
%model. An elaborate discussion on the EBR condition can be found in \cite{Belitser&Nurushev:2015}.
%An important consequence of our local approach is that a panorama 
%of adaptive (global) minimax results (for the both problems: estimation and confidence sets)  
%over \emph{all} biclustering and graphon classes \emph{covered} by $r(\theta)$ %(see Section \ref{implication}) 
%follow from our local results. In particular, our local results imply the same type of  
%adaptive minimax  estimation results over biclustering and graphon classes as in \cite{Gao&Lu&Zhou:2015}.
%and \cite{Klopp&Tsybakov&Verzelen:2017}.

In addition, we also treat the optimality framework with $\Theta_0=\Theta_1=\Theta$ in \eqref{defconfball}
by constructing an alternative confidence ball such that its radius is of the order $\sigma N^{1/4} + r(\theta)$.
According to the negative results of \cite{Li:1989, Baraud:2004} %and \cite{Cai&Low:2004} 
(formulated for the high-dimensional setting $\Theta=\mathbb{R}^N$),
insisting on the overall uniformity in the coverage and size relations leads necessarily to 
the extra term $\sigma N^{1/4}$ in the expression for the effective radial rate 
$R(\theta)=r(\theta)+\sigma N^{1/4}$.
%whereas ideally we want the size relation to hold with the effective radial rate $r(\theta)$.
This fact has also been observed by \cite{Nickl&vandeGeer:2013} for the
case of linear regression with two sparsity classes.
Interestingly, this alternative construction of confidence ball is more preferable 
for some particular models and structures, e.g., biclustering model (stochastic block model), 
dictionary learning; %matrix completion; 
see Part II.  The point is that, for those models and structures,
the extra term $\sigma N^{1/4}$ does not increase the order of the radial rate
because $\sigma N^{1/4} \le c r(\theta)$ for the ``majority'' of $\theta$'s, precisely,
for all $\theta\in\Theta\backslash\tilde{\Theta}$, with some ``thin'' set $\tilde{\Theta}$. 
The set $\tilde{\Theta}$ can be informally described as a set of ``highly structured'' parameters.
This means that, modulo the set $\tilde{\Theta}$ of ``highly structured'' parameters, 
there is no deceptiveness issue  for those cases. 
Speaking informally, these models and structures are already ``too difficult" for the term  
$\sigma N^{1/4}$ to spoil the radial rate.

\subsection{Organization of the rest of the paper}
The rest of the paper is organized as follows. In Section \ref{Preliminaries} we introduce the notation, 
the DDMs, make a link with the penalization method, and provide some conditions. 
Section \ref{main_results}, where we also introduce the EBR, contains the main 
results of the paper.  The proofs %of the lemmas and theorems 
are gathered in Section \ref{sec_proofs}. %\ref{proofs_lemmas} and \ref{proofs_theorems}. respectively. 
In Part II, we demonstrate how the main general results specify to a number of 
examples of model/structure in local and minimax settings.

\section{Preliminaries}
\label{Preliminaries}

In this section we introduce some notation, notions, conditions. %and a mixture normal prior. 
Then, we construct a data dependent measure (DDM) 
which can be associated with an empirical Bayes approach applied to the normal 
likelihood (recall that the true model does not have to be normal), 
%we derive an empirical Bayes posterior which we will use 
We will use this DDM in the construction of the estimator, the structure selector 
and the confidence ball.  

At first reading, one may want to skip this section and go ahead to Section \ref{main_results} 
(one will only need to consult some definitions from Section \ref{Preliminaries}) 
which contains the main results of the paper.

\subsection{Notation}
For $n\in\mathbb{N}$, denote $[n]=\{1,2,\ldots,n\}$ and $[n]_0=\{0\}\cup[n]$;
for a Hilbert space $\mathcal{Y}$, 
$\langle y,z \rangle$ % = \sum_{i=1}^N x_i y_i $ denote 
denotes the scalar product between $y,z \in \mathcal{Y}$, $\mathbb{R}_+=[0,+\infty]$,
$\mathbb{N}_0=\{0\}\cup \mathbb{N}$.
For an $(n_1\times n_2)$-matrix $x=(x_{ij})\in\mathbb{R}^{n_1\times n_2}$, 
we will  interchangeably use the same notation $x$ to denote the vector 
$x=\text{vec}\big[(x_{ij})\big]=(x_{11}, x_{12},\ldots , x_{n_1n_2})^T$. 
Conversely, for any $x\in \mathbb{R}^{n_1n_2}$ we can use 
matricized indexing $x=(x_{11},x_{12}, \ldots,x_{n_1n_2})^T$.
Most of the time the vector notation will be used, and %but we will not specify this as
it should be clear from the context which notation is meant in each expression.
For two nonnegative sequences $(a_l)$ and $(b_l)$, $a_l \lesssim b_l$ means $a_l \le cb_l$ 
for all $l$ (its range should be clear from the context)
with some absolute $c>0$, and $a_l \asymp b_l$ means that $a_l \lesssim b_l$ and $b_l \lesssim a_l$.
%$c^{-1} b_l \le a_l \le cb_l$ with some absolute $c>0$.

For a set $S$, $|S|$ denotes its cardinality. We will often denote matrices and operators 
by upright capital letters, the identity matrix is denoted by $\mathrm{I}$, $1_E=1\{E\}$ stands for 
the indicator function of the event $E$. As usual, $\mathrm{N}(\mu, \mathrm{\Sigma})$ is the 
multivariate normal distribution with mean $\mu$ and covariance matrix $\mathrm{\Sigma}$, 
its density at point $x$ is denoted by $\varphi(x,\mu,\mathrm{\Sigma})$. The dimensions of matrices 
and normal distributions should be clear from the context. Let $\mathrm{P}^\perp_{I}=
\mathrm{I}-\mathrm{P}_{I}$ be the projection operator onto the orthogonal complement 
$\mathbb{L}_I^\perp$ of $\mathbb{L}_I$. We use both notation $\mathrm{P}_{I}$ and 
$\mathrm{P}_{\mathbb{L}_{I}}$ ($\mathrm{P}^\perp_{I}$ and $\mathrm{P}_{\mathbb{L}_{I}}^\perp$) 
to denote the projection operator onto the linear subspace $\mathbb{L}_{I}$ 
(onto the orthogonal complement $\mathbb{L}_I^\perp$ of $\mathbb{L}_I$). 

The symbol $\triangleq$ will refer to equality by definition, for $a,b\in \mathbb{R}$, $(a\vee b)=\max\{a,b\}$, 
$(a\wedge b)=\min\{a,b\}$, $\lfloor a\rfloor=\max\{m\in\mathbb{Z}:m\le a\}$.
%$\lceil a\rceil=\min\{m\in\mathbb{Z}:m\ge a\}$. 
Throughout we assume the conventions: $|\varnothing|=0$, $\sum_{I\in\varnothing}a_I=0$ 
for any $a_I\in\mathbb{R}$ and $0\log(a/0)=0$ (hence $(a/0)^0=1$) for any $a>0$.

\subsection{Conditions} 
The structure $I\in \mathcal{I}$ on each $\theta\in\Theta$ is represented by the slicing 
$\Theta\subseteq\cup_{I \in \mathcal{I}} \mathbb{L}_I$, where 
$\mathcal{I}$ is a finite (or countable) family of possible structures and 
$\{\mathbb{L}_I, I\in \mathcal{I}\}$ is an associated  family of linear subspaces 
$\{\mathbb{L}_I, I\in \mathcal{I}\}$ of $\mathcal{Y}$.
If $\|\theta -\mathrm{P}_{I^*} \theta\|^2=\min_{I\in\mathcal{I}} \|\theta -\mathrm{P}_I \theta\|^2$ 
(for the true $\theta$) is close to zero, we say that $\theta$ is 
``approximately structured'' according 
to the family $\mathcal{I}$ . If 
$\|\theta -\mathrm{P}_{I^*} \theta\|^2=0$ (i.e., $\theta\in\mathbb{L}_{I^*}$), 
the true $\theta$ happens to be exactly structured and $I^*$ has the meaning of 
the ``true structure'' of $\theta$.
The structure $I^*$ of $\theta$ is always determined via the corresponding linear space 
$\mathbb{L}_{I^*} \ni \theta$.
%so by saying ``$I^*$ is the structure of $\theta$'' we mean that
%$\theta\in\mathbb{L}_{I^*}$ (or can be well approximated by $\mathrm{P}_{I^*}\theta$). 

\begin{remark}
\label{remark1}
There may be $\mathbb{L}_I=\mathbb{L}_{I'}$ for different $I,I'\in\mathcal{I}$, 
in other words, the family of structures $\mathcal{I}$ can have redundancy. 
Without loss of generality, we could assume that the family $\mathcal{I}$ is ``cleaned up" 
in the sense that each subspace $\mathbb{L}\in \mathcal{L}_{\mathcal{I}}$ is represented 
in $\mathcal{I}$ by only one (arbitrary) element  
$J=J(\mathbb{L})$ from the set $\{I\in\mathcal{I}:\; \mathbb{L}_I=\mathbb{L} \}$. 
Mathematically, this means that  the resulting ``cleaned up" family $\mathcal{I}$ of structures consists of 
equivalence classes on the original collection of all structures with the equivalence relation:
$I_1\sim I_2$ if and only if $\mathbb{L}_{I_1} =\mathbb{L}_{I_2}$, so that 
$|\mathcal{I}|=|\mathcal{L}_{\mathcal{I}}|$ in this case.

However, in general $|\mathcal{L}_{\mathcal{I}}|\le |\mathcal{I}|$ and 
this redundancy can be beneficial in some practical situations when searching 
(or optimizing in an inference procedure) over a possibly redundant family of structures $\mathcal{I}$  
can be described and realized easier than over the ``cleaned up" version of it. 
The only price for this redundancy is a bigger sum (because of more terms) in 
Condition \eqref{(A2)}, resulting in a bigger constant $C_\nu$ in Condition \eqref{(A2)} 
for a redundant $\mathcal{I}$. In many situations this is a mild price, as demonstrated 
for several particular models and structures in Part II. 
\end{remark}

Throughout the rest of the paper we impose the following  condition on $\xi$ from \eqref{model}.  \smallskip\\
{\sc Condition \eqref{cond_nonnormal}.}
For some nonnegative sequence $(d_I)_{I\in\mathcal{I}}$ and $\alpha>0$,  
\begin{align}
\tag{A1}
\label{cond_nonnormal}
\sup_{\theta\in\Theta}d_I(\theta) \le d_I, \quad \text{where}\quad
d_I(\theta)=\log\big(\mathbb{E}_\theta 
\exp\big\{\alpha \|\mathrm{P}_{I}\xi\|^2\big\}\big).
%\mathbb{E}_\theta \exp\big\{\alpha \|\mathrm{P}_{I}\xi\|^2\big\} \le \exp\{d_{s(I)}\},
%\quad I \in \mathcal{I}, \;\;\theta\in\Theta.
\end{align}  
Without loss of generality, assume $\alpha\in (0,1]$.

\begin{remark} 
\label{rem2}
%Notice that \eqref{initial_cond} actually always holds for any $\alpha>0$, if the $d_I$'s are 
%allowed to be infinite. This condition is useful only when all the $d_I$'s are finite. Further, 
It is desirable to have the bound \eqref{cond_nonnormal} in the tightest possible form, by determining 
the smallest sequence $(d_I)_{I\in \mathcal{I}}$ for which \eqref{cond_nonnormal} holds with 
a given $\alpha>0$. 
Notice that in general \eqref{cond_nonnormal} always holds for any $\alpha>0$, if the $d_I$'s are 
allowed to be infinite, but it is only useful when all the $d_I$'s are finite.
Thus, instead of \eqref{cond_nonnormal}, we could equivalently assume  
$\sup_{\theta\in\Theta}\mathbb{E}_\theta \exp\big\{\alpha \|\mathrm{P}_{I}\xi\|^2\big\} <\infty$
for all $ I\in\mathcal{I}$. Then the smallest $d_I$'s for which \eqref{cond_nonnormal} holds
are $d_I=\sup_{\theta\in\Theta}\log\big(\mathbb{E}_\theta 
\exp\big\{\alpha \|\mathrm{P}_{I}\xi\|^2\big\}\big)$, $I\in\mathcal{I}$.
The quantity $d_I$  can be seen as {\em statistical dimension} of the space $\mathbb{L}_I$,
reflecting in a way the \emph{complexity} of the structure $I$ (space $\mathbb{L}_I$): the bigger $d_I$, 
the more complex the structure $I$. 
If the distribution of $\xi$ does not depend on $\theta$, then there is no 
$\sup_{\theta\in\Theta}$ in the above definition of $d_I$.
Typically, in such cases $d_I \asymp\dim(\mathbb{L}_I)$. The bound \eqref{cond_nonnormal} holds, 
for example, for standard normal $\xi$ with $\alpha=0.43$ and $d_I=\dim(\mathbb{L}_I)$; see Remark \ref{rem_cond_A1} below.
\end{remark}

\begin{remark}
\label{rem_cond_A1}
Condition \eqref{cond_nonnormal} holds for high-dimensional independent normal  
$\xi_{i}$'s with $d_I = \dim(\mathbb{L}_{I})$, irrespective of the linear spaces 
$\mathbb{L}_I$, $I\in\mathcal{I}$. Indeed, if $\xi_{i} \overset{\rm ind}{\sim}\mathrm{N}(0,1)$, 
then $\|\mathrm{P}_{I} \xi\|^2\sim \chi_{\dim(\mathbb{L}_{I})}^2$, the chi-squared distribution with 
$d_I=\dim(\mathbb{L}_{I})$ degrees of freedom. Hence, for any $t < \frac{1}{2}$ we have that
$\mathbb{E}\exp\big\{ t\|\mathrm{P}_{I} \xi\|^2\big\}=(1-2t)^{-d_I/2}$. 
Since $(1-2t)^{-d_I/2} \le e^{d_I}$ for any $t \le (1-e^{-2})/2\approx 0.432$.
By taking  $t=0.4$, we derive $\mathbb{E} e^{0.4\|\mathrm{P}_{I} \xi\|^2}
\le e^{d_I}$. 
Hence, Condition \eqref{cond_nonnormal}  is fulfilled with $\alpha=0.4$ and $d_I=\dim(\mathbb{L}_{I})$. 
\end{remark}

\begin{remark}
Importantly, Condition \eqref{cond_nonnormal} allows quite some flexibility, 
which is crucial when treating concrete models and significantly broadens the 
range of models falling into our general framework; see the examples of models 
in Part II. The distribution of $\xi$ may depend on $\theta$,
%(note however that in many important models from Part II, 
%the distribution of $\xi$ does not depend on $\theta$). 
the coordinates $\xi_i$'s of $\xi$ do not have to be iid and may even be 
non-independent. For example, for the ``signal+noise'' model with the sparsity structure, 
it was shown in \cite{Belitser&Nurushev:2015} that Condition \eqref{cond_nonnormal} is fulfilled 
for the $\xi_i$'s generated according to an autoregressive model.
In this case, in \cite{Belitser&Nurushev:2015} we showed that, for independent $\xi_i$'s,  
Condition \eqref{cond_nonnormal} is equivalent to
the so called \emph{sub-gaussianity} condition on $\xi$ (see the definition in Remark \ref{rem15}).
For dependent $\xi_i$'s, the sub-gaussianity condition and Condition \eqref{cond_nonnormal} 
are close, but in general incomparable. 
For example, if $\xi_i = \xi_0$, $i\in[n]$, for some bounded random variable $\xi_0$ (say, uniform on $[-1,1]$), then, for the sparsity structure, Condition \eqref{cond_nonnormal}  
trivially holds whereas the sub-gaussianity condition is not fulfilled.
\end{remark}

Introduce a function $\rho: \mathcal{I} \mapsto \mathbb{R}_+$, called 
\emph{majorant of the structure complexity}. The idea of introducing this function is to measure 
the amount of structure complexity, which is expressed by the following condition on this function. 
%which we will need in the theorems.
\smallskip\\
{\sc Condition \eqref{(A2)}.} For some $\nu,C_\nu>0$, the function $\rho(I)$ satisfies
\begin{align}
\tag{A2}
\label{(A2)}
\sum\nolimits_{I\in\mathcal {I}} e^{-\nu \rho(I)}\le C_\nu,
\quad \text{and} \quad \rho(I)\ge d_I, \; I \in\mathcal{I},
\end{align}
where the sequence $(d_I)_{I\in\mathcal{I}}$ is from Condition \eqref{cond_nonnormal}.

\begin{remark}
Informally, \eqref{(A2)} means that the function $\rho$ must be large enough 
(that is why called majorant)  to match the total complexity of the family of structures $\mathcal{I}$. 
The total complexity of $\mathcal{I}$  
is a combination of two parts: the ``massiveness'' part, reflected by the cardinality $|\mathcal{I}|$ 
(or rather $|\mathcal{L}_{\mathcal{I}}|$, see Remark \ref{remark1}), 
and the ``effective dimension'' part, reflected by the sequence  of the statistical dimensions 
$(d_I)_{I\in\mathcal{I}}$. 
\end{remark}

\begin{remark}
In Condition \eqref{(A2)}, we can use an up-to-a-constant majorant  $\rho(I)\gtrsim d_I$ 
instead of $\rho(I)\ge d_I$
by %accommodating the multiplicative constant by the 
adjusting  $\alpha\in(0,1]$ in \eqref{cond_nonnormal}, but without loss of generality 
we stick to $\rho(I)\ge d_I$ for the sake of a clean mathematical exposition.
\end{remark}

\begin{remark}
\label{rem5}
For each particular model and structure, we need somehow to find a majorant $\rho(I)$ 
satisfying Condition \eqref{(A2)}, preferably in a \emph{constructive way}.  
Here we propose a way to construct a majorant $\rho(I)$. For that,
introduce a surjective function $s: \mathcal{I} \mapsto \mathcal{S}$, for some set 
$\mathcal{S}$, called the \emph{structural slicing mapping}. 
This function slices the family $\mathcal{I}$ in \emph{layers} 
$\mathcal{I}_s=\{I\in \mathcal{I}: s(I) =s\}$, $s\in \mathcal{S}$,
i.e., $\mathcal{I}=\cup_{s\in\mathcal{S}} I_s$,
$\mathcal{S}$ marks the collection of all layers $\mathcal{I}_s$.
Clearly, any partition %(union of disjoint subsets) 
of $\mathcal{I}$ can be realized by appropriate function $s(I)$, 
and the structure $I$ always belongs to the layer $\mathcal{I}_{s(I)}$.
The quantity $s(I)$ typically describes some features of the space $\mathbb{L}_{I}$,
for example, $s(I)$ can be the dimension (or some function of it) of $\mathbb{L}_I$. 
For a slicing mapping $s(I)$, denote 
\begin{align}
\label{D_s}
D_s=\max_{I\in\mathcal{I}_s} d_I, \quad \text{where} \quad \mathcal{I}_s=\{I\in \mathcal{I}: s(I) =s\}.
\end{align}
If $\sum_{s\in\mathcal{S}}  e^{-\nu D_s}\le C_\nu$ for $\nu\ge 1$, 
then Condition \eqref{(A2)} is fulfilled for any $\rho(I)\ge D_{s(I)}+|\mathcal{I}_{s(I)}|$ 
(one can think of $D_s+|\mathcal{I}_s|$ is the \emph{complexity of the layer} $\mathcal{I}_s$).
Indeed, for $\nu \ge 1$ we obtain  
\[
\sum_{I\in\mathcal {I}} e^{-\nu \rho(I)}
=\sum_{s\in\mathcal{S}}\sum_{I\in\mathcal {I}_s} e^{-\nu \rho(I)}
\le \sum_{s\in\mathcal{S}}  e^{-\nu D_s -(\nu-1)\log|\mathcal{I}_s|}
\le \sum_{s\in\mathcal{S}}  e^{-\nu D_s}\le C_\nu.
\]

Later on, the majorant $\rho(I)$ will enter the local (oracle) rate. In order to derive stronger results, 
it is therefore desirable to use the smallest possible majorant $\rho(I)$ that satisfies Condition \eqref{(A2)}.
In this light, the best majorant $\rho(I)\ge D_{s(I)}+|\mathcal{I}_{s(I)}|$  is the layer complexity itself 
$\rho(I)=D_{s(I)}+\log|\mathcal{I}_{s(I)}|$ for the ``cleaned up'' family of structures 
$\mathcal{I}$ (see Remark \ref{remark1}).
On the other hand, any majorant $\rho(I)$ that satisfies Condition \eqref{(A2)} will do the job.
%see  \eqref{I_max_ddm} and \eqref{oracle} below. 
The reason to allow an arbitrary majorant $\rho(I)$ is that $D_s$ and $|\mathcal{I}_{s}|$ 
may be difficult to compute, whereas some closed form 
upper bounds can be derived. Of course, this comes at the price of a bigger resulting 
local rate because this majorant will then enter the local rate.

It is desirable to use a slicing $s: \mathcal{I} \mapsto \mathcal{S}$
that is parsimonious in the sense that the maximum 
$D_s = \max_{I\in\mathcal{I}_s} d_I$ degenerates, i.e., $d_I =d_J$ for all $I,J \in\mathcal{I}_s$,
so $D_{s(I)}=d_I$. In other words, $d_I=h(s(I))$, $I\in\mathcal{I}$, for some function 
$h:\mathcal{S} \mapsto\mathbb{R}_+$.
%Since typically $d_I \asymp \dim(\mathbb{L}_I)$, in many examples from the follow up paper
%\cite{Belitser&Nurushev:2019b} %Section \ref{applications} we will determine a slicing $s$ 
%such that $\dim(\mathbb{L}_I)=h(s(I))$, $I\in\mathcal{I}$, for some function $h$. 
In this case, we can choose $\rho(I)\ge d_I+ \log|\mathcal{I}_{s(I)}|$. 
Since we always use parsimonious slicings, it is this choice of majorant 
$\rho(I)$ that we used in almost all the examples from Part II.
\end{remark}

The last condition is needed for the UQ results. \smallskip\\
{\sc Condition \eqref{(A3)}.}
%&\text{{\sc Condition \eqref{(A3)}.} 
For any $I_0,I_1\in\mathcal{I}$ there exists  $I'=  I'(I_0,I_1)\in\mathcal{I}$ such that
\begin{align}
\tag{A3}
\label{(A3)}
(\mathbb{L}_{I_0}\cup\mathbb{L}_{I_1})\subseteq \mathbb{L}_{I'} \;\;\text{and}\;\; 
\rho(I')\le \rho(I_0)+\rho(I_1). 
\end{align}

\begin{remark}
\label{cond_A3'}
Typically, Condition \eqref{(A3)} is fulfilled with $I'=  I'(I_0,I_1)\in\mathcal{I}$ such that 
$\mathbb{L}_{I'}=\mathbb{L}_{I_0}+\mathbb{L}_{I_1}$. This is the case for almost all examples 
in Part II.

Let us formulate a slightly stronger version of Condition \eqref{(A3)} called Condition (A3'):
for any $I_0,I_1\in\mathcal{I}$ there exist $I'=I'(I_0,I_1)\in\mathcal{I}$ 
and $I''=I''(I_0,I_1)\in\mathcal{I}$ such that  $\mathbb{L}_{I'}=\mathbb{L}_{I_0}+\mathbb{L}_{I_1}$, 
$\mathbb{L}_{I''}=\mathbb{L}_{I_0}\cap\mathbb{L}_{I_1}$, $\rho(I')\le \rho(I_0)+\rho(I_1)$
and $\rho(I'')\le \rho(I_0)$.
\end{remark}

The constants $\alpha \in (0,1]$ and $\nu>0$ from Conditions \eqref{cond_nonnormal} 
and \eqref{(A2)}, respectively, will be fixed throughout and we omit  the dependence on 
these constants in all further notation. 

In the proof of Theorem \ref{th1} below, we will need a bound for 
$\big[\mathbb{E}_\theta\| \mathrm{P}_I\xi \|^4\big]^{1/2}$, for each $I\in\mathcal{I}$.  
Condition \eqref{cond_nonnormal} and $\rho(I) \ge d_I$ from \eqref{(A2)} ensure such a bound. 
Indeed, since $x^2 \le e^{2x}$ for all $x\ge 0$, by the H\"older inequality 
and \eqref{cond_nonnormal}, we obtain for any $t\in (0,1/2]$ and $I\in\mathcal{I}$,
\begin{align}
\label{moment_bound}
\mathbb{E}_\theta\| \mathrm{P}_I\xi  \|^4
\le \frac{\mathbb{E}_\theta e^{2t\alpha\| \mathrm{P}_I\xi\|^2}}{(t\alpha)^2}
\le \frac{\big(\mathbb{E}_\theta e^{\alpha\|\mathrm{P}_I\xi\|^2 }\big)^{2t}}{(t\alpha)^2}
\le \frac{e^{2t d_I}}{(t \alpha)^2}
\le \frac{e^{2t \rho(I)}}{(t\alpha)^2}. %\quad \text{for any} \;\; t\in (0,1/2].
\end{align}
In case $\xi_i \overset{\rm ind}{\sim} \mathrm{N}(0,1)$, Condition \eqref{cond_nonnormal}
is fulfilled with $d_I=\dim(\mathbb{L}_I)$ and $\alpha=0.4$, see Remark \ref{rem_cond_A1}. 
As $\|\mathrm{P}_{I}\xi\|^2 \sim \chi^2_{d_I}$, instead of \eqref{moment_bound}, 
a better bound can be used in this case: $\big[\mathbb{E}\|\mathrm{P}_{I}\xi\|^4\big]^{1/2}
=\big(d^2_I+2d_I\big)^{1/2}\le d_I+1 \le \rho(I)+1$.

\subsection{Construction of data dependent measure (DDM)}
\label{subsec_ddm}
The following construction of the so called \emph{data dependent measures} (DDMs) 
is motivated by the Bayesian approach. On the other hand, the main and essential difference 
of DDM from posterior is that DDM does not necessarily result from a Bayesian analysis. 
%The DDM notion  generalizes the notion of posterior as all posteriors are DDMs, 
%but not all DDMs are posteriors. 
DDM is an arbitrary random measure dependent on the data, 
whereas posterior distribution has special structure resulting from prior and model assumptions.
In the appendix to Part I (Section \ref{sec_empirical_Bayes}), 
we present the detailed construction of one important example of DDM 
as the result of an \emph{empirical Bayesian approach}, based on certain normal prior and normal 
model (although the true model is not known, so it certainly does not have to be normal) as building 
blocks. This DDM, explicitly constructed in Section \ref{sec_empirical_Bayes},
as empirical Bayes posterior, covers all the situations 
where the statistical dimensions $(d_I)_{I\in\mathcal{I}}$ from Condition \eqref{cond_nonnormal} are such that 
$d_I \gtrsim \dim(\mathbb{L}_I)$, $I\in\mathcal{I}$.

In this section we provide concise formal construction of the DDMs used for inference. 
Define first a DDM on $\mathcal{I}$:
\begin{align}
\label{ddm_P(I|Y)}
\tilde{\pi}(I|Y)&=\tilde{\pi}_{I}
=\frac{\lambda_{I}\exp\{-\tfrac{1}{2\sigma^2}\|(Y-\mathrm{P}_{I}Y\|^2\}}
{\sum_{J \in \mathcal{I}}\lambda_{J}\exp\{-\tfrac{1}{2\sigma^2}\|(Y-\mathrm{P}_{J}Y\|^2\}},
\quad 
\lambda_{I} =e^{- \varkappa \rho(I)} , \quad  I\in\mathcal{I},
\end{align} 
where  $\rho(I)$ satisfies Condition \eqref{(A2)}, the parameter $\varkappa$ satisfies the bound
\begin{align}
\label{cond_technical}
\varkappa>\bar{\varkappa}\triangleq (32\nu+10+\alpha)/(4\alpha),
\end{align}  
$\alpha$ and $\nu$ are from Conditions \eqref{cond_nonnormal} and \eqref{(A2)}, respectively.

Let $\mathbb{P}_Z$ be a probability measure such that if $\xi\sim \mathbb{P}_Z$,
Condition \eqref{cond_nonnormal} is fulfilled. Under Condition \eqref{cond_nonnormal}, 
such a measure must exist. Note that $\mathbb{P}_Z$ is an \emph{arbitrary} 
probability measure for which Condition \eqref{cond_nonnormal} is fulfilled, it does not have 
to coincide with the true probability law of $\xi$.
Now, introduce random vector $Z\sim \mathbb{P}_Z$. Notice that $\mathbb{E}_Z Z =0$.
For a deterministic vector $y\in\mathcal{Y}$, we can compute 
the probabilities $\mathrm{P}_Z \big(\mathrm{P}_{I}y+\sigma \mathrm{P}_{I}Z \in B \big)$
for any measurable set $B\subseteq \mathcal{Y}$. Finally, by 
substituting  the observed $Y$ instead of $y$ in 
the distribution of $\mathrm{P}_{I}y+\sigma \mathrm{P}_{I}Z$
we obtain the following family of DDMs
\begin{align}
\label{ddm_I}
\tilde{\pi}_{I}(\vartheta \in B|Y)=\tilde{\pi}(\vartheta\in B| Y,I)=
\mathbb{P}_Z\big(\mathrm{P}_Iy+\sigma\mathrm{P}_I Z \in B\big) \big|_{y=Y}\; , 
\quad I \in \mathcal{I}.
\end{align}
 
Using the two families of DDMs $\{\tilde{\pi}(\vartheta|Y,I), I\in\mathcal{I}\}$, and 
$\{\tilde{\pi}(I|Y), I\in\mathcal{I}\}$ we create the following new DDM as a mixture 
(a la Bayesian approach):
\begin{align}
\label{ddm_mixture}
\tilde{\pi}(\vartheta|Y)=\tilde{\pi}_\varkappa(\vartheta|Y)=
\sum_{I \in \mathcal{I}}
\tilde{\pi}_I(\vartheta|Y)\tilde{\pi}(I|Y),
\end{align}
called \emph{model averaging DDM} (MA-DDM).
Let $\tilde{\mathbb{E}}$ and $\tilde{\mathbb{E}}_{I}$ be the expectations 
with respect to the DDMs $\tilde{\pi}(\vartheta|Y)$ and $\tilde{\pi}_{I}(\vartheta|Y)$, 
respectively. Introduce the \emph{MA-DDM mean estimator}
\begin{align}
\label{estimator}
\tilde{\theta}&=\tilde{\mathbb{E}}(\vartheta|Y)=\sum_{I \in \mathcal{I}}
\tilde{\mathbb{E}}_{I}(\vartheta|Y) \tilde{\pi}(I|Y)=
\sum_{I \in \mathcal{I}}
(\mathrm{P}_{I}Y) \tilde{\pi}(I|Y).
\end{align}

Consider yet alternative DDM. First derive a DD structure selector $\hat{I}$ 
by maximizing the DDM $\tilde{\pi}(I|Y)$ over $ I\in\mathcal{I}$.
This boils down to
\begin{align}
\label{I_max_ddm}
 \hat{I}&=\argmax_{I\in\mathcal{I}}{\tilde{\pi}}(I|Y)=\argmin_{I\in\mathcal{I}} 
 \big\{\|Y-\mathrm{P}_IY\|^2+\sigma^2\text{pen}(I)\big\},
\end{align}
which is essentially the \emph{penalization method} with the penalty 
$\text{pen}(I)=2\varkappa \rho(I)$. 
Plugging in $\hat{I}$ into DDM $\tilde{\pi}_{I}(\vartheta|Y)$  instead of $I$
gives the corresponding \emph{model selection DDM} (MS-DDM),
and the \emph{MS-DDM mean estimator}  for $\theta$:
\begin{align}
\label{ms_ddm}
\check{\pi} (\vartheta|Y)=\tilde{\pi}_{ \hat{I}} (\vartheta|Y), \quad 
\check{\theta}= \check{\mathbb{E}}(\vartheta|Y)= \mathrm{P}_{\hat{I}} Y,
\end{align}
where $\check{\mathbb{E}}$ denotes the expectation with respect to the 
DDM $\check{\pi}(\vartheta|Y)$. Notice that, like \eqref{ddm_mixture},
$\check{\pi} (\vartheta|Y)$ defined by \eqref{ms_ddm} 
can also be seen formally as mixture 
\begin{align}
\label{ms_ddm_I}
\check{\pi} (\vartheta|Y)=\tilde{\pi}_{ \hat{I}} (\vartheta|Y)=
\sum_{I \in\mathcal{I}} \tilde{\pi}_I(\vartheta|Y)\check{\pi}(I|Y), \quad \check{\pi}(I|Y)=1\{I=\hat{I}\},
\end{align}
where the mixing distribution $\check{\pi}(I|Y)=1\{I=\hat{I}\}$, the 
DDM for $I$, is degenerate at $\hat{I}$.

\begin{remark}
\label{rem1} 
Notice that we have constructed not just one DDM but a whole family of DDMs, as we can take any distribution 
$\mathbb{P}_Z$ (satisfying Condition \eqref{cond_nonnormal}) in \eqref{ddm_I}.

In a way, the DDM $\tilde{\pi}(\vartheta|Y)$ defined by \eqref{ddm_mixture} 
is of a more Bayesian flavor than the DDM 
$\check{\pi}(\vartheta|Y)$ defined by \eqref{ms_ddm} which is more of a ``model selection'' 
flavor, as it is based on the penalization method. Note however that, while the penalization method gives only an estimator, we also provide a DDM.  
\end{remark}

From now on, by $\hat{\pi}(\vartheta|Y)$ we denote
either $\tilde{\pi}(\vartheta|Y)$ defined by \eqref{ddm_mixture} 
or $\check{\pi}(\vartheta|Y)$ defined by \eqref{ms_ddm};
by $\hat{\pi}(I|Y)$ we denote either $\tilde{\pi}(I|Y)$ defined by \eqref{ddm_P(I|Y)}
or $\check{\pi}(I|Y)$ defined by \eqref{ms_ddm_I};
and $\hat{\theta}$ will stand either for $\tilde{\theta}$  defined by 
\eqref{estimator} or for $\check{\theta}$ defined by \eqref{ms_ddm}.
In case $\hat{\pi}(I|Y) =\check{\pi}(I|Y) =1\{I=\hat{I}\}$, the meaning of $\hat{\pi}(I \in \mathcal{G}|Y)$ 
for any $\mathcal{G}\subseteq\mathcal{I}$ is as follows: $\hat{\pi}(I \in \mathcal{G}|Y)
=\check{\pi}(I \in \mathcal{G}|Y)=\mathrm{1}\{ \hat{I}\in\mathcal{G}\}$
so that $\mathbb{E}_{\theta}\check{\pi}(I\in \mathcal{G}| Y)=\mathbb{P}_\theta( \hat{I}\in\mathcal{G})$.

\section{Main results}
\label{main_results}
In this section we present the main results of the paper.

\subsection{Oracle rate}
\label{sec_oracle_rate}
For $I\in\mathcal{I}$, consider the projection estimator $\mathrm{P}_{I} Y$ for estimating $\theta$. 
%Exactly in the same way as \eqref{E_max} follows from \eqref{initial_cond}, condition \eqref{cond_nonnormal} 
%entails that $\mathbb{E}_{\theta}\max_{J\in\mathcal{I}_{s(I)}}\|\mathrm{P}_{J}\xi\|^2\le C\rho(I)$ 
%for some $C>0$, where the complexity majorant $\rho(I)$ satisfies Condition \eqref{(A2)}. 
By Condition \eqref{cond_nonnormal} and Jensen's inequality, 
we obtain the following upper bound for the estimator 
$\mathrm{P}_{I} Y$: for some $C>0$,
\begin{align*}
\mathbb{E}_{\theta}\|\theta-\mathrm{P}_{I} Y\|^2 &=\|\theta-\mathrm{P}_{I}\theta\|^2+\sigma^2
\mathbb{E}_{\theta}\|\mathrm{P}_I\xi\|^2\le\|\theta-\mathrm{P}_{I}\theta\|^2+C\sigma^2  d_I.
%\le\|\theta-\mathrm{P}_{I}\theta\|^2+C\sigma^2 \rho(I).
%\label{risk}
\end{align*}
Ideally, we would like to mimic the local rate for the best (oracle) choice of the projection 
structure $\min_{I\in\mathcal{I}}\big(\|\theta-\mathrm{P}_{I}\theta\|^2+\sigma^2  d_I\big)$, 
uniformly in $\theta\in\Theta$.
However, as is shown for some particular models, this is impossible unless, instead of just 
$d_I$, we use a majorant $\rho(I)\ge d_I$  that satisfies Condition \eqref{(A2)}. In particular, 
according to Remark \ref{rem5}, we can use $\rho(I)\ge D_{s(I)}+\log|\mathcal{I}_{s(I)}|$, 
where $D_s$ and $\mathcal{I}_s$ are defined by \eqref{D_s}
for some appropriate slicing mapping $s(I)$. 
The extra layer complexity term $\log|\mathcal{I}_{s(I)}|$ reflects the  ``price'' for not knowing 
the structure. This motivates the following definition. Introduce the  family of
local rates
\begin{align*}
r^2 (I,\theta)=\|\theta-\mathrm{P}_I\theta\|^2+ \sigma^2\rho(I),\quad I\in\mathcal{I},
\end{align*}
for some $\rho(I)$ satisfying Condition \eqref{(A2)}. For each $\theta$ there exists the best 
structure $I_o= I_o(\theta)=I_o(\theta,\sigma^2)$ (if not unique, take any minimizer) 
corresponding to the fastest local rate
\begin{align}
\label{oracle}
r^2(\theta)&=\min_{ I\in\mathcal{I}}  r^2(I,\theta)=r^2(I_o,\theta)
=\|\theta-\mathrm{P}_{I_o}\theta\|^2+ \sigma^2\rho(I_o),
\end{align}
representing the optimal trade-off between the approximation term 
$\|\theta-\mathrm{P}_{I_o}\theta\|^2$  
and the complexity term $\rho(I_o)$ satisfying Condition \eqref{(A2)}.
We call $I_o$ by \emph{oracle structure} (or just \emph{oracle})
and the quantity $r^2(\theta)$ by \emph{oracle rate}. %and $s_o=s(I_o)$ by \emph{oracle layer}.
%Note that $s_o=s_o(\theta)$ depends on $\theta$ since $I_o$ depends on $\theta$.
%Sometimes we will call the pair $(s_o,I_o)=(s_o(\theta),I_o(\theta))$ 
%by \emph{oracle} (or \emph{oracle structure}).

\begin{remark}
%\label{rem_Nsigma^2}
Often we will have $D_{s(I)}=\dim(\mathbb{L}_I)=d_I$ and $\rho(I)=d_I+ \log|\mathcal{I}_{s(I)}|$
for some appropriate slicing mapping $s(I)$. 
This is the case in many particular models and structures that we consider 
in Part II. %Section \ref{applications}
If $I^*\in\mathcal{I}$ is the true structure, i.e., $\theta\in\mathbb{L}_{I^*}$, 
$s^*=s(I^*)$ and $\mathcal{I}_{s^*}=\{I^*\}$, then, by the oracle definition \eqref{oracle}
and the facts that $\|\theta-\mathrm{P}_{I^*}\theta\|^2=0$ and 
$\mathbb{L}_{I^*}\subseteq \mathbb{R}^N$, we have
\begin{align}
\label{dim_N}
r^2(\theta) \le r^2(I^*,\theta) =\sigma^2 \rho(I^*)=\sigma^2 d_{I^*} 
=\sigma^2\dim(\mathbb{L}_{I^*}) \le %\sigma^2\dim(\mathbb{R}^N) =
N\sigma^2.
\end{align}
If such a true structure does not exist, we can assume without loss of generality that there is an
 $\bar{I}\in\mathcal{I}$ such that $\mathbb{L}_{\bar{I}} =\mathbb{R}^N$,
$\bar{s}=s(\bar{I})$ and  $\mathcal{I}_{\bar{s}}=\{\bar{I}\}$. This would lead again to the bound \eqref{dim_N}:
$r^2(\theta) \le r^2(\bar{I},\theta)=\sigma^2\dim(\mathbb{L}_{\bar{I}})= N\sigma^2$.
This is of course not surprising as the oracle performance should not be worse than that of
the simplistic procedure $\hat{\theta} =Y$.
\end{remark}

\begin{remark}
\label{family_cover}
Suppose we have two different family of structures $\mathcal{I}_1$ and $\mathcal{I}_2$,
with corresponding (different) families of linear spaces $\mathbb{L}_I$ and 
(different) majorants.
We say that the family $\mathcal{I}_1$ \emph{covers} the family $\mathcal{I}_2$ 
if for any $I\in\mathcal{I}_2$ there exists $I'=I'(I) \in \mathcal{I}_1$ such that 
$r^2(I',\theta) \le r^2(I,\theta)$ for all $\theta\in\Theta$ (up-to-a-constant relation will do as well). 
If $\mathcal{I}_1$ covers $\mathcal{I}_2$, there is no point in considering the family 
$\mathcal{I}_2$, one should use the family $\mathcal{I}_1$. 
%One could say structures $\mathcal{I}_2$ are embedded in structures $\mathcal{I}_1$. 
The family $\mathcal{I}_1$ and the family $\mathcal{I}_2$ covered by $\mathcal{I}_1$ could be 
of very different natures. But sometimes $\mathcal{I}_1$ can be a subfamily of 
$\mathcal{I}_2$. 
This happens when a chunk of structures in $\mathcal{I}_2$ can be dominated 
by just one structure. Then we can remove those structures without any harm, obtaining 
a new adjusted family $\mathcal{I}_1$. The complexity term 
in the majorant gets adjusted for some $I\in\mathcal{I}_1$, leading to an \emph{elbow effect} 
in the rate and improving the resulting oracle rate. 
We will see how this \emph{elbow effect} is exhibited  for several cases of model/structure from Part II. %Section \ref{applications}.
%This is a useful simple observation, and we are going to exploit 
%this in several examples of Section \ref{applications}, when 
%adjusting the original family of structures in such a way that the new adjusted family 
%exhibits an \emph{elbow effect} in the complexity 
%term of the rate, improving the resulting oracle rate. 
\end{remark}

\subsection{Estimation and DDM contraction results with oracle rate}

Recall the quantities: the DDM $\hat{\pi}(\vartheta|Y)$, 
which is either MA-DDM $\tilde{\pi}(\vartheta|Y)$ defined by \eqref{ddm_mixture} 
or MS-DDM $\check{\pi}(\vartheta|Y)$ defined by \eqref{ms_ddm}; 
the DDM mean $\hat{\theta}$, which is either $\tilde{\theta}$ 
defined by \eqref{estimator} or $\check{\theta}$ defined by \eqref{ms_ddm};
and the oracle rate $r(\theta)$ defined by \eqref{oracle}.
The following theorem establishes that the DDM $\hat{\pi}(\vartheta|Y)$ 
contracts (from the frequentist $\mathbb{P}_\theta$-perspective) to $\theta$ with 
the oracle rate $r(\theta)$, and the DDM mean $\hat{\theta}$
converges to $\theta$ with the oracle rate $r(\theta)$, uniformly over the entire parameter space.

\begin{theorem} %[Oracle posterior contraction and estimation]
\label{th1} 
Let Conditions \eqref{cond_nonnormal} and \eqref{(A2)} be fulfilled.
%the oracle rate $r(\theta)$ be defined by \eqref{oracle}. 
Then there exist constants $M_0$,$M_1$,$H_0$,$H_1$,$m_0$,$m_1>0$ %(depending on $\varkappa$) 
such that for any $\theta\in\Theta$ and any $M\ge 0$,
\begin{align}
%\tag{i}
\label{th1_i}
\mathbb{E}_{\theta}\hat{\pi}\big( \|\vartheta-\theta\|^2
\ge M_0r^2(\theta) +M\sigma^2 | Y\big)  &\le H_0 e^{-m_0 M},\\
%\tag{ii}
\label{th1_ii}
\mathbb{P}_\theta\big( \|\hat{\theta}-\theta\|^2\ge M_1 r^2(\theta)+M\sigma^2\big)
&\le H_1 e^{-m_1 M}.
\end{align}
\end{theorem}
The constants in the theorem depend only on $\alpha$ and some also on $\varkappa$, 
the exact expressions can be found in the proof.

\begin{remark}
\label{remark}
Notice that already claim \eqref{th1_i} of Theorem \ref{th1} contains an oracle bound 
for the estimator $\hat{\theta}$. Indeed, by Jensen's inequality, we get the oracle inequality 
in expectation:
\begin{align}
\mathbb{E}_{\theta}\|\hat{\theta}-\theta\|^2 &\le
\mathbb{E}_{\theta} \hat{\mathbb{E}}(\|\vartheta-\theta\|^2|Y)
\le M_0r^2(\theta)+H_0 \int_0^{+\infty} \!\!\!\! e^{-m_0u/\sigma^2} du \notag\\&
=M_0r^2(\theta)+\tfrac{H_0\sigma^2}{m_0}. 
\label{cor_est}
\end{align}
Similarly we can show that also \eqref{th1_ii} implies \eqref{cor_est}. 
This means that claim \eqref{th1_ii} is actually stronger than \eqref{cor_est} and 
therefore requires a separate proof.
\end{remark}

\begin{remark}
%\item[\rm{(ii)}]
%Notice that for a local rate to be mimicable  it should include the penalty term $\sigma^2$ 
%which is not included in the oracle rate $r^2(\theta)$ if $\rho(I_o) =0$. But the quantity 
%$M_0r^2(\theta) +M\sigma^2$ from Theorem \ref{th1} does include this term.
The non-asymptotic exponential probability bounds in the both claims 
of the theorem provide a very refined characterization of the quality of the DDM 
$\hat{\pi}(\vartheta|Y)$ and estimator $\hat{\theta}$, finer than, e.g., the traditional 
oracle inequalities in expectation like \eqref{cor_est} 
(since \eqref{cor_est} follows from \eqref{th1_ii}, see Remark \ref{remark}).
This refined formulation allows for subtle analysis in various asymptotic regimes ($N \to \infty$, 
$\sigma \to 0$, or their combination) %and various global minimax risks (see Subsection \ref{implication}) 
as we can let $M$ depend in any way on $N$, $\sigma$, or both. 
\end{remark}

Now we give several technical definitions which we will need in the claims.
For the constants $\alpha$ from Condition \eqref{cond_nonnormal} and 
$\varkappa$ from \eqref{ddm_P(I|Y)}, define 
\begin{align}
\label{def_bar_tau}
\bar{\tau}=\bar{\tau}(\varkappa,\alpha)\triangleq
3(1+\varkappa\alpha)/\alpha.
\end{align}
Next, for some $\delta\in(0,1)$,  fix some $\tau_0 =\tau_0(\delta)$ such that 
$\tau_0> \tfrac{1+\delta}{1-\delta} \bar{\tau}$, where
$\bar{\tau}$ is defined by \eqref{def_bar_tau}. For example, take $\delta = 0.1$ and 
$\tau_0=\tfrac{11}{9}\bar{\tau}+0.1$. For this $\tau_0$ and any $\theta\in\Theta$, define 
\begin{align}
\label{I_*}
 I_*= I_*(\theta)= I_*(\theta,\delta) \triangleq  I_o^{\tau_0}(\theta)=I_o(\theta,\tau_0\sigma^2),
\end{align}
where $ I_o(\theta,\sigma^2)$ is defined by \eqref{oracle}.
We call the quantity  $I_o^\tau=I_o^\tau(\theta)=I_o(\theta, \tau\sigma^2)$, for $\tau\ge 0$, 
by \emph{$\tau$-oracle},  which is just the oracle defined by \eqref{oracle} 
with  $\sigma^2$ substituted by $\tau\sigma^2$. 
%The oracle itself is the $\tau$-oracle with $\tau=1$: $I_o(\theta)=I_o^1(\theta)$. 
Notice that $\rho(I_o^{\tau_1})\ge \rho(I_o^{\tau_2})$ for $\tau_1\le\tau_2$. 
%For $\tau \downarrow 0$, $r^2(\theta, I_o^\tau)\downarrow 0$  and
%the ``limiting'' $\tau$-oracle recovers the true structure %(or the \emph{sparsity} of $\theta$) 
%$I^*=I^*(\theta)$ in the sense that  
%$I_o^\tau \uparrow I^*$ as $\tau \downarrow 0$.
%Informally, the smaller the $\tau$, the ``closer'' the $\tau$-oracle structure 
%to the true structure, we say the $\tau$-oracle is ``stronger'' for smaller $\tau$. 
%This is not surprising since the $\tau$-oracle can do a better job for the smaller error variance 
%$\tau\sigma^2$. 
%%Informally, since the $\tau$-oracle is defined by substituting $\tau\sigma^2$ instead of 
%%$\sigma^2$ in the oracle rate, one can think of the $\tau$-oracle with $\tau\in[0,1)$ 
%%as if $Y$ was observed with a ``magnifying glass'' %``strength'' of 
%%since the error variance is reduced by the factor $\tau$ so that the $\tau$-oracle 
%%can do a better job. The case $\tau>1$ corresponds to a bigger observation variance resulting 
%%in a worse performance of the $\tau$-oracle.
%%Throughout we assume the smallest ``mimicable'' \emph{oracle $\tau$-rate} 
%%is $\tau\sigma^2$ in case $\rho(s(I_o^\tau))=0$.
%We gained certain flexibility in the rate by introducing parameter 
%$\tau\ge 0$ which re-weights the approximation and complexity parts in the rate.  
%However, 
All $\tau$-oracle rates are related to the oracle rate by the trivial relations:
$r^2(\theta)\le r^2(I_o^\tau, \theta)\le \tau r^2(\theta)$ for $\tau\ge 1$,
and $r^2(I_o^\tau, \theta)\le r^2(\theta)\le\tau^{-1}r^2(I_o^\tau, \theta)$ for $0<\tau< 1$.

When proving Theorem \ref{th1}, as byproduct we also obtain the following theorem about the frequentist 
behavior of the DDM $\hat{\pi}(I|Y)$. 
\begin{theorem} %[Posterior oracle control] 
\label{th2} 
Let Conditions \eqref{cond_nonnormal} and \eqref{(A2)} %and \eqref{cond_technical} 
be fulfilled, $\nu, C_\nu$ be from Condition \eqref{(A2)}. The following relations hold
for any $\theta\in\Theta$ and $M\ge 0$.
\begin{itemize}	
\item[(i)] 
Let $c_1, c_2, c_3$ be the constants defined in Lemma \ref{main_lemma}. Then 
\begin{align*}
\mathbb{E}_{\theta}\hat{\pi}(  I \in \mathcal{I}: \, r^2(I,\theta)\ge c_3 r^2(\theta) + M \sigma^2\big| Y) 
\le  C_\nu e^{-c_2 M}.
\end{align*}

\item[(ii)] 
Let $\varkappa\ge\alpha^{-1}\nu$ (implied by \eqref{cond_technical})
and Condition \eqref{(A3)} be fulfilled. %Fix   any  $I_1\in\mathcal{I}$. 
Then there exists 
%$ m'_0>0$ and  $I'=  I'(I,I_1)\in\mathcal{I}$  such that  
%\begin{align*}
%&\mathbb{E}_{\theta}\hat{\pi}\big( I \in \mathcal{I}: \theta^{T}[\mathrm{P}_{I'}-\mathrm{P}_{I}]\theta\ge \sigma^2(\bar{\tau}\rho(I')+M)\big|Y\big) 
%\le  C_\nu e^{-m'_0M},
%\end{align*}
%where $\bar{\tau}$ is defined by \eqref{def_bar_tau}.
%In particular,  there exists 
$m'_1> 0$ such that 
\begin{align}
\label{th2_ii}
\mathbb{E}_{\theta}\hat{\pi}\big( I\in\mathcal{I}: \rho(I)\le 
\delta \rho(I_*)-M \big|Y\big)  \le 
C_\nu e^{-m'_1M},
\end{align}
where $I_*=I_*(\theta,\delta)$ is defined by \eqref{I_*}.

\item[(iii)] 
%$\mathbb{L}_{I_0}=\mathbb{L}_I\cap \mathbb{L}_{I_o}$.
Let $\varkappa\ge \frac{2\nu+2\alpha+3}{2\alpha}$ (implied by \eqref{cond_technical})
and Condition {\rm (A3')} be fulfilled (given in Remark \ref{cond_A3'}).
Then there exists $M'_0>0$  such that 
\begin{align*}
\mathbb{E}_{\theta}\hat{\pi}\big(I \in \mathcal{I}: \rho(I)\ge  
M'_0 \rho(I_o)+M \big| Y\big)\le C_\nu e^{-M/2}.
\end{align*}
\end{itemize}
\end{theorem}
We can interpret the above theorem as \emph{structure recovery}, but in a somewhat weak sense. 
Namely, Theorem \ref{th2} says basically that the DDM $\tilde{\pi}(I|Y)$ 
and the structure selector $\hat{I}$ ``live'' in the set of structures that are,
%and the structure selector $\hat{I}$ for $I$, saying basically that $\tilde{\pi}(I|Y)$ and 
%$\hat{I}$  ``live'' on a set that is, 
in a sense, almost as good as the oracle structure $I_o$. 
%resembling the oracle structure $I_o$, in the sense that the rates and complexities for 
%the structures from this set are in a proximity of the oracle rate and complexity, respectively. 
Recall that in general the oracle structure is not the same as the true structure. 

Notice that, apart from Conditions \eqref{cond_nonnormal} and \eqref{(A2)}, the above weak 
structure recovery results 
do not require any extra conditions on $\theta$. This is in contrast with the ``strong'' 
structure recovery. 
For example, for the sparsity pattern recovery in linear regression model with sparsity structure, 
one needs the so called ``beta-min'' condition.

\subsection{Confidence ball under EBR}
\label{subsec_conf_set}

Theorem \ref{th1} establishes strong local optimal properties 
of the DDM $\hat{\pi}(\vartheta| Y)$ and the DDM mean $\hat{\theta}$, 
%This is a necessary ingredient  for solving the uncertainty quantification problem 
%\eqref{defconfball}, but not the only one.
but this is not enough to solve the UQ problem yet.
As a first candidate for confidence ball, let us construct a credible ball by using 
the DDM $\hat{\pi}(\vartheta|Y)=\check{\pi}(\vartheta|Y)$ defined by \eqref{ms_ddm}. 
According to its definition and Condition \eqref{cond_nonnormal}, 
the DDM $\check{\pi}(\vartheta|Y)$ is concentrated around its mean $\check{\theta}$ 
with the rate $\sigma d_{\hat{I}}^{1/2}$.
%As $\check{\pi}(\vartheta| Y)=\mathrm{N}\big(\check{\theta},\tfrac{\kappa}{\kappa+1}\sigma^2
%\mathrm{P}_{\hat{I}}\big)$, denoting by $\chi^2_{k,\alpha}$ the $(1-\alpha)$-quantile of 
%$\chi^2_k$-distribution, we have 
%\[
%\check{\pi}\big(\|\vartheta-\check{\theta}\|^2 \le\sigma^2 
%\chi^2_{\text{dim}(\mathbb{L}_{\hat{I}}),\alpha} |Y\big)\ge
%\check{\pi}\big(\|\vartheta-\check{\theta}\|^2\le
%\tfrac{\kappa}{\kappa+1}\sigma^2\chi^2_{\text{dim}(\mathbb{L}_{\hat{I}}),\alpha}
%|Y\big)=1-\alpha.
%\]
%Since $\chi^2_{\text{dim}(\mathbb{L}_{\hat{I}}),\alpha}$ is bounded by a constant multiple of 
%$\text{dim}(\mathbb{L}_{\hat{I}})$, for simplicity the latter can replace the former 
%to obtain a credible ball. 
Then $B(\check{\theta},M\sigma d_{\hat{I}}^{1/2})$ is a DDM credible ball for $\theta$, 
which can be guaranteed to have a given level of DDM mass 
by choosing a sufficiently large $M$. 
However, $B(\check{\theta},M\sigma d_{\hat{I}}^{1/2})$
cannot have a guaranteed coverage, since otherwise
in some particular models (cf.\ \cite{Belitser&Nurushev:2015}) this would mean that 
the estimator $\check{\theta}$ would converge to $\theta$ uniformly in $\theta\in\Theta$
at the smaller oracle rate with $\tilde{\rho}(I_o)=d_{I_o}$ instead of 
$\rho(I_o)=D_{s(I_o)}+\log |\mathcal{I}_{s(I_o)}|$. 
But $\log |\mathcal{I}_{s(I_o)}|$ can be the dominating term in $\rho(I_o)$, e.g., 
for sparsity  structures (see \cite{Belitser&Nurushev:2015} and the corresponding 
cases of model/structure from Part II). 
This would contradict the lower bounds from the literature.  
Basically, the DDM $\check{\pi}(\vartheta|Y)$ is well concentrated 
(in fact, `` too concentrated''), but not around the truth, rather around
its mean $\check{\theta}$ which can be further away from the truth than the DDM contraction rate. 
To guarantee coverage, the radius of confidence balls must be at least of the order
$\rho(I_o)$. The oracle structure $I_o$ 
is not known, but we have the structure selector 
$\hat{I}$ defined by \eqref{I_max_ddm}.
%which is in a way close to the oracle structure, according to Theorem \ref{th2}. 

The above heuristics suggests to use $\rho(\hat{I})$ as a proxy for $\rho(I_o)$.
According to Theorem \ref{th2}, $\hat{I}$ lives in the ``complexity shell'' 
$\delta\sigma^2 \rho(I_*) -M \sigma^2 \le 
\sigma^2 \rho(\hat{I})\le M'_0\rho(I_o)+M \sigma^2$ with a large probability.
So, if we want the size of confidence ball to be not of a bigger order than oracle rate, 
it seems reasonable to use the following data dependent (quadratic) radius
\begin{align}
\label{check_r}
\hat{r}^2=\hat{r}^2(Y)= \sigma^2+\sigma^2\rho(\hat{I}).
\end{align}
%So we use the EBMS method for constructing the radius of the confidence ball. 
%We can also construct a confidence ball by using the posterior $\tilde{\pi}(\vartheta|Y)$ 
%defined by \eqref{ddm_mixture} with the same resulting properties, 
%but with more involved mathematical derivations.
%
%which gives a partial control of the oracle rate. But this is all we have without further conditions, 
%in general it is impossible the have a good estimate of the bias part 
%$\|\theta - \mathrm{P}_{I_o}\theta\|^2$ of the oracle rate, further analysis 
%must rely on that partial control.
We will show that the size property holds for the radial rate equal to the oracle rate, 
uniformly over $\theta\in\Theta$. 
But then there is an inevitable problem with coverage: the coverage property does not 
hold uniformly. Indeed, the complexity shell can be too wide if  
$\sigma^2 \rho(I_*)\ll r^2(\theta)$. 
If this happens (for deceptive $\theta$'s), then the coverage 
property of a ball with radius of order $\hat{r}$ cannot be guaranteed because its radius 
can be of a smaller order than the oracle rate $r^2(\theta)$. % is too small. 
This problem will not occur for those $\theta$'s (called non-deceptive) 
for which the approximation term of the oracle rate is within a multiple of its complexity term. 
%Indeed, then $\sigma^2 \rho(I_*)$ must be at least some multiple of 
%$r^2(I_*,\theta)$ which is in turn bigger than the oracle rate $r^2(\theta)$ by the definition 
%\eqref{oracle} of the oracle. This means that $\sigma^2 \rho(I_*)$ will be at least 
%of the oracle rate so that (by (i) of Theorem \ref{th2})  $\hat{r}$ will be at least of the oracle rate order,
%resulting in a good coverage of the confidence ball $B(\check{\theta}, M_2\hat{r}+M\sigma^2)$  for
%some $M_ 2$ and sufficiently large $M$. 
This discussion motivates introducing the following condition. 
%which provides control over the quantity $b(\theta)$.  
%Let us fix the parameter $\varrho\in(0,1]$, for example, take $\varrho=1/2$.
\smallskip

\noindent
{\sc Condition EBR.}
We say that $\theta\in\Theta$ satisfies the 
\emph{excessive bias restriction} (EBR) condition with structural parameter $t\ge 0$ if  
$\theta \in \Theta_{\rm eb}(t)$, where the corresponding set (called the \emph{EBR class}) is 
\begin{align}
\label{cond_ebr}
\Theta_{\rm eb}(t)=\Theta_{\rm eb}(t,\tau_0)
=\big\{\theta\in\Theta : 
\|\theta-\mathrm{P}_{I_*}\theta\|^2\le t \sigma^2\big(1+\rho(I_*)\big)\big\},
\end{align}
where the $\tau_0$-oracle structure $I_*=I_o(\theta,\tau_0\sigma^2)$ is defined 
by \eqref{I_*}. The condition EBR essentially requires that the approximation term of the 
$\tau_0$-oracle rate $r^2(I_*,\theta)$ 
is dominated by a multiple of its complexity term (additional  $\sigma^2$ is needed to handle 
the case $\rho(I_*)=0$). Clearly, $\Theta_{\rm eb}(t_1)\subseteq\Theta_{\rm eb}(t_2)$ for $t_1 \le t_2$.

Now we use  the center $\hat{\theta}$ and the radius $\hat{r}$ to
construct a confidence ball for $\theta$. 
The following theorem %, which is the main result in the paper, 
describes the coverage and size properties of the confidence ball based on
$\hat{\theta}$ and $\hat{r}$. 
%$B\big(\hat{\theta},[(b_\tau(\theta)+1)M_2\hat{r}^2+(b_\tau(\theta)+2)M\sigma^2]^{1/2}\big)$.

\begin{theorem}
\label{th3}
Let Conditions \eqref{cond_nonnormal}, \eqref{(A2)} and \eqref{(A3)} %, \eqref{cond_technical} 
be fulfilled, $\Theta_{\rm eb}(t)$ be defined by \eqref{cond_ebr}. 
Then there exist constants $M_2, M_3, H_2, H_3, m_2, m_3>0$ such that for any 
$t,M\ge 0$, with $\hat{R}^2_M=\hat{R}^2_M(M_2)=(t+1)M_2\hat{r}^2+(t+2)M\sigma^2$,
\begin{align*}
\sup_{\theta\in\Theta_{\rm eb}(t)} \mathbb{P}_\theta
\big(\theta\notin B(\hat{\theta},\hat{R}_M)\big) 
&\le H_2 e^{-m_2M}, \\ %\quad 
\mathbb{P}_\theta\big(\hat{r}^2\ge M_3 r^2(\theta)+(M+1)\sigma^2\big) 
&\le H_3 e^{-m_3 M}.
\end{align*}
The size (second) relation  holds uniformly in $\theta\in\Theta$ without Condition \eqref{(A3)}. 
%Condition (A3') implies Condition \eqref{(A3)} (needed for the coverage relation), and under (A3') a stronger 
%version of the size relation holds: with $\sigma^2\rho(I_o)$ instead of $r^2(\theta)$.

Moreover, if, instead of Condition \eqref{(A3)}, stronger Condition {\rm(A3')} is fulfilled, 
then a stronger version of the size relation holds:
$\mathbb{P}_\theta\big(\hat{r}^2\ge M'_0\rho(I_o)\sigma^2+(M+1)\sigma^2\big)
\le  C_\nu e^{-M/2}$, 
where the constants $M'_0$ and $C_\nu$ are from Theorem \ref{th2}.

\end{theorem}

\begin{remark}
\label{rem_Theta_ebr}
Recall that $I_*$ from \eqref{cond_ebr} is actually the $\tau_0$-oracle. It may be
desirable to impose an EBR condition in terms of the ``standard'' oracle $I_o$ rather 
than the $\tau_0$-oracle. 
By rewriting  the original model \eqref{model} as 
$Y\tau_0^{-1/2}=\theta\tau_0^{-1/2} + \sigma\tau_0^{-1/2} \xi$,
it is not difficult to see that we can construct a confidence ball 
with the radius $\sqrt{\tau_0} \hat{R}_M$ satisfying the coverage property as above, but now uniformly 
over $\Theta_{\rm eb}(t,1)$. %, which is a version of the EBR condition in terms of the oracle $I_o$.
%$\Theta_{\rm eb}(t/\tau_0,1)\subseteq\Theta_{\rm eb}(t,\tau_0)$, so that 
%$\Theta_{\rm eb}(t/\tau_0,1)$ is a version ofEBR condition in terms of the oracle $I_o$. 
%Let $\Theta_{\rm eb}(t)$ denote either $\Theta_{\rm eb}(t,\tau_0)$ or $\Theta_{\rm eb}(t/\tau_0,1)$. 
\end{remark}

\begin{remark}
\label{rem4}
When proving the coverage relation of Theorem \ref{th3}, 
we actually established the following uniform local assertion: % concerning the coverage. 
there exist constants $M_2,m''_1,H_2, m_2>0$  such that for any $\theta\in\Theta$ 
and any $M\ge 0$,
\begin{align}
&\mathbb{P}_\theta
\big(\theta\notin B(\hat{\theta},[(b(\theta)+1)M_2\hat{r}^2+(b(\theta)+2)M\sigma^2]^{1/2}\big)\notag\\
& \qquad \qquad \le H_1 e^{-m_1M}+C_\nu e^{-m''_1M}
\le H_2 e^{-m_2M},
\label{th3_coverage}
\end{align}
where the constants $H_1, m_1$ are defined in Theorem \ref{th1}, $C_\nu$ is from Condition 
\eqref{(A2)}, and the quantity $b(\theta)$ (called \emph{excessive bias ratio}) is defined by
\begin{align}
\label{def_b}
b(\theta)=b(\theta,\tau_0)=\frac{\|\theta-\mathrm{P}_{I_*}\theta\|^2}
{\sigma^2+\sigma^2\rho(I_*)}.
\end{align}
Although the newly formulated coverage relation \eqref{th3_coverage} is now uniform 
over the entire space $\theta\in\Theta$, the main (and unavoidable) problem is its dependence on
$b(\theta)$. That is why we introduced the EBR condition which essentially 
provides control over the quantity $b(\theta)$: indeed, $\Theta_{\rm eb}(t)=\{\theta\in\Theta: b(\theta)\le t\}$. 
\end{remark}

\begin{remark}
%\label{rem5}
Smaller the constant $\tau_0$ (involved in the definition of the EBR condition) is, 
the less restrictive the EBR condition is, the limiting case 
$\tau_0\downarrow 0$ corresponds basically to no condition. 
We treat a general situation, with only Condition \eqref{cond_nonnormal} assumed 
for $\xi$, so that we have a lower bound for $\tau_0$ in terms of $\alpha$ which is possibly too conservative 
for each specific distribution of $\xi$. 
However, even for any specific distribution of $\xi$, the value of the constant $\tau_0>0$ in the EBR condition 
is always bounded away from zero (further from zero for ``bad'' $\xi$'s).
%Besides, knowing or not knowing that 
%best constant would not help in constructing the optimal confidence set. After all, it is never 
%known to observer whether the underlying $\theta$ satisfies to the EBR condition with this 
%threshold $\tau_0$ or not. 
\end{remark}
\begin{remark}
The EBR leads to the new  \emph{EBR-scale} $\{\Theta_{\rm eb}(t), t\ge 0\}$ 
which gives a slicing of the entire space $\Theta= \cup_{t\ge 0} \Theta_{\rm eb}(t)$.
This slicing is very suitable for uncertainty quantification and provides a new perspective at 
the deceptiveness issue (discussed in the Introduction): basically, each parameter 
$\theta$ is  deceptive (or non deceptive) to some extent. It is the parameter $t$ that measures 
the deceptiveness in $\Theta_{\rm eb}(t)$ and affects the size of the confidence ball needed to 
provide a guaranteed high coverage uniformly over $\Theta_{\rm eb}(t)$.
\end{remark}

\subsection{Confidence ball of $N^{1/4}$-radius without EBR}
\label{sec_without_EBR}

%According to Theorem ref{th3}, the coverage relation holds for our constructed 
%confidence ball only under the EBR condition. In general, 
%$\hat{r}^2$ is of the order $\rho(I_o)$, possibly underestimating the oracle rate $r^2(\theta)$. 
%The difference is the bias part of the oracle rate which may in general  be of a bigger order 
%than the variance part, leading to a bad coverage. 

Suppose we want to construct a confidence ball of a full coverage uniformly 
over the whole space $\Theta$. Recall however that for ``signal+noise'' models, in view 
of the negative results of \cite{Li:1989, Cai&Low:2004, Baraud:2004, 
Nickl&vandeGeer:2013} mentioned in the Introduction, no data dependent ball can have 
uniform coverage and 
adaptive size simultaneously. When insisting on the uniform coverage, one must have 
an additional term of the order $\sigma N^{1/4}$ in the radial rate.
%can at best adapt to layers only if $r(\theta)\ge C\sigma N^{1/4}$  and  
%This indeed happens as we shall see below. 
Let us give a heuristics behind this.
An idea is to mimic the quantity $\|\theta-\hat{\theta}\|^2$ by $\hat{R}^2=\|Y-\hat{\theta}\|^2$.
Clearly, there is a lot of bias in $\hat{R}^2$, the biggest part of which is due to the term 
$\sigma^2\|\xi\|^2$ contained in $\hat{R}$. To de-bias for that part, we need to subtract its expectation 
$\sigma^2\mathbb{E}\|\xi\|^2$.
However, even the de-biased version of $\hat{R}^2$ can only be controlled up to a margin 
of the order $\sigma^2 \sqrt{N}$. That is why a term of the order $\sigma N^{1/4}$ is necessary in the radius of the confidence ball to provide coverage uniformly over the whole space $\Theta$.

%the idea is to estimate the bias part $B(I_o,\theta)$ by its empirical counterpart 
%$\hat{B}=\sum_{i\in\hat{I}^c} (X_i^2 - \sigma_i^2)=\|X-\check{\theta}\|^2 - \sum_{i\in\hat{I}^c}\sigma_i^2$ where $\sigma_i^2 =\mathrm{Var}(X_i)$, arriving at a seemingly good candidate 
%for a data dependent radius $\sum_{i\in\hat{I}^c} (X_i^2 - \sigma_i^2)+\hat{r}^2$. 
%For simplicity, throughout this paragraph we assume that $\sigma_i^2 =\sigma^2 
%\mathrm{Var}(\xi_i)=\sigma^2$, the general case (under the assumption 
%$\mathrm{Var}(\xi_i)\le C_\xi$ for some $C_\xi>0$) can also be handled.

To handle some technical issues, we impose the following condition.\smallskip\\
{\sc Condition \eqref{cond_A4}.} Besides $Y$ given by \eqref{model}, we also observe 
$Y'= \theta+ \sigma \xi'$ independent of $Y$, where 
the random vector $\xi'$ satisfies the following relations:
\begin{equation}
\label{cond_A4}
\tag{A4}
\begin{aligned}
&\mathbb{P}\big(|\langle v, \xi' \rangle|\ge \sqrt{M}\big)\le\psi_1(M)\;\;
\forall\, v\in\mathbb{R}^N: \,  
\|v\| =1;\notag\\
&\mathbb{P}\big(\big|\|\xi'\|^2- V(Y',Y) %\mathbb{E} \|\xi'\|^2
\big| \ge M \sqrt{N}\big) \le \psi_2(M),\quad
\text{for some statistic} \;\; V(Y',Y).
\end{aligned}
\end{equation} 
Here $\psi_1(M), \psi_2(M)$ are some %positive monotonically 
decreasing functions such that $\psi_1(M)\downarrow 0$ and $\psi_2(M)\downarrow 0$ 
as $M\uparrow \infty$.

\begin{remark}
\label{rem15}
Typically,  $\mathbb{E} \xi'_{i} =0$, $\mathrm{Var}(\xi'_{i})=1$, $i\in [N]$, then 
$V(Y',Y) =N$.
Condition \eqref{cond_A4} is satisfied for independent normals 
$\xi_i \overset{\rm ind}{\sim} \mathrm{N}(0,1)$ even if we do not have the sample $Y'$ at our 
disposal. Indeed, in this case we can ``duplicate'' the observations by randomization
at the cost of doubling the variance in the following manner: 
create samples $Y' = Y+\sigma Z$ and $Y''=Y-\sigma Z$, for a
$Z=(Z_1,\ldots,Z_N)$ (independent of $Y$) such that $Z_i\overset{\rm ind}{\sim} \mathrm{N}(0,1)$. 
Relations \eqref{cond_A4} are then fulfilled with exponential  functions
$\psi_l(M)=C_le^{-c_l M}$ for some $C_l,c_l>0$, $l=1,2$  and $V(Y',Y)=N$.
\end{remark}
\begin{remark}
The vector $Z$ is called 
\emph{sub-gaussian} with parameter $\rho >0$ if  
$\mathbb{P}( |\langle v, Z \rangle|> t) \le e^{-\rho t^2}$
for all $t\ge 0$ and $v\in\mathbb{R}^{N}$ such that $\|v\| =1$.
If the sub-gaussianity condition %\eqref{subgaussian} 
is fulfilled for $\xi'$, then the first relation in \eqref{cond_A4} holds with
$\psi_1(M) = e^{-\rho M}$.
%The  random vector $\xi$ is called \emph{sub-gaussian} with parameter $\rho >0$ if 
%(cf.\ \cite{Ravisplata:2012}). 
By Chebyshev's inequality, we see that the second relation in \eqref{cond_A4} 
is fulfilled with function $\psi_2(M) = cM^{-2}$ and $V(Y',Y)=N$
for any zero mean independent  $\xi'_{i}$'s with $\mathbb{E} \xi_i'^2=1$ and 
$\mathbb{E}[\xi_{i}']^4\le C$. 
\end{remark}
\begin{remark}
\label{rem_bin_model}
For the \emph{biclustering model} (in particular, the \emph{stochastic block model}), 
given in Section \ref{sec_biclustering}, the case of binomial 
observations ($Y_{i}\overset{\rm ind}{\sim}\text{Bernoulli}(\theta_{i})$ and 
$Y'_{i}\overset{\rm ind}{\sim}\text{Bernoulli}(\theta_{i})$) is important in relation to 
\emph{network modeling}. By using Hoeffding's inequality, we see that Condition \eqref{cond_A4} 
holds with exponential  functions $\psi_l(M)=C_le^{-c_l M}$ for some $C_l,c_l>0$, $l=1,2$ 
and $V(Y',Y)=\sum_{i\in[n]} Y'_i-\sum_{i\in[n]} Y'_i Y_i$, (the function $V$ is by \cite{Ghosal:2019}).
\end{remark}

Coming back to the problem of constructing a confidence ball of full coverage 
uniformly over $\Theta$, let $\hat{\theta}$ and $\hat{I}$ be based on the sample $Y$ and 
defined as before. We propose to mimic $\|\theta-\hat{\theta}\|^2$ by the
de-biased quantity $\|Y'-\hat{\theta}\|^2- \sigma^2V(Y',Y)$ 
plus additional $\sigma^2 \sqrt{N}$-order term to control its oscillations, 
leading us to the following data dependent radius
\begin{align}
\label{radius2}
\tilde{R}^2_M=\big(\|Y'-\hat{\theta}\|^2\!-\! \sigma^2V(Y',Y)\!+\!2\sigma^2G_M\sqrt{N}
%\textstyle{\sum_{i\in\hat{I}^c}}(X_i'^2 - \sigma_i^2)+\sigma^2G_M\sqrt{n}
%-\textstyle{\sum_{i\in\hat{I}}} \sigma_i^2 
\big)_+,
\end{align}
where  $ G_M=\sqrt{M(M+M_1)}$, $x_+ =x\vee 0$ and the constant $M_1$ is from Theorem \ref{th1}.
The next theorem establishes the coverage and size properties of the confidence ball 
$B(\hat{\theta},\tilde{R}_M)$. 
%, where by convention $B(\hat{\theta},\tilde{R}_M)=\varnothing$ if $\tilde{R}_M<0$.
\begin{theorem}
\label{th4}
Let Conditions \eqref{cond_nonnormal}, \eqref{(A2)} and \eqref{cond_A4} be fulfilled 
and $\tilde{R}^2_M$ be defined by \eqref{radius2}. Then  for any $M\ge 0$
\begin{align*}
\sup_{\theta\in\Theta} \mathbb{P}_\theta
\big(\theta\notin B(\hat{\theta},\tilde{R}_M)\big) 
&\le \psi_1(M/4)+\psi_2(M)+H_1e^{-m_1M},\\
\sup_{\theta\in\Theta}  \mathbb{P}_\theta\big(\tilde{R}_M^2\ge g_M(\theta,N)\big)
&\le\psi_1(M/4)+\psi_2(M)+2H_1e^{-m_1M},
\end{align*}
where $g_M(\theta,N)=M_1r^2(\theta)+M\sigma^2+4\sigma^2 G_M \sqrt{N}$ 
and the constants $H_1, m_1, M_1$ are defined in Theorem \ref{th1}.
\end{theorem}

By taking large enough $M$, %(which is essentially a multiple by the term $N^{1/2}$) 
we can ensure the coverage and size relations uniformly over the entire space $\Theta$. 
Thus, the results of Theorem \ref{th4} are to be interpreted as the coverage and size relations 
in the optimality framework \eqref{defconfball} with $\Theta_0=\Theta_1=\Theta$ 
and the effective radial rate $R(\theta)=\sqrt{g_M(\theta,N)}\asymp r(\theta)+\sigma N^{1/4}$ 
(for now disregarding the constants and the inflating factor $M$ as we consider only the order 
of the radial rate). 
Since both sets $\Theta_0=\Theta_1=\Theta=\mathbb{R}^N$ 
are the biggest possible, the deceptiveness phenomenon manifests itself only in the effective 
radial rate $R(\theta)$,
%the only ingredient of the optimality frameworks \eqref{defconfball} that 
%can be affected by the deceptiveness phenomenon is the effective radial rate $R(\theta)$. 
%Indeed, we would ideally want the effective radial rate to be of the order $r(\theta)$,
%but there is an extra term $\sigma N^{1/4}$ 
which can be of a bigger order than the oracle rate $r(\theta)$ for 
$\theta \in \tilde{\Theta}$, where (for some $c>0$)
\begin{align}
\label{tilde_Theta}
\tilde\Theta=\tilde\Theta(c)=\{\theta\in\Theta:  r^2(\theta)\le c\sigma^2 N^{1/2}\}.
\end{align}
Equivalently, this can be seen as the optimality framework \eqref{defconfball} with 
$\Theta_0=\Theta=\mathbb{R}^N$, $\Theta_1 =\Theta\backslash \tilde{\Theta}=
\mathbb{R}^N\backslash \tilde{\Theta}$ and 
the effective radial rate  $R(\theta)\asymp r(\theta)+\sigma N^{1/4}\lesssim Cr(\theta)$
is of the oracle rate order for $\theta\in\Theta_1$. Now the deceptiveness phenomenon 
manifests itself in the fact that $\Theta_1 =\Theta\backslash \tilde{\Theta}$, not the 
whole $\Theta$.

In fact, the massiveness of the set $\tilde\Theta$ measures how much the deceptiveness 
phenomenon is present in particular models and structures.
Loosely speaking, models and structures, %(some are discussed in Section \ref{applications}), 
where ``good'' estimation ($r^2(\theta)\lesssim\sigma^2 N^{1/2}$) is possible for ``many'' $\theta$'s 
($\tilde\Theta$ is massive), suffer more from the deceptiveness phenomenon. For example, 
these are all models with smoothness and sparsity structures from Part II.
the set $\tilde{\Theta}$ is a substantial part of $\Theta=\mathbb{R}^N$ in those cases.  
On the other hand, the deceptiveness phenomenon becomes effectively marginal for some 
``uninformative'' particular models and structures, e.g., biclustering model (stochastic block model), 
dictionary learning %matrix completion 
(see Part II), %Section \ref{applications}), 
because in these cases the set $\tilde{\Theta}$ is a very ``thin'' subset of $\mathbb{R}^N$ and 
can informally  be described as  set of \emph{highly structured} parameters. 
In these cases the extra term $\sigma N^{1/4}$ in the radial rate $R(\theta)$ 
does not increase its order as $\sigma N^{1/4} \lesssim r(\theta)$ for the ``majority'' of $\theta$'s: 
$\theta\in\Theta_1=\Theta\backslash \tilde{\Theta}$. 
This means that, modulo the set $\tilde{\Theta}$ of highly structured parameters, 
there is no deceptiveness issue  for those cases. 
Indeed, there is no payment in terms of removing deceptive parameters from the parameter space 
$\Theta$ in the coverage relation and the size relation holds uniformly over 
$\Theta_1=\Theta\backslash \tilde{\Theta}$ which is ``almost'' the whole space $\Theta$.
 
%But in some other examples the set $\tilde{\Theta}$ is 
%a very ``thin'' subset of  $\mathbb{R}^N$ consisting of ``highly structured parameters'';
%for instance, in biclustering and stochastic block models (see Section \ref{sec_biclustering}).
%In those cases, this minor defect of the radial rate $R(\theta)$ is the only sacrifice in the size 
%relation of the optimality framework \eqref{defconfball}.
%We can also see this as the optimality framework \eqref{defconfball} with 
%$\Theta_0=\mathcal{Y}$, $\Theta_1 =\mathcal{Y}\backslash \tilde{\Theta}$ and 
%the effective radial rate $r(\theta)$, then a minor sacrifice occurs in the set 
%$\Theta_1 =\mathcal{Y}\backslash \tilde{\Theta}$ which is 
%``smaller'' than the entire space $\mathcal{Y}$. 

%This is not a problem for the ``majority'' of $\theta$'s 
%as this extra term does not increase the order of the radial rate:
%$N^{1/4} \le c r(\theta)$ for all $\theta \in \mathbb{R}^N\backslash\tilde{\Theta}$ 
%for some ``thin'' set $\tilde{\Theta}$. The set $\tilde{\Theta}$ can be informally described as a set 
%of ``highly structured'' parameters: namely, when either the (oracle) number of row or column 
%blocks of the underlying parameter $\theta$ is 1.

\section{Proofs} 
\label{sec_proofs}
In this section we gather all the proofs.

\subsection{Technical lemmas}
\label{proofs_lemmas}
First we provide a couple of technical lemmas used in the proofs of the main results. 
Recall that  $\hat{\pi}(I|Y)$ is either $\tilde{\pi}(I|Y)$ defined by \eqref{ddm_P(I|Y)} or 
$\check{\pi}( I|Y)=\mathrm{1}\{I=\hat{I}\}$ defined by \eqref{ms_ddm_I}.
In the latter case $\mathbb{E}_{\theta}\hat{\pi}(I|Y)=\mathbb{P}_{\theta}(\hat{I}=I)$.
In what follows, denote $\hat{p}_I=\hat{\pi}(I| Y)$ for brevity.

\begin{lemma}
\label{secondary_lemma}  
Let Condition \eqref{cond_nonnormal} be fulfilled.
%let $\hat{\pi}(I|Y)$ be either $\tilde{\pi}(I|Y)$ defined by \eqref{ddm_P(I|Y)} or
%$\check{\pi}( I|Y)=\mathrm{1}\{I=\hat{I}\}$ defined by \eqref{ms_ddm_I}.
Then for any $\theta\in\Theta$ and $I,I_0\in \mathcal{I}$ %the following bounds hold:
\begin{align*}
\mathbb{E}_{\theta}&\hat{p}_I %\\&
\le \big(\tfrac{\lambda_{I}}{\lambda_{I_0}}\big)^{h}
\exp\big\{\!-\!\tfrac{1}{\sigma^2}(A_h\|\mathrm{P}_{I}^\perp\theta\|^2\!-
\!B_h\|\mathrm{P}_{I_0}^\perp\theta\|^2)+C_h\rho(I)\!+\!D_h\rho(I_0)\big\},
\end{align*}
where $h =\tfrac{\alpha}{4}$ and the constants
$A_h=\frac{\alpha}{16}$, $B_h=\frac{3\alpha}{16}$  and $C_h=\frac{5}{8}$, 
$D_h= \frac{3}{8}$.

If $\mathbb{L}_{I}\subseteq \mathbb{L}_{I_0}$,  then
\begin{align*}
\mathbb{E}_{\theta}&\hat{p}_I\le \big(\tfrac{\lambda_{I}}{\lambda_{I_0}}\big)^{\alpha} 
\exp\big\{-\tfrac{\alpha}{3}\sigma^{-2}\big(\|\mathrm{P}_{I}^\perp\theta\|^2
-\|\mathrm{P}_{I_0}^\perp\theta\|^2\big)+\rho(I_0)\big\}.
\end{align*}
	
If $\mathbb{L}_{I_0}\subseteq \mathbb{L}_{I}$,  then
\begin{align*}
\mathbb{E}_{\theta}&\hat{p}_I \le \big(\tfrac{\lambda_{I}}{\lambda_{I_0}}\big)^{\alpha}
\exp\big\{\alpha\sigma^{-2}\big(\|\mathrm{P}_{I}\theta\|^2
-\|\mathrm{P}_{I_0}\theta\|^2\big)+\rho(I)\big\}.
\end{align*}
\end{lemma}
\begin{proof}Recall that $\mathrm{P}_{I}$ is the projection onto $\mathbb{L}_{I}$.
Since $\mathrm{P}_{I}-\mathrm{P}_{I_0}=\mathrm{P}_{I_0}^\perp-\mathrm{P}_{I}^\perp$, the bound
\begin{align}
\label{equation_1}
&Y^T(\mathrm{P}_I-\mathrm{P}_{I_0})Y=
\theta^{T}(\mathrm{P}_{I}-\mathrm{P}_{I_0})\theta+2\theta^{T}(\mathrm{P}_{I}-\mathrm{P}_{I_0})\sigma\xi
+\sigma^2\xi^T(\mathrm{P}_{I}-\mathrm{P}_{I_0})\xi\notag\\
&\le  
-\|\mathrm{P}_{I}^\perp\theta\|^2+\|\mathrm{P}_{I_0}^\perp\theta\|^2
+2|\sigma\theta^{T}(\mathrm{P}_{I}-\mathrm{P}_{I_0})\xi|+\sigma^2\|\mathrm{P}_{I}\xi\|^2
-\sigma^2\|\mathrm{P}_{I_0}\xi\|^2
\end{align}
holds for any $I, I_0\in\mathcal{I}$.
Using the relations 
$\mathrm{P}_{I}-\mathrm{P}_{I_0}=(\mathrm{P}_{I}-\mathrm{P}_{I_0})
\mathrm{P}_{\mathbb{L}_{I}+\mathbb{L}_{I_0}}$, $\|\mathrm{P}_{\mathbb{L}_{I}+\mathbb{L}_{I_0}}x\|^2 \le 
\|\mathrm{P}_{I}x\|^2+\|\mathrm{P}_{I_0}x\|^2$, $x\in\mathcal{Y}$,
and the  inequality $2ab\le a^2/4+4b^2$ (for any $a,b\in\mathbb{R}$), we derive
\begin{align*}
%\label{equation_2}
2|\theta^{T}&(\mathrm{P}_{I}-\mathrm{P}_{I_0})\sigma\xi|
=2|\theta^{T}(\mathrm{P}_{I}-\mathrm{P}_{I_0})
\mathrm{P}_{\mathbb{L}_{I}+\mathbb{L}_{I_0}}\sigma\xi |\\
&\le 
2\|\theta^{T}\big(\mathrm{P}_{I}-\mathrm{P}_{I_0})\|
\|\sigma\mathrm{P}_{\mathbb{L}_{I}+\mathbb{L}_{I_0}}\xi\|\le \tfrac{1}{4} \|(\mathrm{P}_{I}-\mathrm{P}_{I_0})\theta\|^2+
4\sigma^2\|\mathrm{P}_{\mathbb{L}_{I}+\mathbb{L}_{I_0}}\xi\|^2\\
&=\tfrac{1}{4} \|(\mathrm{P}_{I_0}^\perp-\mathrm{P}_{I}^\perp)\theta\|^2+
4\sigma^2\|\mathrm{P}_{\mathbb{L}_{I}+\mathbb{L}_{I_0}}\xi\|^2\\
&\le \tfrac{1}{2}\|\mathrm{P}_{I}^\perp\theta\|^2
+\tfrac{1}{2}\|\mathrm{P}_{I_0}^\perp\theta\|^2
+4\sigma^2\|\mathrm{P}_{I}\xi\|^2+4\sigma^2\|\mathrm{P}_{I_0}\xi\|^2.
\end{align*}

The last bound and \eqref{equation_1}   imply that
\begin{align}
\label{equation_2}
Y^T(\mathrm{P}_I-\mathrm{P}_{I_0})Y\le -\tfrac{1}{2}\|\mathrm{P}_{I}^\perp\theta\|^2+
\tfrac{3}{2}\|\mathrm{P}_{I_0}^\perp\theta\|^2+5\sigma^2\|\mathrm{P}_{I}\xi\|^2+3\sigma^2\|\mathrm{P}_{I_0}\xi\|^2.
	\end{align}

In case $\hat{p}_I=\hat{\pi}(I|Y)=\check{\pi}(I|Y)=\mathrm{1}\{\hat{I}=I\}$, \eqref{ddm_P(I|Y)}, 
the definition \eqref{I_max_ddm} of $\hat{I}$ and the Markov inequality imply that, for any $I,I_0\in\mathcal{I}$
and any $h\ge 0$,
\begin{align}
\label{lemm2_rel2}
\mathbb{E}_{\theta}\hat{p}_I 
=\mathbb{P}_{\theta}(\hat{I}= I)\le\mathbb{P}_{\theta}\Big(\frac{\tilde{\pi}(I|Y)}
{\tilde{\pi}(I_0|Y)}\ge 1\Big)\le 
\mathbb{E}_{\theta}\Big[\frac{\tilde{\pi}(I|Y)}{\tilde{\pi}(I_0|Y)}\Big]^h.
\end{align}
In case $\hat{p}_I=\tilde{\pi}(I|Y)$, \eqref{ddm_P(I|Y)} 
implies $\mathbb{E}_{\theta}\hat{p}_I \le \mathbb{E}_{\theta}
\Big[\frac{\tilde{\pi}(I|Y)}{\tilde{\pi}(I_0|Y)}\Big]^h$
for any $I,I_0\in\mathcal{I}$, $h\in[0,1]$, which again establishes 
\eqref{lemm2_rel2}, now for any $h\in[0,1]$.

Combining \eqref{equation_2} and \eqref{lemm2_rel2}, we derive 
for  any $I,I_0\in \mathcal{I}$ and any $h\in[0,1]$,
\begin{align}
\mathbb{E}_{\theta}\hat{p}_I &\le 
%\mathbb{E}_{\theta}\Big[\frac{\tilde{\pi}(I|X)}{\tilde{\pi}(I_0|X)}\Big]^h\!\!=
\mathbb{E}_{\theta}\bigg[\frac{\lambda_{I}\exp\{-\tfrac{1}{2\sigma^2}
\|Y-\mathrm{P}_{I}Y\|^2\}}
{\lambda_{I_0}\exp\{-\tfrac{1}{2\sigma^2}
\|Y-\mathrm{P}_{I_0}Y\|^2\}}\bigg]^h %\notag\\& 
=\big(\tfrac{\lambda_{I}}{\lambda_{I_0}}\big)^{h}
\,\mathbb{E}_{\theta}\exp\big\{\tfrac{h}{2\sigma^2} \big(Y^T(\mathrm{P}_{I}-\mathrm{P}_{I_0})Y
 \big)\big\} \notag\\
%\label{formula3}
\label{formula2}
&\le\big(\tfrac{\lambda_{I}}{\lambda_{I_0}}\big)^{h}
\exp\big\{-\tfrac{h}{4\sigma^2} \|\mathrm{P}_{I}^\perp\theta\|^2+\tfrac{3h}{4\sigma^2}
\|\mathrm{P}_{I_0}^\perp\theta\|^2\big\} 
\mathbb{E}_\theta\exp\big\{\tfrac{h}{2}(5\|\mathrm{P}_{I}\xi\|^2+3\|\mathrm{P}_{I_0}\xi\|^2)\big\}. %\notag
\end{align} 

The lemma follows for $h=\frac{\alpha}{4}$ from the last display and the relation
\begin{align*}
\mathbb{E}_\theta&\exp\big\{\tfrac{5\alpha}{8}\|\mathrm{P}_{I}\xi\|^2
+\tfrac{3\alpha}{8}\|\mathrm{P}_{I_0}\xi\|^2\big\} 
\le\big[\mathbb{E}_\theta e^{\alpha \|\mathrm{P}_{I}\xi\|^2}\big]^{\frac{5}{8}}
\big[\mathbb{E}_\theta e^{\alpha\|\mathrm{P}_{I_0}\xi\|^2}\big]^{\frac{3}{8}}\\
&\le \exp\big\{\tfrac{5}{8}\rho(I)+\tfrac{3}{8}\rho(I_0)\big\},
\end{align*}
which is in turn obtained by using the H\"older inequality and Condition \eqref{cond_nonnormal}.

In case $\mathbb{L}_{I}\subseteq \mathbb{L}_{I_0}$, 
take $h=\alpha$ in \eqref{formula2} and, 
instead of \eqref{equation_2} use $Y^T(\mathrm{P}_{I}-\mathrm{P}_{I_0})Y=
-\|\mathrm{P}_{\mathbb{L}_{I}^\perp\cap\mathbb{L}_{I_0}} Y\|^2 
\le -\frac{2}{3} \|\mathrm{P}_{\mathbb{L}_{I}^\perp\cap\mathbb{L}_{I_0}} \theta\|^2
+2\sigma^2 \|\mathrm{P}_{\mathbb{L}_{I}^\perp\cap\mathbb{L}_{I_0}}\xi\|^2
\le \frac{2}{3} \big(\|\mathrm{P}_{I}\theta\|^2-\|\mathrm{P}_{I_0}\theta\|^2\big)+
2\sigma^2 \|\mathrm{P}_{I_0}\xi\|^2=-\tfrac{2}{3}\|\mathrm{P}_{I}^\perp\theta\|^2
+\tfrac{2}{3}\|\mathrm{P}_{I_0}^\perp\theta\|^2
+2\sigma^2 \|\mathrm{P}_{I_0}\xi\|^2$ as $(a+b)^2\ge 2a^2/3-2b^2$ and 
$\mathrm{P}_{I_0}-\mathrm{P}_{I}=\mathrm{P}_{\mathbb{L}_{I}^\perp\cap\mathbb{L}_{I_0}}$. 
	
In case $\mathbb{L}_{I_0}\subseteq \mathbb{L}_{I}$, 
take $h=\alpha$ in \eqref{formula2} and, 
instead of \eqref{equation_2} use $Y^T(\mathrm{P}_{I}-\mathrm{P}_{I_0})Y=
\|\mathrm{P}_{\mathbb{L}_{I}\cap\mathbb{L}_{I_0}^\perp} Y\|^2 
\le 2 \|\mathrm{P}_{\mathbb{L}_{I}\cap\mathbb{L}_{I_0}^\perp} \theta\|^2
+2\sigma^2 \|\mathrm{P}_{\mathbb{L}_{I}\cap\mathbb{L}_{I_0}^\perp}\xi\|^2
\le 2 \big(\|\mathrm{P}_{I}\theta\|^2-\|\mathrm{P}_{I_0}\theta\|^2\big)+
2\sigma^2 \|\mathrm{P}_{I}\xi\|^2=-2\|\mathrm{P}_{I}^\perp\theta\|^2
+2\|\mathrm{P}_{I_0}^\perp\theta\|^2
+2\sigma^2 \|\mathrm{P}_{I}\xi\|^2$ as $(a+b)^2\le 2a^2+2b^2$ and 
$\mathrm{P}_{I}-\mathrm{P}_{I_0}=\mathrm{P}_{\mathbb{L}_{I}\cap\mathbb{L}_{I_0}^\perp}$. 
\end{proof}
%\begin{remark} 
%	\label{rem1a}
%	Notice that we proved the above lemma for the both cases
%	$\hat{\pi}(I|Y)=\tilde{\pi}(I|Y)$ and $\hat{\pi}(I|Y)=\check{\pi}(I|Y)$.
%	From this point on, 
%	the proof of the properties of $\check{\pi}(\vartheta|Y)$ and $\check{\theta}$ 
%	proceeds exactly in the same way as for $\tilde{\pi}(\vartheta|Y)$ 
%	and $\tilde{\theta}$, with the only difference that 
%	everywhere (in the claims and in the proofs)
%	$\hat{\pi}((I)\in \mathcal{G}|Y)$ should be read as 
%	$\tilde{\pi}((I) \in \mathcal{G}|Y)$ in case $\hat{\pi}=\tilde{\pi}$; 
%	and as $\mathrm{1}\{(\hat{s}, \hat{I})\in\mathcal{G}\}$ in case $\hat{\pi}=\check{\pi}$, 
%	for all $\mathcal{G}\subseteq \{s\in\mathcal{S}, I\in\mathcal{I}_s\}$ that appear in the proof.
%	Hence, $\mathbb{E}_{\theta}\hat{\pi}((I)\in\mathcal{G}|Y)
%	=\mathbb{E}_{\theta}\tilde{\pi}((I)\in\mathcal{G}|Y)$ in the former case,
%	and $\mathbb{E}_{\theta}\hat{\pi}((I)\in\mathcal{G}|Y)
%	=\mathbb{P}_\theta((\hat{s}, \hat{I})\in\mathcal{G}|Y)$ in the latter case.
%\end{remark} 

Note that above lemma holds for any  $I_0\in\mathcal{I}$. 
By taking $I_0=I_o$ defined by \eqref{oracle}, we derive the next lemma.
\begin{lemma}
\label{main_lemma}
Let Condition \eqref{cond_nonnormal}  be fulfilled.
Then there exist positive constants $c_1=c_1(\varkappa)>2\nu, c_2$ and $c_3=c_3(\varkappa)$ 
such that for any $\theta\in\Theta$
\begin{align*}
\mathbb{E}_{\theta}\hat{p}_I
\le\exp\big\{-c_1\rho(I)-c_2\sigma^{-2} \big[r^2(I,\theta)-c_3r^2(\theta)\big]\big\}.
\end{align*}
\end{lemma} 

\begin{proof}
With constants $h$, $A_h,B_h,C_h,D_h$ defined in Lemma \ref{secondary_lemma}, define 
the constant $c_1=c_1(\varkappa)=h\varkappa-C_{h}-A_{h} =\frac{\varkappa\alpha}{4}-\frac{5}{8}
-\frac{\alpha}{16}>2\nu$ as $\varkappa>\bar{\varkappa}$ by \eqref{cond_technical}.  
The definition \eqref{ddm_P(I|Y)} of  $\lambda_{I}$ entails that
\begin{align*}
\big(\lambda_I/\lambda_{I_o}\big)^{h}
=\exp\big\{h\varkappa\rho(I_o)-\big(c_1+C_h+A_h\big)\rho(I)\big\}.
\end{align*}
Combining the last relation with Lemma \ref{secondary_lemma} 
(for $I_0= I_o$), we derive that
\begin{align*}
\mathbb{E}_{\theta}\hat{p}_I&\le
\exp\big\{-c_1\rho(I)-\sigma^{-2}\big[A_hr^2(I,\theta)
-\max\{B_h,D_h+h\varkappa\} r^2(\theta)\big]\big\}\\
&=
\exp\big\{-c_1\rho(I)-
c_2\sigma^{-2} \big[r^2(I,\theta)-c_3r^2(\theta)\big]\big\},
\end{align*}
which completes the proof with the constants $c_1=c_1(\varkappa)
=\frac{\varkappa\alpha}{4}-\frac{5}{8}-\frac{\alpha}{16}$, %\frac{4\varkappa\alpha-10-\alpha}{16}$, 
$c_2=A_h=\frac{\alpha}{16}$ and $c_3=c_3(\varkappa)=A_h^{-1}\max\big\{B_h,D_h+h\varkappa\big\}=
\tfrac{16}{\alpha}\max\big\{\tfrac{3\alpha}{16},\tfrac{3}{8}+\tfrac{\alpha\varkappa}{4}\big\}
=\max\big\{3,\tfrac{6}{\alpha}+4\varkappa\big\}=\tfrac{6}{\alpha}+4\varkappa$
because $\varkappa>\bar{\varkappa}\ge1$ by \eqref{cond_technical}.  
\end{proof}

\subsection{Proofs of the theorems}
\label{subsec_proofs}
\label{proofs_theorems}
Here we give the proofs of all the theorems. By $C_1, C_2$ etc., we denote constants 
which are different in different proofs.
\begin{proof}[Proof of Theorem \ref{th1}]
Recall the constants $c_1, c_2, c_3$ defined in the proof of Lemma \ref{main_lemma}
and the notation  $\hat{p}_I=\hat{\pi}(I|Y)$. 
For any $\theta\in\Theta$, $M\ge 0$ and some constant $M_0$ to be chosen later, denote
$\Delta_{M}=\Delta_{M}(\theta)=M_0r^2(\theta)+M\sigma^2$.  
Next, introduce the set  $\mathcal{O}_M=\mathcal{O}_M(\theta)
=\{ I \in \mathcal{I}: \, r^2(I,\theta)\le c_3 r^2(\theta)+C_1M\sigma^2\}$
%and the event $\mathcal{A}_M=\mathcal{A}_M(\theta)=
%\big\{\max_{I\in\mathcal{S}_M}\|\mathrm{P}_{I}\xi\|^2\le 
%\tfrac{2}{\alpha} c_3 r^2(\theta)+C_2M\big\}$
and  the events  $A_M(I)=\big\{\alpha\|\mathrm{P}_I\xi\|^2\le (\nu+1)\rho(I)
+C_2M\big\}$, $I\in\mathcal{I}$, where constants $C_1,C_2>0$ are to be chosen later. We have
\begin{align}
&\hat{\pi}(\|\vartheta-\theta\|^2 \ge \Delta_{M}| Y)=
\sum_{I\in\mathcal{I}}
\hat{\pi}_{I}(\|\vartheta-\theta\|^2 \ge \Delta_{M}| Y)\hat{p}_I\notag\\
&\le \sum_{I\in\mathcal{I}}\hat{p}_I 1_{A^c_M(I)} 
+ \sum_{I\in \mathcal{O}^c_M}\hat{p}_I %\notag\\
%\hat{\pi}((I)\in\mathcal{S}_M^c| Y) \\&\quad 
+ \sum_{I\in\mathcal{O}_M}\hat{\pi}_{I}\big(\|\vartheta-\theta\|^2\ge \Delta_{M}| Y\big)
\hat{p}_I 1_{A_M(I)}\notag\\
&=T_1+T_2+T_3.
\label{th1_rel1}
\end{align}
Now we need to bound the quantities $\mathbb{E}_{\theta}T_1$, $\mathbb{E}_{\theta}T_2$ 
and $\mathbb{E}_{\theta}T_3$.
	
%We first bound $\mathbb{E}_{\theta}T_1$. 
By using  the Markov inequality and Condition \eqref{cond_nonnormal}, we have
%By Lemma \ref{max_projection_1},
\begin{align*}
\mathbb{P}_\theta(A^c_M(I)) &=\mathbb{P}_\theta\big(e^{\alpha \|\mathrm{P}_I\xi\|^2}
>e^{(\nu+1)\rho(I)+C_2M }\big) \le e^{-\nu\rho(I)-C_2M}.
\end{align*}
The last relation and Condition \eqref{(A2)} yield the bound for $\mathbb{E}_{\theta}T_1$:   
\begin{align}
\label{th1_T1}
\mathbb{E}_{\theta}T_1
&\le 
\sum_{I\in\mathcal{I}} \mathbb{P}_\theta\big(A^c_M(I)\big)
\le \sum_{I\in\mathcal{I}} \exp\{-\!\nu\rho(I)-C_2M\}\le C_\nu e^{-C_2M}.
\end{align}

If $I\in\mathcal{O}^c_M$, then $r^2(I,\theta)>c_3 r^2(\theta)+C_1M\sigma^2$.
Using this, Lemma \ref{main_lemma} and the fact that $\sum_{I\in\mathcal{I}}
e^{-c_1\rho(I)}\le C_\nu$ (in view of Condition \eqref{(A2)} and because $c_1>2\nu$),  
we  bound $\mathbb{E}_{\theta} T_2$ as follows:
\begin{align}
\label{th1_T2}
\mathbb{E}_{\theta} T_2
&=\sum_{I\in \mathcal{O}^c_M}\mathbb{E}_{\theta}\hat{p}_I %\notag\\&
\le \sum_{I\in\mathcal{O}^c_M}  
\exp\big\{-c_1\rho(I)-c_2\sigma^{-2}\big[r^2(I,\theta)-c_3r^2(\theta)\big]\big\}\notag\\ &\le\sum_{I\in\mathcal{I}}
\exp\big\{-c_1\rho(I)-c_2C_1M\big\}\le C_\nu\exp\big\{-c_2C_1M\big\}.
\end{align}
	
It remains to establish the last bound for $\mathbb{E}_{\theta} T_3$. 
For $I \in \mathcal{O}_M$, we have that
\begin{align*}
%\label{th7_rel1}
A_M(I)\subseteq\Big\{\!\|\theta\!-\!\mathrm{P}_{I}\theta\|^2\!+\!\sigma^{2}\!\|\mathrm{P}_{I}\xi\|^2
\le c_3 r^2(\theta)+\tfrac{\nu+1}{\alpha}\sigma^{2}\rho(I)\!+\!(C_1\!+\!\tfrac{C_2}{\alpha})M\sigma^{2}\Big\}.
\end{align*}
Recall  the definition \eqref{ddm_I} of the DDM $\tilde{\pi}_{I}(\vartheta \in B|Y)$, 
which is expressed in terms of the measure $\mathbb{P}_Z$. 
The measure $\mathbb{P}_Z$ satisfies Condition \eqref{cond_nonnormal}, 
which implies $\mathbb{P}_Z\big(\|\mathrm{P}_IZ\|^2 \ge \alpha^{-1}d_I+M\big)\le e^{-\alpha M}$,  
$I\in\mathcal{I}$. Using this, the last display  and the fact that 
$\frac{r^2(\theta)}{\sigma^{2}}\ge c_3^{-1}(\rho(I)-C_1M)$ for 
$I\in\mathcal{O}_M$, %and Lemma \ref{max_projection_1}, 
we obtain that, for any $I\in\mathcal{O}_M$,
\begin{align*}
\hat{\pi}_{I}&\big(\|\vartheta-\theta\|^2\ge\Delta_M| Y\big)1_{A_M(I)}\\&
=\mathbb{P}_Z(\|\mathrm{P}_{I}Y+(1-e^{-1})^{1/2}\sigma\mathrm{P}_{I}Z-\theta\|^2 
\ge \Delta_{M} )1_{A_M(I)} \\
&\le\mathbb{P}_Z\big(2\sigma^{2}\|\mathrm{P}_{I}Z\|^2+2\|\mathrm{P}_IY-\theta\|^2
\ge M_0r^2(\theta)+M\sigma^2\big)1_{A_M(I)} \\
& =
\mathbb{P}_Z\big(\sigma^{2}\|\mathrm{P}_{I}Z\|^2+\|\theta-\mathrm{P}_{I}\theta\|^2
+\sigma^{2}\|\mathrm{P}_{I}\xi\|^2\ge\tfrac{M_0}{2}r^2(\theta)+\tfrac{M\sigma^2}{2}\big)
1_{A_M(I)}\\
& \le
\mathbb{P}_Z\Big(\|\mathrm{P}_{I}Z\|^2 \ge (\tfrac{M_0}{2}-c_3)\tfrac{r^2(\theta)}{\sigma^{2}}
-\tfrac{\nu+1}{\alpha}\rho(I)+\tfrac{M}{2}-(C_1+\tfrac{C_2}{\alpha})M\Big)\\
&\le
\mathbb{P}_Z\Big(\|\mathrm{P}_IZ\|^2\ge 
\big(\tfrac{M_0}{2c_3}-\tfrac{\nu+\alpha+1}{\alpha}\big)\rho(I) +\tfrac{M}{2}-\tfrac{C_2}{\alpha}M
-\tfrac{M_0C_1}{2c_3}M\Big)\\
&= \mathbb{P}_Z\Big(\|\mathrm{P}_IZ\|^2 \ge\alpha^{-1}\rho(I)+\tfrac{M}{4}\Big)\le e^{-\alpha M/4},
\end{align*}
where we have chosen $M_0=\tfrac{c_3(2\nu +2\alpha+4)}{\alpha}$, 
$C_1=\tfrac{\alpha}{4(2\nu +2\alpha+4)}$ and $C_2=\tfrac{\alpha}{8}$
(so that $\tfrac{M_0}{2c_3}-\tfrac{\nu+\alpha+1}{\alpha}=\tfrac{1}{\alpha}$, $\tfrac{C_2}{\alpha}=\tfrac{1}{8}$, 
$\tfrac{M_0C_1}{2c_3} =\tfrac{1}{8}$).
Thus we have derived
\begin{align*}
\mathbb{E}_{\theta} T_3 &= 
\mathbb{E}_{\theta} \sum_{I\in\mathcal{O}_M}
1_{A_M(I)}\hat{\pi}_{I}\big(\|\vartheta-\theta\|^2 
\ge \Delta_{M}| Y\big) \hat{p}_I %\\&
\le\mathbb{E}_{\theta}\sum_{I\in\mathcal{I}}
e^{-\tfrac{\alpha M}{4}}\hat{p}_I\le e^{-\alpha M/4}.
\end{align*}
This completes the proof of the first assertion since, in view of  \eqref{th1_rel1}, 
\eqref{th1_T1}, \eqref{th1_T2} and the last display,  we established the claim 
\eqref{th1_i}: $\mathbb{E}_{\theta}\hat{\pi}\big(\|\vartheta-\theta\|^2\ge M_0
r^2(\theta)+M\sigma^2 | Y\big) \le \mathbb{E}_{\theta}(T_1+T_2+T_3) 
\le  H_0e^{-m_0 M}$, with the constants $M_0=\tfrac{c_3(2\nu +2\alpha+4)}{\alpha}  $, 
$H_0=1+2C_\nu$ and $m_0= \min\{C_2,c_2C_1, \alpha/4\}$.
\medskip	
	
The proof of the assertion \eqref{th1_ii} proceeds along similar lines.
%We use the same notation as in the proof of claim \eqref{th1_i}
Introduce the set  $\mathcal{J}_M=\mathcal{J}_M(\theta)=\{ I \in \mathcal{I}: \, r^2(I,\theta)\le 
2c_3 r^2(\theta)+C_3M\sigma^2\}$ and the events $B_M(I)=
\big\{\alpha\|\mathrm{P}_I\xi\|^2\le 2(\nu+1)\rho(I)+C_4M\big\}$,
$I\in\mathcal{I}$, where constants $C_3,C_4>0$ are to be chosen later.
	
%Let $\mathcal{O}_M=\mathcal{O}_M(\theta)=\big\{s\in\mathcal{S},\, I \in \mathcal{I}_{s}: 
%\, r^2(I,\theta)\le c_3 r^2(\theta)+C_5M\big\}$ for some $C_5$ to be defined later.
%Clearly, for each $I\in\mathcal{O}_M$, $\dim(\mathbb{L}_{I}) \le c_3 r^2(\theta)+C_5M$
%and $\log|\mathcal{O}_M| \le n_1\log k_1+n_2\log k_2 \le c_3 r^2(\theta)+C_5M$.
	
If $M\in[0,1]$, the claim (ii) holds for $H_1=e^{m_1}$. Let  $M\ge1$.	Denote for brevity 
$R^2_{I} =R^2_{I}(\theta,Y)=\|\theta-\mathrm{P}_{I}Y\|^2
=\|\theta-\mathrm{P}_{I}\theta\|^2+\sigma^2\|\mathrm{P}_{I}\xi\|^2$, 
$\Delta'_{M}=\Delta'_M(\theta)=M_1r^2(\theta)+M\sigma^2$ and $\hat{p}_{I}=  \hat{\pi}(I|Y)$,
where $M_1>0$ is to be chosen later. Applying the Cauchy-Schwarz inequality,
%and the fact that  $\|\theta-\mathrm{P}_{I}X\|^2\le 2r^2_{I,s}$, 
we have
\begin{align}
%\label{relation11}
&\mathbb{P}_\theta(\|\hat{\theta}-\theta\|^2 \ge  \Delta'_{M}) 
\le 
\mathbb{P}_\theta\Big(\sum_{I \in \mathcal{I}} R^2_{I}  \hat{p}_{I} \ge \Delta'_{M}\Big) \notag\\
&\le \mathbb{P}_\theta\Big(\sum_{I\in\mathcal{J}_M}R^2_{I}\hat{p}_{I} (1_{B_M(I)}+1_{B^c_M(I)})
+\sum_{I\in\mathcal{J}_M^c}R^2_{I} \hat{p}_I \ge \Delta'_{M}\Big)  \notag\\
&\le 
\mathbb{P}_\theta\Big(\sum_{I\in\mathcal{J}_M}R^2_{I}\hat{p}_I 1_{B_M(I)}
\ge \tfrac{\Delta'_{M}}{3}\Big)+\mathbb{P}_\theta\Big(\sum_{I\in\mathcal{J}_M}
R^2_{I}\hat{p}_I 1_{B^c_M(I)}\ge \tfrac{\Delta'_{M}}{3}\Big)  \notag\\
\label{relation11}
& \quad 
+\mathbb{P}_\theta\Big(\sum_{I\in\mathcal{J}_M^c}R^2_{I} 
\hat{p}_I \ge \tfrac{\Delta'_{M}}{3}\Big) =\bar{T}_1+\bar{T}_2+\bar{T}_3.
\end{align}

Let us evaluate $\bar{T}_1$. 
%Let $\mathcal{K}= \mathcal{K}(\mathcal{O}_M)=
%\{s \in \mathcal{S}:\, \exists I\in\mathcal{I}_{s}: 
%I\in \mathcal{O}_M\}$ and $\mathcal{I}_{s}(\mathcal{O}_M)
%=\mathcal{I}_{s} \cap \mathcal{O}_M$.  
For any $I\in \mathcal{J}_M$,
%$s\in\mathcal{I}_s(\mathcal{O}_M)$ and $I\in\mathcal{I}_{s}(\mathcal{O}_M)$, 
under $B_M(I)$, 
we have that $R^2_{I}=\|\theta-\mathrm{P}_{I}\theta\|^2+\sigma^2\|\mathrm{P}_{I}\xi\|^2
\le \|\theta-\mathrm{P}_{I}\theta\|^2+ \tfrac{2(\nu+1)}{\alpha} \sigma^2\rho(I)+\tfrac{C_4}{\alpha}M\sigma^2
\le  \tfrac{2(\nu+1)}{\alpha} r^2(I,\theta)+\tfrac{C_4}{\alpha}M\sigma^2 
\le\tfrac{4c_3(\nu+1)}{\alpha}r^2(\theta)+\tfrac{2C_3(\nu+1)+C_4}{\alpha}M\sigma^2$. 
Using this, we derive  
\begin{align}
\bar{T}_1&=\mathbb{P}_\theta\Big(\sum_{I\in\mathcal{J}_M}
R^2_{I} \hat{p}_I1_{B_M(I)}\ge\tfrac{\Delta'_{M}}{3}\Big)\notag\\
\label{relation12}
&\le
\mathbb{P}_\theta\Big(
\tfrac{4c_3(\nu+1)}{\alpha}r^2(\theta)+\tfrac{2C_3(\nu+1)+C_4}{\alpha}M\sigma^2
\ge\tfrac{\Delta'_{M}}{3}\Big)=0,
\end{align}
as $\frac{4c_3(\nu+1)}{\alpha}=\tfrac{M_1}{3}$ and $\tfrac{2C_3(\nu+1)+C_4}{\alpha}<\tfrac{1}{3}$
because we choose $M_1=\frac{12c_3(\nu+1)}{\alpha}$, 
$C_3=\tfrac{\alpha}{12(\nu+1)}$ and $C_4=\tfrac{\alpha}{7}$.

Next, we evaluate $\bar{T}_2$. By Condition \eqref{cond_nonnormal} and the Markov inequality,
\begin{align*}
\mathbb{P}_\theta(B^c_M(I)) &= 
\mathbb{P}_\theta\big(\alpha\|\mathrm{P}_I \xi\|^2>(2\nu+2)\rho(I)+C_4M\big)%\\&
\le e^{-(2\nu+1)\rho(I)-C_4M}.
\end{align*}
It follows from  \eqref{moment_bound} with $t=\tfrac{1}{2}$ that 
$\big[\mathbb{E}_\theta\|\mathrm{P}_{I}\xi\|^4\big]^{1/2}
\le \frac{2}{\alpha}\exp\{\rho(I)/2\}$ for any  $I\in  \mathcal{I}$. 
By Condition \eqref{(A2)},
$\sum_{I\in\mathcal{I}} 
\exp\{-\nu\rho(I)\} \le  C_\nu$.
Besides, for any $I\in\mathcal{J}_M$, 
$\|\theta-\mathrm{P}_{I}\theta\|^2/\Delta'_{M} \le 
(2c_3r^2(\theta) +C_3M\sigma^2)/(M_1r^2(\theta) + M\sigma^2) \le \tfrac{2c_3}{M_1}+C_3$ and  
$\Delta'_{M} \ge M\sigma^2 \ge \sigma^2$ (as $M\ge1$).
Collecting all the derived relations for evaluating $T_2$ and using the Markov and  
Cauchy-Schwarz inequalities, we obtain
\begin{align}
\bar{T}_2&= \mathbb{P}_\theta\Big(\!\sum\nolimits_{I\in\mathcal{J}_M}
R^2_{I}\hat{p}_I1_{B^c_M(I)}\ge \Delta'_{M}/3\Big)\notag\\
&\le
\frac{\mathbb{E}_\theta \sum_{I\in\mathcal{J}_M} 
\big(\|\theta-\mathrm{P}_{I}\theta\|^2+\sigma^2\|\mathrm{P}_{I}\xi\|^2\big)\hat{p}_I1_{B^c_M(I)}}
{\Delta'_{M}/3}
\notag\\
&\le 
\frac{\sum_{I\in\mathcal{J}_M}\!\! \|\theta-\mathrm{P}_{I}\theta\|^2\mathbb{P}_\theta(B^c_M(I))}
{\Delta'_{M}/3}+\frac{\sigma^2\sum_{I\in\mathcal{J}_M} \!\!
\big[\mathbb{E}_\theta \|\mathrm{P}_{I}\xi\|^4\big]^{1/2} 
\big[\mathbb{P}_\theta(B^c_M(I))\big]^{1/2}}{\Delta'_{M}/3} \notag \\
&\le 3\big(\tfrac{2c_3}{M_1} +C_3\big)  e^{-C_4 M}
\sum_{I\in\mathcal{J}_M} e^{-(2\nu+1)\rho(I)}
+ \tfrac{6}{\alpha}e^{-C_4 M/2}
\sum_{I\in\mathcal{J}_M} e^{-\nu\rho(I)}
\notag\\
&\le 3C_\nu\big(\tfrac{2c_3}{M_1} +C_3\big) e^{-C_4 M}
+ \tfrac{6C_\nu}{\alpha}e^{-C_4 M/2}.
\label{bound_T2}
\end{align}

It remains to bound $T_3$.
Applying first the Markov inequality and then the Cauchy-Schwarz inequality, we have
\begin{align}
\label{relation13}
\bar{T}_3 &=\mathbb{P}_\theta\Big(\sum\nolimits_{I\in\mathcal{J}_M^c}R^2_{I} \hat{p}_I
\ge \Delta'_{M}/3\Big)\le \frac{\sum_{I\in\mathcal{J}_M^c} 
\|\theta-\mathrm{P}_{I}\theta\|^2 \mathbb{E}_{\theta}\hat{p}_I}{\Delta'_{M}/3} 
\notag\\
&\quad + \frac{\sigma^2\sum_{I\in\mathcal{J}_M^c} 
\big(\mathbb{E}_{\theta}\|\mathrm{P}_{I}\xi\|^4\big)^{1/2}
\big[\mathbb{E}_{\theta}\hat{p}_I\big]^{1/2}}{\Delta'_{M}/3}=\bar{T}_{31}+\bar{T}_{32}.
\end{align}
For each $I \in \mathcal{J}^c_M$, we have $c_3 r^2(\theta) 
\le \tfrac{1}{2}r^2(I,\theta)-\tfrac{C_3}{2}M\sigma^2$, yielding the bound 
\begin{align*}
\frac{c_2}{2\sigma^2}\big(r^2(I,\theta)-c_3 r^2(\theta)\big) &\ge
\frac{c_2}{4\sigma^2}r^2(I,\theta)+\frac{c_2C_3}{4}M.
\end{align*}
The last relation and Lemma \ref{main_lemma} entail that, for each $I \in \mathcal{J}^c_M$,
\begin{align}
\label{th1_rel12}
\big[\mathbb{E}_{\theta}\hat{p}_I \big]^{1/2}\le  
\exp\Big\{-\frac{c_1}{2}\rho(I)-\frac{c_2}{4\sigma^2} r^2(I,\theta)-\frac{c_2C_3}{4}M\Big\}.
\end{align}

Since $M\ge1$, $\Delta'_{M} \ge M\sigma^2 \ge \sigma^2$. 
Using this, the relation \eqref{th1_rel12}, the facts that 
$\max_{x\ge 0}\{xe^{-cx} \}\le (ce)^{-1}$ (for any $c>0$) and $\sum_{I\in\mathcal{I}} 
e^{-c_1\rho(I)}\le C_\nu$ (in view of Condition \eqref{(A2)} as $c_1>2\nu$), 
we bound the term $T_{31}$ as follows:
\begin{align}
\bar{T}_{31}&=
\frac{\sum_{I\in \mathcal{J}^c_M} 
\|\theta-\mathrm{P}_{I}\theta\|^2\mathbb{E}_{\theta}\hat{p}_I}{\Delta'_M/3}
\notag\\
&\le 
3\sum_{I\in\mathcal{J}^c_M}\! \frac{r^2(I,\theta)}{\sigma^2}   
\exp\big\{-\!c_1\rho(I)-\tfrac{c_2}{2\sigma^2}r^2(I,\theta_0)
-\tfrac{c_2C_3}{2}M\big\}\notag\\
&\le \frac{6C_\nu}{c_2e} e^{-c_2C_3M/2}.
\label{relation14}
\end{align}

Using \eqref{moment_bound} with $t_0=\min\{1/2,c_2/4\}$,
we have that $\big[\mathbb{E}_{\theta}\|\mathrm{P}_{I}\xi\|^4\big]^{1/2}
\le \frac{1}{\alpha t_0}\exp\{\tfrac{c_2}{4}\rho(I)\}$.  
Besides, $\Delta'_{M}\ge \sigma^2$, $r^2(I,\theta)\ge\sigma^2\rho(I)$
and, as $c_1>2\nu$, $\sum_{I\in\mathcal{I}} 
e ^{-c_1\rho(I)/2}\le C_\nu$. Piecing all these together with \eqref{th1_rel12}, we obtain
\begin{align*}
\bar{T}_{32} &=\frac{\sigma^2 \sum_{I\in\mathcal{J}_M^c} 
\big(\mathbb{E}_{\theta}\|\mathrm{P}_{I}\xi\|^4\big)^{1/2}
\big[\mathbb{E}_{\theta}\hat{p}_I\big]^{1/2}}{\Delta'_{M}/3} \\
&\le \frac{3e^{-c_2C_3M/4}}{\alpha t_0}
\sum_{I\in\mathcal{J}^c_M}
\exp\big\{-\tfrac{c_1}{2}\rho(I)+
\tfrac{c_2}{4} \rho(I)
-\tfrac{c_2}{4\sigma^2}r^2(I,\theta)\big\}\notag\\ 
&\le 
\frac{3e^{-c_2C_3M/4}}{\alpha t_0}
\sum_{I\in\mathcal{I}}\exp\big\{-\tfrac{c_1}{2}\rho(I)\big\}
\le
\frac{3C_\nu}{\alpha t_0}\ e^{-c_2C_3M/4}.
\end{align*}

Combining  \eqref{relation11}, \eqref{relation12}, \eqref{bound_T2}, \eqref{relation13},  
\eqref{relation14} and the last  relation finishes the proof of claim \eqref{th1_ii} 
with the constants $M_1=\frac{12c_3(\nu+1)}{\alpha}$, 
$H_1=\max\{C_\nu\big[3(\frac{2c_3}{M_1} +C_3) 
+ \frac{6}{\alpha}+\frac{6}{c_2e}+\frac{3}{\alpha \min\{1/2,c_2/4\}}\big], e^{m_1}\}$ 
and $m_1= \min\{\frac{C_4 }{2},\frac{c_2C_3}{4}\}$.
\end{proof}

\begin{proof}[Proof of Theorem \ref{th2}]  
First  we prove (i). 
Denote $\mathcal{G}_1=\mathcal{G}_1(\theta,M)=\{I \in \mathcal{I}: \, r^2(I,\theta) \ge c_3 r^2(\theta) + M \sigma^2\}$, 
where the constants $c_1>2\nu$, $c_2$, $c_3$ 
are defined in Lemma \ref{main_lemma}. Applying Lemma \ref{main_lemma} 
and Condition \eqref{(A2)}, 
we obtain
\begin{align*}
&\mathbb{E}_{\theta}\hat{\pi}\big(I \in\mathcal{G}_1 \big| Y\big)=
\sum_{I\in\mathcal{G}_1} \mathbb{E}_{\theta}\hat{\pi}(I|Y)\le
e^{-c_2M} \sum_{I\in\mathcal{I}} e^{-c_1\rho(I)}\le C_\nu e^{-c_2 M},
\end{align*}
which completes the proof of (i).

Now  we prove (ii). By Condition \eqref{(A3)}, for any $I,I_1\in\mathcal{I}$ there exists  $I'=I'(I,I_1)\in\mathcal{I}$ 
such that $(\mathbb{L}_{I}\cup\mathbb{L}_{I_1}) 
\subseteq \mathbb{L}_{I'}$. % and $\rho(I')\le \rho(I)+\rho(s(I_1))$.
Fix $I_1\in\mathcal{I}$ and define $\mathcal{G}_2(M,I_1)=\{I\in\mathcal{I}: 
\theta^{T}[\mathrm{P}_{I'}-\mathrm{P}_{I}]\theta\ge\bar{\tau}\rho(I')\sigma^2+M\sigma^2\}$, 
where $\bar{\tau}$ is defined by \eqref{def_bar_tau}.

As $\mathbb{L}_{I}\subseteq \mathbb{L}_{I'}$, by using \eqref{ddm_P(I|Y)}, \eqref{def_bar_tau}
and applying Lemma \ref{secondary_lemma} 
with $h=\alpha$ and $I_0=I'$,  we obtain that, for each $I\in \mathcal{G}_2(M,I_1)$,
\begin{align*}
\mathbb{E}_{\theta}\hat{p}_I &\le 
\big(\tfrac{\lambda_{I}}
{\lambda_{I'}}\big)^{\alpha} 
\exp\big\{-\tfrac{\alpha}{3\sigma^{2}}\big[\theta^{T}(\mathrm{P}_{I'}-\mathrm{P}_{I})\theta\big]
+\rho(I')\big\}\\
&=
\exp\big\{-\varkappa\alpha \rho(I)-\tfrac{\alpha}{3\sigma^{2}}
\big[\theta^{T}(\mathrm{P}_{I'}-\mathrm{P}_{I})\theta\big]
+(1+\varkappa\alpha)\rho(I')\big\}\\
&\le
\exp\big\{-\varkappa\alpha \rho(I)-\big[\tfrac{\alpha\bar{\tau}}{3}
-(1+\varkappa\alpha)\big] \rho(I')
-\tfrac{\alpha}{3}M\big\}\\
&=e^{-\varkappa\alpha \rho(I)-\tfrac{\alpha}{3}M}.
\end{align*}

Since $\varkappa\ge\alpha^{-1}\nu$, by Condition \eqref{(A2)} we have that 
$\sum_{ I\in\mathcal{I}}e^{-\varkappa\alpha \rho(I)}\le C_\nu$. 
This relation and the last display imply that, with $m_0'=\alpha/3$,
\begin{align}
\label{th2_iia}
\mathbb{E}_{\theta} \hat{\pi}\big(I\in \mathcal{G}_2(M,I_1)|Y\big)
=\sum_{I\in \mathcal{G}_2(M,I_1)} \mathbb{E}_{\theta}\hat{p}_I\le 
C_\nu\exp\big\{-m_0'M\big\}.
\end{align}

Now take $I_1=I_*$ defined by \eqref{I_*}.
By  Condition \eqref{(A3)}  there exists  $I'(I, I_*)\in\mathcal{I}$ 
such that $(\mathbb{L}_{I}\cup\mathbb{L}_{I_*})\subseteq \mathbb{L}_{I'}$ and 
$\rho(I')\le \rho(I)+\rho(I_*)$. If $\rho(I)\le \delta\rho(I_*) - M$, then 
$\rho(I')\le \rho(I)+\rho(I_*)\le (1+\delta) \rho(I_*) -M$. Hence, 
$\rho(I_*) \ge \tfrac{1}{1+\delta}\rho(I')+\tfrac{M}{1+\delta}$ and 
$\mathrm{P}_{I'}\ge \mathrm{P}_{I_*}$, which, together with %\eqref{def_I*}, 
the definition of the $\tau$-oracle, imply
\begin{align*}
%\label{proof_th2_ii}
\theta^{T}&[\mathrm{P}_{I'}-\mathrm{P}_{I}]\theta \ge\theta^{T}[\mathrm{P}_{I_*}
-\mathrm{P}_{I}]\theta\ge\tau_0\sigma^{2}[\rho(I_*)-\rho(I)]\notag\\
&\ge\tau_0\sigma^{2}(1-\delta)\rho(I_*)+\tau_0 M\sigma^{2}\ge
\tfrac{1-\delta}{1+\delta} \tau_0\sigma^2 \rho(I')+\tau_0M \sigma^2 
\ge \bar{\tau} \sigma^2 \rho(I')
+\tau_0M \sigma^2 ,
\end{align*}
as $\tfrac{1-\delta}{1+\delta} \tau_0> \bar{\tau}$ by the definition \eqref{I_*} 
of $I_*$. It follows that $\{I\in\mathcal{I}: \rho(I)\le \delta\rho(I_*) - M\} 
\subseteq \mathcal{G}_2(\tau_0 M, I_*)$.
Thus, we obtain 
\begin{align*}
\mathbb{E}_{\theta}\hat{\pi}\big(I\in\mathcal{I}: \rho(I)\le 
\delta\rho(I_*) - M \big|Y\big)
%\\& \le\mathbb{E}_{\theta}\hat{\pi}\big(I\in\mathcal{I}: \theta^{T}[\mathrm{P}_{I'}-\mathrm{P}_{I}]\theta 
%\ge \tau'\sigma^2\rho(I')+\tfrac{2\tau_0}{1+\delta} \sigma^2 M |Y\big)\\&
\le 
\mathbb{E}_{\theta}\hat{\pi}\big(\mathcal{G}_2(\tau_0 M, I_*)|Y\big).
\end{align*}
The last relation and \eqref{th2_iia} %with $I_1=I_*$ 
imply claim (ii) with
%$\alpha_1' = (\tau'-\bar{\tau})c_0\alpha/3>0$ and 
$m'_1=\tau_0 m'_0=\tau_0\alpha/3$.

Finally, we prove (iii). Condition (A3') implies that $ \mathbb{L}_{I'} =\mathbb{L}_{I_o}+\mathbb{L}_I
=\mathbb{L}_{I_o} \oplus(\mathbb{L}_I \cap\mathbb{L}^\perp_{I_o})$.  
%We have $\mathbb{L}_I=\mathbb{L}_I \cap\mathbb{L}_{I_o}
%+\mathbb{L}_I\cap\mathbb{L}_{I_o}^\perp=\mathbb{L}_I\cap\mathbb{L}_{I_o}
%+\mathbb{L}_{I''}$.   
%By Condition \eqref{(A3)}, there exists $I'\in\mathcal{I}$  
%such that $\mathbb{L}_{I_o} \cup \mathbb{L}_I'' \subseteq \mathbb{L}_{I'}$
%and $ \rho(I')\le \rho(I_o)+\rho(s(I''))$.
%Now, by the complementary relation of Condition \eqref{(A3)},� there exists $I'''\in\mathcal{I}$ such that 
%$\mathbb{L}_{I'} \cap \mathbb{L}^\perp_{I_o} = \mathbb{L}_{I'''}$.
%Then $\mathbb{L}_{I'} = \mathbb{L}_{I_o} \cup \mathbb{L}_{I'''}$ is a partition so that 
%$ \rho(I')\le \rho(I_o)+\rho(s(I'''))$.
If the  inequality $\sigma^2\rho(I)<\|\mathrm{P}_{\mathbb{L}_I 
\cap\mathbb{L}^\perp_{I_o}}\theta\|^2$ would hold, then  
%then using this and previous relations we have that 
%By the oracle definition \eqref{oracle}, we have that for any $I\in\mathcal{I}$
%\[
%\|\mathrm{P}_I \theta\|^2-\|\mathrm{P}_{I_o} \theta\|^2\le\sigma^2\big(\rho(I) - \rho(I_o)\big). 
%\]
\begin{align*}
r^2(I',\theta) 
&= \|\theta-\mathrm{P}_{I'}\theta\|^2+\sigma^2 \rho(I')\\
& \le\|\theta-(\mathrm{P}_{I_o}
+\mathrm{P}_{\mathbb{L}_I \cap\mathbb{L}^\perp_{I_o}})\theta\|^2+\sigma^2( \rho(I_o)+\rho(I)\\
&< 
\|\mathrm{P}_{\mathbb{L}_I \cap\mathbb{L}^\perp_{I_o}}\theta\|^2+
\|\theta-(\mathrm{P}_{I_o}+\mathrm{P}_{\mathbb{L}_I \cap\mathbb{L}^\perp_{I_o}})\theta\|^2+\sigma^2 \rho(I_o)\\
&= \|\theta-\mathrm{P}_{I_o}\theta\|^2+\sigma^2 \rho(I_o)=r^2(\theta),
\end{align*}
which contradicts the definition of the oracle. Hence, 
$\|\mathrm{P}_{\mathbb{L}_I \cap\mathbb{L}^\perp_{I_o}}\theta\|^2\le \sigma^2\rho(I)$.

Take $I_0\in\mathcal{I}$ such that $\mathbb{L}_{I_0}=\mathbb{L}_I \cap\mathbb{L}_{I_o}$.
Using $\varkappa\ge \frac{2\nu+2\alpha+3}{2\alpha}$, the fact that
 $\theta^T(\mathrm{P}_{\mathbb{L}_{I}}-\mathrm{P}_{I_0})\theta
=\|\mathrm{P}_{\mathbb{L}_I\cap\mathbb{L}_{I_o}^\perp}\theta\|^2\le \sigma^2\rho(I)$
and Lemma \ref{secondary_lemma} 
(in case $\mathbb{L}_{I_0}\subseteq \mathbb{L}_{I}$) with $h=\alpha$, 
we obtain for each $I\in\mathcal{G}_0=\{I\in\mathcal{I}: \rho(I) \ge 
M'_0 \rho(I_0)+M\}$ with $M'_0=2\varkappa\alpha$,
\begin{align*}
\mathbb{E}_\theta&\hat{p}_I 
\le\big(\tfrac{\lambda_I}{\lambda_{I_0}}\big)^\alpha
\exp\big\{\alpha\sigma^{-2}\theta^T(\mathrm{P}_{\mathbb{L}_{I}}
-\mathrm{P}_{\mathbb{L}_{I_0}})\theta+\rho(I)\big\}\\
&\le
\exp\big\{-(\varkappa\alpha-\alpha-1)\rho(I)+\alpha \varkappa\rho(I_0)\big\}\\
&\le
\exp\big\{-(\nu+\tfrac{1}{2})\rho(I)+\alpha \varkappa\rho(I_0)\big\}\\
&\le 
\exp\big\{-\nu\rho(I)-(\tfrac{M'_0}{2}-\varkappa\alpha)
\rho(I_0)-\tfrac{M}{2}\big\}
= %e^{-\alpha'_0|I_0| \log(\tfrac{en}{|I_0|})-m'_0M}
e^{-\nu\rho(I)-M/2}.
\end{align*}

Combining the last display with Condition \eqref{(A2)} completes the proof:
\begin{align*} 
\mathbb{E}_{\theta}\hat{\pi}\big(I \in \mathcal{I}: \rho(I)\ge  
M'_0 \rho(I_o)+M \big| Y\big)&\le 
\mathbb{E}_\theta\hat{\pi}(I \in\mathcal{G}_0| Y)\\ \vspace{-\baselineskip} 
&=
\sum_{I \in\mathcal{G}_0}\mathbb{E}_\theta\hat{p}_I\le  
C_\nu e^{-M/2}. \qedhere
\end{align*} 
\end{proof}

\begin{proof}[Proof of Theorem \ref{th3}]  
We first establish the coverage property. The constants $M_1$, $H_1$ and $m_1$ 
are defined in Theorem \ref{th1}. Take $M_2=\tfrac{M_1}{\delta}$ where $\delta\in(0,1)$ is from 
\eqref{I_*}. 
From  \eqref{oracle}, it follows that 
$r^2(\theta)\le r^2(I_*,\theta)=(b(\theta)+1)\sigma^2 \rho(I_*)+b(\theta)\sigma^2\le(b(\theta)+1)
\sigma^2(\rho(I_*)+1)$,
where $b(\theta)$ is given by \eqref{def_b}.
Combining this with the claim \eqref{th1_ii} from Theorem \ref{th1}, the claim (ii) 
from Theorem \ref{th2} %(and Remark \ref{rem1a}), 
and the definition \eqref{check_r} of $\hat{r}$ yields the coverage property:
\begin{align*}
&\mathbb{P}_\theta\big(\theta  \notin B(\hat{\theta},[(b(\theta)+1)M_2\hat{r}^2
+(b(\theta)+2)M\sigma^2]^{1/2})\big) \\
%&=\mathbb{P}_\theta\big(\|\hat{\theta}-\theta\|^2>(b_\tau(\theta)+1)(\tfrac{M_1}{\varrho}
%\hat{r}^2+M\sigma^2)\big) \\
&\le\mathbb{P}_\theta\big(\|\hat{\theta}\!-\!\theta\|^2>(b(\theta)\!+\!1)M_2\hat{r}^2
\!+\!(b(\theta)\!+\!2)M\sigma^2, \hat{r}^2\ge \delta\sigma^2\rho(I_*)
\!+\!\sigma^2\!-\!\tfrac{M\sigma^2}{M_2}\big) \\
&\qquad +\mathbb{P}_\theta\big(\hat{r}^2<\delta \sigma^2\rho(I_*)
+\sigma^2-\tfrac{M\sigma^2}{M_2}\big)\notag\\
&\le \mathbb{P}_\theta\big(\|\hat{\theta}-\theta\|^2> M_1 r^2(\theta)+M\sigma^2\big)
+ \mathbb{P}_\theta\big(\rho(\hat{I})< 
\delta\rho(I_*)-\tfrac{M}{M_2}\big)\\ 
&\le H_1 e^{-m_1M}
+C_\nu e^{-m''_1M}
\le  H_2 e^{-m_2M},
\end{align*}
where  $m''_1= m'_1/M_2$, $H_2=H_1+C_\nu$, 
$m_2=m_1\wedge m''_1$; $m'_1$ is defined in Theorem \ref{th2}.
%The first claim follows with $M'_2 = M_2 (t+\tau)$, $H'_2 = H_0 + H_2$, $m'_2= \min\{m_0, m_2/(t+\tau)\}/2$.   
Since $b(\theta)\le t$ for all $\theta\in\Theta_{\rm eb}(t)$, the coverage relation follows.

Let us show the size property. 
For $M\ge 0$, introduce the set  $\mathcal{G}(M)=\mathcal{G}(M,\theta)=\{ I\in \mathcal{I}: 
\sigma^2\rho(I)\ge c_3 r^2(\theta) + M\sigma^2\}$,
where $c_3$ is defined in Lemma \ref{main_lemma}.
Then for all $I \in \mathcal{G}(M)$,
\begin{align*}
r^2(I,\theta) -c_3 r^2(\theta)
&\ge \sigma^2\rho(I)-c_3 r^2(\theta)\ge M\sigma^2.
\end{align*}
%In view of Remark \ref{rem1}, 
Remind the notation $\hat{p}_I=\check{\pi}( I|Y)=\mathrm{1}\{I=\hat{I}\}$ 
defined by \eqref{ms_ddm_I}. From Lemma \ref{main_lemma} and the last relation, it follows that 
for all $I \in \mathcal{G}(M)$
\begin{align*}
\mathbb{E}_{\theta}\hat{p}_I
&\le 
\exp\big\{-c_1\rho(I)-c_2\sigma^{-2} \big[r^2(I,\theta) -c_3 r^2(\theta)\big]\big\}\le  
e^{-c_1\rho(I)-c_2M}.
\end{align*}
The last display implies that, for any $\theta\in\Theta$,
\begin{align*}
&\mathbb{P}_\theta(\hat{r}^2 \ge c_3 r^2(\theta)+(M+1)\sigma^2 )=\mathbb{P}_\theta(\sigma^2\rho(\hat{I}) \ge c_3 r^2(\theta)+M\sigma^2 )\\ 
&\le\sum_{I \in \mathcal{G}(M)} \mathbb{E}_{\theta}\hat{p}_I\le e^{-c_2 M} \sum_{I\in\mathcal{I}} e^{-c_1\rho(I)}\le H_3 e^{-c_2M}, 
\end{align*}
because $\sum_{I\in\mathcal{I}} 
e^{-c_1\rho(I)}\le C_\nu$ in view of Condition \eqref{(A2)} as $c_1>2\nu$.
The size relation follows with  $M_3=c_3, H_3=C_\nu $ and $m_3=c_2$.

If, instead of Condition \eqref{(A3)}, stronger Condition {\rm(A3')} is fulfilled, 
then the stronger version of the size relation follows immediately from 
property (iii) of Theorem \ref{th2}:
$\mathbb{P}_\theta\big(\hat{r}^2\ge M'_0\sigma^2\rho(I_o)+(M+1)\sigma^2\big)
\le  C_\nu e^{-M/2}$, 
where the constants $M'_0$ and $C_\nu$ are defined in Theorem \ref{th2}.
%Let us show the size property. 
%In view of Remark \ref{rem1}, from the definition \eqref{check_r} of 
%$\hat{r}$ and the property (i) of Theorem \ref{th2}, it follows that  for any $\theta\in\Theta$,
%\begin{align*}
%&\mathbb{P}_\theta\big(\hat{r}^2\ge 
%\sigma^2 M'_0\rho(s_o)+(M+1)\sigma^2\big)
%=\mathbb{P}_\theta\big(\rho(\hat{s})\ge  
%M'_0 \rho(s_o)+M \big)\\ 
%&\le\mathbb{P}_\theta\big(\rho(\hat{s})\ge  
%M'_0 \rho(s_0)+M \big)
%=\mathbb{E}_{\theta}\hat{\pi}\big(s\in\mathcal{S},  I \in \mathcal{I}_s: \rho(s)\ge  
%M'_0 \rho(s_0)+M \big| Y\big)\\
%&\le C_\nu e^{-M},
%\end{align*}
%where $(s_0, I_0)$ is defined in the property (i) of Theorem \ref{th2}.
%The size relation follows.
\end{proof}

\begin{proof}[Proof of Theorem \ref{th4}]   
Since $Y'=\mathrm{P}_{I^*} \theta+\xi'$, we rewrite \eqref{radius2} as 
\begin{align}
&\tilde{R}^2_M=\big(\|Y'-\hat{\theta}\|^2 -\sigma^2V(Y',Y)  
+2\sigma^2G_M\sqrt{N}\big)_+ \notag\\
&=\big(\|\theta-\hat{\theta}\|^2+\sigma^2\big(\|\xi'\|^2-V(Y',Y)\big)+2\sigma 
\langle\xi',(\theta-\hat{\theta})\rangle+2\sigma^2G_M\sqrt{N}\big)_+.
\label{R_M}
\end{align}

Introduce the events $D_M=D_M(\theta)=
\big\{\|\hat{\theta}-\theta\|^2\ge M_1r^2(\theta) +M\sigma^2\big\}$ and 
$E_M=E_M(\theta)=\big\{2|\langle\xi',(\theta-\hat{\theta})\rangle|\ge
\sqrt{M(M_1r^2(\theta) +M\sigma^2)} \big\}$. 
According to Condition \eqref{cond_A4}, $\hat{\theta}$ and $\hat{I}$ are based on $Y$ 
and independent of $\xi'$.
Using this fact, the first relation from \eqref{cond_A4} and Theorem \ref{th1}, 
we obtain that
\begin{align}
\label{expression}
&\mathbb{P}_\theta(E_M)=
\mathbb{E}_\theta\mathbb{P}_\theta(E_M\cap D_M^c|Y)
+\mathbb{P}_\theta(E_M\cap D_M)\notag\\
&\le \mathbb{E}_\theta \Big[\psi_1\big(\tfrac{M(M_1r^2(\theta)+M\sigma^2)}{4\|\hat{\theta}-\theta\|^2}\big)
\mathrm{1}_{D_M^c}\Big]+\mathbb{P}_\theta(D_M)\le 
\psi_1(M/4)+H_1e^{-m_1M}.
\end{align}
Since, by \eqref{dim_N}, $r^2(\theta)\le\sigma^2N$, the event $E_M^c$ implies 
that $2\sigma\langle\xi',(\theta-\hat{\theta})\rangle>
-\sigma\sqrt{M(M_1 \sigma^2N +M\sigma^2)} \ge -\sigma^2 G_M \sqrt{N}$.
Combining this with \eqref{R_M}, \eqref{expression} and the second relation from \eqref{cond_A4}
yields the coverage property:
\begin{align*}
\mathbb{P}_\theta& \big(\theta  \notin B(\hat{\theta},\tilde{R}_M)\big)
=\mathbb{P}_\theta\big(\theta  \notin B(\hat{\theta},\tilde{R}_M),E_M^c\big)
+\mathbb{P}_\theta\big(\theta\notin B(\hat{\theta},\tilde{R}_M),E_M\big)\\
&\le \mathbb{P}_\theta \big(\|\theta-\hat{\theta}\|^2\ge \tilde{R}^2_M,E_M^c\big)
\!+\!\mathbb{P}_\theta(E_M)\\  
&\le\mathbb{P}_\theta\big(0\ge\sigma^2(\|\xi'\|^2\!-V(Y'))+\sigma^2G_M \sqrt{N}\big)\!
+\!\mathbb{P}_\theta(E_M)\\
&\le\mathbb{P}_\theta\big(\|\xi'\|^2\!-V(Y',Y)\le\!-\!M\sqrt{N} \big)
+\psi_1(M/4)+H_1e^{-m_1M}\\  
&\le \psi_2(M)\!+\!\psi_1(M/4)\!+\!H_1e^{-m_1M}.
\end{align*}
	
Let us show the size property. By \eqref{expression}, 
$\mathbb{P}_\theta\big(2\sigma\langle\xi',(\theta-\hat{\theta})\rangle\ge 
\sigma^2G_M\sqrt{N}\big) \le
\mathbb{P}_\theta\big( 2\langle\xi',(\theta-\hat{\theta})\rangle>
\sqrt{M(M_1r^2(\theta) +M\sigma^2)}\big)\le \mathbb{P}_\theta(E_M)
\le\psi_1(M/4)+H_1e^{-m_1M}$. This, Theorem \ref{th1} and \eqref{R_M} imply
%By applying the inequality $\mathbb{P}(A+B\ge a+b)\le \mathbb{P}(A\ge a)+\mathbb{P}(B\ge b)$ 
%for any random variables $A,B$ and numbers $a,b \in \mathbb{R}$, we derive that 
\begin{align*}
\mathbb{P}_\theta&\big(\tilde{R}_M^2 \ge g_M(\theta,N)\big) 
\le 
\mathbb{P}_\theta \big( \|\theta-\hat{\theta}\|^2 \ge M_1  r^2(\theta)+M\sigma^2\big) \\
&\quad %\hspace{-2.5cm}
+\mathbb{P}_\theta\big(\sigma^2\big(\|\xi'\|^2\!-V(Y',Y)\big) 
\ge \sigma^2G_M\sqrt{N} \big) %\\& \qquad
+\mathbb{P}_\theta\big(2\sigma\langle\xi',(\theta-\hat{\theta})\rangle\ge 
\sigma^2G_M\sqrt{N}\big)  \\
& \le H_1e^{-m_1M} +\psi_2(M)+\psi_1(M/4)+H_1e^{-m_1M}. \qedhere
\end{align*}  
\end{proof}

\section{Appendix: empirical Bayes posterior construction}
\label{sec_empirical_Bayes}
Here we present the detailed construction of the DDMs for $\theta$ and $I$, 
as the result of an \emph{empirical Bayesian approach}, based on certain (mixture of) normal prior and 
normal model, although the true model is not known. These DDMs are explicitly 
constructed below as empirical Bayes posteriors, with some links to  the penalization method. 
As we show below, these posteriors (and the derived quantities: estimators for $\theta$ and 
structure selector $\hat{I}$) handle all models and structures for which the statistical dimensions 
$(d_I)_{I\in\mathcal{I}}$ from Condition \eqref{cond_nonnormal} satisfy $d_I \gtrsim \dim(\mathbb{L}_I)$, 
$I\in\mathcal{I}$ (actually, we can always take a sufficiently large majorant $\rho(I)$ for which this holds, 
but the resulting oracle rate may be too large). This is the case for all the particular models and structures considered in Part II.

\paragraph{Prior: mixture of normals.}
%For the general setting \eqref{model}, we construct an empirical Bayesian procedure and 
%make connection with the penalization method. %for making inference on $\theta$. 
Recall that the true parameter $\theta$ is assumed to be well approximated by its structured version
$\mathrm{P}_{I^*}\theta$ (e.g., $\theta=\mathrm{P}_{I^*}\theta$), 
for some ``true'' structure $I^*\in\mathcal{I}$. The true structure $I^*$ is unknown, so at a later stage 
we will put a prior on the family of structures $\mathcal{I}$. 
For now, given a structure $I$, consider the model 
$Y=\mathrm{P}_I \theta + \sigma\xi$, approximating the original model \eqref{model}, 
where $\mathrm{P}_{I}$ is the projection operator 
onto space $\mathbb{L}_{I}$, and put first an ``unstructured'' prior 
$\Pi$ on the ``unstructured'' $\theta\in\Theta$: 
$\theta \sim \Pi=\mathrm{N}(\mu,\kappa\sigma^2 \mathrm{I})$, 
where $\kappa=e-1$ and the parameter $\mu\in\mathcal{Y}$ is to be chosen by the empirical 
Bayes method later. 
%Recall that $\mathrm{P}_{I}$ is the projection operator onto space $\mathbb{L}_{I}$.
The ``unstructured'' prior $\Pi$ on  $\theta$ leads to the ``structured'' prior 
$\pi_I$ on the ``structured'' $\theta^I\triangleq\mathrm{P}_I\theta$:
%First, we propose a prior on $\theta$ given structure $I$:
\begin{align}
\label{gen_norm_prior}
\pi_{I}(\vartheta)= \mathrm{N}(\mathrm{P}_I\mu, \kappa\sigma^2\mathrm{P}_I),
\quad I \in \mathcal{I}, \quad \kappa=e-1.
\end{align}
In this way, we constructed the conditional prior on $\theta$ given $I$: 
$\theta| I \sim \pi_{I}(\vartheta)$.
The rather specific choice of $\kappa=e-1$ is made only for the sake of clean 
mathematical exposition in later calculations, many other choices are actually possible. 

The next very important step in the Bayesian analysis below is that we use the normal likelihood 
$\ell(\theta, Y)=\bigotimes_{i} \mathrm{N}(\theta_i, \sigma^2)$, whereas the ``true'' model 
$Y \sim \mathbb{P}_{\theta}$ is not assumed to be normal, but only satisfying Condition 
\eqref{cond_nonnormal}. Formally applying Bayesian approach to this prior 
and the normal likelihood $\ell(\theta, Y)$ 
delivers the marginal distribution $Y\sim\mathbb{P}_{Y,I}=\mathrm{N}(\mathrm{P}_{I}\mu,\mathrm{I}
+\kappa\mathrm{P}_{I})$ and the following posterior distribution on $\theta$:
\begin{align}
\label{norm_poster_I}
\pi_{I}(\vartheta|Y)= \mathrm{N} \big(\tfrac{1}{\kappa+1}\mathrm{P}_I\mu
+\tfrac{\kappa}{\kappa+1}\mathrm{P}_I Y,\tfrac{\kappa\sigma^2}{\kappa+1}\mathrm{P}_I\big).
\end{align}
Note that in general the covariance matrix in \eqref{gen_norm_prior} is not invertible, 
but the Bayes formula for the conjugate normal-normal model still holds with the Moore-Penrose inverse
$\mathrm{P}_{I}^-$ of $\mathrm{P}_{I}$ instead of the usual inverse 
(recall that $\mathrm{P}^-=\mathrm{P}$ for any projection operator $\mathrm{P}$).

Let us now put a prior on $I$:  
\begin{align}
\label{prior_lambda}
\lambda_{I} =
c_{\varkappa}e^{- \varkappa \rho(I)} , \quad  I\in\mathcal{I},
\end{align} 
where $c_{\varkappa}$ is the normalizing constant (i.e., $\sum_{I\in\mathcal{I}} \lambda_I=1$), 
$\rho(I)$ satisfies Condition \eqref{(A2)},
the parameter $\varkappa$ satisfies the relation \eqref{cond_technical}.
 
Combining \eqref{gen_norm_prior} and \eqref{prior_lambda} gives the mixture prior on $\theta$: 
$\pi =\sum_{I \in \mathcal{I}} \lambda_{I} \pi_{I}$. This leads to the marginal distribution of $Y$:
$\mathbb{P}_{Y}=\sum_{I \in \mathcal{I}}
\lambda_{I}\mathbb{P}_{Y,I}$, %\quad
$\mathbb{P}_{Y,I}=\mathrm{N}(\mathrm{P}_{I}\mu,\sigma^2(\mathrm{I}
+\kappa\mathrm{P}_{I}))$,
where the density of the distribution $\mathbb{P}_{Y,I}=
\mathrm{N}(\mathrm{P}_{I}\mu,\sigma^2(\mathrm{I}+\kappa\mathrm{P}_{I}))$ is 
\begin{align}
\label{norm_posterior_density}
\varphi(y,\mathrm{P}_{I}\mu,\sigma^2(\mathrm{I}+\kappa\mathrm{P}_{I}))=
\frac{e^{-(y-\mathrm{P}_{I} \mu)^T(\mathrm{I} -\tfrac{\kappa}{\kappa+1}\mathrm{P}_{I})
(y-\mathrm{P}_{I}\mu)/(2\sigma^2)}}
{(2\pi\sigma^2)^{n/2} (1+\kappa)^{\text{dim}(\mathbb{L}_{I})/2}},
\end{align}
because $(\mathrm{I}+\kappa\mathrm{P})^{-1}=\mathrm{I} -\tfrac{\kappa}{\kappa+1}\mathrm{P}$, 
$\text{det}(\mathrm{I}+\kappa\mathrm{P})=(1+\kappa)^{\text{rank}(\mathrm{P})}$ 
for any projection operator $\mathrm{P}$, and $\text{rank}(\mathrm{P}_{I})
=\text{dim}(\mathbb{L}_{I})$. The posterior of $\theta$ becomes  
\begin{align}
\label{norm_gen_posterior}
\pi(\vartheta|Y)=\pi_\varkappa(\vartheta|Y)=\sum_{I \in \mathcal{I}}
\pi(\vartheta,I|Y)=\sum_{I \in \mathcal{I}} \pi_I(\vartheta|Y)\pi(I|Y),
\end{align}
where $\pi_{I}(\vartheta|Y)$ is defined by (\ref{norm_poster_I}) 
and the posterior for $I$ is
\begin{align}
\label{norm_P(I|Y)}
\pi(I|Y)=\frac{\lambda_{I} \mathbb{P}_{Y,I}}
{\sum_{J \in \mathcal{I}}\lambda_{J}\mathbb{P}_{Y,J}}.
\end{align}

\paragraph{Empirical Bayes posterior.} 
The parameter $\mu$ is yet to be chosen in the prior. We apply the empirical Bayes approach. 
The marginal likelihood $\mathrm{P}_{Y}$ is readily maximized with respect to 
$\mu$:
$\argmin_{\mu} \big\{(Y-\mathrm{P}_{I} \mu)^T(\mathrm{I}-\tfrac{\kappa}{\kappa+1}
\mathrm{P}_{I})(Y-\mathrm{P}_{I}\mu)\big\}=Y$.
%\argmin_{\mu_I} \big\{-2Y\mathrm{P}_{I}\mu+\mu_I^T\mathrm{P}_{I}\mu\big\}=\mathrm{P}_{I}Y$. 
%\begin{align*}
%&(Y-\mathrm{P}_{I} \mu)^T(\mathrm{I} -\tfrac{\kappa}{\kappa+1}\mathrm{P}_{I})
%(Y-\mathrm{P}_{I}\mu) 
%=(Y-\mathrm{P}_{I} \mu)^T(\mathrm{I} -\mathrm{P}_{I}+\tfrac{1}{\kappa+1}\mathrm{P}_{I})
%(Y-\mathrm{P}_{I}\mu) 
%\to \min_\mu\\
%&-2Y\mathrm{P}_{I}\mu+\mu^T\mathrm{P}_{I}\mu\to \min_\mu
%\end{align*}
%$\tilde{\mu}=Y$.  
Substituting $Y$ instead of $\mu$ 
in the expressions \eqref{norm_poster_I}, \eqref{norm_gen_posterior} 
and \eqref{norm_P(I|Y)} yields the empirical Bayes posterior
\begin{align}
\label{emp_norm_posterior}
\tilde{\pi}(\vartheta|Y)=\tilde{\pi}_\varkappa(\vartheta|Y)=
\sum_{I \in \mathcal{I}}
\tilde{\pi}(\vartheta|Y,I)\tilde{\pi}(I|Y),
\end{align}
called \emph{empirical Bayes model averaging} (EBMA) posterior,
where the EBMA posterior for $\theta$ given $I$ is
\begin{align}
\tilde{\pi}(\vartheta| Y,I)&=\tilde{\pi}_{I}(\vartheta|Y)
=\mathrm{N}\big(\mathrm{P}_{I}Y,\tfrac{\kappa\sigma^2}{\kappa+1}\mathrm{P}_{I}\big)
\label{emp_poster_I}
\end{align}
and the empirical Bayes posterior for $I$ is
\begin{align}
\label{emp_P(I|Y)}
\tilde{\pi}(I|Y)&=\tilde{\pi}_{I}
=\frac{\lambda_{I}\exp\{-\tfrac{1}{2\sigma^2}[\|(\mathrm{I}-\mathrm{P}_{I})Y\|^2
+\sigma^2\text{dim}(\mathbb{L}_{I})]\}}
{\sum_{J \in \mathcal{I}}\lambda_{J}\exp\{-\tfrac{1}{2\sigma^2}[\|(\mathrm{I}-\mathrm{P}_{J})Y\|^2
+\sigma^2\text{dim}(\mathbb{L}_{J})]\}}.
\end{align} 
When deriving \eqref{emp_P(I|Y)}, we used \eqref{norm_posterior_density}, $\kappa=e-1$ and 
the fact that $(\mathrm{I}-\mathrm{P})(\mathrm{I}-\tfrac{\kappa}{1+\kappa}\mathrm{P})
(\mathrm{I}-\mathrm{P})=(\mathrm{I}-\mathrm{P})$ for any projection operator $\mathrm{P}$. 
Let $\tilde{\mathbb{E}}$ and $\tilde{\mathbb{E}}_{I}$ be the expectations 
with respect to the EBMA measures $\tilde{\pi}(\vartheta|Y)$ and $\tilde{\pi}_{I}(\vartheta|Y)$, 
respectively. Then $\tilde{\mathbb{E}}_{I}(\vartheta|Y)=\mathrm{P}_IY$, $I\in\mathcal{I}$.
Introduce the \emph{EBMA posterior mean} estimator
\begin{align}
\label{estimator_suppl}
\tilde{\theta}&=\tilde{\mathbb{E}}(\vartheta|Y)=\sum_{I \in \mathcal{I}}
\tilde{\mathbb{E}}_{I}(\vartheta|Y) \tilde{\pi}(I|Y)=
\sum_{I \in \mathcal{I}}
(\mathrm{P}_{I}Y) \tilde{\pi}(I|Y).
\end{align}

Consider yet alternative empirical Bayes posterior. First derive an empirical 
Bayes structure selector $\hat{I}$ by maximizing $\tilde{\pi}(I|Y)$ over $ I\in\mathcal{I}$.
This boils down to
\begin{align}
\label{I_MAP}
 \hat{I}&=\argmax_{I\in\mathcal{I}}{\tilde{\pi}}(I|Y)=\argmin_{I\in\mathcal{I}} 
 \big\{\|Y-\mathrm{P}_IY\|^2+\sigma^2\text{pen}(I)\big\},
\end{align}
which is essentially the \emph{penalization method} with the penalty 
$\text{pen}(I)=2\varkappa \rho(I)+\text{dim}(\mathbb{L}_{I})$. 
Note however that,
while the penalization method gives only an estimator, our method also yields   
a posterior.  
Indeed, plugging in $\hat{I}$ (defined by \eqref{I_MAP}) into $\tilde{\pi}_{I}(\vartheta|Y)$ 
defined by \eqref{emp_poster_I} gives the corresponding empirical Bayes posterior, 
called \emph{empirical Bayes model selection} (EBMS) posterior, 
and the \emph{EBMS mean estimator}  for $\theta$:
\begin{align}
\label{emp_emp_posterior}
\check{\pi} (\vartheta|Y)=\tilde{\pi}_{ \hat{I}} (\vartheta|Y)
=\mathrm{N}\big(\mathrm{P}_{\hat{I}}Y,\tfrac{\kappa\sigma^2}{\kappa+1}\mathrm{P}_{\hat{I}}\big), \quad 
\check{\theta}= \check{\mathbb{E}}(\vartheta|Y)= \mathrm{P}_{\hat{I}} Y,
\end{align}
where $\check{\mathbb{E}}$ denotes the expectation with respect to the 
EBMS measure $\check{\pi}(\vartheta|Y)$. Notice that, like \eqref{emp_norm_posterior},
$\check{\pi} (\vartheta|Y)$ defined by \eqref{emp_emp_posterior} 
can also be seen formally as mixture 
\begin{align}
\label{emp_emp_I}
\check{\pi} (\vartheta|Y)=\tilde{\pi}_{ \hat{I}} (\vartheta|Y)=
\sum_{I \in\mathcal{I}} \tilde{\pi}_I(\vartheta|Y)\check{\pi}(I|Y), \quad \check{\pi}(I|Y)=1\{I=\hat{I}\},
\end{align}
where the mixing distribution $\check{\pi}(I|Y)=1\{I=\hat{I}\}$, the 
empirical Bayes posterior for $I$, is degenerate at $\hat{I}$.
In a way, the EBMA posterior  $\tilde{\pi}(\vartheta|Y)$ 
defined by \eqref{emp_norm_posterior} is ``more Bayesian'' than the EBMS posterior 
$\check{\pi}(\vartheta|Y)$ defined by \eqref{emp_emp_I},
although both are formally mixtures.

%\begin{remark}
Now notice that if $d_I \gtrsim \dim(\mathbb{L}_I)$, $I\in\mathcal{I}$, where $d_I$'s are from 
Condition \eqref{cond_nonnormal}, then the terms 
$e^{\dim(\mathbb{L}_I)/2}$ and $e^{\dim(\mathbb{L}_J)/2}$ in the numerator and denominator 
of the right hand side of \eqref{emp_P(I|Y)} can be absorbed into $\lambda_I$ and $\lambda_J$ 
respectively (for example, by making $\rho(I)$ larger by adding a multiple of $d_I$).  
Then \eqref{emp_P(I|Y)} can be expressed in the same form as \eqref{ddm_P(I|Y)}.
Next, in the above construction we used the normal likelihood 
$\ell(\theta, Y)=\bigotimes_{i} \mathrm{N}(\theta_i, \sigma^2)$
(i.e., as if $\xi_i\overset{\rm ind}{\sim}\mathrm{N}(0,1)$), which corresponds to 
$\mathrm{P}_Z = \bigotimes_{i\in[n]}\mathrm{N}(0,1)$ in \eqref{ddm_I}. 
In \eqref{ddm_I}, we need  $\mathrm{P}_Z$ to satisfy Condition \eqref{cond_nonnormal}.
According to Remark \ref{rem_cond_A1}, $\mathrm{P}_Z = \bigotimes_{i\in[n]}\mathrm{N}(0,1)$ 
does satisfy Condition \eqref{cond_nonnormal} with $\alpha=0.4$ and 
$d_I=\dim(\mathbb{L}_I)$, $I\in\mathcal{I}$ (also for $d_I \gtrsim \dim(\mathbb{L}_I)$,
$I\in\mathcal{I}$, with a different $\alpha$). 
This means that in this case the above constructed empirical 
Bayes posteriors and estimators are all particular cases of the corresponding 
DDMs and estimators constructed in Section \ref{subsec_ddm}.

\newpage

\part{\Large %General framework for projection structures: 
Applications}

In Part I we developed a theory for a \emph{general framework for projection structures}. 
and studied the following inference problems within this framework:
the {\em estimation}, {\em DDM contraction}, \emph{uncertainty quantification}, 
and the (weak) \emph{structure recovery} problems. 
In Part II, we apply the developed theory to a number of various models and structures, 
interesting and important on their own right. %which fall in the studied general framework. 
%In this follow up paper we introduce a number of particular models and structures which fall in 
%the general framework studied in \cite{Belitser&Nurushev:2019a}.
%and for which  local and adaptive (global) minimax results 
%(for all the three problems: estimation, posterior contraction rate and confidence sets) 
%can be derived as consequences of our local results for the general framework.
We present a whole avenue of results
(many new ones, some are known in the literature, some are improved) for particular combinations 
of model/structure as consequences of the general framework results from Part I.
Almost all the results on uncertainty quantification (and weak structure recovery) 
are new, we obtain some new results on estimation and DDM (posterior) contraction. 
%The results on the weak structure recovery are also new, they holds without any extra conditions. 
Besides, we obtain stronger versions of many known results on estimation and posterior (DDM) contraction 
since our results are \emph{local}, \emph{refined} (non-asymptotic exponential probability bound) 
and hold in the \emph{distribution-free} setting. 
%Sometimes by using our obtained local rates, we improve upon some global rates from the literature. 

Actually, we have not spelled out all the results in the form of theorems, but have done all the preparatory 
work so that  the reader should be able to formulate formal assertions when desired. 
For some models, the preparatory work is more elaborate as it involves transforming the original 
data and the use of some tools from the literature. This is the case for the density estimation 
and covariance matrix estimation problems.   
%Because of the page limit, computations for some models and structures 
%are moved to Supplement  \cite{Belitser&Nurushev:2019e}.  
%This especially concerns the models and structures considered 
%in our earlier papers. Note however that also for those cases in this paper we sometimes 
%present stronger versions of the results, when they are obtained in a more refined form 
%and in more general settings.
 
%In what follows, we use the notations and cross-references 
%to numbered elements (like equations, sections) from the previous paper 
%\cite{Belitser&Nurushev:2019a}. 
%To avoid  confusion in numbering displays and to link the two papers, 
%we continue the numbering of sections: there are 5 sections in 
%the first paper \cite{Belitser&Nurushev:2019a}, so we start  this paper with Section \ref{applications}.

\section{Combinations model/structure}

There are numerous examples of combinations model/structure falling into our general framework.
The full list of the combinations %model/structure 
considered in this paper is given in Section \ref{subsec_general_framework}. 
Almost all the studied cases result from combining the 4 basic models (\emph{singal+noise}, 
\emph{linear regression}, \emph{matrix+noise} and \emph{matrix linear regression}) with the 4 basic 
structures (\emph{smoothness}, \emph{sparsity}, \emph{clustering} and \emph{shape restriction}). 
Table \ref{table1} gives an overview of the sections dedicated to corresponding combinations. 
\begin{table}[]
\caption{Considered combinations  model/structure (here: mixt. = mixture; 
DL= dictionary learning).}
\label{table1}
%\caption{Classification of model/structure}
\begin{tabular}{|l|c|c|c|c|cc}
\hline %\cline{1-5}
\textit{\textbf{model \textbackslash structure}} &\textit{smoothness} &\textit{sparsity}&\textit{clustering} &\textit{shape restr.}  & \multicolumn{1}{c|}{\textit{mixt.}*} & \multicolumn{1}{c|}{\textit{DL}*} \\  \hline %\cline{1-5}
\textit{signal+noise} & \ref{sec_smoothness}, \ref{sec_graph_with_smoothness}, \ref{sec_density}*& \ref{sec_regression_wavelet}*, \ref{sec_sparsity}& \ref{sec_multi-level} &\ref{sec_shape_restr} & & \\   \cline{1-5}%\hline
\textit{linear regression} & & \ref{sec_lin_regr_sparsity} &   & \ref{sec_aggregation}& & \\ \cline{1-5} %\hline
\textit{matrix+noise} & \ref{subsec_banding}* & \ref{subsec_cov_sparsity}*, \ref{sec_matrix_sparsity} & 
\ref{sec_biclustering}  &  & & \\ \hline
\textit{matrix linear regr.} &  &\ref{sec_reg_group_sparsity}&\ref{sec_group_clustering} & 
& \multicolumn{1}{c|}{\ref{sec_mixture_model}} & \multicolumn{1}{c|}{\ref{sec_dict_learning}} 
\\ \hline
%\textit{regr.\ under wavelet basis}  &\multicolumn{2}{c|}{Sec. \ref{sec_regression_wavelet}*} & &  \\ \hline
\end{tabular}
\end{table}

We now comment on the items marked by * in  Table \ref{table1}. 
In Section \ref{sec_regression_wavelet}, the regression model 
under wavelet basis is considered, which can be reduced to the signal+noise model where a ``better'' choice 
of structure would be some combination of smoothness and sparsity structures (not just sparsity, see Remark \ref{rem_better_structure}).  
%(or, as matrix+noise with sparsity structure); 
In Sections \ref{sec_density} (density estimation) and \ref{sec_covariance_matrix} (covariance matrix estimation), 
we had to transform the original data in order to obtain the resulting signal+noise and matrix+noise models, respectively. 
Next, we needed to employ certain additional tools to derive the results for these models, see 
Sections \ref{sec_density} and \ref{sec_covariance_matrix} for details. 
Finally, the structures \emph{mixture} and \emph{dictionary learning} are only possible for the 
matrix linear regression model. In both cases, the vector $\beta$ and the design matrix 
(in the matrix linear regression) are both unknown, but in the mixture structure we impose some 
structure on the design matrix, whereas in the dictionary learning structure we impose some structure 
on the vector $\beta$. 

Further, one can perform the computations for empty boxes in the table and derive the corresponding 
results for those cases as well, or come up with new models and/or structures, for example, by combining 
the basic 4 structures from Table \ref{table1}.

\section{Specifying the general results to particular models, structures and scales}
\label{applications}

For each particular model and a particular structure, 
we specify the structures $\mathcal{I}$, the corresponding linear spaces 
$\{\mathbb{L}_I, I\in \mathcal{I}\}$, and the majorant $\rho(I)$;
next we verify Conditions \eqref{cond_nonnormal}, \eqref{(A2)}, \eqref{(A3)} and \eqref{cond_A4}
for these quantities. %whenever appropriate (depending on which claim we want to establish). 
For almost all cases, 
we construct the majorant $\rho(I)$ according to Remark \ref{rem5}: 
$\rho(I) \ge d_I+ \log|\mathcal{I}_{s(I)}|$ for some parsimonious structural slicing mapping 
$s: \mathcal{I}\mapsto\mathcal{S}$ (i.e., $d_I=d_J$ for all $I,J\in\mathcal{I}_s$, $s\in\mathcal{S}$). 
 In view of Remarks \ref{rem_cond_A1} and \ref{rem15}, Conditions \eqref{cond_nonnormal} 
and \eqref{cond_A4} hold with $d_I=\dim(\mathbb{L}_I)$ in all %particular 
models with $\xi_i \overset{\rm ind}{\sim}\mathrm{N}(0,1)$. 
Hence, we will not verify Conditions \eqref{cond_nonnormal} and \eqref{cond_A4} 
for the models where $\xi_i \overset{\rm ind}{\sim}\mathrm{N}(0,1)$.

We keep the same notation for all the quantities involved as for the general framework from Part I, 
with the understanding that these are specialized for the particular models and structures, 
and some constants must be adjusted. Let us first summarize the results of 
Theorems \ref{th1}, \ref{th2}, \ref{th3} and \ref{th4} by the following corollary.

\begin{corollary}
\label{meta_theorem}
Let Conditions \eqref{cond_nonnormal} and \eqref{(A2)} be fulfilled. Then for any $M\ge 0$
\begin{align*}
\tag{i}
&\sup_{\theta\in\Theta}\mathbb{E}_{\theta}\hat{\pi}\big(\|\vartheta-\theta\|^2
\ge M_0r^2(\theta)+M\sigma^2 | Y\big)  \le H_0 e^{-m_0 M}, 
\\&
\tag{ii}\sup_{\theta\in\Theta}\mathbb{P}_\theta\big(\|\hat{\theta}-\theta\|^2 
\ge M_1 r^2(\theta)+M\sigma^2\big)\le H_1 e^{-m_1 M},\\
%	\tag{iii}
%	&\sup_{\theta\in\Theta}\mathbb{E}_{\theta}\hat{\pi}\big(s\in\mathcal{S},  I \in \mathcal{I}_s: \rho(I)\ge  
%	M'_0 \rho(s_0)+M \big| Y\big)\le C_\nu %(\tfrac{ne}{|I_o|})^{-\alpha'_0|I_o|} 
%	e^{-M},\\
%	\tag{iv}
%	&\sup_{\theta\in\Theta}\mathbb{E}_{\theta}\hat{\pi}\big(s\in\mathcal{S},  I \in \mathcal{I}_s: \theta^{T}[\mathrm{P}_{I'}\!-\!\mathrm{P}_{I}]\theta\ge \sigma^2(\bar{\tau}\rho(s')\!+\!M)\big|Y\big)\!\le\!  C_\nu e^{-m'_0M},\\
\tag{iii}
&\sup_{\theta\in\Theta}\mathbb{E}_{\theta}\hat{\pi}\big( I\in \mathcal{I}: r^2(I,\theta)
\ge c_3 r^2(\theta) + M\sigma^2|Y\big)
\le C_\nu e^{-c_2 M},\\
\tag{iv}
&\sup_{\theta\in\Theta}\mathbb{P}_\theta\big(\hat{r}^2\ge M_3 r^2(\theta)+(M+1)\sigma^2\big)
\le  H_3 e^{-m_3 M}.
\end{align*}

If in addition Condition\eqref{(A3)} is fulfilled, then for any $M,t\ge 0$
\begin{align*}
\tag{v}
&\sup_{\theta\in\Theta_{\rm eb}(t)} \mathbb{P}_\theta
\big(\theta\notin B(\hat{\theta},\hat{R}_M)\big) 
\le H_2 e^{-m_2M}.
\end{align*}

If in addition Condition {\rm \eqref{cond_A4}} is fulfilled, then for any $M\ge 0$, %the claims $(i), (ii), (iii)$ of  Corollary \ref{meta_theorem} hold true and
\begin{align*}
\tag{vi}
\sup_{\theta\in\Theta} \mathbb{P}_\theta
\big(\theta\notin B(\hat{\theta},\tilde{R}_M)\big) 
&\le \psi_1(M/4)+\psi_2(M)+H_1e^{-m_1M},\\
\tag{vii}
\sup_{\theta\in\Theta}  \mathbb{P}_\theta\big(\tilde{R}_M^2\ge g_M(\theta,N)\big)
&\le\psi_1(M/4)+\psi_2(M)+2H_1e^{-m_1M}.
\end{align*}
\end{corollary}
\begin{remark}
The properties (ii) and (iii) of Theorem \ref{th2} can also be included in Corollary \ref{meta_theorem}, 
but we omit them, because these properties are only auxiliary results used for proving 
the size relations of Theorem \ref{th3}. 
If additionally Condition (A3') is assumed for the property (iv), then the stronger uniform 
version of (iv) holds:  $\sup_{\theta\in\Theta}\mathbb{P}_\theta\big(\hat{r}^2\ge 
M'_0\rho(s(I_o))\sigma^2+(M+1)\sigma^2\big)\le  C_\nu e^{-M/2}$.
Claim (v) of Corollary \ref{meta_theorem} can be formulated 
for the local version of coverage relation of  Theorem \ref{th3} in terms of $b(\theta)$ (given by
\eqref{def_b}) if needed. 
\end{remark}

%In case of radius of confidence ball of   order $r(\theta)+\sigma N^{1/4}$ the main results of this paper 
%can be summarized  in  the following corollary.  
%\begin{corollary}
%\label{meta_theorem_full}
%Let conditions of Theorems \ref{th1}, \ref{th2} and \ref{th4} be fulfilled.
%Then for any $M\ge 0$, the claims $(i), (ii), (iii)$ of  Corollary \ref{meta_theorem} hold true and
%\begin{align*}
%\tag{vi}
%\sup_{\theta\in\Theta} \mathbb{P}_\theta
%\big(\theta\notin B(\hat{\theta},\tilde{R}_M)\big) 
%&\le \psi_1(M/4)+\psi_2(M)+H_1e^{-m_1M},\\
%\tag{vii}
%\sup_{\theta\in\Theta}  \mathbb{P}_\theta\big(\tilde{R}_M^2\ge g_M(\theta,N)\big)
%&\le\psi_1(M/4)+\psi_2(M)+2H_1e^{-m_1M}.
%\end{align*}
%\end{corollary}

Consider scales of classes $\{\Theta_\beta,\, \beta\in\mathcal{B}\}$, 
where $\beta\in\mathcal{B}$ is the structural parameter, for instance,
$\beta$ could measure the amount of \emph{smoothness} or \emph{sparsity}
of $\theta\in\Theta_\beta$.
The above local results imply adaptive (global) minimax results for estimation and
posterior contraction rate problems  over \emph{all} scales 
$\{\Theta_\beta,\, \beta\in\mathcal{B}\}$ at once, whose minimax rate 
\[
r^2(\Theta_\beta)\triangleq \inf_{\tilde{\theta}}\sup_{\theta\in\Theta_\beta}
\mathbb{E}_\theta \|\tilde{\theta}-\theta\|^2
\]
is bounded from below by a multiple of the local rate, namely 
\begin{align}
\label{oracle_minimax}
r^2(\Theta_\beta)\ge cr^2(\theta) \quad \text{for all} \;\; \theta\in\Theta_\beta, \; \beta\in\mathcal{B}.  
\end{align}

\begin{remark}
\label{rem_appr_term}
Typically, \eqref{oracle_minimax} is established by comparing the oracle rate with the rate 
for some appropriately chosen structure $I^*=I^*(\theta)$. The reasoning goes usually as follows: 
first show that $\sup_{\theta\in\Theta_\beta}\|\theta-\mathrm{P}_{I^*} \theta\|^2\lesssim r^2(\Theta_\beta)$ 
and $\sigma^2\rho(I^*)\lesssim r^2(\Theta_\beta)$, then argue
$r^2(\theta)\le r^2(I^*,\theta)=\|\theta-\mathrm{P}_{I^*} \theta\|^2+\sigma^2\rho(I^*)
\lesssim r^2(\Theta_\beta)$ uniformly in $\theta\in\Theta_\beta$.
Often $I^*$ is the so called ``true structure'', i.e., $\theta\in\mathbb{L}_{I^*}$, then $r^2(I^*,\theta)=\sigma^2\rho(I^*)\lesssim r^2(\Theta_\beta)$.  
\end{remark}

If \eqref{oracle_minimax} holds, we say that the oracle $r^2(\theta)$ \emph{covers} the scale 
$\{\Theta_\beta,\, \beta\in\mathcal{B}\}$. Under \eqref{oracle_minimax},
 the adaptive (with respect to the structural 
parameter $\beta\in\mathcal{B}$) minimax result follows immediately from Theorem \ref{th1}:
$\sup_{\theta\in\Theta_\beta} \mathbb{P}_\theta\big(\|\hat{\theta}-\theta\|^2 \ge 
\tfrac{M_1}{c} r^2(\Theta_\beta)+M\sigma^2\big)\le H_1 e^{-m_1 M} $. 
Moreover, Theorems \ref{th1} and  \ref{th3} imply the minimax versions of the posterior contraction 
result, the estimation result and the size relation in the uncertainty quantification problem, 
which are summarized by the following corollary.
%and the same type of global minimax results 
%on contraction posterior rates as in \cite{Gao&vanderVaart&Zhou:2015}.
\begin{corollary}
\label{minimax_results}
Let  \eqref{oracle_minimax}, Conditions \eqref{cond_nonnormal} and \eqref{(A2)}
be fulfilled. Then for any $M\ge 0$,
\begin{align*}
&\sup_{\theta\in\Theta_\beta}
\mathbb{E}_{\theta}\hat{\pi}\big(\|\vartheta-\theta\|^2\ge 
M_0 c^{-1}r^2(\Theta_\beta)+M\sigma^2|Y\big)  \le H_0 e^{-m_0 M}, 
\\&
\sup_{\theta\in\Theta_\beta}\mathbb{P}_\theta\big(\|\hat{\theta}-\theta\|^2 
\ge M_1c^{-1}r^2(\Theta_\beta)+M\sigma^2\big)\le H_1 e^{-m_1 M},\\
&\sup_{\theta\in\Theta_\beta}
\mathbb{P}_\theta\big(\hat{r}^2\ge M_3c^{-1} r^2(\Theta_\beta)+(M+1)\sigma^2\big)\le H_3e^{-m_3 M}.
\end{align*}
\end{corollary}
In case the radius  of confidence ball is of the order $r(\theta)+\sigma N^{1/4}$, we assume that the conditions of Theorem \ref{th4} instead of Theorem \ref{th3}  are fulfilled and  the third claim of Corollary \ref{minimax_results} is replaced as follows:   
\begin{align*}
\sup_{\theta\in\Theta_\beta}\mathbb{P}_\theta\big(\tilde{R}_M^2\ge g_M'(\theta,N)\big)
\le\psi_1(M/4)+\psi_2(M)+2H_1e^{-m_1M},
\end{align*}
where $g_M'(\theta,N)=M_1c^{-1}r^2(\Theta_\beta)+M\sigma^2+4\sigma^2 G_M \sqrt{N}$.  
We do not  specialize Theorem \ref{th2} and the coverage relation of Theorems \ref{th3} and \ref{th4}
for the scale $\{\Theta_\beta, \, \beta\in\mathcal{B}\}$, because 
it does not make much sense to specialize these claims for any scale. 
Theorem \ref{th2} holds uniformly in $\theta\in\Theta$, hence uniformly 
over  any $\Theta_\beta$.
The coverage relation in Theorem \ref{th3} holds uniformly over the EBR class $\Theta_{\rm eb}$, 
so it will certainly hold uniformly over the intersection $\Theta_{\rm eb}\cap \Theta_\beta$. 
Similarly,  the coverage relation in Theorem \ref{th4} will certainly hold uniformly over $\Theta_\beta$.  
%Moreover, Corollary \ref{meta_theorem} and \ref{minimax_results} can be extended 
%to the case when $\theta\in\ell_2$ for some specific models (these models will be discussed later) 
%with only difference that any projection operator 
%$\mathrm{P}_I\subseteq\mathbb{L}_I=\{x\in\ell_2: x_{lI+1}=x_{lI+2}=\ldots=0,\; lI\in\mathbb{N}\}$.   

%Depending on the structure in these examples, 
Below we perform the computations to obtain Corollaries \ref{meta_theorem} 
and \ref{minimax_results} for concrete models and structures. 
For brevity sake, for some cases and some claims of Corollaries \ref{meta_theorem} 
and \ref{minimax_results}, we will not present all the computations  
for verifying the required conditions, since 
these computations can be done similarly to the previously considered cases. 
%Sometimes we will use the usual $o$, $O$ notation to describe the asymptotic behavior of certain quantities. 
%with respect to what should be clear from the context. 

\section{Signal+noise model with smoothness structure}
 \label{sec_smoothness}
%Assume that the data $Y=(Y_i)_{i\in\mathbb{N}}$ come from the model
Consider the observations
\begin{align*}
Y_i=\theta_i+\tfrac{1}{\sqrt{n}}\xi_i,\; \; i\in\mathbb{N},
\end{align*}
where $\xi_i \overset{\rm ind}{\sim}\mathrm{N}(0,1)$ and 
$\theta=(\theta_i)_{i\in\mathbb{N}}\in\Theta=\ell_2$ is an unknown parameter with 
the \emph{smoothness structure}.

There is a vast literature on estimation (see, for example, references in \cite{Johnstone:2017}), 
a few papers on global posterior contraction, very few on 
global uncertainty quantification, but only one (to the best of our knowledge) on local uncertainty quantification.
%Note that for this model (and also for many others below, we will not mention this anymore), 
%our results improve many existing results upon the above mentioned aspects such as 
%1) our results are \emph{local}, 2) hold in the \emph{refined formulation} (non-asymptotic 
%exponential probability bounds) and 3) \emph{distribution-free setting}. 
Bayesian global results for smoothness scales are studied in \cite{Belitser&Ghosal:2003, Szabo&vanderVaart&vanZanten:2013, Szabo&vanderVaart&vanZanten:2015} and many others 
(see further reference therein). A local approach for this model,
delivering also the adaptive minimax results for many smoothness scales simultaneously, is
considered by \cite{Babenko&Belitser:2010, Gao&vanderVaart&Zhou:2015} 
for estimation and posterior contraction problems, and by \cite{Belitser:2016} also 
for uncertainty quantification problem (in the inverse problem context which is a more general setting).

Admittedly, this is an infinite dimensional model (with $\sigma^2=n^{-1}$) as compared with 
the default high-dimensional general framework \eqref{model}, but in this case all the results go through with one 
minor adjustment: all the sums over $I\in\mathcal{I}$ become countable infinite instead of finite.  
Alternatively, we could consider a finite dimensional model approximating the original infinite dimensional 
model with arbitrary accuracy. 
%For example, if $\theta\in \Theta_{\bar{\alpha}}$ for some grand 
%Sobolev space  $\Theta_{\bar{\alpha}}$ (see \eqref{sobolev_ellipsoid} below), then we can study the 
%following high-dimensional model $Y_i=\theta_i+\tfrac{1}{\sqrt{n}}\xi_i$, $i=1,\ldots, n^{1/(2\bar{\alpha})}$,
%which will do the job as good as the original infinite dimensional model.

In this case, the smoothness structure is modeled by the linear spaces 
\begin{align}
\label{L_I_6.1}
\mathbb{L}_I=\big\{x\in\ell_2:\, x_{i}=0\;\;\text{for all}\; i\ge I+1\big\},\quad I\in\mathcal{I} = \mathbb{N}_0.
\end{align}
We have $\|\theta-\mathrm{P}_I \theta\|^2=\sum_{i=I+1}^\infty \theta_i^2$, $\sigma^2=n^{-1}$,
$d_I=\dim(\mathbb{L}_I)=I$, 
the structural 
slicing mapping is taken to be $s(I)=I$, so that $\mathcal{S} =\mathcal{I}=\mathbb{N}_0$
and $\mathcal{I}_{s(I)}=\{I\}$. Hence $\log| \mathcal{I}_s|=0$ for all $s\in\mathcal{S}$.
We thus take the majorant $\rho(I)=d_I+\log| \mathcal{I}_{s(I)}|=d_I =I$.
The oracle rate is 
\[
r^2(\theta) =\min_{I \in \mathbb{N}_0} \big(\sum\nolimits_{i\ge I+1} \theta_i^2+\tfrac{I}{n}\big)
=\sum\nolimits_{i\ge I_o+1} \theta_i^2+\tfrac{I_o}{n}.
\]
Recall that, in view of Remarks \ref{rem_cond_A1} and \ref{rem15}, 
Conditions \eqref{cond_nonnormal} and \eqref{cond_A4} hold with $d_I =\dim(\mathbb{L}_I)$. 
Condition \eqref{(A2)} is  fulfilled since, in view of Remark \ref{rem5}, 
$\sum_{I\in\mathcal{I}} e^ {-\nu\rho(I) }=\sum_{s\in\mathcal{S}} e^ {-\nu s}
%=\sum_{k\in\mathbb{N}_0} e^ {-\nu k}
=\frac{e^{\nu}}{e^{\nu}-1}= C_\nu$ for any $\nu>0$. Finally, Condition \eqref{(A3)} is also fulfilled. 
Indeed, for any  $I_0, I_1\in\mathcal{I}$ define $I'(I_0,I_1)=I_0 \vee I_1$, then 
$(\mathbb{L}_{I_0}\cup\mathbb{L}_{I_1}) \subseteq \mathbb{L}_{I'}$ and 
$\rho(I')=I_0\vee I_1\le I_0+I_1=\rho(I_0)+\rho(I_1)$.

As consequence of our general results, we obtain the local results of Corollary \ref{meta_theorem}
for this case with the local rate $r^2(\theta)$ defined above. 
In turn, by virtue of Corollary \ref{minimax_results} the local results will  imply global minimax 
adaptive results at once over all scales $\{\Theta_\beta,\, \beta\in\mathcal{B}\}$ covered by 
the oracle rate $r^2(\theta)$ (i.e., for which \eqref{oracle_minimax} holds).
Below we present a couple of examples of scales $\{\Theta_\beta,\, \beta\in\mathcal{B}\}$ covered 
by the oracle rate $r^2(\theta)$.
%so that adaptive global minimax results follow immediately by Corollary \ref{minimax_results}.
%whose minimax rate is bounded from below by a multiple of the local rate.

%\paragraph{Sobolev ellipsoids.}
\subsection{Minimax results for the Sobolev ellipsoids} 
For $\beta,Q>0$, introduce the Sobolev ellipsoids 
\begin{align}
\label{sobolev_ellipsoid}
\Theta_\beta=\Theta_\beta(Q)=\{\theta\in\ell_2: \sum\nolimits_{i\in\mathbb{N}} i^{2\beta}\theta_i^2\le Q\}.
\end{align}
It is well known that the  corresponding minimax rate is $r^2(\Theta_\beta)\asymp 
Q^{1/(2\beta+1)} n^{-2\beta/(2\beta+1)}$; see \cite{Pinsker:1980} or, e.g., \cite{Belitser&Levit:1995}.  
The adaptive minimax results for Sobolev ellipsoids were considered by \cite{Babenko&Belitser:2010, 
Szabo&vanderVaart&vanZanten:2013} (see further references therein) for posterior 
contraction rates, and by \cite{Belitser:2016, Robins&vanderVaart:2006,
Szabo&vanderVaart&vanZanten:2015} (see also further references therein)
for constructing optimal confidence balls.
By taking $I_0=\lfloor (Qn)^{1/(2\beta+1)}\rfloor$, we obtain \eqref{oracle_minimax}: 
\begin{align*}
\sup_{\theta\in\Theta_\beta(Q)}r^2(\theta)&=\sup_{\theta\in\Theta_\beta(Q)}
\Big(\sum_{i=I_o+1}^\infty\theta_i^2+\tfrac{I_o}{n}\Big)
\le \sup_{\theta\in\Theta_\beta(Q)}\sum_{i=I_0+1}^\infty\tfrac{i^{2\beta}\theta_i^2}{I_0^{2\beta}}+\tfrac{I_0}{n}\\
&\le \tfrac{I_0}{n}+\tfrac{Q}{I_0^{2\beta}}\lesssim Q^{1/(2\beta+1)} n^{-2\beta/(2\beta+1)}
\asymp r^2(\Theta_\beta).
\end{align*} 
Corollary \ref{minimax_results} follows for this case with the minimax rate $r^2(\Theta_\beta)$ 
defined above.

\begin{remark}
Notice that, besides adaptation with respect to the smoothness $\beta$, the local result yields 
adaptation also with respect to the ellipsoid size $Q$, which is important when either $Q\to 0$ or $Q \to \infty$. 
%Indeed, the minimax rate over $\Theta_\beta(Q)$ (with $Q$ taken into account) is  
%$r^2(\Theta_\beta(Q))\asymp n^{-2\beta/(2\beta+1)} Q^{1/(2\beta+1)}$.
%Then taking  $I_0=\lfloor (Q n)^{1/(2\beta+1)}\rfloor$ leads along similar lines to the bound
%\[
%\sup_{\theta\in\Theta_\beta(Q)}r^2(\theta) \lesssim Q^{1/(2\beta+1)} n^{-2\beta/(2\beta+1)} 
%\asymp r^2(\Theta_\beta(Q)),
%\]
%so that the minimax rate with respect to the both smoothness and size 
%of the Sobolev ellipsoid is covered by the oracle rate. 
The same holds for examples below, where %, besides the smoothness $\beta\in\mathcal{B}$, 
the local oracle results deliver adaptation 
with respect to the both smoothness $\beta\in\mathcal{B}$ and size parameter $Q>0$.
\end{remark}

%\paragraph{Sobolev hyperrectangles.}
\subsection{Minimax results for the Sobolev hyperrectangles} 
Consider the so called Sobolev hyperrectangles  in $\ell_2$:
\begin{align*}
\Theta_\beta=\Theta_\beta(Q)=\{\theta\in\ell_2: |\theta_i|\le \sqrt{Q}i^{-\beta}\}, \; \; \beta>1/2.
\end{align*}
It is known that the  corresponding minimax rate is $r^2(\Theta_\beta)\asymp Q^{1/(2\beta)} n^{-(2\beta-1)/(2\beta)}$.  
The adaptive minimax results for Sobolev hyperrectangles were
considered by \cite{Babenko&Belitser:2010, Belitser:2016} 
for posterior contraction rates, and by \cite{Belitser:2016, Robins&vanderVaart:2006, 
Szabo&vanderVaart&vanZanten:2015} (see further references therein) 
for constructing optimal confidence balls.
By taking $I_0=\lfloor (Qn)^{1/2\beta}\rfloor$, we obtain \eqref{oracle_minimax}: 
\begin{align*}
\sup_{\theta\in\Theta_\beta(Q)}r^2(\theta)&=\sup_{\theta\in\Theta_\beta(Q)}
\sum_{i=I_o+1}^\infty\theta_i^2+\tfrac{I_o}{n}
\le \sup_{\theta\in\Theta_\beta(Q)}\sum_{i=I_0+1}^\infty\tfrac{Q}{i^{2\beta}}+\tfrac{I_0}{n}\\
&\le \tfrac{I_0}{n}+\tfrac{Q}{(2\beta-1)I_0^{(2\beta-1)}}\lesssim Q^{1/(2\beta)}n^{-(2\beta-1)/(2\beta)}
\asymp r^2(\Theta_\beta).
\end{align*} 
Corollary \ref{minimax_results} follows for this case with the minimax rate $r^2(\Theta_\beta)$ 
%(with respect to the both smoothness and size of the Sobolev hyperrectangle)
defined above.

\subsection{Minimax results for the analytic and tail classes}
Similarly, we can derive the adaptive minimax results for two more scales of
\emph{exponential ellipsoids} (or \emph{analytic classes}) and \emph{tail classes}.
Exponential ellipsoids  are defined as follows:
\begin{align*}
\Theta_\beta=\Theta_\beta(Q)=\big\{\theta\in\ell_2: \sum\nolimits_{k\in\mathbb{N}} 
e^{2\beta k}\theta_k^2\le Q\big\}, \quad \beta>0.
\end{align*}
For the analytic scale, the relation \eqref{oracle_minimax} is
$\sup_{\theta\in\Theta_\beta}r^2(\theta) \lesssim r^2(\Theta_\beta)\asymp \frac{\log n}{\beta n}$
(here $I_0= \frac{\log n}{\beta}$).

The tail classes are 
\begin{align*}
\Theta_\beta=\Theta_\beta(Q)=\big\{\theta\in\ell_2: \sum\nolimits_{k=m+1}^\infty\theta_k^2\le 
Qm^{-\beta}, m\in\mathbb{N}\big\},
\quad \beta>0.
\end{align*}
In this case, the relation \eqref{oracle_minimax} is $\sup_{\theta\in\Theta_\beta}r^2(\theta) 
\lesssim r^2(\Theta_\beta)\asymp Q^{1/(1+\beta)}n^{-\beta/(\beta+1)}$ 
(here $I_0=\lfloor (Qn)^{1/(1+\beta)}\rfloor$).

Corollary \ref{minimax_results} follows for the both scales with the corresponding minimax 
rates $r^2(\Theta_\beta)$. % defined above.

\section{Smooth function on a graph} 
\label{sec_graph_with_smoothness}

We adopt  the notation and conventions from \cite{Kirichenko&vanZanten:2017}.
Let $G$ be a connected, simple (i.e., no loops, multiple edges or weights), 
undirected graph with $n$ vertices labelled as $1,\ldots,n$. Following \cite{Kirichenko&vanZanten:2017}, 
a function on the graph $G$ can be represented by a mapping $f:[n]\mapsto
\mathbb{R}$.  We write $f$ both for the function and for the
associated vector of function values $(f(1),f(2),...,f(n))$ in $\mathbb{R}^n$.  
Then the observations $Y_1,\ldots,Y_n$ at the vertices of the graph $G$ are modeled as
\begin{align*}
Y_i = f(i)+\tfrac{1}{\sqrt{n}}\xi_i,\quad i\in[n],
\end{align*}
where $\xi_i \overset{\rm ind}{\sim}\mathrm{N}(0,1)$, and the mapping $f$ on the graph $G$ has 
a \emph{smoothness} structure. This is basically the finite-dimensional version of the model 
from Section \ref{sec_smoothness}, in this case $\theta=f\in\Theta\triangleq\mathbb{R}^n$. 
%To the best of our knowledge, there are no local results on estimation, posterior contraction 
%rate and uncertainty quantification problems for this model. 
Below we derive the local results on estimation, posterior contraction 
rate and uncertainty quantification as consequence of our general 
local results from Part I.

The smoothness structure of function $f$ is described by the linear spaces 
\begin{align*}
\mathbb{L}_I=\big\{x\in\mathbb{R}^n:\, x_{i}=0\;\;\text{for all}\; i=I+1\ldots n\big\},\quad I\in\mathcal{I} = [n]_0.
\end{align*}
In this case, $\|f-\mathrm{P}_I f\|^2=\sum_{i=I+1}^n f^2(i)$, the structural 
slicing mapping $s(I)=I$, so that $\mathcal{S} =\mathcal{I}=[n]_0$
and $\mathcal{I}_s =\mathcal{I}_I=\{I\}$. Hence $\log| \mathcal{I}_s|=0$. 
Further, in view of Remark \ref{rem_cond_A1}, Condition \eqref{cond_nonnormal} is fulfilled 
with $\alpha =0.4$, $d_I=\dim(\mathbb{L}_I)=I$, and we arrive at the majorant
$\rho(I)= d_I =I$. The oracle rate is 
\[
r^2(f) =\min_{I \in  [n]_0} \big(\sum\nolimits_{i= I+1}^n f^2(i)+\sigma^2 I\big)
=\sum\nolimits_{i=I_o+1}^n f^2(i)+\sigma^2 I_o.
\]

Further, Condition \eqref{cond_A4} holds in view of Remark \ref{rem15}.
Condition \eqref{(A2)} is fulfilled since, according to Remark \ref{rem5}, 
for any $\nu>0$,
\[
\sum\nolimits_{I\in\mathcal{I}} e^ {-\nu\rho(I) }=\sum\nolimits_{s\in\mathcal{S}} e^ {-\nu s}
%=\sum_{k\in\mathbb{N}_0} e^ {-\nu k}
=\tfrac{e^{\nu}}{e^{\nu}-1}= C_\nu.
\]
Condition \eqref{(A3)} is also fulfilled. Indeed, for any  $I_0, I_1\in\mathcal{I}$ 
define $I'(I_0,I_1)=I_0\vee I_1$,
then $(\mathbb{L}_{I_0}\cup\mathbb{L}_{I_1})\subseteq \mathbb{L}_{I'}$ 
and $\rho(I')=I_0\vee I_1\le I_0+I_1=\rho(I_0)+\rho(I_1)$.

As consequence of our general framework results, we obtain the local results of 
Corollary \ref{meta_theorem} for this case of model/structure with the local rate 
$r^2(f)$ defined above. The reader is invited to formulate all these claims.
In turn, by virtue of Corollary \ref{minimax_results} the local results will  imply global minimax 
adaptive results over all scales $\{\Theta_\beta,\, \beta\in\mathcal{B}\}$ at once, 
covered by the oracle rate $r^2(f)$ (i.e., for which \eqref{oracle_minimax} holds).
Below we present the Laplacian scale $\{H^\beta,\, \beta>0\}$ and show that it is covered 
by the oracle rate $r^2(f)$.

\subsection{Minimax results for the  Laplacian graph} 
One common approach to learn functions on graphs 
is Laplacian regularization; see, for example, \cite{Belkin&Matveeva&Niyogi:2004, 
Kirichenko&vanZanten:2017}. The graph Laplacian is defined as $\mathrm{L}=\mathrm{D}
-\mathrm{A}$, where $\mathrm{A}$ is the adjacency matrix of the graph and $\mathrm{D}$ is 
the diagonal matrix with the degrees of the vertices on the diagonal. When viewed as 
a linear operator, the Laplacian acts on a function $f$ as
\begin{align*}
\mathrm{L}f(i)=\sum\nolimits_{j\sim i}\big(f(i)-f(j)\big),
\end{align*}
where we write $i\sim j$ if vertices $i$ and $j$ are connected by an edge. 
Denote the Laplacian eigenvalues, ordered by magnitude, by $\lambda_1\le \lambda_2
\le\ldots \le\lambda_n$. As in \cite{Kirichenko&vanZanten:2017}, we  assume without loss 
of generality that there exist $i_0\in\mathbb{N}$, $C_1>0$ such that for all $n$ large enough 
and $r\ge 1$,
\begin{align*}
&\lambda_i\ge C_1\big(\tfrac{i}{n}\big)^{2/r}, \quad i>i_0,
\end{align*}
and $f\in H^\beta=H^\beta(Q)=\{f:\sum_{i=1}^n(1+n^{2\beta/r}\lambda_i^\beta)f^2(i)\le Q\}$, 
with smoothness $\beta>0$. The minimax estimation rate over the class $H^\beta$ is 
$r^2(H^{\beta})=\inf_{\tilde{f}}\sup_{f \in H^{\beta}}\mathbb{E}_f\|\tilde{f}-f\|^2 
\asymp n^{-\frac{2\beta}{2\beta+r}}$; see \cite{Kirichenko&vanZanten:2018}. 

%Take $I_0=\lfloor n^{r/(2\alpha+r)}\rfloor$ and $c>0$ large 
%enough such that  $i_0< I^*$ to derive  
By taking $I_0=\lfloor n^{r/(2\beta+r)}\rfloor$, we establish \eqref{oracle_minimax} in this case: 
\begin{align*}
\sup_{f\in H^\beta} r^2(f)&=\sup_{f\in H^\beta} \sum_{i=I_o+1}^n f^2(i)+\tfrac{I_o}{n}
\le \sup_{f\in H^\beta} \sum_{i=I_0+1}^n f^2(i)+\tfrac{I_0}{n}\\
&\le \tfrac{Q}{1+n^{2\beta/r}\lambda_{I_0}^\beta}+\tfrac{I_0}{n}\lesssim 
n^{-\frac{2\beta}{2\beta+r}}\asymp r^2(H^\beta).  
\end{align*}
Hence, Corollary \ref{minimax_results} follows for this case with the minimax rate $r^2(H^\beta)$ 
defined above. As compared to the Theorems $3.2$ and $3.3$ in \cite{Kirichenko&vanZanten:2017}, 
there is no restriction on the range of the smoothness $\beta$ in Corollary \ref{minimax_results}, 
and we do not have any extra logarithmic factor in the rate.
 
\section{Density estimation with smoothness structure}
\label{sec_density}

We observe $X_1,\ldots, X_n\sim f$, where $f$ is a density on $[0,1]$. 
Let $\{\varphi_i, i\in\mathbb{N}\}$ be an orthonormal basis in $L_2[0, 1]$. 
For simplicity, consider a basis $\{\varphi_i, i\in\mathbb{N}\}$ such that 
$\sup_{x\in[0,1]}|\varphi_i(x)|\le c_\varphi$ for some $c_\varphi>0$; 
e.g., for the trigonometric basis $c_\varphi=\sqrt{2}$. 
We can expand the density function $f$ in a Fourier series $f(x)=\sum_{i=1}^\infty\theta_i\varphi_i(x)$, 
$x\in [0,1]$, in the $L_2$-sense. Due to Parseval's identity, the problem of estimating 
the density function $f$ in the $L_2$-sense can be converted into the problem of estimating 
the parameter $\theta=(\theta_i)_{i\in\mathbb{N}}$ in the $\ell_2$-sense:
\begin{align}
\label{model_dens}
Y_i=\theta_i+\sigma_n\xi_i, \;\;i\in\mathbb{N},
\end{align}
where $\theta=(\theta_i)_{i\in\mathbb{N}}$ is an unknown high-dimensional parameter 
of interest with $\theta_i=\mathbb{E}_fY_i=\int_0^1\varphi_i(x)f(x)dx$, $Y_i=\frac{1}{n}\sum_{l=1}^n\varphi_i(X_l)$, 
and $\sigma_n\xi_i = Y_i -\theta_i$. 
Since $|Y_i|=|\frac{1}{n}\sum_{l=1}^n\varphi_i(X_l)|\le c_\varphi$,
we have $\sigma_n|\xi_i|\le |Y_i|+|\theta_i|\le 2c_\varphi$ 
and $\mbox{Var} (\sigma_n \xi_i)\le \frac{c_\varphi^2}{n}$. 
The parameter $\sigma_n$ will be chosen later, for now it is any sequence $\sigma_n \in [0,1]$.

Notice that we reduced the original density estimation problem to a finite dimensional version of 
the model from Section \ref{sec_smoothness}, however the errors $\xi_i$'s are now not iid normals, 
which complicates the study of the present model.  
Consider the same smoothness structure %\eqref{L_I_6.1} 
as in Section \ref{sec_smoothness}, 
with the difference that we restrict the family of structures $I\in\mathcal{I} = [n]_0$. 
The oracle rate becomes
\[
r^2(\theta) =\min_{I \in [n]_0} \big(\sum\nolimits_{i\ge I+1} \theta_i^2+\sigma^2_n I\big)
=\sum\nolimits_{i\ge I_o+1} \theta_i^2+\sigma^2_n I_o.
\]
Conditions \eqref{(A2)} and \eqref{(A3)} are met in the same way as for the signal+noise model from 
Section \ref{sec_smoothness}.
However, in order to derive at least the local estimation and posterior contraction results,
we also need Condition \eqref{cond_nonnormal}. This condition is now 
not immediate since the errors $\xi_i$'s are non-normal and dependent in the model \eqref{model_dens} 
(actually, the $\xi_i$'s are asymptotically normal, but we are not going to rely on this).
We apply the following strategy: introduce certain event 
and establish that the probability of this event is exponentially small (in $n$);
next, under this event establish Condition \eqref{cond_nonnormal}; finally, combine these 
two facts to derive the local estimation and posterior contraction results.

The following proposition is a direct consequence of McDiarmid's inequality; 
see, for instance Theorem 6.2 in \cite{Boucheron&Lugosi&Massart:2012}.
\begin{proposition}
\label{proposition_density}
For any $t>0$ and $i\in[n]$,
\begin{align*}
\mathbb{P}(\sigma_n|\xi_i|\ge t)\le 2\exp\big\{-c_\varphi^{-2}t^2n/2\big\}.
\end{align*}
\end{proposition}
The relation $\mathbb{P}(\max_{i\in[n]}|\xi_i|\ge t)\le\sum_{i\in[n]}
\mathbb{P}(|\xi_i|\ge t)$ and Proposition \ref{proposition_density} imply that,
for the event $E=\{\max_{i\in[n]}|\xi_i|\le \sqrt{2} c_\varphi\}$,
\begin{align}
\label{rel_max_dens}
\mathbb{P}(E^c)=\mathbb{P}\big(\max_{i\in[n]}|\xi_i|>\sqrt{2}c_\varphi\big)\le 
2\exp\{-n\sigma_n^2+\log n\}.
\end{align}
Now, by using \eqref{rel_max_dens}, we ensure Condition \eqref{cond_nonnormal} 
under the event $E=\{\max_{i\in[n]}|\xi_i|\le \sqrt{2}c_\varphi\}$ with 
$\alpha=1\wedge 1/(2c_\varphi^2)$.
Exactly, for any $I\in[n]_0$, 
\begin{align}
\mathbb{E}\exp\big\{&\alpha \|\mathrm{P}_I\xi\|^2\big\}\mathrm{1}_E=
\mathbb{E} \exp\big\{\alpha\sum_{i=1}^I \xi_i^2\big\}\mathrm{1}\{\max_{i\in[n]}|\xi_i|\le
\sqrt{2}c_\varphi\} \notag\\&
\le \exp\big\{\alpha 2c_\varphi^2 I \big\}=e^I=\exp\{d_I\}.
\label{conditional_A1_dens}
\end{align}

We have thus verified the conditional version of Condition \eqref{cond_nonnormal} 
(under event $E$) and Conditions \eqref{(A2)} and \eqref{(A3)} for the model \eqref{model_dens}. 
This means that we can derive results on estimation, posterior 
contraction and uncertainty quantification for the density $f$ in terms of the model 
\eqref{model_dens}. These are the counterparts of claims (i)-(v) of 
Corollary \ref{meta_theorem} summarized by Theorem \ref{meta_theorem_dens} below. 
To the best of our knowledge, local results on uncertainty quantification for the density 
are new. In the below theorem, we keep the same notation for all the quantities involved as 
in the general framework, with the understanding that these are specialized for the model 
\eqref{model_dens} with the smoothness structure and the oracle rate $r^2(\theta)$.

\begin{theorem}
\label{meta_theorem_dens}
Let the constants $M_0, M_1, M_3, H_0, H_1, H_2, H_3$, $m_0, m_1,m_2,m_3$, 
$c_2, c_3, C_\nu$ 
be defined in Theorems \ref{th1}-\ref{th3} and \eqref{rel_max_dens}.
Then for any $M \ge 0$,
\begin{align*}
&\sup_{\theta\in\ell_2}\mathbb{E}_\theta
\hat{\pi}\big(\|\theta-\vartheta\|^2\ge M_0r^2(\theta)+M\sigma^2_n|Y\big) 
\le 2e^{-n\sigma^2_n+\log n}\!+\! H_0 e^{-m_0 M},\\
&\sup_{\theta\in\ell_2}\mathbb{P}_\theta
\big(\|\hat{\theta}-\theta\|^2 \ge M_1 r^2(\theta)+M\sigma^2_n\big)
\le 2e^{-n\sigma^2_n+\log n}+H_1 e^{-m_1 M},\\
&\sup_{\theta\in\ell_2}\mathbb{E}_\theta
\hat{\pi}\big(I: r^2(I,\theta)\ge c_3 r^2(\theta)+M\sigma^2_n|Y\big) 
\le 2e^{-n\sigma^2_n+\log n}+ C_\nu e^{-c_2 M},\\
&\sup_{\theta\in\ell_2}
 \mathbb{P}_\theta\big(\hat{r}^2\ge M_3 r^2(\theta)+(M+1)\sigma^2_n\big)
\le 2e^{-n\sigma^2_n+\log n}+H_3 e^{-m_3 M},\\
&\sup_{\theta\in\ell_2\cap\Theta_{\rm eb}}\mathbb{P}_\theta
\big(\theta\notin B(\hat{\theta},\hat{R}_M)\big) 
\le 2e^{-n\sigma^2_n+\log n}+H_2 e^{-m_2M}.
\end{align*}
\end{theorem}

Let us outline the idea of the proof (which is omitted) of the first claim of the above theorem;
the same reasoning applies to the remaining claims.
The expectation of the empirical Bayes posterior probability $\mathbb{E}_\theta \Pi=
\mathbb{E}_\theta\hat{\pi}\big(\|\theta-\vartheta\|^2\ge 
M_0r^2(\theta)+M\sigma^2_n|Y\big)$ is bounded by the sum of two terms 
$\mathbb{E}_\theta \Pi\le \mathbb{P}_\theta(E^c)
+\mathbb{E}_\theta \Pi 1_E$. The first term is evaluated by using \eqref{rel_max_dens}
(obtaining the bound $2e^{-n\sigma^2_n+\log n}$); the second term is evaluated exactly 
in the same way as in the proof Theorem \ref{th1}, because Condition \eqref{cond_nonnormal} 
is fulfilled under the event $E$ according to \eqref{conditional_A1_dens}.   
Counterparts of assertions (ii) and (iii) of Theorem \ref{th2} can also be formulated and proved
in the same way. Notice that the results that rely on Condition \eqref{cond_A4} are 
not claimed as we are unable to verify this condition at the moment.

As to the choice of $\sigma^2_n$ in the oracle rate, clearly, we would want 
it to be as small as possible. On the other hand, we want the claims of the theorem to be non-void,
which is ensured only if $\sigma^2_n n \ge C \log n$, or $\sigma^2_n\ge \tfrac{C\log n}{n}$, 
for sufficiently large $C>0$. In the sequel we take therefore $\sigma^2_n=\tfrac{C\log n}{n}$.
An extra log factor thus appeared which will also enter the minimax rates in the global results.   
We conjecture that one can get rid of that factor by using more accurate concentration inequalities 
when establishing Condition \eqref{cond_nonnormal}.

As usually, the local results of Theorem \ref {meta_theorem_dens} will imply global minimax 
adaptive results simultaneously over all scales $\{\Theta_\beta,\, \beta\in\mathcal{B}\}$ covered by 
the oracle rate $r^2(\theta)$ (i.e., for which \eqref{oracle_minimax} holds).
Hence, the same adaptive minimax results for the same scales as in Section \ref{sec_smoothness} follow,
up to a log factor as we have $\sigma^2_n\asymp \frac{\log n}{n}$ in the model \eqref{model_dens}
instead of $n^{-1}$ in the model from Section \ref{sec_smoothness}. 
The reader is invited to formulate a number of local and adaptive minimax 
results for this case. We should mention that it seems possible to 
extend the results to other structures (e.g., sparsity) and scales (e.g., Besov scales).   

%Assume that   the following structure: $\theta_i=0$  for any $i=I+1,\ldots $, $I\in\mathcal{I}=\mathbb{N}_0$. Then similar to Subsection  \ref{sec_smoothness}, we can easily derive Corollaries \ref{meta_theorem}, \ref{meta_theorem_full} (oracle results) and \ref{minimax_results} (minimax results of Sobolev ellipsoids and Sobolev hyperrectangles) for density estimation. Notice that the Condition (A2) is always fulfilled for these cases, since the errors $\xi_i, i\in\mathbb{N}$, are bounded from above by $\sqrt{2}(K+1)$.

\section{Regression under wavelet basis (smoothness+sparsity structure)} 
\label{sec_regression_wavelet}

Consider the observations
\begin{align}
\label{model_wavelet}
Y_{jk}=\theta_{jk}+\tfrac{1}{\sqrt{n}}\xi_{jk},\quad \xi_{jk} \overset{\rm ind}{\sim}\mathrm{N}(0,1),
\quad (jk)\in \mathcal{K}=\{(jk): j\in\mathbb{N}_0,\; k\in [2^j]\}.
\end{align}
%random noise variables $\xi_{jk}$'s  satisfy  Condition (A2) ($\mathbb{L}_{I}, \rho(I), \mathcal{S}$ 
%and  $\mathcal{I}_s $ for this model will be specified later) and  wavelet coefficients $\theta_{jk}$'s  
%have the following structure: let $j_0\in\mathbb{N}_0$ and   $\theta_{jk}=0$ for any   $j>j_0$, $k\in[2^j]$, 
%and  for any  $j\in \{\{0\}\cup[j_0]\}$, $k\in I_j^c$ (for some $I=(I_0, I_1,\ldots , I_{j_0}), 
%\; I_j\subseteq [2^j], \; j\in\{\{0\}\cup [j_0]\}$ ). Note that $|I_j|\ge 1, j\in\{\{0\}\cup [j_0]\}$.
This model is obtained as the result of the orthogonal wavelet transform of an additive
regression function observed in Gaussian noise with $\sigma^2=n^{-1}$, or just as 
a sequence version (with respect to some 
wavelet basis) of the continuous \emph{white noise model}. We could also consider 
a high dimensional ``projected'' (see (9.57) in \cite{Johnstone:2017}) variant of \eqref{model_wavelet}, where 
$j\in[J_o]$ with $2^{J_0+1}= n$. For further references, details, many interesting connections 
and relations of the above model to the function estimation theory, we refer to the very 
comprehensive and insightful account \cite{Johnstone:2017} on this topic. 
We adopt  the notation and conventions from \cite{Johnstone:2017}.  

The model here is of the type signal+noise but can also be regarded as matrix+noise.
The structure studied here is some kind of smoothness, but different from the previous 
case, now geared towards describing functions from \emph{Besov scales}; in a way, 
it is combined \emph{smoothness+sparsity} structure. For this model, there is vast literature on 
estimation, especially in the global settings related to the Besov scales, much less literature 
on uncertainty quantification; we mention some relevant references below.

The smoothness+sparsity structure of $\theta=(\theta_{jk}, (jk)\in\mathcal{K})$ 
is modeled by the linear spaces 
\begin{align*}
\mathbb{L}_{I}=\big\{(x_{jk}, (jk)\in\mathcal{K}): x_{jk}=0\;\, \forall\, j\in [j_0]_0,k\in I_j^c\,
\;\text{and}\;\, \forall \, j> j_0,\, k\in[2^j]   \big\},
\end{align*}
where $I=(j_0,I_0,\ldots, I_{j_0}) \in\mathcal{I}=\big\{(j_0,I_0,\ldots, I_k): j_0 \in\mathbb{N}_0, 
I_j\subseteq [2^j], j\in[j_0]_0\big\}$. %with $I_j\subseteq [2^j]$, $j\in[j_0]_0$.
The structural slicing mapping is $s(I)=(j_0,|I_0|,\ldots,|I_{j_0}|)$ and 
$d_I=\dim(\mathbb{L}_{I})=\sum_{m=0}^{j_0}|I_m|$.
Compute $|\mathcal{I}_{s(I)}|=\prod_{k=0}^{j_0}\binom{2^k}{|I_k|}$, hence 
$\log |\mathcal{I}_{s(I)}|=\sum_{k=0}^{j_0}\log \binom {2^k}{|I_k|}
\le \sum_{k=0}^{j_0}|I_k|\log(\frac{e 2^k}{|I_k|})$.
%In view of Remark \ref{rem_cond_A1}, Condition \eqref{cond_nonnormal} is fulfilled with $\alpha =0.4$, 
%$d_I=\dim(\mathbb{L}_I)$, and Condition \eqref{cond_A4} holds in view of Remark \ref{rem15}.
Since $d_I+\log |\mathcal{I}_{s(I)}|
\le 2\sum_{k=0}^{j_0}|I_k|\log(\frac{e 2^k}{|I_k|})$, we take the majorant  
$\rho(I)=2\sum_{k=0}^{j_0}|I_k|\log(\frac{e 2^k}{|I_k|})$.

Conditions \eqref{cond_nonnormal} and \eqref{cond_A4} hold with $d_I =\dim(\mathbb{L}_I)$
in view of Remarks \ref{rem_cond_A1} and \ref{rem15}. Condition \eqref{(A2)} is also fulfilled, 
since, according to Remark \ref{rem5}, for any $\nu>2$ 
\begin{align*}
\sum_{I\in\mathcal {I}} e^{-\nu \rho(I)}
&\le \sum_{s\in\mathcal{S}}  e^{-(\nu-1) \rho(I)}
\le \sum_{j_0=0}^\infty \sum_{k_0=1}^{2^0} \ldots \sum_{k_m=1}^{2^{j_0}}  e^{-(\nu-1)(k_0+\ldots+k_m)}\\
&\le 
\sum_{j_0=0}^\infty \big(\tfrac{1}{e^{\nu-1}-1}\big)^{j_0+1}\le \tfrac{1}{e^{\nu-1}-2}=C_\nu.
%&\le\sum_{j_0=0}^\infty 2^{j_0+1} e^{-j_0}\le C_\nu.
\end{align*}
%\begin{align*}
%&\sum_{I\in\mathcal{I}}e^{-\nu\rho(I)}=\sum_{s=0}^\infty\Big(\sum_{|I_0|=1}^1\tbinom{1}{|I_0|}
%\Big(...\Big(\sum_{|I_{s-1}|=1}^{2^{s-1}}\!\!\tbinom{2^{s-1}}{|I_{s-1}|} 
%\Big(\sum_{|I_s|=1}^{2^s}\tbinom{2^s}{|I_s|} e^{-\nu\rho(I)}\Big)...\Big)\\
%&\le \sum_{s=0}^\infty\Big(\sum_{|I_0|=1}^1\tbinom{1}{|I_0|}
%\Big(...\Big(\sum_{|I_{s-1}|=1}^{2^{s-1}}\tbinom{2^{s-1}}{|I_{s-1}|} e^{-\nu\rho(s-1)} 
%\Big(\sum_{|I_s|=1}^{2^s} \Big(\frac{e2^s}{|I_s|}\Big)^{-(\nu-1)|I_s|}\Big)...\Big)\\
%&\le \sum_{s=0}^\infty\frac{1}{e^{\nu-1}-1}\Big(\sum_{|I_0|=1}^1\tbinom{1}{|I_0|}
%\Big(...\Big(\sum_{|I_{s-2}|=1}^{2^{s-2}}\!\!\tbinom{2^{s-2}}{|I_{s-2}|} 
%\Big(\sum_{|I_{s-1}|=1}^{2^{s-1}}\!\!\tbinom{2^{s-1}}{|I_{s-1}|} e^{-\nu\rho(s-1)}\Big)...\Big)\\
%&\le \sum_{s=0}^\infty \Big(\frac{1}{e^{\nu-1}-1}\Big)^{s+1}\le \frac{1}{e^{\nu-1}-2}=C_\nu.
%\end{align*}
Finally, for  any  $I^0, I^1\in\mathcal{I}$  define  $j''_0=\min\{j_0^0,j_0^1\}$, $j'_0=\max\{j_0^0,j_0^1\}$ and $I'(I^0,I^1)\in\mathcal{I}$ such that 
\[
I'(I^0,I^1)=(I_0^0\cup I_0^1, I_1^0\cup I_1^1,\ldots , I_{j''_0}^0\cup I_{j''_0}^1,
I_{j''_0+1}^{\mathrm{1}\{j'_0=j_0^1\}},\ldots, I_{j'_0}^{\mathrm{1}\{j'_0=j_0^1\}} ).
\]
Then  $(\mathbb{L}_{I^0}\cup\mathbb{L}_{I^1})
\subseteq \mathbb{L}_{I'}$ and 
\[
\sum_{m=0}^{j'_0}|I_m'|\log \big(\tfrac{e 2^m}{|I_m'|}\big)
\le\sum_{m=0}^{j_0^0}|I_m^0|\log \big(\tfrac{e 2^m}{|I_m^0|}\big)+ \sum_{m=0}^{j_0^1}|I_m^1|
\log \big(\tfrac{e 2^m}{|I_m^1|}\big),
\]
which entails Condition \eqref{(A3)}.

As consequence of our general results, we obtain Corollary \ref{meta_theorem}
for this case with the local rate $r^2(\theta)= \min_{ I \in \mathcal{I}} \big\{\|\theta-\mathrm{P}_I \theta\|^2 +
\tfrac{1}{n}\rho(I)\big\}$. 
%Notice that the best oracle rate we can get is of the order $\frac{j_o}{n}$.
Below we present the example of Besov scale, for which the global minimax adaptive results follow
from the local results. We should mention that there are of course more scales covered 
by the oracle rate $r^2(\theta)$, the reader is invited to make computations for other 
interesting scales. Besides, the results can be extended to non-normal, not independent 
$\xi_{jk}$'s, but only satisfying Condition \eqref{cond_nonnormal}.

\subsection{Minimax results for the Besov scale} 
\label{subsec_Besov_scale}
Assume that  the true signal $\theta$ belongs to a Besov ball 
\begin{align}
\label{Besov_class}
\Theta_{p,q}^\beta(Q)=\Big\{\theta: \sum_{j=0}^\infty2^{ajq}\big(\sum_{k=1}^{2^j}
\theta_{jk}^p\big)^{q/p}\le Q^q\Big\}, \quad a=\beta+\tfrac{1}{2}-\tfrac{1}{p},
\end{align}
for some $p,q, Q>0$ and $\beta\ge 1/p$. 
The minimax rate over $\Theta_{p,q}^\beta(Q)$ is known to be $r^2(\Theta_{p,q}^\beta(Q))\asymp 
n^{-\frac{2\beta}{2\beta+1}}$. The adaptive minimax results for the scale of the 
class $\Theta_{p,q}^\beta(Q)$ were considered by \cite{Rivoirard&Rousseau:2012, Hoffmann&Rousseau&Schmidt-Hieber:2015, 
Gao&vanderVaart&Zhou:2015} and many others for posterior contraction rates, 
and \cite{Bull&Nickl:2013} for constructing optimal confidence balls.

Let $j_*=\lfloor\log_2 n\rfloor$. Define $\mathcal{I}_*
=\{I\in\mathcal{I}: j_0(I)=j_*\}$ and note that $\mathcal{I}_*\subset\mathcal{I}$. 
Hence, for any $\theta\in \Theta_{p,q}^\beta(Q)$,
%By  taking some $I^*\in\mathcal{I}_{j_0}$, where $j_0=\lfloor\log_2 n\rfloor$,  we obtain 
%for some $\bar{C}(\alpha,p)>0$ that 
\begin{align*}
r^2(\theta) &\le \min_{ I \in \mathcal{I}_*}\big\{\|\theta-\mathrm{P}_I \theta\|^2+\tfrac{1}{n}\rho(I)\big\} \\
&\le 
\sum_{j=0}^{j_*}\sum_{k\in I_{oj}^c}\theta_{jk}^2\!+\!\sum_{j=j_*+1}^\infty  \sum_{k=1}^{2^j}\theta_{jk}^2
+\sum_{j=0}^{j_*} \tfrac{|I_{oj}|}{n}\log \big(\tfrac{e 2^j}{|I_{oj}|}\big)\\
&\le
\sum_{j=0}^{j_*}\min_{0\le k\le 2^j}\Big(\sum_{l>k}\theta_{j(l)}^2+
C_1\frac{k}{n}\log(e2^j/k)\Big)+\sum_{j=j_*+1}^\infty  \sum_{k=1}^{2^j}\theta_{jk}^2\\
&\le C_2 n^{-\frac{2\beta}{2\beta+1}} +C_3n^{-1} \lesssim n^{-\frac{2\beta}{2\beta+1}}
\asymp r^2(\Theta_{p,q}^\beta(Q)),
\end{align*}
where  $\theta_{j(l)}^2$ denotes the $l$-th largest value among  $\{\theta_{jk}^2, j\in[2^k]\}$. 
The third inequality of the last display follows from Theorem 12.1 in \cite{Johnstone:2017} 
under the assumption $\beta\ge 1/p$. 
We thus established the relation \eqref{oracle_minimax} for the Besov scale,
and Corollary \ref{minimax_results} follows with the minimax 
rate $r^2(\Theta_{p,q}^\beta(Q))$ defined above.

\begin{remark}
Interestingly, as is shown in Section \ref{subsec_besov_rev} (see also \cite{Belitser&Nurushev:2015}),
the global results on Besov scales for the model \eqref{model_wavelet} can also be derived as consequence 
of the local approach to the signal+noise model with sparsity structure.
\end{remark}

\section{Signal+noise with sparsity structure} 
\label{sec_sparsity}
Consider the observations
\begin{align}
\label{model_sparsity}
Y_i=\theta_i+\sigma\xi_i,\quad i\in[n],
\end{align}
where $\theta=(\theta_i)_{i\in[n]}\in\Theta=\mathbb{R}^n$ is an unknown parameter
and $\xi_i \overset{\rm ind}{\sim}\mathrm{N}(0,1)$.
According to the local approach, the goal is to fully exploit 
all the \emph{sparsity structure} in the high-dimensional vector 
$\theta$. There is a vast literature on estimation and posterior contraction, 
some relevant references can be found below.
The local approach for this model,
delivering also the adaptive minimax results for various sparsity scales simultaneously, is
considered in \cite{Belitser&Nurushev:2015, Gao&vanderVaart&Zhou:2015} 
for posterior contraction rates (in \cite{Belitser&Nurushev:2015}, 
also for uncertainty quantification problem).

The classical sparsity structure is modeled by the linear spaces 
\begin{align*}
\mathbb{L}_I=\big\{x\in\mathbb{R}^n: x_i=0, i\in I^c  \big\},\quad I\in\mathcal{I} = 
\{J: J \subseteq [n]\}.
\end{align*}
In this case, $d_I=\dim(\mathbb{L}_I)=|I|$, $\|\theta-\mathrm{P}_I\theta\|^2=\sum_{i\in I^c}\theta_i^2$,  
the structural slicing mapping is defined to be $s(I)=|I|\in\mathcal{S}\triangleq[n]_0$. 
Compute $|\mathcal{I}_{s(I)}|=\binom{n}{|I|}$, hence $\log |\mathcal{I}_{s(I)}|=\log \binom {n}{|I|}
\le |I|\log(\frac{en}{|I|})$. Since $d_I+\log |\mathcal{I}_{s(I)}|
\le |I|+|I|\log(\frac{en}{|I|})$, we take the majorant $\rho(I)=2|I|\log(\frac{en}{|I|})$.

Conditions \eqref{cond_nonnormal} and \eqref{cond_A4} hold with $d_I =\dim(\mathbb{L}_I)$
in view of Remarks \ref{rem_cond_A1} and \ref{rem15}. 
Condition \eqref{(A2)} is fulfilled, since, according to Remark \ref{rem5}, for any $\nu>1$ 
\begin{align*}
\sum_{I\in\mathcal {I}} e^{-\nu \rho(I)}
&\le \sum_{s\in\mathcal{S}}  e^{-(\nu-1) \rho(I)}
\le \sum_{s=0}^n\big(\tfrac{en}{s}\big)^{-(\nu-1)s}\le 
\tfrac{1}{1-e^{1-\nu}}= C_\nu.
%&\le\sum_{j_0=0}^\infty 2^{j_0+1} e^{-j_0}\le C_\nu.
\end{align*}
%\begin{align*}
%&\sum_{I\in\mathcal{I}}e^{-\nu\rho(I)}=\sum_{s=0}^\infty\Big(\sum_{|I_0|=1}^1\tbinom{1}{|I_0|}
%\Big(...\Big(\sum_{|I_{s-1}|=1}^{2^{s-1}}\!\!\tbinom{2^{s-1}}{|I_{s-1}|} 
%\Big(\sum_{|I_s|=1}^{2^s}\tbinom{2^s}{|I_s|} e^{-\nu\rho(I)}\Big)...\Big)\\
%&\le \sum_{s=0}^\infty\Big(\sum_{|I_0|=1}^1\tbinom{1}{|I_0|}\Big(...\Big(\sum_{|I_{s-1}|=1}^{2^{s-1}}
%\tbinom{2^{s-1}}{|I_{s-1}|} e^{-\nu\rho(s-1)} \Big(\sum_{|I_s|=1}^{2^s}
%\Big(\frac{e2^s}{|I_s|}\Big)^{-(\nu-1)|I_s|}\Big)...\Big)\\
%&\le \sum_{s=0}^\infty\frac{1}{e^{\nu-1}-1}\Big(\sum_{|I_0|=1}^1\tbinom{1}{|I_0|}
%\Big(...\Big(\sum_{|I_{s-2}|=1}^{2^{s-2}}\!\!\tbinom{2^{s-2}}{|I_{s-2}|}
%\Big(\sum_{|I_{s-1}|=1}^{2^{s-1}}\!\!\tbinom{2^{s-1}}{|I_{s-1}|} e^{-\nu\rho(s-1)}\Big)...\Big)\\
%&\le \sum_{s=0}^\infty \Big(\frac{1}{e^{\nu-1}-1}\Big)^{s+1}\le \frac{1}{e^{\nu-1}-2}=C_\nu.
%\end{align*}
Finally, for  any  $I_0, I_1\in\mathcal{I}$  define  $I'=I_0\cup I_1$.
Then  $(\mathbb{L}_{I_0}\cup\mathbb{L}_{I_1})
\subseteq \mathbb{L}_{I'}=\mathbb{L}_{I_0}+\mathbb{L}_{I_1}$ and $ |I'|\log \big(\frac{e n}{|I'|}\big)
\le |I_0|\log \big(\frac{e n}{|I_0|}\big)+ |I_1|\log \big(\frac{e n}{|I_1|}\big)$, which entails Condition \eqref{(A3)}.

\begin{remark}
\label{rem_rho'}
We can take a slightly better majorant, $\rho'(I)=\max\{|I|,\log\binom{n}{|I|}\}$. 
%Condition \eqref{(A2)} is fulfilled for this majorant as well, since, 
%according to Remark \ref{rem5}, for any $\nu>1$ 
%\begin{align*}
%\sum_{I\in\mathcal{I}}e^{-\nu\rho'(I)}\le \sum_{s\in\mathcal{S}}  e^{-(\nu-1) \rho'(I)}
%\le \sum_{s=0}^n e^{-(\nu-1)s}\le 
%\frac{1}{1-e^{1-\nu}}= C_\nu.
%\end{align*}
%Clearly, $(\mathbb{L}_{I_0}\cup\mathbb{L}_{I_1})
%\subseteq \mathbb{L}_{I'}$ and  $|I'|\le |I_0|+|I_1|\le \rho'(s(I_0))+\rho'(s(I_1))$. 
%Since $ \binom{n}{|I'|} \le\binom{n}{|I_0|}\binom{n}{|I_1|}$, then
%$\log \binom{n}{|I'|}\le \log \binom{n}{|I_0|}+\log \binom{n}{|I_1|} \le \rho'(s(I_0))+\rho'(s(I_1))$. 
%%Since $(\frac{n}{s})^s \le \binom{n}{s} \le (\frac{en}{s})^s$, we obtain that
%%\begin{align*}
%%\log\binom{n}{s'}&\le s_0\log(\tfrac{en}{s_0})+ s_1\log(\tfrac{en}{s_1})
%%=s_0+s_0\log(\tfrac{n}{s_0})+ s_1+s_1\log(\tfrac{n}{s_1})\\&\le 2(\rho'(s_0)+\rho'(s_1)). 
%%\end{align*}
%Using this and $|I'|\le \rho'(s(I_0))+\rho'(s(I_1))$, we have $\rho'(s(I'))\le \rho'(s(I_0))+\rho'(s(I_1))$, 
%which also entails Condition \eqref{(A3)} for the majorant $\rho'(s)$. 
\end{remark}

As a consequence of our general results, we obtain Corollary \ref{meta_theorem}
with the local rate $r^2(\theta)= \min_{ I \in \mathcal{I}}\big\{\|\theta
-\mathrm{P}_I\theta\|^2 +\sigma^2\rho(I)\big\}$.
In view of Remark \ref{rem_rho'}, the results hold also with the local rate 
$r^2(\theta)= \min_{ I \in \mathcal{I}} \big\{\|\theta
-\mathrm{P}_I\theta\|^2 +\sigma^2\rho'(I)\big\}$. 
%Notice that the best oracle rate we can get is of the order $\frac{j_o}{n}$.
As $\rho'(I)\le \rho(I)$ for all $I\in\mathcal{I}$, the local rate with 
$\rho'(s)$ is smaller than the rate with $\rho(I)$ implying a stronger version of 
Corollary \ref{meta_theorem}. However, the quantity $\rho(I)$ is easier to compute, so 
we will use the majorant $\rho(I)$.
%However, there is one drawback in  $\rho'(s)$, namely,  $\log\binom{n}{s}$ is difficult to 
%compute, e.g., if $n$ is relatively big, whereas the upper bound for it can be easily derived. 

Below we present a couple of examples of scales $\{\Theta_\beta,\, \beta \in\mathcal{B}\}$, 
for which the global minimax adaptive results follow from the local results. 
There are of course more scales covered by the
oracle rate $r^{2}(\theta )$, one can establish the relation
\eqref{oracle_minimax} for other scales, for example for smoothness
scales (with a log factor in the minimax rate for smoothness scales).
Recall also that the results can be extended to non-normal and not necessarily independent 
$\xi_{i}$'s, but only satisfying Condition \eqref{cond_nonnormal}. For example, 
as demonstrated in \cite{Belitser&Nurushev:2015}, $\xi_i$'s originating from a 
certain AR(1)-model also satisfy Condition \eqref{cond_nonnormal}.

\subsection{Minimax results for the nearly black vectors $\ell_0$} 
%\paragraph{Nearly black vectors $\ell_0$}
By $I^*(\theta)$ and $s(\theta)$, we denote respectively the active index set and 
the sparsity of $\theta\in\mathbb{R}^n$. For $p\in[n]$, introduce the \emph{sparsity class} 
(also called \emph{nearly black vectors})
\begin{align}
\label{nearly_black_vectors}
\ell_0[p]=\{\theta\in\mathbb{R}^n: \|\theta\|_0=|I^*(\theta)|\le p\},\;\; 
I^*(\theta)=\{i\in[n]: \theta_i\!\not\! =0\}.
\end{align}
The minimax estimation rate over the class  of nearly black vectors $\ell_0[p]$ with the sparsity 
parameter $p$  is known to be $r^2(\ell_0[p])\asymp\sigma^2p\log(\frac{n}{p})$ 
(usually in the literature $p= p_n=o(n)$ as $n\to \infty$, but we do not impose this restriction); 
see \cite{Donoho&Johnstone&Hoch&Stern:1992}. 
The adaptive minimax results for \emph{nearly black vectors} were
considered in \cite{Belitser&Nurushev:2015, Castillo&vanderVaart:2012, Martin&Walker:2014,  vanderPas&Kleijn&vanderVaart:2014} and many others  for posterior contraction rates, and in    \cite{Belitser&Nurushev:2015, vanderPas&Szabo&vanderVaart:2017} 
for constructing optimal confidence balls.

By the definition  \eqref{oracle} of the oracle rate $r^2(\theta)$,
we have that $r^2(\theta) \le r^2(I^*(\theta),\theta)$. Then 
we obtain trivially that  
\[
\sup_{\theta\in \ell_0[p]} r^2(\theta) 
\le \sup_{\theta\in \ell_0[p]} r^2(I^*(\theta),\theta)
\le \sigma^2 p \log\big(\tfrac{en}{p}\big)\lesssim r^2(\ell_0[p]).
\]
We thus established the relation \eqref{oracle_minimax} for the scale $\{\ell_0[p],\,0\le p\le n\}$,
and Corollary \ref{minimax_results} follows with the minimax 
rate $r^2(\ell_0[p])$ defined above.

%\paragraph{Weak $\ell_q$-balls}
\subsection{Minimax results for the weak $\ell_q$-balls} 
For $q\in(0,2)$, the \emph{weak $\ell_q$-ball} of sparsity $p_n$ is defined by
\begin{align}
\label{weak_balls}
m_q[p_n]=\big\{\theta\in\mathbb{R}^n: 
\theta^2_{[i]}\le(p_n/n)^2(n/i)^{2/q}, \, i\in[n]\big\},
\end{align} 
where $p_n=o(\sigma n) \; \text{as} \; n\to \infty$,  $\theta_{[1]}^2\ge\ldots\ge \theta_{[n]}^2$ 
are the ordered %values of 
$\theta_1^2,\ldots, \theta_n^2$. 
This class %scale of weak $\ell_q$-balls %$\{m_q[p_n], q\in(0,2), 0\le p_n=o(\sigma n)\}$ 
can be thought of as Sobolev hyperrectangle for ordered (with unknown locations) coordinates:
$m_q[p_n]=\mathcal{H}(\beta,\delta_n)=
\{\theta\in\mathbb{R}^n: |\theta_{[i]}| \le \delta_n i^{-\beta}\}$, with 
$\delta_n =p_n n^{-1+1/q}$ and $\beta=1/q > 1/2$.

Denote $j=O_\theta(i)$ if $\theta_i^2=\theta_{[j]}^2$, with 
the convention that in the case $\theta_{i_1}^2=\ldots =\theta_{i_k}^2$ for $i_1 < \ldots < i_k$ 
we let $O_\theta(i_{l+1})=O_\theta(i_l)+1$, $l=1,\ldots, k-1$. 
The minimax estimation rate  over this class is 
$r^2(m_q[p_n])=n(\tfrac{p_n}{n})^q[\sigma^2\log(\tfrac{n\sigma}{p_n})]^{1-q/2}$ 
when $n^{2/q}(\tfrac{p_n}{n})^2 \ge \sigma^2 \log n$, and $r^2(m_q[p_n])
=n^{2/q}(\tfrac{p_n}{n})^2 +\sigma^2$ when $n^{2/q}(\tfrac{p_n}{n})^2<\sigma^2 \log n$,
as $n\to \infty$; see \cite{Donoho&Johnstone:1994b, Birge&Massart:2001}. 
The adaptive minimax results for the scale of \emph{weak $\ell_q$-balls} were
considered in \cite{Belitser&Nurushev:2015, Castillo&vanderVaart:2012}  for posterior contraction rates and in \cite{Belitser&Nurushev:2015}  for constructing optimal confidence balls.
We take $I_0(\theta)=\{i \in[n]: O_\theta(i) \le p^*_n\}$,
with $p_n^*=en (\frac{p_n}{n\sigma})^q[\log(\frac{n\sigma}{p_n})]^{-q/2}$ 
in the case $n^{2/q}(\tfrac{p_n}{n})^2\ge\sigma^2 \log n$, to derive \eqref{oracle_minimax}:  
\begin{align*}
\sup_{\theta\in m_q[p_n]} r^2(\theta) &\le \sup_{\theta\in m_q[p_n]} r^2(I_0(\theta),\theta)
\le \sigma^2 p_n^*\log(\tfrac{en}{p^*_n})+n^{2/q} (\tfrac{p_n}{n})^2\sum_{i>p_n^*} i^{-2/q} 
\notag\\
&\le C_1\sigma^2p_n^*\log(\tfrac{n\sigma}{p_n})+C_2 n^{2/q} (\tfrac{p_n}{n})^2 (p^*_n)^{1-2/q} \notag\\
&\lesssim n (\tfrac{p_n}{n})^q\big[\sigma^2 \log (\tfrac{n\sigma}{p_n})\big]^{1-q/2}
\lesssim r^2(m_q[p_n]).  
\end{align*}
The case $n^{2/q}(\tfrac{p_n}{n})^2<\sigma^2 \log n$ is treated similarly by taking
$p_n^*=0$. 
Corollary \ref{minimax_results} follows for this case with the minimax rate $r^2(m_q[p_n])$  defined above.

\subsection{Minimax results for Besov scales}
\label{subsec_besov_rev}
Consider again the model \eqref{model_wavelet} with $j \in [J_0]_0$, where $J_{0}\in \mathbb{N}$ 
is such that $2^{J_{0}+1}=n$. 
We can see \eqref{model_wavelet} as $J_{0}+1$ models of type
\eqref{model_sparsity}, where $\sigma ^{2} = n^{-1}$ and the $j$-th model has
$2^{j}$ observations, $j \in [J_{0}]_{0}$. Let $\theta ^{j} =\bigl(\theta_{jk}, k \in [2^{j}]\bigr)$ 
and $r^{2}(\theta ^{j}, I_{oj})$ denote the oracle
rate in $j$-th model. Then aggregating the oracle results over these $J_{0}+1=\log _{2}n$ models leads to the results for the whole model \eqref{model_wavelet} with the aggregated oracle rate $r^{2}(\theta )
= \sum_{i \in [J_{0}]_{0}} r^{2}(\theta ^{j}, I_{oj})$. Because of the
aggregation, in Corollary \ref{meta_theorem} we get $\frac{\log _{2} n}{n} M$
instead of $\sigma ^{2} M$ and $(\log _{2} n)H_{l}$ instead of
$H_{l}$, $l=0,1$.

Assume that the true signal $\theta $ belongs to a Besov ball $\Theta _{p,q}^{\beta }(Q)$ defined 
by \eqref{Besov_class}, for some $p,q, Q>0$, $\beta \ge 1/p$.
Now, exactly in the same way as in Section \ref{subsec_Besov_scale}, we derive that
%The minimax rate  over $\Theta _{p,q}^{\beta }(Q)$ is known to be 
%$r^{2}\bigl(\Theta _{p,q}^{\beta}(Q)\bigl)\asymp n^{-\frac{2\beta }{2\beta +1}}$. 
for any $\theta\in \Theta _{p,q}^{\beta }(Q)$, %$r^{2}(\theta ) \lesssim r^{2}\bigl(\Theta _{p,q}^{\beta }(Q)\bigr)$. 
\begin{align*}
r^{2}(\theta ) &\le \sum_{j\in [J_{0}]_{0}}\sum
_{k\in I_{oj}^{c}} \theta _{jk}^{2}+ \sum
_{j\in [J_{0}]_{0}} \frac{| I_{oj} | }{n}
\log \Bigl(\frac{e 2^{j}}{ | I_{oj} |}\Bigr) 
\\&
\le \sum_{j\in [J_{0}]_{0}}\min_{0\le k\le 2^{j}}
\biggl(\sum_{l>k} \theta _{j(l)}^{2}+
\frac{k}{n}\log \bigl(e2^{j}/k\bigr) \biggr) %\\&
\le C
n^{-\frac{2
\beta }{2\beta +1}} \asymp r^{2}\bigl(\Theta _{p,q}^{\beta }(Q)
\bigr),
\end{align*}
where $\theta _{j(l)}^{2}$ denotes the $l$-th largest value among $\{\theta _{jk}^{2}, j\in [2^{k}]\}$. 
The third inequality of the last display follows from Theorem 12.1 in \cite{Johnstone:2017} under the
assumption $\beta \ge 1/p$. We thus established the relation \eqref{oracle_minimax} for the Besov scale, 
so that the global minimax adaptive results for the Besov scale follow by Corollary
\ref{minimax_results} with $\Theta _{\beta }= \Theta _{p,q}^{\beta }(Q)$ and the minimax rate 
$r^{2}(\Theta _{\beta}) =r^{2}\bigl(\Theta _{p,q}^{\beta }(Q)\bigr)\asymp n^{-\frac{2\beta }{2\beta +1}}$. 
Recall that we have to set $\frac{\log _{2} n}{n} M$ instead of $\sigma ^{2} M$ and $(\log _{2} n)H_l$ 
instead of $H_{l}$, $l=0,1$, because of the aggregation. In this case, the asymptotic regime $n\to \infty $ 
is of interest. Let us formulate the first claim of Corollary \ref{minimax_results} in this case
(other claims can be formulated similarly): for some $C>0$ and any $M\ge 0$,
\[
\sup_{\theta \in \Theta _{p,q}^{\beta }(Q)}
\mathrm{E}_{\theta } \hat{\pi } \Bigl(
\| \vartheta -\theta \|^{2}\ge C n^{-\frac{2\beta }{2\beta +1}}+
M\frac{\log _{2} n}{n} \big| Y \Bigr) \le H_{0} (\log _{2} n) e^{-m_{0}M}.
\]
Take for example $M=M_{n}=n^{1/(2\beta +1)}/\log _{2} n$ to obtain a well
interpreted asymptotic relation. 
\begin{remark}
\label{rem_better_structure}
Notice that we consider minimax results over Besov scales also in Section \ref{subsec_Besov_scale}, 
and the  results obtained in this section are slightly weaker than the ones from Section \ref{subsec_Besov_scale}, 
in view of the log factors. This is because the structure here is sparsity, whereas in 
Section \ref{subsec_Besov_scale} it is smoothness+sparsity that is better suited for Besov scales.
\end{remark}
 
\section{Signal+noise with clustering (multi-level sparsity) structure}
\label{sec_multi-level}

Consider the same model \eqref{model_sparsity}, but now with
the so called \emph{clustering} (or \emph{multi-level sparsity})  structure, an extension of the traditional sparsity 
structure. In the usual one-level sparsity structure we have just one known sparsity level, 
which is by default zero. The first attempt to study a version of 
such structure has been undertaken in \cite{Belitser&Nurushev:2017} (by a different approach),  
here we propose a systematic approach to this from the general perspective 
of the linear spaces for the first time.   
To the best of our knowledge, this structure has never been systematically studied in the literature.

First we extend the classical sparsity structure by allowing the sparsity level to be 
an unknown constant, not necessarily zero. This extended \emph{unknown level sparsity} structure
is described by the linear spaces:
\begin{align*}
\mathbb{L}_I=\big\{x\in\mathbb{R}^n: x_i=x_j,\, \forall\, i,j\in I^c\big\},\quad I\in\mathcal{I} = 
\{J: J \subseteq [n]\}.
\end{align*}
Then $d_I=\dim(\mathbb{L}_I)=(|I|+1)\wedge n$, $\|\theta-\mathrm{P}_I\theta\|^2
=\sum_{i\in I^c}(\theta_i-\bar{\theta}_{I^c})^2$ (where 
$\bar{\theta}_{I^c} =|I^c|^{-1} \sum_{i\in I^c} \theta_i$), and the structural slicing mapping
$s(I)=|I|\in\mathcal{S}\triangleq[n]_0$. 
Compute $|\mathcal{I}_s|=\binom{n}{s}$, hence 
$d_I+\log |\mathcal{I}_{s(I)}|=d_I+\log \binom {n}{|I|}
\le (|I|+1)\wedge n+|I|\log(\frac{en}{|I|})$ and the majorant is $\rho(I)=(|I|+1)\wedge n+|I|\log(\frac{en}{|I|})$.

Next, we extend the one-level sparsity structure to the multi-level 
sparsity structure (with unknown sparsity levels) by introducing  the following linear spaces:
for a partition $I=(I_i, \, i\in [m]_0)$ of the set $[n]$ into $m+1$ parts,
%$I\in\mathcal{I}=\{(I_i,i\in[m]_0): I_i \subseteq [n], I_i\cap I_j=\varnothing, \cup_{i\in[m]_0} I_i=[n]\}$,
\begin{align*}
\mathbb{L}_I=\big\{x\in\mathbb{R}^n: x_j=x_{j'},\, \forall j,j'\in I_i,\, 
i\in [m] \big\}, \quad I\in\mathcal{I},
\end{align*}
where $\mathcal{I}=\mathcal{I}_m$ is the family of all partitions of $[n]$ into $m+1$ parts 
(some possibly empty), and $m=2,\ldots, n-1$. 
This can also be seen as \emph{clustering structure}, where the partition $I$ determines 
clustering of the coordinates of $\theta$ into $m+1$ groups. 
In this case,  compute $\|\theta-\mathrm{P}_I\theta\|^2
=\sum_{k=1}^m \sum_{i\in I_k}(\theta_i-\bar{\theta}_{I_k})^2$ with the group averages
$\bar{\theta}_{I_k} = \tfrac{1}{|I_k|}\sum_{i\in I_k} \theta_i$, 
the structural slicing mapping is taken to be $s(I)=(|I_i|, i\in[m]_0)\in\mathcal{S}$, 
where $\mathcal{S}=\mathcal{S}(n,m+1)=\{(n_i,i\in[m]_0):\, n_i\in[n]_0, \sum_{i\in[m]_0} 
n_i=n\}$ is the family of the so called \emph{weak compositions} of $n$ into $m+1$ parts. 
It is well known that $|\mathcal{S}|={n+m  \choose m}$.
Further we have $d_I=\dim(\mathbb{L}_I)=(|I_0|+m)\wedge n$ and 
$|\mathcal{I}_{s(I)}|=\binom{n}{|I_0|,\ldots, |I_m|}$ is the multinomial coefficient.

\begin{remark}
An interesting variation of the above structure is when we insist on consecutive clusters:
for $t_i \in \mathbb{N}$ such that $1=t_1\le \ldots \le t_m<  t_{m+1}=n+1$, 
let $I_i =\{t_i,t_i+1, \ldots, t_{i+1}-1\}$, $i\in[m]$, with the convention that $I_i=\varnothing$
if $t_i=t_{i+1}$. 
One can do the computations for this case, %presumably
also when allowing the number of clusters vary: $m\in[n]$. This will be studied elsewhere 
in the context of a \emph{change point} problem.
\end{remark}

Conditions \eqref{cond_nonnormal} and \eqref{cond_A4} hold with $d_I =\dim(\mathbb{L}_I)$
in view of Remarks \ref{rem_cond_A1} and \ref{rem15}. 
To ensure Condition \eqref{(A2)}, we have to compensate for the number $|\mathcal{S}|$ (which can 
be big in general) by adding the term $\log|\mathcal{S}|=\log{n+m  \choose m}$ in the complexity 
majorant $\rho(I)$. Hence,  we take the majorant $\rho(I)=(|I_0|+m)\wedge n
+\log \binom{n}{|I_0|,\ldots, |I_m|}+\log{n+m  \choose m}$, so that
Condition \eqref{(A2)} is now fulfilled for any $\nu\ge 1$: 
\begin{align*}
\sum_{I\in\mathcal {I}} e^{-\nu \rho(I)}=\sum_{s\in\mathcal{S}}\sum_{I\in\mathcal {I}_s} 
e^{-\nu\rho(I)}\le \sum_{s\in\mathcal{S}} e^{-\nu \log |\mathcal{S}|}\le 1. 
%\frac{1}{1-e^{1-\nu}}= C_\nu.
%&\le\sum_{j_0=0}^\infty 2^{j_0+1} e^{-j_0}\le C_\nu.
\end{align*}
%\begin{remark}
Notice that the factor $\log{n+m  \choose m}$ in $\rho(I)$ is too conservative for some 
$I\in\mathcal{I}$, for example, we can set this factor to zero if $|I_i| =n$ for some $i\in[m]$.
%\end{remark}

Unfortunately, we were unable to establish Condition \eqref{(A3)} for this structure, which is needed 
for the uncertainty quantification results under the EBR condition. What we can claim are 
the relations (i)--(iv) and (vi)--(vii) of Corollary \ref{meta_theorem} with the local rate 
\begin{align*}
r^2(\theta) &= \min_{ I \in \mathcal{I}}\big\{\|\theta
-\mathrm{P}_I\theta\|^2 +\sigma^2\rho(I)\big\}%\\&
=\min_{I \in \mathcal{I}}\Big\{ \sum_{k=1}^m \sum_{i\in I_k}(\theta_i-\bar{\theta}_{I_k})^2  \\
& \qquad\qquad
+\sigma^2 \big[(|I_0|+m)\wedge n+\log \tbinom{n}{|I_0|,\ldots, |I_m|}+\log \tbinom{n+m}{m}\big]\Big\}.
\end{align*}
For $m=1$ we get the classical one-level local sparsity results  which also imply the global 
minimax results over sparsity scales, as is considered in the previous paragraph. 
For $m\ge 2$, the obtained local results (i)--(iv) and (vi)--(vii) of Corollary \ref{meta_theorem} 
are new to the best of our knowledge. 
%Consider the case $|I_i|\asymp n/(m+1)$, then 
%\[
%\log \binom{n}{n/(m+1),\ldots,n/(m+1)} \asymp \log \Big((2\pi \tfrac{n}{m+1})^{(1-(m+1))/2}
%(m+1)^{n+1/2} \Big).
%\]
The most problematic term is $\log \binom{n}{|I_0|,\ldots, |I_m|}$, this term is of a smaller order 
than $n$ if $|I_0|$ and any $m-1$ values among $|I_1|,\ldots,|I_m|$ 
(e.g., $|I_0|,|I_1|,\ldots,|I_{m-1}|$) are themselves of the smaller order than $n$.

\begin{remark}
It is an open problem to establish Condition \eqref{(A3)}. This is important in the uncertainty quantification 
problem, namely, the coverage relation  (v) from Corollary \ref{meta_theorem} relies on this. 
If we are to verify Condition \eqref{(A3)}, for any  $I,I'\in\mathcal{I}$  we would define %the partition refinement 
\[
I''=I''(I,I')=\big(I_0\cup I'_0,(I_i\cap I'_{i'}, \, i,i' \in [m])\big).
\] 
Clearly, $\mathbb{L}_{I'}\subseteq \mathbb{L}_{I''}$, $\mathbb{L}_{I}
\subseteq \mathbb{L}_{I''}\subseteq\mathbb{L}_{I} + \mathbb{L}_{I'}$ 
and $ \max\{s(I),s(I')\} \le s(I'') \le s(I)+s(I')$,
implying $ \rho(I'')\le\rho(I)+\rho(I')$, and seems that Condition \eqref{(A3)} is fulfilled. 
However, the problem is that  the resulting $I''$ may in general not lie in $\mathcal{I}$ but 
rather in $\mathcal{I}_{m^2}$. An idea to fix this would be to let the number $m$ of parts 
in partitions $I\in\mathcal{I}$ free (any integer from $0$ to $n$).
But then the problem will emerge in another place: there are too many choices as the 
family $\mathcal{S}$ of all compositions of 
$n$ becomes $|\mathcal{S}| = 2^{n-1}$. Then we will have to put the term 
$\log|\mathcal{S}| \asymp n $ in the complexity majorant $\rho(I)$ to meet Condition \eqref{(A2)},
which makes the local rate $r^2(\theta) \gtrsim n\sigma^2$ trivially large and therefore uninteresting. 
\end{remark}

\subsection{Minimax results for the clustering (multi-level sparsity)} 
The global minimax results are not going to be useful, at least if we try to extend one-level sparsity 
scales to multi-level sparsity scales in the usual way. Indeed, even if we assume sparsity in the sense 
that $|I_0| \le s$ for some small $s=s_n \ll n$, i.e., $\theta \in \Theta_s =
\cup_{I \in \mathcal{I}: |I_0|\le s} \mathbb{L}_I$, 
%(or even $\Theta_s=\cup_{I \in \mathcal{I}: |I_0|=s} \mathbb{L}_I$), 
the minimax rate over $\Theta_s$ will presumably be (one will have to prove the lower bound also) 
\[
r^2(\Theta_s) \asymp \sigma^2 \max_{I \in \mathcal{I}: |I_0|\le s} \rho(I) 
\gtrsim \sigma^2 \max_{I \in \mathcal{I}: |I_0|\le s} \log \tbinom{n}{|I_0|,\ldots, |I_m|}\gtrsim n \sigma^2,
\]
which would not be useful. This means basically that the multilevel counterpart $\Theta_s$ for the 
traditional one-level sparsity class $\ell_0[s]$  is too ``massive'' in the minimax sense. 

One can propose other scales $\{\Theta_\beta, \beta\in\mathcal{B}\}$ with more structure, for which at least minimax consistency would hold, i.e., $r^2(\Theta_\beta) \ll \sigma^2 n$.
For example, consider \[
\Theta_{s,m}=\cup\big\{ \mathbb{L}_I: I \in \mathcal{I}, 
|I_j|\le s_j, j\in [m]_0\backslash \{i\} \text{ for some } i\in[m]\big\}, 
\]
with $s=(s_j, j\in[m])$ and $m\in[n]$ such that 
$(|s|+m)  \log n \ll n$, where $|s| =\sum_{j\in[m]}s_j$.
Then it is easy to see that for any $\theta \in \Theta_{s,m}$ 
\[
r^2(\theta) \le \sigma^2 \big[(|I_0|+m)+ \log \tbinom{n}{|I_0|,\ldots, |I_m|}+\log \tbinom{n+m}{m}\big]
\lesssim \sigma^2 (|s|+m)\log n.
\] 
One needs to establish the corresponding lower bound for the minimax rate over 
$\Theta_{s,m}$.

\section{Signal+noise with shape structure: isotonic, unimodal and convex regressions}
\label{sec_shape_restr}

Consider the standard signal+noise model \eqref{model_sparsity}, but now assume that 
the parameter $\theta=(\theta_i, i \in [n]) \in \mathbb{R}^n$ possibly belongs to
%Assume that the observations are $Y=(Y_i)_{i\in[n]}$ according to the model
%\begin{align*}
%Y_i=\theta_i+\sigma\xi_i,\; \; i\in[n],
%\end{align*}
%where $\xi_i \overset{\rm ind}{\sim}\mathrm{N}(0,1)$ and $\theta=(\theta_1,\ldots, \theta_n)\in\mathbb{R}^n$ 
%is the unknown high-dimensional parameter, possibly belonging to  
one of the three classes:
\begin{align}
\label{isotonic}
&\mathcal{S}^{\uparrow}=\{\theta\in\mathbb{R}^n:\,\theta_i\le \theta_{i+1},\, i=1,\ldots, n-1\}, 
\;\; n \ge 2;\\
\label{unimodal}
&\mathcal{U}^m=\{\theta\in\mathbb{R}^n:\,\theta_1\ge\ldots\ge\theta_{m}\le\theta_{m+1}\le\ldots\le \theta_n\}, 
\;m\in[n],\;n \ge 2;\\
\label{convex}
&\mathcal{C}=\{\theta\in\mathbb{R}^n:\, 2\theta_i\le \theta_{i+1}+\theta_{i-1},\,i=2,\ldots ,n-1\}, 
\;\; n\ge3.
\end{align} 
The \emph{isotonic}, \emph{unimodal} and \emph{convex} regression 
problems concern the classes  $\mathcal{S}^{\uparrow}$, $\mathcal{U}^m$ and $\mathcal{C}$, respectively. Recently,  oracle estimation   results for these problems  were derived by \cite{Chatterjeeetal:2015, 
Bellec&Tsybakov:2015, Bellec:2018}.  To the best of our knowledge, there are 
no local results on  posterior contraction rate and uncertainty quantification problems 
for these structures.

First, to model parameters from $\mathcal{S}^{\uparrow}$ and $\mathcal{U}^m$, 
introduce the linear spaces 
\begin{align*}
\mathbb{L}_{I}=\big\{x\in\mathbb{R}^n: x_i=x_{i+1}, \; i\not\in I \big\},\quad I\subseteq\mathcal{I}=[n-1],
\end{align*}
where $d_I=\dim(\mathbb{L}_I)=%(|I|+1) 1\{n\not\in I\} + |I|1\{n\in I\} \le 
|I|+1$. 
The structural slicing mapping is $s(I)=|I|$, so that $\mathcal{S}=[n-1]$.
Compute $|\mathcal{I}_{s(I)}|=\binom{n-1}{|I|}$, hence $\log |\mathcal{I}_{s(I)}|\le |I|\log (\frac{en}{|I|})$.
Since $d_I + |\mathcal{I}_{s(I)}| \le |I|+1+\log \binom{n-1}{|I|} \le
1+2|I|\log (\frac{en}{|I|}) $, we take the majorant $\rho(I)=1+2|I|\log (\frac{en}{|I|})$. 
 
Next, the parameters from $\mathcal{C}$ are modeled by the linear spaces
\begin{align*}
\mathbb{L}'_{I}=\big\{x\in\mathbb{R}^n: 2x_i=x_{i+1}+x_{i-1}, \; i\not\in I\big\}, 
\quad I\subseteq\mathcal{I}'=\{2,\ldots,n-1\},
\end{align*}
where $d_I=\dim(\mathbb{L}_I)\le(2|I|\vee 1)\wedge n\le 2|I|+1$. 
The structural slicing mapping in this case is $s'(I)=|I|$, 
so that $\mathcal{S}'=\{2,\ldots n-1\}$. Compute $|\mathcal{I}_{s'(I)}|=\binom{n-2}{|I|}$, 
hence $\log |\mathcal{I}_{s'(I)}|\le |I|\log (\frac{en}{|I|})$. Since 
$d_I+|\mathcal{I}_{s'(I)}|\le 2|I|+1+|I|\log (\frac{en}{|I|})\le 1+3|I|\log (\frac{en}{|I|})$, 
we take the majorant $\rho'(I)=1+3|I|\log (\frac{en}{|I|})$. 

\begin{remark}
The traditional approach to shape structures is by projecting the data on 
one corresponding convex (or closed) set. We instead work with a family of 
linear spaces which, in a way, reproduces the shape structure. Moreover, 
at the price of a log factor, our approach has certain universality feature; 
see Section \ref{subsec_universality}.
\end{remark}

We introduced two different families of structures with two corresponding (different) 
families of linear spaces, but the majorants in the both cases can be chosen the same 
(up to a multiplicative constant). Conditions \eqref{cond_nonnormal}--\eqref{cond_A4} 
for the both cases are fulfilled in the same way 
as for the model considered in Section \ref{sec_sparsity}, we omit the argument and 
computations that are very much along the same lines as in Section \ref{sec_sparsity}. 
As consequence of our general results, we obtain the local results of Corollary \ref{meta_theorem}
for the both cases with the local rate $r^2(\theta)=\min_{I\in\mathcal{I}}
\big\{\|\theta-\mathrm{P}_I \theta\|^2+\sigma^2\rho(I)\big\}$ and
$r'^2(\theta)=\min_{I\in\mathcal{I}'}\big\{\|\theta-\mathrm{P}_I\theta\|^2+\sigma^2\rho'(I)\big\}$. 
In turn, by virtue of Corollary \ref{minimax_results} the local results will  imply global minimax 
adaptive results at once over all scales $\{\Theta_\beta,\, \beta\in\mathcal{B}\}$ covered by 
the oracle rate $r^2(\theta)$ and $r'^2(\theta)$ (i.e., for which \eqref{oracle_minimax} holds).
Below we present a couple of examples of such scales $\{\Theta_\beta,\, \beta\in\mathcal{B}\}$.

\subsection{Minimax results for isotonic, unimodal  and convex regressions} %up to a logarithmic factor.} 
Following \cite{Bellec&Tsybakov:2015}, for $\theta\in\mathbb{R}^n$,
denote the number of relations $\theta_{i} \neq \theta_{i+1}$ for $i\in[n-1]$ by $k(\theta)-1$ 
(number of jumps of $\theta$), and for $\theta \in \mathcal{C}$, 
the number of inequalities $2\theta_i \le \theta_{i+1} + \theta_{i-1}$ that are strict for 
$i = 2,\ldots, n-1$ by $q(\theta)-1$. Let $\mathcal{U}=\cup_{m=1}^n\mathcal{U}^m$, 
where $\mathcal{U}^m$ is defined by \eqref{unimodal}.
Define the classes of monotone and unimodal parameters with at most $k$ jumps 
and the class of piecewise linear convex parameters with at most 
$q$ linear pieces as follows: for $k,q\ge1$,
\begin{align*}
&\mathcal{S}_k^{\uparrow}=\{\theta\in\mathcal{S}^{\uparrow}: k(\theta)\le k\},\;\; 
\mathcal{U}_k=\{\theta\in\mathcal{U}: k(\theta)\le k\},\;\; %\quad %\\&
\mathcal{C}_q=\{\theta\in\mathcal{C}:q(\theta)\le q\},
\end{align*}
where $\mathcal{S}^{\uparrow}$ and  $\mathcal{C}$ are defined by \eqref{isotonic} and \eqref{convex}.
Define $\Theta^\uparrow_k=\cup_{I\in\mathcal{I}:|I|+1\le k} \mathbb{L}_I$ 
and notice that for each $\theta\in\mathcal{S}_k^{\uparrow}$ (or $\theta\in\mathcal{U}_k$)
there exists $I_*\in\mathcal{I}$ such that 
$\theta\in\mathbb{L}_{I_*}$ and $k(\theta)=|I_*|+1$, implying that 
$\mathcal{S}_k^{\uparrow}\subseteq\Theta^\uparrow_k$
(and $\mathcal{U}_k\subseteq\Theta^\uparrow_k$).
Similarly, we define $\Theta_q=\cup_{I\in\mathcal{I}':|I|+1\le q} \mathbb{L}'_I$
and derive that $\mathcal{C}_q\subseteq\Theta_q$.

As is shown in \cite{Bellec&Tsybakov:2015}, the minimax rates over $\mathcal{S}_k^{\uparrow}$ 
and $\mathcal{C}_q$, with $k,q\ge 1$, are $r^2(\mathcal{S}_k^{\uparrow})\triangleq
\inf_{\hat{\theta}}\sup_{\theta \in \mathcal{S}_k^{\uparrow}}\mathbb{E}_\theta\|\hat{\theta}-\theta\|^2 
\asymp \sigma^2 k$ and $r^2(\mathcal{C}_q) \asymp \sigma^2 q$, respectively. Due to the fact that 
$\mathcal{S}^\uparrow\subseteq \mathcal{U}$, we also have $r^2(\mathcal{U}_k)\gtrsim\sigma^2 k$. 
%(cf.\ \cite{Bellec:2018}).  
Now, for each $\theta \in\mathcal{S}_k^{\uparrow}$ (or $\theta \in \mathcal{U}_k$) there exists 
$I_*\in\mathcal{I}$ such that $\theta\in\mathbb{L}_{I_*}$ (so that $\mathrm{P}_{I_*}\theta =\theta$)
and $|I_*|+1 =k(\theta) \le k$. Hence, $r^2(\theta)\le r^2(I_*,\theta)=
\sigma^2\big(1+2|I_*|\log(\tfrac{en}{|I_*|})\big)\lesssim\sigma^2k\log(\frac{en}{k})$ for all 
$\theta \in\mathcal{S}_k^{\uparrow}$ and $\theta \in \mathcal{U}_k$. 
Similarly, we show that $r'^2(\theta)\lesssim\sigma^2q\log(\frac{en}{q})$ for all 
$\theta\in\mathcal{C}_q$. We thus established the relation \eqref{oracle_minimax} 
for the classes $\mathcal{S}_k^{\uparrow}$, $\mathcal{U}_k$ and $\mathcal{C}_q$, which implies 
the minimax results (up to a logarithmic factor) of Corollary \ref{minimax_results} 
for all these three classes.

Finally introduce the classes of (shape-restricted) monotone, unimodal 
and convex parameters $\theta$ with bounded total variation:
$\mathcal{S}^{\uparrow}(V)=\{\theta\in\mathcal{S}^{\uparrow}: V(\theta)\le V\}$,
$\mathcal{U}(V)=\{\theta\in\mathcal{U}: V(\theta)\le V\}$,
$\mathcal{C}(V)=\{\theta\in\mathcal{C}: V(\theta)\le V\}$,
where $V(\theta)=\max_{i,j}(\theta_i-\theta_j)$ (notice that $V(\theta)=\theta_n-\theta_1$ 
for $\theta\in\mathcal{S}^{\uparrow}$). It is known that the minimax rates over 
$\mathcal{S}^{\uparrow}(V)$, $\mathcal{U}(V)$ and $\mathcal{C}(V)$ 
are respectively 
$r^2(\mathcal{S}^{\uparrow}(V))\asymp \max\{n^{1/3}(\sigma^{2}V)^{2/3},\sigma^2\}$,
$r^2(\mathcal{U}(V))\asymp \max\{n^{1/3}(\sigma^{2}V)^{2/3},\sigma^2\}$ and
$r^2(\mathcal{C}(V))\gtrsim n^{1/5}(\sigma^4V)^{2/5}$ if $V\ge \sigma$.
To derive the Corollary \ref{minimax_results} for these classes, 
we need the next proposition, where claim (i) is Lemma 2 from \cite{Bellec&Tsybakov:2015}
and claim (ii) is Lemma 4.1 from \cite{Bellec&Tsybakov:2015}. 
We give these claims here (in our notation) for completeness, 
the proofs can be found in the mentioned references. 
\begin{proposition}
\label{proposition_isotonic}
Let $k,q\in[n]$. Then the following properties hold.
\begin{itemize}
\item[(i)] For any $\theta\in\mathcal{S}^{\uparrow}$ (or $\theta\in\mathcal{U}$) 
there exists a $\theta^*=\theta^*(\theta) \in \Theta^\uparrow_k$ such that 
$\| \theta - \theta^*\|^2 \le C_1 \tfrac{nV^2(\theta)}{k^2}$ for some absolute 
constant $C_1>0$. \\ 
\item[(ii)] 
For any $\theta\in\mathcal{C}$ there exists a $\theta^*=\theta^*(\theta)\in\Theta_q$
such that $\| \theta - \theta^*\|^2 \le C_2\tfrac{nV^2(\theta)}{q^4}$ for some 
absolute constant $C_2>0$.
\end{itemize}
\end{proposition}

By Proposition \ref{proposition_isotonic}, 
$\theta^*\in\Theta^\uparrow_k$, then there exists an $I_*\in\mathcal{I}$ such that 
$\theta^*\in\mathbb{L}_{I_*}$ and
\begin{align*}
\|\theta-\mathrm{P}_{I_*}\theta\|^2\le \|\theta-\theta_*\|^2\le C_1 \tfrac{nV^2(\theta)}{k^2}.
\end{align*}
It follows therefore that  for any  $k \in [n]$ and any $\theta\in\mathcal{S}^{\uparrow}$ 
(or $\theta\in\mathcal{U}$) ,
\[
r^2(\theta) \le r^2(I_*,\theta) \lesssim \tfrac{nV^2(\theta)}{k^2} + \sigma^2 k\log(\tfrac{en}{k}).
\]

Let $\theta\in\mathcal{S}^{\uparrow}$ (or $\theta\in\mathcal{U}$) and let us 
take $k=k_*=\lfloor \big(\frac{nV^2(\theta)}{\sigma^2\log (en)}\big)^{1/3}\rfloor+1$.
If $k^*=1$, we have $nV^2(\theta)\le \sigma^2\log(en)$. If $k^*>1$, then by definition 
of $k^*$, $\tfrac{nV^2(\theta)}{(k^*)^2} \le n^{1/3}(\sigma^2 V(\theta)\log(en))^{2/3}$. 
%Since, according to Proposition \ref{proposition_isotonic}, 
%$\theta^*\in\Theta^\uparrow_k$, there exists an $I_*\in\mathcal{I}$ such that 
%$\theta^*\in\mathbb{L}_{I_*}$. Then 
%\begin{align*}
%&\|\theta-\mathrm{P}_{I_*}\theta\|^2\le \|\theta-\theta_*\|^2\le
%\frac{1}{4}\max\Big\{n^{1/3}(\sigma^2V(\theta)\log (en))^{2/3},\sigma^2\log(en)\Big\},\\
%&\sigma^2s_*\log \big(\frac{en}{s_*}\big)\le \max\Big\{n^{1/3}(\sigma^2V(\theta)\log (en))^{2/3},\sigma^2\log(en)\Big\}.
%\end{align*}
We conclude that if $V(\theta)\le V$, then for all $\theta\in \mathcal{S}^{\uparrow}(V)$ 
and $\theta\in\mathcal{U}(V)$, 
\[
r^2(\theta)\lesssim\max
\big\{n^{1/3}(\sigma^2V\log (en))^{2/3},\sigma^2\log(en)\big\}.
\] 
Thus, we established the relation \eqref{oracle_minimax}, and Corollary \ref{minimax_results} 
follows with the minimax  (up to a log factor) rate $\max\{n^{1/3}(\sigma^{2}V)^{2/3},\sigma^2\}$
for the both classes 
$\mathcal{S}^{\uparrow}(V)$ and $\mathcal{U}(V)$ simultaneously.

Similarly, we establish that for all $\theta\in \mathcal{C}(V)$ 
\[
r'^2(\theta)\lesssim \max
\big\{n^{1/5}(\sigma^4V)^{2/5}[\log(en)]^{4/5} ,\sigma^2\log(en)\big\}.
\]
This means that we established the relation \eqref{oracle_minimax} and hence also  
Corollary \ref{minimax_results} with the minimax rate (up to a logarithmic factor) 
for the class $\mathcal{C}(V)$.

%Below we present some further remarks and extensions for isotonic, unimodal 
%and convex regressions.

\subsection{Log factor and universality of the results}
\label{subsec_universality}
It should be recognized that we attain the minimax rates for the classes $\mathcal{S}^{\uparrow}_k$,
$\mathcal{U}_k$, $\mathcal{C}_q$, $\mathcal{S}^{\uparrow}(V)$,
$\mathcal{U}(V)$ and $\mathcal{C}(V)$ only up to a logarithmic factor.
On the other hand, we obtain the optimal rates over the bigger scales 
$\{\Theta^{\uparrow}_k, k\in[n]\}$ and $\{\Theta_k, k\in[n]\}$. 
Moreover, as consequence of our general results we have also solved 
the uncertainty quantification problem and the problem of structure recovery (in a weak sense).
Our constants in the estimation results may be worse than those from the above mentioned 
references, but on the other hand we do not require that the vector $\xi$ is normal 
and its coordinates are independent, only mild Condition \eqref{cond_nonnormal} is to be fulfilled.  
    
Interestingly, the extra log factor in the local rate can also be seen as ``price'' for certain 
\emph{universality} of the results. Indeed, recall that the results for the family of 
structures $\mathcal{I}$ with corresponding linear spaces $\mathbb{L}_I$, $I\in\mathcal{I}$, 
cover the scale $\{\Theta^{\uparrow}_k, k\in[n]\}$. This in turn implies the minimax results for 
the scales $\{\mathcal{S}^{\uparrow}_k, k\in[n]\}$ and $\{\mathcal{U}_k, k\in[n]\}$ (adaptively 
with respect to $k\in[n]$) and over the global shape-restricted classes $\mathcal{S}^{\uparrow}(V)$ 
and $\mathcal{U}(V)$ of monotone and unimodal parameters, \emph{simultaneously} 
for all the mentioned scales. Thus, at the log factor price, one approach handles several 
structures at once.

Actually our approach allows to extend the universality property even further. Indeed, let us unite 
the two structures families $\bar{\mathcal{I}} = \mathcal{I}\cup\mathcal{I}'$ and the corresponding 
families of the linear spaces $\{\mathbb{L}_I, I\in\bar{\mathcal{I}}\}$ and consider the resulting procedure. 
This procedure makes sense because the majorants for the both families are of the same order, 
so we only need to adjust a multiplicative constant in front of the majorant $\rho(I)=
1+2|I|\log (\tfrac{en}{|I|})$ that will now handle the both families of structures. 
In doing so, we get the local result with the oracle rate over the both families at the price 
of a bigger multiple of the majorant. This means that the resulting procedure will mimic the 
oracle structure over the union of the two families, i.e., the resulting oracle rate will cover both scales 
$\{\Theta^{\uparrow}_k, k\in[n]\}$ and $\{\Theta_k, k\in[n]\}$ simultaneously.
This in turn implies the minimax results for 
the scales $\{\mathcal{S}^{\uparrow}_k, k\in[n]\}$, $\{\mathcal{U}_k, k\in[n]\}$ and 
$\{\mathcal{C}_q, q\in[n]\}$ (adaptively with respect to $k,q\in[n]$) and over the global 
shape-restricted classes $\mathcal{S}^{\uparrow}(V)$, $\mathcal{U}(V)$ and $\mathcal{C}(V)$ of 
monotone, unimodal  and convex parameters, \emph{simultaneously} for all the mentioned scales. 

\subsection{No EBR-like condition for shape-restricted structures}
The last important aspect to discuss for this case of model/structure  
%and the consideredshape-restricted structures 
is one peculiar phenomenon recently discovered by some researchers 
in related settings: for certain shape-restricted classes, the uniform coverage and 
optimal size properties in the uncertainty quantification problem  can be derived without 
imposing any EBR-like condition. It turns out to be possible to construct a confidence ball for monotone 
$\theta$'s with a high coverage and a radius of the optimal order 
$n^{1/3} (\sigma^2 V)^{2/3}$, uniformly over monotone $\theta\in\mathcal{S}^{\uparrow}(V)$ 
and without any EBR-like condition.

Let us show that we can also achieve this (up to a logarithmic factor) by using our approach. 
We consider only the family of structures $\mathcal{I}$ (and the corresponding family of linear spaces) 
for modeling monotone and unimodal $\theta$'s, similar argument can be given 
for the family of structures $\mathcal{I}'$. To ensure the EBR-condition, we 
simply restrict the family of structures $\mathcal{I}$ to the subfamily 
$\mathcal{I}_1=\{I \in \mathcal{I}: |I|\ge C(n/\sigma^2)^{1/3}\}$ for some 
sufficiently large $C>0$.
Then the results go through in the same way as before with the difference that 
the oracle rate is now $r^2(\theta)=\min_{I\in\mathcal{I}_1}
\big\{\|\theta-\mathrm{P}_I \theta\|^2+\sigma^2\rho(I)\big\}$, 
with respect to the family $\mathcal{I}_1$, rather than $\mathcal{I}$.
Since for any $I\in\mathcal{I}_1$, $|I|\ge C(n/\sigma^2)^{1/3}$,
this and Proposition \ref{proposition_isotonic} imply that for any $I\in\mathcal{I}_1$ and any 
$\theta \in \mathcal{S}^\uparrow(V) $ (or $\theta\in\mathcal{U}(V)$) there exists a 
$\theta_*\in\Theta^\uparrow_I$ such that 
\[
\|\theta - \mathrm{P}_I\theta\|^2 \le \| \theta - \theta_*\|^2 \le 
C_1\tfrac{nV^2}{|I|^2} \lesssim n^{1/3}\sigma^{4/3} \lesssim
\sigma^2 |I|\lesssim\sigma^2\rho(I),
\]
which ensures the EBR condition \eqref{cond_ebr}. Thus, the EBR condition is fulfilled 
automatically for the family of structures $\mathcal{I}_1$. At the same time, 
the oracle rate $r^2(\theta)$ covers the both scales $\mathcal{S}^{\uparrow}(V)$ and 
$\mathcal{U}(V)$. Indeed, by taking 
$I_*\in \mathcal{I}_1$ such that $|I_*|= \lfloor C(n/\sigma^2)^{1/3} \rfloor+1$, we obtain 
that uniformly over $\theta\in\mathcal{S}^{\uparrow}(V)\cup\mathcal{U}(V)$  
\[
r^2(\theta) \le r^2(I_*,\theta) \le C_1\tfrac{nV^2}{|I_*|^2}+\sigma^2\rho(I_*)
\lesssim n^{1/3}\sigma^{4/3} \log(en),
\]
which is the minimax rate (up to a logarithmic factor) over the both classes 
$\mathcal{S}^{\uparrow}(V)$ and $\mathcal{U}(V)$ simultaneously.

\section{Matrix linear regression}
\label{sec_multiple_linear_regression}
First we introduce the the \emph{matrix linear regression}:
\begin{align}
\label{matrix_regression}
Y=\mathrm{X}\beta+\sigma\xi, 
\end{align}
where $\mathrm{X}=\text{diag}(\mathrm{X}^1,\ldots , \mathrm{X}^m)\in\mathbb{R}^{mn\times mp}$ 
is a block diagonal matrix, whose blocks  $\mathrm{X}^1,\ldots , \mathrm{X}^m\in\mathbb{R}^{n\times p}$ 
are design matrices, $\sigma>0$ is the known noise intensity, $\beta=(\beta^1,\ldots ,\beta^m)\in\mathbb{R}^{mp}$ is a concatenation of $m$ unknown $p$-dimensional vectors 
$\beta^1, \ldots , \beta^m\in\mathbb{R}^p$, $Y=(Y^1, \ldots, Y^m)\in\mathbb{R}^{mn}$ 
is a concatenation of  observed vectors $Y^1, \ldots, Y^m\in\mathbb{R}^{n}$,
$\xi=(\xi_i, i \in[mn])$, $\xi_i \overset{\rm ind}{\sim} \mathrm{N}(0,1)$. 
The name \emph{matrix regression} comes from the fact that 
\eqref{matrix_regression} can be represented in the matrix form 
$\mathrm{Y}=\bar{\mathrm{X}}\mathrm{B}+\sigma\mathrm{Z}$ with appropriate matrices 
$\mathrm{Y},\bar{\mathrm{X}},\mathrm{B}, \mathrm{Z}$, as is usually done in the literature, 
but we will use the vectorized version \eqref{matrix_regression}.

Introduce some notation. In the sequel, by $\mathrm{M}_I$ we denote the submatrix of $\mathrm{M}$ 
with columns $(M_i, \, i \in I)$, %$x^I=(x\in\mathbb{R}^p: \, x_i =0 \text{ for }  i\in I)$, 
$x_I$ is the $|I|$-dimensional subvector of $x\in\mathbb{R}^p$ with
coordinates $i\in I$, % so $\mathrm{M}x^I=\mathrm{M}_Ix_I$. 
$\|\beta\|_0$ denotes the number of non-zero elements of $\beta$, i.e., 
the cardinality of the support $I^*(\beta)=\text{supp}(\beta) = \{i: \beta_i\neq 0\}$ of $\beta$.
Under $\xi_i \overset{\rm ind}{\sim} \mathrm{N}(0,1)$, 
Conditions \eqref{cond_nonnormal} and \eqref{cond_A4} hold with $d_I=\dim(\mathbb{L}_I)$,
%for all cases of this section %(irrespective of the underlying structure)
in view of Remarks \ref{rem_cond_A1} and \ref{rem15}. 

Many particular linear models 
%(the ones considered and studied in the literature and the ones not considered) 
can be put in \eqref{matrix_regression} by choosing 
appropriately $m$, $p$ and $\mathrm{X}$. 
The case $m=1, p=n, \mathrm{X}=\mathrm{I}$ is already considered in Section 
\ref{sec_smoothness} for the smoothness structure and in Section \ref{sec_sparsity} for the sparsity structure.
In the following sections we consider several other specific models and structures in detail.

We should emphasize that in linear regression models of type \eqref{matrix_regression} 
by $\beta$ we denote the vector of unknown parameters, 
notation commonly used in the literature.
This is not to be confused with the structural parameter $\beta$ for indexing 
the scales of classes $\{\Theta_\beta, \beta \in \mathcal{B}\}$ 
which we use for other models.
%for linear models of type \eqref{matrix_regression} we will use other notation for the scales.

\begin{remark}
\label{lin_model2}
As explained in the introduction %of \cite{Belitser&Nurushev:2019a}, 
we can always work with the linear model, even when the true distribution $\mathrm{P}_\theta$
of the observed data does not follow the linear model \eqref{matrix_regression}, 
e.g., $Y =\theta +\sigma\xi$, where $\theta\not= \mathrm{X}\beta$. 
In that case, \eqref{matrix_regression} is an approximating model of 
the true model and all the local results hold  
with $\theta$ substituted everywhere instead of $\mathrm{X}\beta$.
The global minimax results over, say, a class $\Theta_\gamma$ 
will have to be modified by including the approximation 
term $\sup_{\theta\in\Theta_\gamma} \| \theta- \mathrm{P}_{I_o} \theta\|^2$ 
in the minimax rate $r^2(\Theta_\gamma)$.
\end{remark}

\section{Linear regression with sparsity structure}
\label{sec_lin_regr_sparsity}
Consider the classical linear regression model, with $m=1$ in \eqref{matrix_regression}, that is, 
\begin{align}
\label{lin_regr_suppl}
Y=\mathrm{X}\beta+\sigma\xi, \quad \mathrm{X}=(X_1,\ldots , X_p)\in\mathbb{R}^{n\times p}, \;\;
\xi=(\xi_i, i \in[n]),
\end{align}
where $\mathrm{X}$ 
is the design matrix, whose columns $X_1,\ldots , X_p$  are the (observed) predictors, 
$\sigma>0$ is the known noise intensity. We assume 
$\xi_i \overset{\rm ind}{\sim} \mathrm{N}(0,1)$, but this can be relaxed by assuming 
Condition \eqref{cond_nonnormal} (and \eqref{cond_A4}). 
In high-dimensional settings, typically $p\gg n$,  
and to be able to make sensible inference, one needs to exploit a structure on $\beta$. 
For a \emph{sparsity structure} $I\subseteq [p]$, the vector $\beta=(\beta_1,\ldots , \beta_p)^T\in\mathbb{R}^p$ is called \emph{sparse} in the sense that $\beta_i=0$  (or close to zero) 
for $i\in I^c$, in other words, the predictors $(X_i: i\in I^c)$ 
are irrelevant. %from the perspective of structure $I$. 
Recall that we pursue a local approach, in this case this means that we do not impose any specific sparsity 
constraint, but rather exploit as much sparsity as there is in an arbitrary $\beta$.

A local approach to this case of model/structure,
delivering also the adaptive minimax results, %for many sparsity structures simultaneously, 
is considered in \cite{Gao&vanderVaart&Zhou:2015, Belitser&Ghosal:2018} 
for the estimation and posterior contraction problems. In \cite{Belitser&Ghosal:2018} the uncertainty 
quantification problem is extensively treated as well as some other related interesting aspects, 
such as inference  on $\beta$ %(under some compatibility conditions) 
and sparsity recovery. 
%Note that ``signal+noise" model  is a special case of linear regression with 
%$\mathrm{X}=\mathrm{I}$ and $p=n$, and it has been considered  in several previous sections. 
Here we demonstrate that the results obtained 
in \cite{Belitser&Ghosal:2018} follow from our general framework results.  Actually, we obtain stronger 
versions of the results as they hold in the refined formulation (non-asymptotic exponential probability 
bounds) and distribution-free setting (the observations are not necessarily normal and/or independent).

In this model, the sparsity structure is expressed by the linear spaces 
\begin{align}
\label{lin_regr_L_I}
\mathbb{L}_I=\big\{\mathrm{X}_Ix_I\in\mathbb{R}^n: x_I\in\mathbb{R}^{|I|}\big\}
=\big\{\mathrm{X}x\in \mathbb{R}^n: 
x\in\mathbb{R}^{p},\; x_i=0 \text{ for }i \not\in I\big\},
%=\Big\{\sum\nolimits_{i\in I} X_i\beta_i: \beta_i\in\mathbb{R}, \; i\in I\Big\},
\end{align}
$I \in\mathcal{I}$, the family of structures is 
$\mathcal{I}=\mathcal{I}_1\cup \{I_r\}$ with 
$\mathcal{I}_1 =\{I\subseteq [p]: 2|I|\log(ep/|I|) \le r\}$ (where we denote 
$r=r(\mathrm{X})=\rank(\mathrm{X})$)
and $I_r=(i_1,\ldots, i_r) \subseteq [p]$ such that  $(X_{i_1},\ldots, X_{i_r})$ are $r$ 
linearly independent columns of $\mathrm{X}$.
Then $|\mathcal{I}|\le 2^p$, the structural slicing mapping is taken to be
$s(I)=|I| \in\mathcal{S}\triangleq[r]_0$.
Further, we have $\theta=\mathrm{X}\beta$, 
$d_I=\dim(\mathbb{L}_I)\le \min\{|I|,r\}\le \min\{|I|,n,p\}$ for $I\in\mathcal{I}_1$ 
and $d_{I_r}=\dim(\mathbb{L}_{I_r})=r$.
%We have $|\mathcal{I}|=2^p$ and $|I|\le r\le \min\{n,p\}$ for each $I\in\mathcal{I}$. 
Clearly, %$|\mathcal{I}_s|=\binom{p}{s}$, hence 
$\log |\mathcal{I}_{s(I)}|\le\log \binom {p}{|I|}\le |I|\log(\frac{ep}{|I|})$ for $I\in\mathcal{I}_1$ and
$\log |\mathcal{I}_{s(I_r)}|=0$. Since $d_I +\log |\mathcal{I}_{s(I)}|\le |I| + |I| \log(\frac{ep}{|I|})
\le 2|I| \log(\frac{ep}{|I|})$ for $I\in\mathcal{I}_1$ and $d_{s(I_r)}+\log |\mathcal{I}_{s(I_r)}|=d_{I_r}=r$, 
we take the majorant 
\begin{align}
\label{formula_rho}
\rho(I)=2|I|\log(\tfrac{ep}{|I|})1\{I\in\mathcal{I}_1\}+r 1\{I=I_r\}, \quad I\in \mathcal{I}.
\end{align}
Notice that we could use a smaller majorant $\rho'(I)=d_I+\log\binom {p}{|I|}$ for $I\in\mathcal{I}_1$
(the best choice), but this majorant is not practical to use.

\begin{remark} 
In the majorant $\rho(I)$ defined above, we see the \emph{elbow effect} mentioned in Remark \ref{family_cover}, this elbow effect will enter the rate as well. 
Let us explain how this elbow effect has emerged in this model.

Notice that we could consider the more natural full family of structures $\bar{\mathcal{I}}=
\{J:J\subseteq [p]\}$, so that $|\bar{\mathcal{I}}|=2^p$, with the same structural slicing mapping 
$s(I)=|I|\in\mathcal{S}\triangleq[p]_0$, but defined on the family $\bar{\mathcal{I}}$. 
As before, $d_I=\dim(\mathbb{L}_I)\le \min\{|I|,r\}$
and $ |\mathcal{I}_{s(I)}|=\binom {p}{|I|}$.
%We have $|\mathcal{I}|=2^p$ and $|I|\le r\le \min\{n,p\}$ for each $I\in\mathcal{I}$. 
Since $d_I +\log |\mathcal{I}_{s(I)}|\le |I|+|I| \log(\frac{ep}{|I|})\le 2|I| \log(\frac{ep}{|I|})$, 
the majorant would be $\bar{\rho}(I)=2|I|\log(\tfrac{ep}{|I|})$, $I\in\bar{\mathcal{I}}$.
The idea of the family $\mathcal{I}$ is that, even though 
$\mathcal{I}\subseteq \bar{\mathcal{I}}$, the family $\mathcal{I}$ still covers $\bar{\mathcal{I}}$ 
in the sense of Remark \ref{family_cover}. 
Indeed, 
$r^2(I,\beta)=\|\mathrm{X} \beta-\mathrm{P}_I\mathrm{X}\beta\|^2 +\sigma^2\rho(I) 
=\|\mathrm{X} \beta-\mathrm{P}_I\mathrm{X}\beta\|^2 +\sigma^2\bar{\rho}(I)=\bar{r}^2(I,\beta)$ 
for $I\in\mathcal{I}_1$, and $r^2(I_r,\beta)=\sigma^2 r \le \sigma^2 2|I| \log(\frac{ep}{|I|}) \le
\bar{r}^2(I,\beta)$ for all $I\in\bar{\mathcal{I}}\backslash\mathcal{I}_1$, as
$\mathrm{P}_{I_r}\mathrm{X}\beta=\mathrm{X}\beta$ for any $\beta\in\mathbb{R}^p$.

Here we considered an important case when a seemingly right (full) family $\bar{\mathcal{I}}$ 
of structures can be reduced to a subfamily $\mathcal{I}\subset\bar{\mathcal{I}}$ 
that has a reduced complexity but still covers the original family $\bar{\mathcal{I}}$ 
in the sense of Remark \ref{family_cover}, thus improving the resulting oracle rate. 
This is a typical situation exhibiting the  ``elbow effect'' in the complexity term of the rate; 
below there are a couple of more such 
example (Sections \ref{sec_group_clustering}, \ref{sec_mixture_model} and 
\ref{sec_dict_learning}).
\end{remark}

\begin{remark}
\label{rem_32}
We could further reduce the family of structures to
${\mathcal{I}}'=\{I_r\}\cup\mathcal{I}'_1$, with $\mathcal{I}'_1=
\{I\subseteq\mathcal{I}_1: \text{the columns } (X_i, i\in I) 
\text{ are linearly independent}\}$ (with the same 
structural slicing mapping $s(I)=|I|$), so that $\mathcal{I}'\subseteq\mathcal{I}$. 
In this case, we have $d_I=\dim(\mathbb{L}_I)=|I|$, 
$|I|\le r\le \min\{n,p\}$ for each $I\in\mathcal{I}'$,
the layer is $\mathcal{I}_{s(I)}=\{J \subseteq \mathcal{I}_1:  \dim(\mathbb{L}_J) =\dim(\mathbb{L}_I)\}$
for $I\in\mathcal{I}'_1$ and  $\mathcal{I}_{s(I_r)}=\{I_r\}$.
Then we can take the majorant $\bar{\rho}(I)=(|I|+\log |\mathcal{I}_{s(I)}|)1\{I\in\mathcal{I}'_1\} +
r 1\{I=I_r\}$. When implementing the Bayesian or penalization procedure, the majorant 
$\rho(I)=2|I| \log(ep/|I|)1\{I\in\mathcal{I}'_1\}+r1\{I=I_r\}$ is more practical to use also 
for the family $\mathcal{I}'$. But then the families $\mathcal{I}$ and $\mathcal{I}'$ cover each other 
in the sense of Remark \ref{family_cover}, thus yielding  the same resulting oracle rate over the both 
families. Therefore, as soon as we use the same majorant $\rho(I)$, it does not matter which family of 
structures, $\mathcal{I}$ or $\mathcal{I}'$, we take. We will have a slightly bigger constant in 
Condition \eqref{(A2)} for the family $\mathcal{I}$ as there are more terms in the sum. 
We will use the family $\mathcal{I}$. 
\end{remark}

Condition \eqref{(A2)} is fulfilled, since, according to Remark \ref{rem5}, for any $\nu>1$ 
\begin{align*}
\sum_{I\in\mathcal {I}} e^{-\nu \rho(I)}
&\le \sum_{s\in\mathcal{S}}  e^{-(\nu-1) \rho(I)}
\le \sum_{s=0}^p e^{-(\nu-1)s}\le 
\frac{1}{1-e^{1-\nu}}= C_\nu. %&\le\sum_{j_0=0}^\infty 2^{j_0+1} e^{-j_0}\le C_\nu.
\end{align*}
As to Condition \eqref{(A3)}, for any $I_0, I_1\in\mathcal{I}$, take $I'=I_r$ 
if either $I_0=I_r$ or $I_1=I_r$ or $2|I_0\cup I_1| \log (ep/|I_0\cup I_1|)>r$; otherwise 
take $I' =I_0\cup I_1$.
%to be such that the columns $(X_i,\, i \in I_0)\subseteq I_1 \cup I_2$ form the basis  
%of the space $\mathbb{L}_{I_1}+\mathbb{L}_{I_2}$. 
Since $(\mathbb{L}_{I_0}\cup\mathbb{L}_{I_1})\subseteq \mathbb{L}_{I'}=\mathbb{L}_{I_0}
+\mathbb{L}_{I_1}$ and $\rho(I')\le \rho(I_0)+\rho(I_1)$, Condition \eqref{(A3)} is also fulfilled.

%To establish Condition \eqref{(A3)}, we use the following proposition which easily follows from basics 
%of linear algebra.
%\begin{proposition}
%\label{proposition}
%For  any $I_1, I_2\in\mathcal{I}$, there exists $I_0=I_0(I_1, I_2)\in\mathcal{I}$ such that 
%$\mathbb{L}_{I_1}, \mathbb{L}_{I_2}\subseteq \mathbb{L}_{I_0}=\mathbb{L}_{I_1}
%+\mathbb{L}_{I_2}$.
%\end{proposition}
%Indeed, let $(X_{i_1},\ldots, X_{i_l})$ be the basis of $\mathbb{L}_{I_1}$ for some $i_1,\ldots,i_l \in I_1$, 
%and let $(X_{j_1},\ldots, X_{j_m})$ be the columns from $\mathrm{X}_{I_i2}$ which
%form a basis of the space $\mathbb{L}_{I_2} \cap \mathbb{L}_{I_1}^\perp$. 
%the proof of the proposition proceeds via extending the basis $(X_{i_1},\ldots, X_{i_l})$ 
%(with $i_1,\ldots,i_l \in I_1$) of  
%$\mathbb{L}_{I_1}$  to the basis of 
%$\mathbb{L}_{I_1}+\mathbb{L}_{I_2}$. Clearly,  this proposition entails Condition \eqref{(A3)}, 
%since $ \rho(I_0)\le \rho(I_1)+\rho(s(I_2))$.

As consequence of our general results, we obtain Corollary \ref{meta_theorem}
with the local (prediction) rate $r^2(\beta)=\min_{ I \in \mathcal{I}} r^2(I,\beta)=
 \min_{ I \in \mathcal{I}} \big\{\|\mathrm{X}\beta
-\mathrm{P}_I\mathrm{X} \beta\|^2 +\sigma^2\rho(I)\big\}$, where the majorant 
$\rho(I)$ is defined by \eqref{formula_rho}.
In particular,
\begin{align}
r^2(\beta) & \le r^2(I^*(\beta),\beta) \wedge r^2 (I_r,\beta) =
 \sigma^2 \big[\rho(s(I^*(\beta)))\wedge\rho(s(I_r))\big] \notag\\
&\lesssim \sigma^2\big[\big(|I^*(\beta)|\log(\tfrac{ep}{|I^*(\beta)|})\big) \wedge r\big].
\label{local_rate2}
\end{align}

\begin{remark}
Conditions \eqref{cond_nonnormal} and \eqref{cond_A4}  hold in view of 
Remarks \ref{rem_cond_A1} and \ref{rem15}.
Notice that the claims (i)--(vii) of Corollary \ref{meta_theorem} deliver finer and stronger versions 
of the corresponding results from \cite{Belitser&Ghosal:2018}. Besides, we can drop the normality 
and independence assumptions and impose only Conditions 
\eqref{cond_nonnormal} and \eqref{cond_A4} instead. 
\end{remark}

Next, by virtue of Corollary \ref{minimax_results} the local results imply global minimax 
adaptive results at once over all scales $\{\Theta_\gamma,\, \gamma\in\Gamma\}$ covered 
by the oracle rate $r^2(\beta)$ (i.e., for which \eqref{oracle_minimax} holds).
Below we present a couple of scales $\{\Theta_\gamma,\, \gamma\in\Gamma\}$ 
covered by the oracle rate $r^2(\beta)$.

\subsection{Minimax results for the nearly black vectors $\ell_0$} 
For $s \in[p]$, %such that $s=s_n=o(n)$ as $n\to \infty$ 
introduce 
\begin{align*}
\ell_0[s]=\{\beta\in\mathbb{R}^p: \|\beta\|_0=|I^*(\beta)|\le s\},\;\; \text{where} \;\;
I^*(\beta)=\{i\in[p]: \beta_i \not =0\},
\end{align*}
the set of vectors with at most $s$ nonzero elements. Under certain conditions on 
the parameters $s,p,n$ and the design matrix $\mathrm{X}$ (at least, $s\log(ep/s) \lesssim r=
\rank(\mathrm{X})$ has to hold), 
the  minimax prediction estimation rate over $\ell_0[s]$ is known to be $r^2(\ell_0[s])
=\inf_{\hat{\beta}}\sup_{\beta\in \ell_0[s]} \mathbb{E}_\beta\|\mathrm{X} \hat{\beta}-\mathrm{X}\beta\|^2
\asymp\sigma^2s\log(ep/s)$; see \cite{Bunea&Tsybakov&Wegkamp:2007, 
Raskutti&Wainwright&Yu:2011}. The adaptive minimax results for $\ell_0$-balls were considered 
by \cite{Castilloetal:2015, Gao&vanderVaart&Zhou:2015, Martin&etal:2017} for posterior 
contraction rates and by \cite{Nickl&vandeGeer:2013} for uncertainty quantification problem.

If  $\beta\in\ell_0[s]$, then $\mathrm{X}\beta \in \mathbb{L}_{I^*(\beta)}$ for  
$I^*(\beta)\in \mathcal{I}_s$ such that $\mathrm{P}_{I^*(\beta)}\mathrm{X}\beta
=\mathrm{X}\beta$ and $|I^*(\beta)|\le s$,
and hence  $r^2(I^*(\beta),\beta)\le \sigma^2|I^*(\beta)|
\log\big(ep/|I^*(\beta)|\big)\le\sigma^2s\log\big(ep/s\big)$.
By the definition  \eqref{oracle} of the oracle rate $r^2(\beta)$,
we have that $r^2(\beta) \le r^2(I^*(\beta),\beta) \wedge r^2(I_r,\beta)$. 
Then we obtain trivially that  
\[
\sup_{\beta\in\ell_0[s]} r^2(\beta) 
\le \sup_{\beta\in \ell_0[s]}\big[r^2(I^*(\beta),\beta) \wedge r^2(I_r,\beta)\big]
\lesssim \sigma^2\big[r\wedge\big(s\log(\tfrac{ep}{s})\big)\big]\lesssim r^2(\ell_0[s]).
\]
We thus established the relation \eqref{oracle_minimax} for the scale $\ell_0[s]$,
and Corollary \ref{minimax_results} follows with the minimax 
rate $r^2(\ell_0[s])$ defined above.

\subsection{Minimax results for the weak $\ell_q$-balls} 
For $q\in(0,1]$, the \emph{weak $\ell_q$-ball} is defined by
\begin{align*}
\ell_q[R]=\big\{\beta\in\mathbb{R}^p: 
\beta^2_{[i]}\le R^2 i^{-2/q}, \, i\in[p]\big\},\;\; R^2 \ge \sigma^2  \log p,
\end{align*}
where $\beta_{[1]}^2\ge\ldots\ge \beta_{[p]}^2$ are the ordered  
$\beta_1^2,\ldots, \beta_p^2$.  We assume that there exists a constant $L>0$ such that $\max_{i\in[p]}\|X_i\|^2\le n L^2$.   The  minimax prediction estimation rate over $\ell_q[R]$ in $\ell_2$-prediction norm is known to be  $r^2(\ell_q[R])=
\inf_{\hat{\beta}}\sup_{\beta\in\ell_q[R]} \mathbb{E}_\beta
\|\mathrm{X} \hat{\beta}-\mathrm{X}\beta\|^2
=R^qn^{q/2}\sigma^{2-q}[\log(1+\tfrac{p\sigma^q}{n^{q/2}R^q})]^{1-q/2}$ when $R^2 \ge \sigma^2  \log p$; see  \cite{Tsybakov:2014} (cf.\ \cite{Donoho&Johnstone:1994b, Birge&Massart:2001}). 
%In \cite{Tsybakov:2014} it was shown that the minimax rate in the aggregation problem for 
%\emph{weak} $\ell_q$-balls  is $\min(R^qn^{q/2}[\sigma^2\log(1+\tfrac{p\sigma^q}{n^{q/2}R^q})]^{1-q/2},
%\sigma^2 R_{\mathrm{X}})$, where $R_{\mathrm{X}}=\rank(\mathrm{X})$. Since $\rank(\mathrm{X})$ is %unknown in our model, then the aggregation problem becomes the linear regression problem. 
The adaptive minimax results for \emph{weak} $\ell_q$-balls in $\ell_2$-prediction norm were
considered by \cite{Gao&vanderVaart&Zhou:2015} for posterior contraction rates.

Define $j=O_\beta(i)$ if $\beta_i^2=\beta_{[j]}^2$, with 
the convention that in the case $\beta_{i_1}^2=\ldots =\beta_{i_k}^2$ for $i_1 < \ldots < i_k$ 
we set $O_\beta(i_{l+1})=O_\beta(i_l)+1$, $l=1,\ldots, k-1$.
Let $I^*=I^*(\beta)=\{i \in [p]: O_\beta(i) \le R^*\}$ with $R^*=e(\frac{R}{\sigma})^q 
n^{q/2}[\log(\frac{p\sigma^q}{n^{q/2}R^q})]^{-q/2}$, and $\beta^*=\beta^*(\beta)=((\beta^*_i)_{i\in[p]}: \beta^*_i=\beta_i \text{ for } i \in I^*, \beta^*_j=0 \text{ for } j\not\in I^*)$.

There exists ${I^*}\in\mathcal{I}$  such that $\mathrm{X}\beta^*\in\mathbb{L}_{I^*}$ and 
$\|\mathrm{X}\beta-\mathrm{P}_{I^*}\mathrm{X}\beta\|^2=\|\mathrm{X}\beta-\mathrm{X}\beta^*\|^2$. 
By using this and  the fact that 
$\max_{i\in[p]}\|X_i\|^2\le L^2n$, we derive \eqref{oracle_minimax}:
\begin{align*}
\sup_{\beta\in\ell_q[R]}r^2(\beta) &\le\sup_{\beta\in\ell_q[R]} r^2(I^*,\beta)\le
\sigma^2 R^*\log(\tfrac{ep}{R^*})+\sup_{\beta\in\ell_q[R]}\|\mathrm{X}\beta-\mathrm{X}\beta^*\|^2 
\notag\\
&\lesssim \sigma^2R^*\log(\tfrac{p\sigma^q}{n^{q/2}R^q})+L^2n R^2 (R^*)^{1-2/q}\\
&\lesssim R^{q}n^{q/2}\sigma^{2-q}\big[\log (1+\tfrac{p\sigma^q}{n^{q/2}R^q})\big]^{1-q/2}
\asymp r^2(\ell_q[R]).  
\end{align*}
Corollary \ref{minimax_results} follows for this case with the minimax rate 
$r^2(\ell_q[R])$  defined above.

\subsection{Model selection} 
Besides inference on $\theta=\mathrm{X}\beta$, several interesting 
corollaries were established in \cite{Belitser&Ghosal:2018} 
and they follow from our results exactly in the same way, we  
provide them here for completeness. 

The first corollary concerns a bound on the size of the selected model. 
Similar to \cite{Castilloetal:2015} and \cite{Martin&etal:2017}, the following assertion 
shows that  the models with substantially higher size than the true one are unlikely 
according to the posterior $\hat{\pi}(I|Y)$ (which is in essence the penalization method in case 
$\hat{\pi}(I|Y)=\check{\pi}(I|Y)$).
\begin{proposition}
\label{cor1_dimension}
Under the conditions of Corollary \ref{meta_theorem}, for sufficiently large $M'_0$   
\[
\sup_{\beta\in\mathbb{R}^p}\mathbb{E}_{\beta}\hat{\pi}(I: |I|>C_0 \|\beta\|_0|Y)
\le C_\nu\exp\big\{-c_2\big(\tfrac{M'_0}{2}-c_3\big)\|\beta\|_0 \log(\tfrac{ep}{\|\beta\|_0})\big\},
\]
where $C_0=\max\{M'_0,{M'_0}^2/e\}$.
\end{proposition}
\begin{proof}
Note that for any $M'_0>2c_3$,  $|I|\ge M'_0\|\beta\|_0$ implies that 
\begin{align*}
r^2(I,\beta)&\ge \sigma^2|I| \log (\tfrac{en}{|I|})\ge 
M'_0\sigma^2\|\beta\|_0 \log\big(\tfrac{en}{M'_0\|\beta\|_0}\big)
\ge 
\tfrac{M'_0}{2} \sigma^2\|\beta\|_0\log(\tfrac{en}{\|\beta\|_0}),
\end{align*} 
provided $\|\beta\|_0<en/{M'_0}^2$. Since $r^2(\beta) \le r^2(I^*(\beta), \beta) \le 
\sigma^2\|\beta\|_0\log(en/\|\beta\|_0)$, the above display implies that $r^2(I,\beta)
\ge c_3 r^2(\beta)+M''_0\sigma^2$, where $M''_0=(M'_0/2-c_3)\|\beta\|_0\log(en/\|\beta\|_0)$. 
By (i) of Theorem \ref{th2} ($M''_0$ corresponds to $M$ from Theorem \ref{th2}), 
the assertion holds for any $|I|\ge M'_0 \|\beta\|_0$  whenever $\|\beta\|_0<en/{M'_0}^2$. 
If $\|\beta\|_0\ge en/{M'_0}^2$, the result trivially holds for any $|I|\ge {M'_0}^2\|\beta\|_0/e$.
Hence, the choice $C_0=\max\{M'_0,{M'_0}^2/e\}$ ensures the result for any $\beta\in\mathbb{R}^p$.  
\end{proof}
The above claim, being non-asymptotic and uniform in $\beta\in\mathbb{R}^p$, can be specialized 
to certain situations. In particular, it leads to 
an interesting  conclusion under the asymptotic setting $p=p_n\to \infty$ and $\|\beta\|_0
\le s_n=o(p_n)$ as $n\to \infty$. Then the probability bound goes to $0$ as $n\to \infty$, uniformly 
in  $\beta\in\ell_0[s_n]\triangleq \{\beta: \|\beta\|_0\le s_n\}$. 
Further, when $s_n=o(p_n)$, the constant $C_0$ can be chosen smaller, which makes 
the conclusion of the claim stronger. 

\subsection{Inference on $\beta$ under the compatibility condition} 
The next several corollaries concern inference on $\beta$ rather than on $\theta$.
Besides optimal prediction, it is of interest to infer on the parameter $\beta$ itself. 
%with quality measured by the usual Euclidean or the$\ell_1$-norm.  
Because the dimension $p$ may be (and generally is) larger than $n$, the correspondence between $\mathrm{X}\beta$ and $\beta$ is not unique, and hence additional conditions are necessary even in the noiseless situation. As is commonly adopted in the literature (see, e.g., 
\cite{Gao&vanderVaart&Zhou:2015}), we will need to assume a condition lower bounding the norm 
of $\mathrm{X}\beta$ by a positive multiple of a norm on $\beta$ for sparse vectors which 
is in turn a condition on the design matrix $\mathrm{X}$.

There is yet another issue: recall that inference in the general framework is on $\theta$ and is 
based on the posterior $\hat{\pi}(\vartheta|Y)$ 
for $\theta=\mathrm{X}\beta$, not for $\beta$. In order to infer on $\beta$, 
we need to construct a prior $\Pi$ on $\beta$  
that leads to an (empirical Bayes) posterior $\hat{\Pi}(b|Y)$ such that $\vartheta=\mathrm{X}b 
\sim \hat{\pi}(\vartheta|Y)$. This is not difficult: indeed, we use the construction 
of the conditional prior $\pi_I(\vartheta)$ on $\theta$ from \cite{Belitser&Nurushev:2015},
for all $I\in \mathcal{I}'$, where the family $\mathcal{I}'$ is from Remark \ref{rem_32}.
Since in this case the conditional prior 
$\pi_I(\vartheta|Y)$ was formally constructed as prior on ``structured''
$\theta^I=\mathrm{P}_I\theta= \mathrm{P}_I\mathrm{X} \beta = \mathrm{X}_I\beta$,
we can derive the corresponding conditional prior $\Pi_I(b|Y)$ for $\beta=
(\mathrm{X}_I^T \mathrm{X}_I)^{-1}\theta^I$ because 
$\theta^I=\mathrm{X}_I\beta$ is invertible with respect to $\beta$ for any $I\in\mathcal{I}'$, 
by the definition of $\mathcal{I}'$. Thus, the corresponding conditional prior on $\beta$ becomes
\[
\beta|I \sim \Pi_I(b|Y) = \mathrm{N}\big((\mathrm{X}_I^T \mathrm{X}_I)^{-1} \mathrm{X}_I \mu, 
\kappa \sigma^2(\mathrm{X}_I^T \mathrm{X}_I)^{-1}\big) 
\otimes \delta_{0_{|I^c|}},
\]
which means that subvector $\beta_I$ with coordinates in $I$ is normally distributed with the 
above parameters, and the remaining coordinates $I^c$ of $\beta$ are set to zero. 
%Recall that we have actually started with the prior on $\theta$: $\theta\sim \Pi$. 
From this point on, we can apply the empirical Bayesian approach exactly
in the same way as in \cite{Belitser&Ghosal:2018}, yielding
the corresponding empirical Bayes posteriors on $\beta$:
$\tilde{\Pi}_I(b|Y)$, $\Pi(b|Y)$, $\tilde{\Pi}(b|Y)$, $\check{\Pi}(b|Y)$; 
and the estimators $\tilde{\beta}=\sum_{I\in\mathcal{I}}
\hat{\beta}_I \tilde{\pi}(I|Y)$ and $\check{\beta}=\hat{\beta}_{\hat{I}}$, 
where $\hat{I}$ is defined by \eqref{I_MAP} and 
$\hat{\beta}_I=(\mathrm{X}_I^T \mathrm{X}_I)^{-1} \mathrm{X}_IY$ is just the ordinary 
least squares estimator of $\beta$ based on the design matrix $\mathrm{X}_I$ 
of full column rank as $I\in\mathcal{I}$. Similarly, we can define 
$\hat{\Pi}(\beta|Y)$ as being either $\tilde{\Pi}(\beta|Y)$ or $\check{\Pi}(\beta|Y)$, and 
$\hat{\beta}$ as being either $\tilde{\beta}$ or $\check{\beta}$. 
The details of Bayesian construction for $\beta$ can be found in \cite{Belitser&Ghosal:2018}. 
For us what only matters is the fact that if $\beta\sim\hat{\Pi}(b|Y)$ then 
$\theta=\mathrm{X}\beta \sim \hat{\pi}(\vartheta|Y)$.

%The variable $V$ is used in the posteriors on $\upsilon$ to distinguish from the ``true'' $\upsilon$. 
%%For example, $\hat{\pi}(\vartheta\in G|Y)=\sum_{I\in\mathcal{I}} \tilde{\Pi}_I(\mathrm{P}_IV\in G|Y)
%%\hat{\pi}(I|Y)$ and $\hat{\theta}=\sum_{I\in\mathcal{I}} \mathrm{P}_I \hat{\upsilon}\, \hat{\pi}(I|Y)$. 
%
%Interestingly, this gives a Bayesian solution of the  inverse problem 
%when one wants to recover $\upsilon$ whereas only noisy $\theta=\mathrm{P}_{I^*} \upsilon$ 
%is observed. In general, this solution is not going to be good, unless the parameter $\upsilon$
%itself is well structured: say, $\upsilon\in \mathbb{L}_{I^*}$, in which case we are back to 
%the inference problem on $\theta$ as $\upsilon=\mathrm{P}_{I^*} \upsilon=\theta$; 
%or, $\upsilon\in \mathbb{I}_{I_0}$ for some $I_0\in\mathcal{I}$
%such that inverting $\theta=\mathrm{P}_{I^*} \upsilon$ becomes (at least partially) possible.
%Some additional technical issues may arise here. We will not treat this issue in general framework, 
%but only look into this in the specific case of linear regression model with sparsity structure 
%on $\upsilon$ in Section \ref{sec_lin_regr_sparsity}.
 
 %To state the results, 
Introduce some additional notation.
Recall that $\|\beta\|_0$ denotes the number of non-zero elements of $\beta$. 
Further let $\|\beta\|_1=\sum_{j=1}^p|\beta_j|$ be the $\ell_1$-norm of $\beta$
and $\|\mathrm{X}\|_{\rm max}= \max_{k=1,\ldots, p} \|X_k\|$
(notice that if the design matrix $\mathrm{X}$ is normalized
so that $\|X_k\|^2=n$, $k\in[p]$, then $\|\mathrm{X}\|_{\rm max}=\sqrt{n}$).
%(one can  assume $\|\mathrm{X}\|_{\rm max}=\sqrt{n}$ without too much loss of generality). 
For $l\in\mathbb{N}$, let
\begin{align}
\label{phi_1}
\phi_1(l)=\inf\Big\{\frac{\sqrt{l}\|\mathrm{X}\beta\|}{\|\mathrm{X}\|_{\rm max}\|\beta\|_1}:\|\beta\|_0\le l,\; \text{supp}(\beta)\in\mathcal{I}\Big\},\\
\label{phi_2}
\phi_2(l)=\inf\Big\{\frac{\|\mathrm{X}\beta\|}{\|\mathrm{X}\|_{\rm max}\|\beta\|}:\|\beta\|_0\le l,\; \text{supp}(\beta)\in\mathcal{I}\Big\}.
\end{align}
Because $\|\beta\|_1\le \sqrt{\|\beta\|_0}\|\beta\|$, it follows that $\phi_1(l)\ge\phi_2(l)$. 
Positivity of $\phi_1$ at an argument $l$ is called the compatibility condition, and is stronger 
if $\phi_1(l)$ is larger. If any of $\phi_1$ or $\phi_2$ is zero at its argument, then the corresponding 
result below becomes trivial but remains valid.

The following claims say basically that, under the compatibility condition, the (empirical Bayes) 
posterior $\hat{\Pi}(b|Y)$ on $\beta$  contracts around the truth with the optimal rate.
\begin{proposition} Under the conditions of Corollary \ref{meta_theorem},
for sufficiently large $M'_0$ and any $M\ge 0$
\label{estimation_rate}
\begin{align*}
\mathbb{E}_{\beta}&\hat{\Pi}\Big(\|b-\beta\|_1\ge \tfrac{\sqrt{(C_0+1)\|\beta\|_0(M_0r^2(\beta)+M\sigma^2)}}{\|\mathrm{X}\|_{\rm max}
\phi_1((C_0+1)\|\beta\|_0)} \big| Y\Big) \\
& \le H_0 e^{-m_0 M}+C_\nu\exp\big\{-c_2(M'_0/2-c_3)\|\beta\|_0 \log(\tfrac{ep}{\|\beta\|_0})\big\},\\
\mathbb{E}_{\beta}&\hat{\Pi}\Big(\|b-\beta\|\ge \tfrac{\sqrt{M_0r^2(\beta)+M\sigma^2}}{\|\mathrm{X}\|_{\rm max}
\phi_2((C_0+1)\|\beta\|_0)} \big| Y\Big) \\
& \le H_0 e^{-m_0 M}+C_\nu\exp\big\{-c_2(M'_0/2-c_3)\|\beta\|_0 
\log(\tfrac{ep}{\|\beta\|_0})\big\},
\end{align*}
uniformly in $\beta\in\mathbb{R}^p$, where $C_0=\max\{M'_0,{M'_0}^2/e\}$.
\end{proposition}
\begin{proof} 
By the definition of compatibility coefficient, 
on models $I$ with $|I| \le C_0\|\beta\|_0$, the quantity $\|b-\beta\|_1$ is bounded by 
\[
\sqrt{(C_0+1)\|\beta\|_0}\|\mathrm{X}(b-\beta)\|/[\|\mathrm{X}\|_{\rm max}
\phi_1((C_0+1)\|\beta\|_0)],
\]
since the cardinality of  $\text{supp}(b-\beta)$ is at most  $(C_0+1)\|\beta\|_0$. 
By Theorem \ref{th1}, the $\mathbb{E}_\beta$-expectation of the posterior probability 
of $\|\mathrm{X}(b-\beta)\| = \|\vartheta-\theta\|>\sqrt{M_0r^2(\theta)+M\sigma^2}$ is 
bounded by $H_0 e^{-m_0 M}$, while by Proposition \ref{cor1_dimension}, the event 
$\{I: |I|\ge C_0\|\beta\|_0\}$ has probability bounded by 
\[
C_\nu\exp\big\{-c_2(M'_0/2-c_3)\|\beta\|_0\log(\tfrac{ep}{\|\beta\|_0})\big\}.
\] 
The first assertion follows, the proof of the second claim is similar.
\end{proof}
Notice that the above result implies the Corollary 5.4 in \cite{Gao&vanderVaart&Zhou:2015} and 
obtains optimal estimation rates for both $\ell_1$ and $\ell_2$ loss functions.  
Moreover, the dependence on the quantities  $\phi_1(l)$ and $\phi_2(l)$ are optimal; cf.\ \cite{Raskutti&Wainwright&Yu:2011}.
Next, we also obtain the optimal  estimation result  for both $\ell_1$- and $\ell_2$-norms.

\begin{proposition} 
Under the conditions of Corollary \ref{meta_theorem}, for sufficiently large $M'_0$ and any $M\ge 0$
\label{estimation_result}
\begin{align*}
\mathbb{P}_\beta&\Big(\|\hat{\beta}-\beta\|_1 \ge\tfrac{\sqrt{(C_0+1)\|\beta\|_0
(M_1r^2(\beta)+M\sigma^2)}}{\|\mathrm{X}\|_{\rm max}\phi_1((C_0+1)\|\beta\|_0)}\Big)\\
& \le H_1 e^{-m_1 M}+C_\nu\exp\big\{-c_2(M'_0/2-c_3)\|\beta\|_0\log(\tfrac{ep}{\|\beta\|_0})\big\},\\
\mathbb{P}_\beta&\Big(\|\hat{\beta}-\beta\| \ge\tfrac{\sqrt{M_1r^2(\beta)+M\sigma^2}}{\|\mathrm{X}\|_{\rm max}\phi_2((C_0+1)\|\beta\|_0)}\Big)\\
& \le H_1 e^{-m_1 M}+C_\nu\exp\big\{-c_2(M'_0/2-c_3)\|\beta\|_0\log(\tfrac{ep}{\|\beta\|_0})\big\},
\end{align*}
uniformly in $\beta\in\mathbb{R}^p$, where $C_0=\max\{M'_0,{M'_0}^2/e\}$.
\end{proposition}
\begin{proof} 
Consider the case $\hat{\theta}=\check{\theta}
=\mathrm{X}\check{\beta}$, where $\check{\theta}$ is defined 
by \eqref{emp_emp_posterior}.
Denote for brevity $\Delta=\tfrac{\sqrt{(C_0+1)\|\beta\|_0(M_1r^2(\beta)+M\sigma^2)}}
{\|\mathrm{X}\|_{\rm max}\phi_1((C_0+1)\|\beta\|_0)}$ and introduce the event 
$E_M=\{|\hat{I}| \le C_0\|\beta\|_0\}$, where  $\hat{I}$ is defined by $\eqref{I_MAP}$. 
By the definition of compatibility coefficient, 
in case $|\hat{I}| \le  C_0\|\beta\|_0$,  $\|\check{\beta}-\beta\|_1$ is bounded by $\sqrt{(C_0+1)\|\beta\|_0}
\|\mathrm{X}(\check{\beta}-\beta)\|/(\|\mathrm{X}\|_{\rm max}\phi_1((C_0+1)\|\beta\|_0))$,
since the cardinality of  $\text{supp}(\check{\beta}-\beta)$ is at most  $(C_0+1)\|\beta\|_0$. 
By Theorem \ref{th1}, $\|\mathrm{X}(\check{\beta}-\beta)\|>\sqrt{M_1r^2(\beta)+M\sigma^2}$
has probability bounded by $H_1 e^{-m_1 M}$. Using this and Proposition \ref{cor1_dimension}, 
we have 
\begin{align*}
\mathbb{P}_{\beta}(\|\check {\beta}-\beta\|_1 \ge  \Delta)&=
\mathbb{P}_{\beta}(\|\check {\beta}-\beta\|_1 \ge  \Delta,E_M)
+\mathbb{P}_{\beta}(\|\check {\beta}-\beta\|_1 \ge  \Delta, E_M^c)\\
&\le \mathbb{P}_{\beta}(\|X\check {\beta}-X\beta\|^2 \ge  M_1r^2(\beta)+M\sigma^2)
+\mathbb{P}_{\beta}(E_M^c)\\
&\le  H_1 e^{-m_1 M}+\mathbb{P}_{\beta}(E_M^c)\\
&\le H_1 e^{-m_1 M}+\mathbb{E}_{\beta}\check{\pi}(I: |I|>C_0\|\beta\|_0|Y)\\
&\le H_1 e^{-m_1 M}+
C_\nu\exp\big\{-c_2(M'_0/2-c_3)\|\beta\|_0\log(\tfrac{ep}{\|\beta\|_0})\big\}.
\end{align*}
The proof of the second claim for the case  
$\hat{\theta}=\check{\theta}=\mathrm{X}\check{\beta}$ and the proofs of the both claims 
for the case $\hat{\theta}=\tilde{\theta}=\mathrm{X}\tilde{\beta}$ 
are similar and therefore omitted.
\end{proof}

\section{Linear regression with shape structure: aggregation} 
\label{sec_aggregation} 
 
Consider the  regression model with a fixed design: 
\begin{align}
\label{aggregation_model_suppl}
Y_i = f(x_i) + \sigma\xi_i ,\; i\in[n],
\end{align}
where $x_i \in \mathcal{X}$ are nonrandom, $ \mathcal{X}$ is an arbitrary set, $f: \mathcal{X}
\to\mathbb{R}$ is an unknown function, and $\xi_i \overset{\rm ind}{\sim}\mathrm{N}(0,1)$. 
We use the notation $\|f\|^2 = \sum_{i\in[n]} f^2(x_i)$.

Aggregation in nonparametric regression has been considered by \cite{Nemirovski:2000, Bunea&Tsybakov&Wegkamp:2007, Rigollet&Tsybakov:2011, Tsybakov:2014} 
and many others, with estimation as grand problem.
Here we demonstrate that the results obtained 
in the above mentioned papers follow from our general framework results. 
Actually, we obtain stronger versions of the results as they hold in the refined formulation 
and distribution-free setting (the observations are not necessarily normal and/or independent).
Moreover, apart from the estimation results, claims (i) and (iii)--(vii) of Corollary \ref{meta_theorem} 
deliver additional results for the model \eqref{aggregation_model_suppl}, DDM contraction, uncertainty 
quantification and weak structure recovery results, which are new to the best of our knowledge.

%notice that the claim (ii) of Corollary \ref{meta_theorem}, being a uniform exponential 
%inequality in probability, is itself finer and stronger version of the corresponding oracle 
%result in expectation (like \eqref{oracle_ineq}).
%Moreover, we additionally obtain claims (i) and (iv)-(vii) of Corollary \ref{meta_theorem}
%for the posterior concentration rate and uncertainty quantification, 
%and these results are new to the best of our knowledge. Global results over 
%appropriate scales can also be derived as
%consequences of Corollary \ref{minimax_results}. 
%Besides, we can drop the normality and independence assumptions 
%and impose only Condition \eqref{cond_nonnormal} instead. 
%One can readily formulate these results.

Assume we are given a collection of functions $\{f_1, ..., f_p\}$, called \emph{dictionary}.
For $\beta\in\mathbb{R}^p$, let $f_\beta=\sum_{j=1}^p\beta_jf_j$. 
By choosing a rich dictionary $\{f_1, ..., f_p\}$ and an appropriate 
$\beta\in\mathcal{B}\subseteq\mathbb{R}^p$, one can expect
$f_\beta$ to be close to $f$ under some assumptions. 
For a certain choice of $\mathcal{B}$, the so called \emph{aggregation problem}
consists basically in determining the ``best'' $\hat{\beta}\in\mathcal{B}$
on the basis of the data $Y$ such that 
$\big(\sum_{j=1}^p \hat{\beta}_j f_j(x_i), i\in[n]\big)$ well estimates the true $(f(x_i),i\in [n])$.
The appropriate structure here is \emph{sparsity} as in Section \ref{sec_lin_regr_sparsity}.

Introduce the sets $\mathcal{B}$ studied in the literature:
%It will be convenient for us to formulate those sets as sets for $\theta=\mathrm{X}\beta$, i.e., 
%$\Theta = \{\theta = \mathrm{X}\beta: \beta \in \mathcal{B}\}$. We consider $\Theta\in\{\Theta_{(MS)},
%\Theta_{(C)},\Theta_{(L)}, \Theta_{(L_s)}, \Theta_{(C_s)}\}$, where the corresponding 
the sets $\mathcal{B}_{(MS)}$, $\mathcal{B}_{(C)}$, $\mathcal{B}_{(L)}$,
$\mathcal{B}_{(L_s)}$, $\mathcal{B}_{(C_s)}$ are defined as in \cite{Rigollet&Tsybakov:2011}.
Precisely, let $B_1(1) = \{\beta\in\mathbb{R}^p: \|\beta\|_1=\sum_{j=1}^p |\beta_j|\le 1\}$ 
and $B_0(s)=\ell_0[s]=\{\beta\in\mathbb{R}^p: \|\beta\|_0\le s\}$ for $s \in[p]$. 
Next, define $\mathcal{B}_{(MS)}=B_0(1)$, $\mathcal{B}_{(C)}$ is a closed convex subset of 
$B_1(1)$, $\mathcal{B}_{(L)}=B_0(p)=\mathbb{R}^p$, $\mathcal{B}_{(L_s)} = B_0(s)$ and
$\mathcal{B}_{(C_s)}$ as a closed convex subset of $B_0(s) \cap B_1(1)$. 
Thus $\mathcal{B}\in\{\mathcal{B}_{(MS)},\mathcal{B}_{(C)},\mathcal{B}_{(L)},
\mathcal{B}_{(L_s)},\mathcal{B}_{(C_s)}\}$.

%??? Following the notation in \cite{Tsybakov:2014}, 
%define the simplex $\Lambda^p=\{\beta\in\mathbb{R}_+^p: \sum_{j=1}^p \beta_j=1\}$, 
%the ball $B_0(s^*)=\ell_0[s^*]=\{\beta\in\mathbb{R}^p: \|\beta\|_0\le s^*\}$, 
%the model selection aggregation space $\Theta_{(MS)}=B_0(1)\cap\Lambda^p$, 
%the convex aggregation space $\Theta_{(C)}=\Lambda^p$, 
%the linear aggregation space $\Theta_{(L)}=\mathbb{R}^p$, the  sparse aggregation space 
%$\Theta_{(L_{s^*})}=B_0(s^*)$ and  the sparse convex aggregation space  $\Theta_{(C_{s^*})}
%=B_0(s^*)\cap\Lambda^p$. 

First recall the main estimation results from \cite{Rigollet&Tsybakov:2011} 
(lower bounds are also established in that paper).
The so called \emph{exponential screening estimator} 
$f_{\tilde{\beta}^{\rm ES}}$ is proposed in \cite{Rigollet&Tsybakov:2011}. 
Under the assumptions $\max_{j\in[p]}\|f_j\| \le\sqrt{n}$, $p\ge 2$, 
$n \ge 1$, $s\in[p]$, the following oracle estimation result is derived in \cite{Rigollet&Tsybakov:2011}:
%This motivates that the estimator $\hat{f}$ should have the following form: 
%$\hat{f}=\sum_{j=1}^p\hat{\beta}_jf_j$. More precisely,  we want to find estimator $\hat{f}$ such that
for some constant $C>0$,
\begin{align}
\label{oracle_ineq_aggreg}
\mathbb{E}_f\|f_{\tilde{\beta}^{\rm ES}}-f\|^2\le \inf_{\beta\in\mathcal{B}}\|f_\beta-f\|^2
+C\sigma^2\psi_{n,p}(\mathcal{B}),
\end{align}
where  $\psi_{n,p}(\mathcal{B})=\min\{\psi_{n,p}^*(\mathcal{B}), r\}$ is the optimal rate of aggregation for the corresponding classes $\mathcal{B}\in\{\mathcal{B}_{(MS)},\mathcal{B}_{(C)},\mathcal{B}_{(L)},
\mathcal{B}_{(L_{s})},\mathcal{B}_{(C_{s})}\}$, $r=\rank(\mathrm{X})$, $\psi_{n,p}^*(\mathcal{B})$ 
is defined as follows:
\begin{align*}
\psi_{n,p}^*(\mathcal{B})=\begin{cases}
\log p, &\mathcal{B}=\mathcal{B}_{(MS)},\\
\sqrt{n\log \big(1+\tfrac{ep\sigma}{\sqrt{n}}\big)},&\mathcal{B}=\mathcal{B}_{(C)},\\
r,&\mathcal{B}=\mathcal{B}_{(L)},\\
s\log(1+ep/s),&\mathcal{B}=\mathcal{B}_{(L_{s})},\\
\min\Big\{\sqrt{n\log \big(1+\tfrac{ep\sigma}{\sqrt{n}}\big)}, s\log(1+ep/s)\Big\},&\mathcal{B}=\mathcal{B}_{(C_{s})}.
\end{cases}
\end{align*}
%Oracle estimation results  for these aggregation problems were derived by \cite{Rigollet&Tsybakov:2011},
%\cite{Tsybakov:2014}, and posterior contraction rate results by \cite{Gao&vanderVaart&Zhou:2015}.  
An advantageous feature of the result \eqref{oracle_ineq_aggreg} 
is its universality: the aggregation is attained over 
the five classes simultaneously. This result follows from Lemma 8.2 and Theorem 3.1 of \cite{Rigollet&Tsybakov:2011}. The result of Theorem 3.1 from \cite{Rigollet&Tsybakov:2011} in our notation reads as follows:
for any $p,n\ge 1$
\begin{align}
\label{Tsybakov_result}
\mathbb{E}_f\|f_{\tilde{\beta}^{\rm ES}}-f\|^2 &\le
\min_{\beta \in\mathbb{R}^p} \Big\{\|f-f_\beta\|^2 + \sigma^2 \big[r \wedge 
\big(9|I^*(\beta)|\log(1+\tfrac{e p}{|I^*(\beta)|\vee 1})\big)\big] \Big\} \notag\\
& \quad + 8\sigma^2 \log 2.
\end{align}
But %A very useful fact derived in \cite{Rigollet&Tsybakov:2011} is that, 
Lemma 8.2 is fulfilled as soon as Theorem 3.1 holds and $\max_{j\in[p]}\|f_j\| \le\sqrt{n}$; 
see \cite{Rigollet&Tsybakov:2011}.
This means (as is as  concluded in \cite{Rigollet&Tsybakov:2011}) 
that under the condition $\max_{j\in[p]}\|f_j\| \le\sqrt{n}$, any estimator satisfying 
\eqref{Tsybakov_result} (possibly with different constants in the right hand side) 
leads to the universal oracle inequality \eqref{oracle_ineq_aggreg}.

Let us demonstrate that we can derive the same type of estimation results as in 
\cite{Rigollet&Tsybakov:2011}, again as consequences of our general approach 
for particular choice of sparsity structures. In fact, we improve upon 
certain aspects and also provide the results on uncertainty quantification, again as 
consequence of our general framework results.

The aggregation problem considered here for the model \eqref{aggregation_model_suppl}
can be associated with the standard linear regression model \eqref{lin_regr_suppl}.
%(with $m=1$) studied in Section \ref{sec_lin_regr_sparsity}. 
Indeed, let $\theta=(f(x_i),\, i \in [n])$ and notice that
the vector $f_\beta=(f_\beta(x_i),i\in[n])$ can be represented as $\mathrm{X}\beta$, 
where $\beta\in\mathbb{R}^p$ is the unknown high-dimensional parameter and 
the design $(n\times p)$-matrix $\mathrm{X}$ has the entries $X_{ij}=f_j(x_i)$, $(i, j)\in[n]\times[p]$. 
In doing so, we arrive to the general setting $Y=\theta + \sigma\xi$, but now we take 
the family of structures $\mathcal{I}$ and the corresponding family of linear spaces  
$\{\mathbb{L}_I, I\in \mathcal{I}\}$, defined  by \eqref{lin_regr_L_I}. 
% as in Section \ref{sec_lin_regr_suppl} of this Supplement.
%The conditions are already verified in Section \ref{sec_lin_regr_sparsity}.
Then (see Remark \ref{lin_model2})
the general framework results imply Corollary \ref{meta_theorem} with the oracle rate 
\[
r^2(\theta)=\min_{I \in \mathcal{I}} r^2(I,\theta)=
 \min_{ I \in \mathcal{I}} \big\{\|\theta
-\mathrm{P}_I\theta\|^2 +\sigma^2\rho(I)\big\},
\]
where the majorant $\rho(I)$ is defined by \eqref{formula_rho}.

Recall the full family of structures $\bar{\mathcal{I}}=\{J:J\subseteq [p]\}$. 
Since $\|\theta-\mathrm{P}_{I_r}\theta\|^2=\min_{I\in\mathcal{I}}\|\theta-\mathrm{P}_I\theta\|^2
=\min_{I\in\bar{\mathcal{I}}}\|\theta-\mathrm{P}_I\theta\|^2$, it is easy to see that 
\begin{align}
r^2(\theta) &=\big[\min_{I \in \mathcal{I}_1} r^2(I,\theta)\big]\wedge r^2 (I_r,\theta) =
\min_{I \in \bar{\mathcal{I}}}\big\{\|\theta-\mathrm{P}_I\theta\|^2 +
\sigma^2[\rho(I)\wedge r]\big\} \notag\\
&=\min_{\beta\in\mathbb{R}^p} \big\{\|f_\beta-f\|^2 +\sigma^2 \big[r \wedge
\big(2|I^*(\beta)|\log(\tfrac{ep}{|I^*(\beta)|})\big)\big]\big\}.
\label{local_rate3}
\end{align}
In particular, property (ii) of Corollary  \ref{meta_theorem} entails
that for some $C_0,C_1>0$
\[
\mathbb{E}_f\|f-\hat{\theta}\|^2 \le
C_0 r^2(\theta)+ C_1\sigma^2,
\]
where $r^2(\theta)$ is defined by \eqref{local_rate3}, which is in fact property \eqref{Tsybakov_result} 
for our estimator $\hat{\theta}$. As is mentioned above, \cite{Rigollet&Tsybakov:2011} established that
\eqref{Tsybakov_result} (with the additional assumption $\max_{j\in[p]}\|f_j\| \le\sqrt{n}$) 
in turn leads to the universality property \eqref{oracle_ineq_aggreg}. This means that 
\eqref{oracle_ineq_aggreg} holds also for our estimator $\hat{\theta}$: for some $C_0,C_2>0$,
\[
\mathbb{E}_f\|f-\hat{\theta}\|^2\le C_0\inf_{\beta\in\mathcal{B}}\|f_\beta-f\|^2
+C_2\sigma^2\psi_{n,p}(\mathcal{B}).
\]

We should mention that the constants in the universality property 
for our estimator $\hat{\theta}$ may be worse than those for the estimator 
$\tilde{f}_{\beta^{\rm ES}}$. On the other hand, 
notice that the claim (ii) of Corollary \ref{meta_theorem}, being a uniform exponential 
inequality in probability, is itself finer and stronger version of the corresponding oracle 
result in expectation (like \eqref{oracle_ineq_aggreg}).
Moreover, we additionally obtain claims (i) and (iv)--(vii) of Corollary \ref{meta_theorem}
for the DDM (empirical Bayes posterior for the normal case) contraction and uncertainty quantification, 
and these results are new to the best of our knowledge. Global results over 
appropriate scales can also be derived as
consequences of Corollary \ref{minimax_results}. 
Besides, we can drop the normality and independence assumptions 
and impose only Condition \eqref{cond_nonnormal} instead. 
One can readily formulate these results.

\section{Matrix+noise model: covariance matrix estimation} 
\label{sec_covariance_matrix}

Suppose we observe $n$ iid $p$-dimensional vectors $X_1,\ldots, X_n$,
$X_i=(X_i^1,\ldots, X_i^p)^T$, $i\in[n]$, %distributed according to some distribution $F$ 
with $\mathbb{E}X_i = 0$, $\mathbb{E}(X_i^j)^4\le C_X$, 
$(i,j)\in[p]\times[p]$, and the unknown covariance matrix $\mathbb{E}(X_iX_i^T )=\mathrm{\Sigma}$, $i\in[n]$. 
Without loss of generality, we set $C_X=1$. % for the rest of this section. 
Let $\mathcal{C}\subseteq \mathbb{R}^{p\times p}$ denote the set of all $p$-dimensional covariance matrices.
Assume that for some (known and independent of $p$) 
$\varepsilon_0>0$, 
\[
\mathrm{\Sigma} \in\mathcal{C}_{\varepsilon_0}=\{\mathrm{M}\in\mathcal{C}:
\varepsilon_0\le\lambda_{\rm min}(\mathrm{M})\le\lambda_{\rm max}(\mathrm{M})
\le\varepsilon_0^{-1}\}.
\] 
Here, $\lambda_{\rm max}(\mathrm{M})$ and  $\lambda_{\rm min}(\mathrm{M})$ 
are the maximum and minimum eigenvalues of $\mathrm{M}$. 
We assume that $X_i \overset{\rm ind}{\sim}\mathrm{N}(\mathrm{0}_n,\mathrm{\Sigma})$, where 
$\mathrm{0}_n$ is the $n$-dimensional vector of zeros. 
The normality assumption is not important to us, this only plays a role in that we can use certain 
auxiliary result below (Proposition \ref{proposition_cov}) which is available only for the normal case.

We are interested in recovering the covariance matrix $\mathrm{\Sigma}=\{\Sigma_{ij}\}_{1\le i,j\le p}$
which is assumed to have the \emph{banding} or \emph{sparsity} structure, to be specified later. 
The maximum likelihood estimator of $\mathrm{\Sigma}$ is 
$\tilde{\mathrm{\Sigma}}=\frac{1}{n}\sum_{l=1}^n(X_l-\bar{X})(X_l-\bar{X})^T
=\frac{1}{n}\sum_{l=1}^nX_lX_l^T-\bar{X}\bar{X}^T$, where $\bar{X}=\frac{1}{n}\sum_{l=1}^nX_l$. 
Since $\bar{X}\bar{X}^T$ is a higher order term (see Remark 1 in \cite{Caietal:2010}), 
we shall ignore this term and focus on the dominating term $\frac{1}{n}\sum_{l=1}^nX_lX_l^T$ 
for estimating %the covariance matrix 
$\mathrm{\Sigma}$. 

Let $\mathrm{Y}=(Y_{ij})_{i,j\in[p]}=\frac{1}{n}\sum_{l=1}^nX_lX_l^T$, 
$Y=\text{vec}\big[( Y_{ij})\big]=(Y_{11},Y_{12}\ldots,Y_{pp})^T$. 
We obtain the following model:
\begin{align}
\label{model_cov}
Y_{ij}=\Sigma_{ij}+ \sigma_n \xi_{ij}, \quad i,j\in[p],
\end{align}
where $\sigma_n\xi_{ij}=Y_{ij} - \mathbb{E}Y_{ij}=Y_{ij}-\Sigma_{ij}$, so that
%satisfy Condition \eqref{cond_nonnormal}
%($\mathbb{L}_{I}, \rho(I),\mathcal{S}$ and  $\mathcal{I}_s $ for this model will be specified later) with 
$\mathbb{E} \xi_{ij}=0$ and 
\begin{align*}
\sigma^2_n\text{Var}(\xi_{ij})&=\tfrac{1}{n}\mbox{Var}(X_1^iX_1^j)
=\tfrac{1}{n}\mathbb{E}(X_1^iX_1^j)^2\le \tfrac{1}{n}[\mathbb{E}(X_1^i)^4]^{\frac{1}{2}}[\mathbb{E}(X_1^j)^4]^{\frac{1}{2}}\le\tfrac{1}{n}.
\end{align*} 
The parameter $\sigma_n$ will be chosen later, for now it is any sequence $\sigma_n \in [0,1]$.
We thus have a particular case of general framework model \eqref{model}, where the parameter 
of interest is now denoted by $\mathrm{\Sigma}$ instead of $\theta$.
%Then without loss of generality  one can assume $\sigma=n^{-1/2}$. 
Recall that we work with the usual norm of vectorized version of the parameter 
$\mathrm{\Sigma}=(\Sigma_{ij})_{i,j \in [p]}$, that is, if $\mathrm{\Sigma}$ is seen as matrix, 
then $\|\mathrm{\Sigma}\|$ means its Frobenius norm. We denote the probability measure of $Y$ 
from the model (\ref{model_cov}) by $\mathbb{P}_{\mathrm{\Sigma}}$, and the corresponding 
expectation $\mathbb{E}_{\mathrm{\Sigma}}$. 
%Recall that $\mathbb{P}_{\mathrm{\Sigma}}$ satisfies $\mathbb{E}(X_i^j)^4\le 1$, $i\in[n]$, $j \in [p]$.

\subsection{Matrix+noise with smoothness structure: banded covariance matrix} 
\label{subsec_banding}

Assume that the covariance matrix $\mathrm{\Sigma}=(\Sigma_{ij})_{i,j\in[p]}$ has 
a \emph{banding} structure, i.e.,  $\Sigma_{ij}=0$ for all $i,j\in[p]$ such that $|i-j|>I$ 
for some $I\in[p]_0$. To model this structure, define the linear spaces
\begin{align*}
\mathbb{L}_{I}=\big\{\text{vec}(x)\in\mathbb{R}^{p^2}: x_{ij}=x_{ji}\;\, \forall i,j\in[p]; 
x_{ij}=0\;\,\text{if}\, |i-j|>I\big\},\;\, I\in \mathcal{I}=[p]_0.
\end{align*}
Then $\|\mathrm{\Sigma}-\mathrm{P}_I\mathrm{\Sigma}\|^2=\sum_{|i-j|>I}\Sigma_{ij}^2$,
 $d_I=\dim(\mathbb{L}_{I})=p+\sum_{l=1}^I(p-l)=p+I(p-(I+1)/2)$,
the structural slicing mapping $s(I)=I$, $\mathcal{S}=[p-1]_0$, 
$\log |\mathcal{I}_s|=0$, $d_I=p+I(p-(I+1)/2)$ leading to the 
majorant $\rho(I)=d_I=p+I(p-(I+1)/2)$.  

Condition \eqref{(A2)} is fulfilled, since $\sum_{I\in\mathcal{I}} e^ {-\nu\rho(I)}
\le \sum_{s\in\mathcal{S}} e^ {-\nu s}\le\frac{e^\nu}{e^{\nu}-1}= C_\nu$ for any $\nu>0$.  
However, in order to derive at least the local estimation and posterior contraction results,
we also need Condition \eqref{cond_nonnormal}. This condition is now 
not easy to check since the errors $\xi_{ij}$'s are dependent in the model 
\eqref{model_cov}. We apply the following strategy (in the same spirit as in 
Section \ref{sec_density}): introduce certain event 
and establish that the probability of this event is exponentially small (in $n$);
next, under this event establish Condition \eqref{cond_nonnormal}; finally, combine these 
two facts to derive the local estimation and posterior contraction results.

The following proposition (formulated in our notation) is Lemma $12$ from Appendix of %supplementary material 
\cite{Kolar&Liu:2012} and is given here  for completeness,
its proof can be found in \cite{Kolar&Liu:2012}.
\begin{proposition}
\label{proposition_cov}
Let $\nu_{ij}=\max\{(\Sigma_{ii}\Sigma_{jj})^{1/2} -\Sigma_{ij},
(\Sigma_{ii}\Sigma_{jj})^{1/2} +\Sigma_{ij}\}$, $i,j\in[p]$.
Then for any $t\in[0,\nu_{ij}/2)$
\begin{align*}
\mathbb{P}(\sigma_n|\xi_{ij}|\ge t)\le 4\exp\big\{-\tfrac{3nt^2}{16\nu_{ij}^2}\big\}.
\end{align*}
\end{proposition}
The relation $\mathbb{P}(\max_{i,j\in[p]}|\xi_{ij}|\ge t)\le\sum_{i,j\in[p]}
\mathbb{P}(|\xi_{ij}|\ge t)$ and Proposition \ref{proposition_cov} imply that,
for the event $E=\{\max_{i,j\in[p]}|\xi_{ij}|\le t_0\}$ with $t_0=\frac{\min_{i,j\in[p]}\nu_{ij}}{\sqrt{5}}$,
\begin{align}
\label{rel_max_cov}
\mathbb{P}(E^c)=\mathbb{P}\big(\max_{i,j\in[p]}|\xi_{ij}|\ge t_0\big)\le H'\exp\{-\bar{c}_1n\sigma_n^2+\bar{c}_2\log p\},
\end{align}
where $\nu_{ij}$ is defined in Proposition \ref{proposition_cov},
$H'=4$, $0<\bar{c}_1=\frac{3\varepsilon_0^4}{320}\le \tfrac{3}{80}
\tfrac{\min_{i,j\in[p]}\nu^2_{ij}}{\max_{i,j\in[p]}\nu^2_{ij}}$ 
(because $\varepsilon_0 \le \min_{i,j\in[p]}\nu_{ij} \le \max_{i,j\in[p]}\nu_{ij} \le 2\varepsilon_0^{-1}$)
and $\bar{c}_2=2$. Clearly, for \eqref{rel_max_cov} to be useful, 
we need %to assume the asymptotic relation 
$\log p \lesssim n\sigma^2_n$. %as $n\to \infty$.
 
%As $\mathrm{\Sigma} \in\mathcal{C}_{\varepsilon_0}$,
By the assumptions on $\mathrm{\Sigma}$, we have that $\min_{i,j\in[p]}\nu_{ij}\le 2\varepsilon_0^{-1}$,
so that $t_0^2=\min_{i,j\in[p]}\nu^2_{ij}/5\le \tfrac{4}{5\varepsilon_0^2}$.
Using this and \eqref{rel_max_cov}, we ensure Condition \eqref{cond_nonnormal} 
under the event $E=\{\max_{i,j\in[p]}|\xi_{ij}|\le t_0\}$ with $\alpha=1\wedge (5\varepsilon_0^2)/4$.
Exactly,
\begin{align}
\mathbb{E}\exp\big\{&\alpha \|\mathrm{P}_I\xi\|^2\big\}\mathrm{1}\{E\}=
\mathbb{E} \exp\big\{\alpha\sum_{|i-j|>I} \xi_{ij}^2\big\}\mathrm{1}\{\max_{i,j\in[p]}|\xi_{ij}|<t_0\}\notag\\
&\le \exp\big\{\alpha t_0^2(p+I(p-(I+1)/2))\big\}\le\exp\{d_I\}.
\label{conditional_A1}
\end{align}
Condition \eqref{(A3)} holds as well. Indeed, for any $I_0,I_1\in\mathcal{I}$ take $I'=I_0\vee I_1$ and 
verify that $(\mathbb{L}_{I_0}\cup\mathbb{L}_{I_1})\subseteq \mathbb{L}_{I'}=\mathbb{L}_{I_0}
+\mathbb{L}_{I_1}$ and $\rho(I')\le \rho(I_0)+\rho(I_1)$.
 
The oracle rate is in this case
$r^2(\mathrm{\Sigma}) =\min_{I\in\mathcal{I}} r^2(I, \mathrm{\Sigma})$,
where
\[
r^2(I,\mathrm{\Sigma})=\|\mathrm{\Sigma}-\mathrm{P}_I\mathrm{\Sigma}\|^2
+\sigma^2_n \rho(I)=\sum_{|i-j|>I}\Sigma_{ij}^2+\sigma^2_n\big(p+I(p-(I+1)/2)\big),
\]
and the EBR-set $\Theta_{\rm eb}=\Theta_{\rm eb}(t)$ is given by \eqref{cond_ebr}, but now in terms of 
the bias and variance parts of the oracle rate $r^2(\mathrm{\Sigma})$. 

We have thus verified the conditional version of Condition \eqref{cond_nonnormal} 
(under the event $E$) and Conditions \eqref{(A2)} and \eqref{(A3)} for the model \eqref{model_cov} 
with the banding structure. This means that we can derive results on estimation, posterior contraction 
and uncertainty quantification for this model. These are the counterparts of claims (i)--(v) of 
Corollary \ref{meta_theorem} summarized by Theorem \ref{meta_theorem_cov} below. 
To the best of our knowledge, there are no local results on estimation, posterior contraction 
rate and  uncertainty quantification problems for this model. 

A couple of conventions concerning notation in Theorem \ref{meta_theorem_cov}: as compared 
to the general framework notation, in the model \eqref{model_cov}, the parameter of interest 
is denoted by $\mathrm{\Sigma}$ instead of $\theta$ and the corresponding estimator becomes 
$\hat{\mathrm{\Sigma}}$ instead of $\hat{\theta}$; 
in the posteriors for $\mathrm{\Sigma}$ we use the variable $\mathit{\Sigma}$ to distinguish it from 
the ``true''  $\mathrm{\Sigma}\in\mathcal{C}_{\varepsilon_0}$. 
We keep the same notation for all other quantities involved as in the general
framework (like $\hat{r}$, $\hat{R}_M$, $B(\hat{\mathrm{\Sigma}},\hat{R}_M)$), 
with the understanding that these are specialized for the model \eqref{model_cov}
with the banding structure and the oracle rate $r^2(\mathrm{\Sigma})$.

\begin{theorem}
\label{meta_theorem_cov}
Let the constants $M_0, M_1, M_3, H_0, H_1, H_2, H_3, m_0, m_1$, $m_2,m_3$, 
$c_2, c_3$, $C_\nu, H', \bar{c}_1, \bar{c}_2$ 
be defined in Theorems \ref{th1}-\ref{th3} and \eqref{rel_max_cov}.
%There exist constants $\bar{H}_0, \bar{H}_1, \bar{m}_0$ such that 
Then for any $M \ge 0$,
\begin{align*}
&\sup_{\mathrm{\Sigma}\in\mathcal{C}_{\varepsilon_0}}\mathbb{E}_{\mathrm{\Sigma}}
\hat{\pi}\big(\|\mathit{\Sigma}-\mathrm{\Sigma}\|^2\ge M_0r^2(\mathrm{\Sigma})+M\sigma^2_n|Y\big) 
\le H'e^{-\bar{c}_1n\sigma^2_n+\bar{c}_2\log p}\!+\! H_0 e^{-m_0 M},\\
&\sup_{\mathrm{\Sigma}\in\mathcal{C}_{\varepsilon_0}}\mathbb{P}_{\mathrm{\Sigma}}
\big(\|\hat{\mathrm{\Sigma}}-\mathrm{\Sigma}\|^2 \ge M_1 r^2(\mathrm{\Sigma})+M\sigma^2_n\big)
\le H'e^{-\bar{c}_1n\sigma^2_n+\bar{c}_2\log p}+H_1 e^{-m_1 M},\\
&\sup_{\mathrm{\Sigma}\in\mathcal{C}_{\varepsilon_0}}\mathbb{E}_{\mathrm{\Sigma}}
\hat{\pi}\big(I: r^2(I,\mathit{\Sigma})\ge c_3 r^2(\mathrm{\Sigma})+M\sigma^2_n|Y\big) 
\le H'e^{-\bar{c}_1n\sigma^2_n+\bar{c}_2\log p}+ C_\nu e^{-c_2 M},\\
%&\sup_{\mathrm{\Sigma}\in\in\mathcal{C}_{\varepsilon_0}}\mathrm{P}_{\mathrm{\Sigma}}
%\big(\Sigma\notin B(\hat{\Sigma},\bar{R}_M)\big)  
%\le H'e^{-\bar{c}_1n+\bar{c}_2\log p}+\psi_1(M/4)+\psi_2(M)+\bar{H}_1e^{-m_1M},\\ \label{th_v}&
%\sup_{\Sigma\in\mathbb{R}^{p^2}}\mathrm{P}_{\Sigma}\big(\bar{R}_M^2\ge g_M(\Sigma,p^2)\big)
%\le\! H'e^{-\bar{c}_1n\sigma^2_n+\bar{c}_2\log p}\!+ \!\psi_1(M/4)+\psi_2(M)+2\bar{H}_1e^{-m_1M},
&\sup_{\mathrm{\Sigma}\in\mathcal{C}_{\varepsilon_0}}
 \mathbb{P}_{\mathrm{\Sigma}}\big(\hat{r}^2\ge M_3 r^2(\mathrm{\Sigma})+(M+1)\sigma^2_n\big)
\le  H'e^{-\bar{c}_1n\sigma^2_n+\bar{c}_2\log p}+H_3 e^{-m_3 M},\\
&\sup_{\mathrm{\Sigma}\in\mathcal{C}_{\varepsilon_0}\cap\Theta_{\rm eb}}\mathbb{P}_{\mathrm{\Sigma}}
\big(\mathrm{\Sigma}\notin B(\hat{\mathrm{\Sigma}},\hat{R}_M)\big) 
\le H'e^{-\bar{c}_1n\sigma^2_n+\bar{c}_2\log p}+H_2 e^{-m_2M}.
\end{align*}
%where $g_M(\Sigma,p^2)=M_1r^2(\Sigma)+M/n+4 pG_M /n$, $\bar{R}^2_M=
%\big(\|Y'-\hat{\Sigma}\|^2- p^2/n +2pG_M/n\big)_+, G_M=\sqrt{M(M+M_1)}$. 
\end{theorem}

Let us outline the idea of the proof (which is omitted) of the first claim of the above theorem;
the same reasoning applies to the remaining claims.
The expectation of the empirical Bayes posterior probability $\mathbb{E}_{\mathrm{\Sigma}}\Pi=
\mathbb{E}_{\mathrm{\Sigma}}\hat{\pi}\big(\|\mathit{\Sigma}-\mathrm{\Sigma}\|^2\ge 
M_0r^2(\mathrm{\Sigma})+M\sigma^2_n|Y\big)$ is bounded by the sum of two terms 
$\mathbb{E}_{\mathrm{\Sigma}}\Pi\le \mathbb{P}_{\mathrm{\Sigma}}(E^c)
+\mathbb{E}_{\mathrm{\Sigma}}\Pi 1_E$. The first term is evaluated by using \eqref{rel_max_dens}
(obtaining the bound $H'e^{-\bar{c}_1n\sigma^2_n+\bar{c}_2\log p}$); 
the second term is evaluated exactly in the same way as 
in the proof Theorem \ref{th1} because Condition \eqref{cond_nonnormal} 
is fulfilled under the event $E$ according to \eqref{conditional_A1}.   
Counterparts of assertions (ii) and (iii) of Theorem \ref{th2} can also be formulated and proved
in the same way.

As to the choice of $\sigma^2_n$, this quantity is in the oracle rate, so that we would want 
it to be as small as possible. On the other hand, we want the claims of the theorem to be non-void,
which is ensured only if $\sigma^2_n n \ge C \log p$, or $\sigma^2_n\ge \tfrac{C\log p}{n}$, for sufficiently large $C>0$. In the sequel we take therefore $\sigma^2_n=\tfrac{C\log p}{n}$.
An extra log factor thus appeared which will also enter the minimax rates in the global results.   
We conjecture that one can get rid of that factor by using more accurate concentration inequalities 
when establishing Condition \eqref{cond_nonnormal}.

As usually, the local results of Theorem \ref {meta_theorem_cov} will imply global minimax 
adaptive results at once over all scales $\{\Theta_\beta,\, \beta\in\mathcal{B}\}$ covered by 
the oracle rate $r^2(\mathrm{\Sigma})$ (i.e., for which \eqref{oracle_minimax} holds). Below 
we present the example of scales $\{\Theta_\beta,\, \beta\in\mathcal{B}\}$ covered by %the oracle rate 
$r^2(\mathrm{\Sigma})$. 

\subsection{Minimax results for the scale $\{\mathcal{G}_\beta, \beta>0\}$}
For $\beta, L,\varepsilon_0>0$, define  
\begin{align*}
\mathcal{G}_\beta=\mathcal{G}_\beta(L,\varepsilon_0^{-1})=\big\{\Sigma\in
\mathcal{C}_{\varepsilon_0}: |\Sigma_{ij}|\le L|i-j|^{-(\beta+1)}\;\text{for}\; i\neq j\big\}.
\end{align*}
The rate $r^2(\mathcal{G}_\beta)
=\min\big\{pn^{-\frac{2\beta+1}{2(\beta+1)}},p^2 n^{-1}\big\}$ is minimax over 
the class  $\mathcal{G}_\beta$ under the Frobenius norm; see \cite{Caietal:2010}.
If $(\tfrac{n}{\log p})^{\frac{1}{2(\beta+1)}}\le p$, taking $I_*=\lfloor (n/\log p)^{\frac{1}{2(\beta+1)}}\rfloor$ and 
recalling $\sigma^2_n=\tfrac{C\log p}{n}$, we
derive that, uniformly in $\Sigma\in \mathcal{G}_\beta$,
\begin{align*}
%\sup_{\Sigma\in \mathcal{G}_\beta} 
r^2(\mathrm{\Sigma}) 
&\le r^2(I_*,\mathrm{\Sigma})=\sum_{|i-j|>I_*}\Sigma_{ij}^2+\sigma^2_n\big(p+I_*(p-(I_*+1)/2)\big)\\
&\lesssim pI_*^{-(2\beta+1)}+\tfrac{p I_* \log p}{n} \lesssim 
p\big(\tfrac{n}{\log p}\big)^{-\frac{2\beta+1}{2(\beta+1)}}.
\end{align*}
If $(\tfrac{n}{\log p})^{\frac{1}{2(\beta+1)}}> p$, we take $I_*=p$ to derive  
$\sup_{\Sigma\in \mathcal{G}_\beta} r^2(\mathrm{\Sigma})\lesssim p^2(\tfrac{n}{\log p})^{-1}$.

To summarize, we established that 
\[
\sup_{\Sigma\in \mathcal{G}_\beta} r^2(\mathrm{\Sigma})\lesssim 
\min\big\{p\big(\tfrac{n}{\log p}\big)^{-\frac{2\beta+1}{2(\beta+1)}},p^2\big(\tfrac{n}{\log p}\big)^{-1}\big\}
=\tilde{r}^2(\mathcal{G}_\beta),
\]
where $\tilde{r}^2(\mathcal{G}_\beta)$ is the minimax rate (up to a logarithmic factor) 
for the class $\mathcal{G}_\beta$.
Then the last relation and Theorem \ref{meta_theorem_cov} imply the global minimax results 
for the scale $\{\mathcal{G}_\beta, \beta>0\}$. These results will look as the ones from 
Theorem \ref{meta_theorem_cov} with the difference that the class 
$\mathcal{G}_\beta$ stands instead  $\mathcal{C}_{\varepsilon_0}$ and the rate 
$\tilde{r}^2(\mathcal{G}_\beta)$ stands instead of $r^2(\mathrm{\Sigma})$.
For the results to be most useful, we take $\sigma^2_n=\tfrac{C\log p}{n}$ with sufficiently large $C>0$
and $M=M_n\to \infty$ as $n\to \infty$ such that $M_n\sigma^2_n\asymp \tilde{r}^2(\mathcal{G}_\beta)$.

\begin{remark}
\label{rem_A4_issue}
The obtained local and global results on uncertainty quantification 
for the covariance matrix with a banding structure are new to the best of our knowledge.
Notice however that we derived only the uncertainty quantification 
results based on the EBR condition, whereas counterparts of claims (vi)--(vii)
of Corollary \ref{meta_theorem} are not established 
because we were unable to verify Condition \eqref{cond_A4}. 

The point is that the set $\tilde\Theta$ of highly structured parameters defined by \eqref{tilde_Theta}
is empty in this case: as $N=p^2$,
\[
r^2(\mathrm{\Sigma})\ge \sigma^2_n 
\rho(I) \gtrsim \sigma^2_n p=\sigma^2_n N^{1/2}.
\]
This means that the uncertainty quantification claims based on 
Condition \eqref{cond_A4} would be more valuable for this model because they are free 
of the deceptiveness phenomenon. 
Indeed, if we would have established Condition \eqref{cond_A4}, then 
the confidence ball $B(\hat{\mathrm{\Sigma}}, \tilde{R}_M)$ would have been of asymptotically 
full coverage and of the optimal oracle size, uniformly over $\mathcal{C}_{\varepsilon_0}$ (because 
$\tilde\Theta$ turns out to be empty in this case).
It is an open problem to verify Condition \eqref{cond_A4} for the model \eqref{model_cov},
the main issue is to find an appropriate statistics $V(Y')$ for which the second relation of 
Condition \eqref{cond_A4} is fulfilled.
\end{remark}

\subsection{Matrix+noise with sparsity structure: sparse covariance matrix} 
\label{subsec_cov_sparsity}

Here we briefly discuss the case of sparsity structure for the model \eqref{model_cov}.
Denote by $\Sigma_{-i}$ the $i$-th column of $\mathrm{\Sigma}$ with $\Sigma_{ii}$
removed. Let $p\ge 2$.  For any $i\in[p]$ the vector $\Sigma_{-i}\in\mathbb{R}^{p-1}$ is 
assumed to be \emph{sparse} so that $\Sigma_{ki}=0,\; k \not\in I_i$ (or $\Sigma_{ik}=0$), 
where $I_i\subseteq [p]\backslash\{i\}$. 
To model this sparsity structure, introduce the linear spaces
\begin{align*}
\mathbb{L}_{I}=\big\{\text{vec}(x)\in\mathbb{R}^{p^2}: x_{ki}=0,\; k \not\in I_i,\; i\in[p]\big\},
\end{align*}
where the structure is $I=(I_1,\ldots, I_p)\in\mathcal{I}\triangleq \{(J_1, \ldots, J_p): 
J_i\subseteq [p]\backslash\{i\}, i\in[p]\}$. 
Then $\|\mathrm{\Sigma}-\mathrm{P}_I\mathrm{\Sigma}\|^2=\sum_{i\in[p]}\sum_{j\not\in I_i}
\Sigma_{ji}^2$, $\dim(\mathbb{L}_{I})=p+\sum_{i=1}^p|I_i|$. Take the structural mapping 
$s(I)=(|I_1|,\ldots,|I_p|)\in \otimes_{i\in[p]} [p-1]_0=\mathcal{S}$, 
$\log |\mathcal{I}_{s(I)}|=\sum_{i\in[p]}\log \binom{p-1}{|I_i|}\le\sum_{i\in[p]} |I_i|
\log\big(\tfrac{e(p-1)}{|I_i|}\big)$. 

Next, along the same lines as in Section \ref{subsec_banding}, we can verify the conditional version 
of Condition \eqref{cond_nonnormal} (under the same event $E$) with $d_I=\dim(\mathbb{L}_{I})
=p+\sum_{i\in[p]}|I_i|$. Thus, we take the majorant $\rho(I)=p+\sum_{i\in[p]}|I_i|\log\big(\tfrac{e(p-1)}{|I_i|}\big)$. 
Condition \eqref{(A2)} is fulfilled for the majorant $\rho(I)$, since, according 
to Remark \ref{rem5}, for sufficiently large $\nu>1$,
\begin{align*}
\sum_{I\in\mathcal{I}}e^{-\nu\rho(I)}
&\le e^{-\nu p}\sum_{s_1=0}^{p-1} e^{-(\nu-1)s_1}\ldots\sum_{s_p=0}^{p-1}
e^{-(\nu-1)s_p} \le\tfrac{e^{-\nu p}}{(1-e^{1-\nu})^p} \le C_\nu.
\end{align*}
Condition \eqref{(A3)} follows %also for the model \eqref{model_cov} with the sparsity structure 
in the same way as in Section \ref{subsec_banding}.

This means that we can derive results on estimation, posterior contraction 
and uncertainty quantification (and weak structure recovery) for the model \eqref{model_cov}, 
now with the sparsity structure, in the same way as for banding structure in Section \ref{subsec_banding}. 
We can readily formulate a theorem containing the local results for this structure:
it will take the form of Theorem \ref{meta_theorem_cov} with the oracle rate
$r^2(\mathrm{\Sigma}) =\min_{I\in\mathcal{I}} \big\{\|\mathrm{\Sigma}-\mathrm{P}_I\mathrm{\Sigma}\|^2
+\sigma^2_n \rho(I)\big\}$, where again $\sigma^2_n=\tfrac{C\log p}{n}$ with sufficiently large $C>0$.
To the best of our knowledge, there are no local results on estimation, posterior contraction 
rate and uncertainty quantification for the covariance matrix with sparsity structure. 
Also for the sparsity structure we have the same issue (described in Remark \ref{rem_A4_issue})
with Condition \eqref{cond_A4} as for the banding structure.

Finally, consider one scale covered by the oracle rate for the sparsity structure.

\subsection{Minimax results for the weak $\ell_q$-balls}
Recall the weak $\ell_q$ ball of radius $c$ in $\mathbb{R}^m$ containing elements with fast decaying 
ordered magnitudes of components,
\begin{align*}
B_q^m(c)=\{\zeta\in\mathbb{R}^m: |\zeta|_{(k)}^q\le ck^{-1},\, k\in[m]\},
\end{align*}
where $|\zeta|_{(k)}$ denotes the $k$th largest element in magnitude of the vector $\zeta$. 
For $0\le q<1$, define the class $\mathcal{G}_q(c_{n,p})$ of covariance matrices by
\[
\mathcal{G}_q(c_{n,p})=\{\mathrm{\Sigma}\in\mathcal{C}_{\varepsilon_0}: 
\Sigma_{-j}\in B_q^{p-1}(c_{n,p}), j\in[p]\},
\]
that is, each column $\Sigma_{-j}$ of $\mathrm{\Sigma} \in\mathcal{G}_q(c_{n,p})$ 
must be in a weak $\ell_q$-ball, $j\in[p]$. The minimax estimation rate  over $\mathcal{G}_q(c_{n,p})$ %this class 
is $r^2(\mathcal{G}_q(c_{n,p}))=pc_{n,p}\big(\frac{\log p}{n}\big)^{1-q/2}+\frac{p}{n}$; 
see \cite{Caietal:2012}. Recall $\sigma^2_n=\tfrac{C\log p}{n}$ and take 
$I^*=I^*(\mathrm{\Sigma})=(I_1^*,\ldots, I_p^*)$ such that $|I_i^*| = p^*\triangleq 
\lfloor c_{n,p}\big(\frac{\log p}{n}\big)^{-q/2}\rfloor $, $i\in[p]$,  to derive  
\begin{align*}
\sup_{\mathrm{\Sigma}\in \mathcal{G}_q(c_{n,p})}\!\!\! r^2(\mathrm{\Sigma})&\le
\sup_{\mathrm{\Sigma}\in \mathcal{G}_q(c_{n,p})}\!\!\! r^2(I^*,\mathrm{\Sigma}) %\\&
\le \sup_{\Sigma\in \mathcal{G}_q(c_{n,p})} \sum_{i\in[p]}\sum_{j\not\in I_i^*}
\Sigma_{ji}^2+ \sigma^2_n\big[p+pp^*\log \big(\tfrac{e (p-1)}{p^*}\big)\big]\\
&\lesssim p c_{n,p}^{2/q} \sum_{j>p^*} j^{-2/q}+\sigma^2_np\, c_{n,p}n^{q/2}(\log p)^{1-q/2}+\sigma^2_np \\&
\lesssim\big( pc_{n,p}\big(\tfrac{\log p}{n}\big)^{1-q/2}+\tfrac{p}{n}\big)  \log p .   
\end{align*}
This relation and the local results imply the global minimax results 
(up to the logarithmic factor $\log p$)
on estimation, posterior contraction and uncertainty quantification for
the model (79) for the scale $\mathcal{G}_q(c_{n,p})$.

\section{Matrix+noise with sparsity structure}
\label{sec_matrix_sparsity}

Suppose we observe a matrix $Y=(Y_{ij})\in \mathbb{R}^{n_1\times n_2}$:
\begin{align*}
Y_{ij}=\theta_{ij}+\sigma\xi_{ij}, \quad i\in[n_1], \;\;  j\in[n_2],
\end{align*}
where $\sigma>0$ is the known noise intensity,  $\xi_{ij} \overset{\rm ind}{\sim}\mathrm{N}(0,1)$,  $\theta=(\theta_{ij}) \in \mathbb{R}^{n_1\times n_2}$ is an unknown high-dimensional parameter 
of interest with at most $k_1$ nonzero rows and $k_2$ nonzero columns, %which are
not necessarily consecutive.  To the best of our knowledge, there are no local results on estimation, posterior contraction rate and  uncertainty quantification problems for this case of model/structure.

 The submatrix sparsity structure is modeled by the linear subspaces
\begin{align*}
\mathbb{L}_{I}=\big\{\text{vec}(x)\in\mathbb{R}^{n_1n_2}: x_{ij}=0\;\; \forall (i,j)\in \big((I_1^c\times[n_2])\cup ([n_1]\times I_2^c)\big)  \big\},
\end{align*}
where  $I=(I_1,I_2)\in\mathcal{I}=\{(I_1',I_2'): I_1'\subseteq[n_1],I_2'\subseteq[n_2]\}$ 
and $d_I=\dim(\mathbb{L}_{I})=|I_1||I_2|$. The structural slicing mapping 
is $s(I)=(|I_1|, |I_2|)$, so that $\mathcal{S}=([n_1]_0,[n_2]_0)$.
Compute $|\mathcal{I}_{s(I)}|=\prod_{i=1}^2\binom{n_i}{|I_i|}$, hence 
\[
\log |\mathcal{I}_{s(I)}|=\log \tbinom{n_1}{|I_1|}+\log \tbinom{n_2}{|I_2|}
\le\sum_{i\in[2]} |I_i|\log(\tfrac{en_i}{|I_i|}).
\]
Since $d_I=|I_1||I_2|$ and $d_I+\log |\mathcal{I}_{s(I)}| \le 
|I_1||I_2|+\sum_{i=1}^2|I_i|\log(\tfrac{en_i}{|I_i|})$, 
we take the majorant 
$\rho(I)=|I_1||I_2|+|I_1|\log(\tfrac{en_1}{|I_1|})+|I_2|\log(\tfrac{en_2}{|I_2|})$.

Conditions \eqref{cond_nonnormal} and \eqref{cond_A4} hold with $d_I=\dim(\mathbb{L}_I)$
in view of Remarks \ref{rem_cond_A1} and \ref{rem15}. 
Condition \eqref{(A2)} is fulfilled, since, according to Remark \ref{rem5}, for any $\nu>1$ 
\begin{align*}
 \sum_{I\in\mathcal{I}}e^{-\nu\rho(I)}&\le\sum_{|I_1|=0}^{n_1} 
 \big(\tfrac{en_1}{|I_1|}\big)^{-(\nu-1)|I_1|}\sum_{|I_2|=0}^{n_2}\big(\tfrac{en_2}{|I_2|}\big)^{-(\nu-1)|I_2|} %\\&
 \le\tfrac{1}{(1-e^{1-\nu})^2}= C_\nu.
\end{align*}
For any $I^0, I^1\in\mathcal{I}$ define $I'=I'(I^0, I^1)=(I_1^0\cup I_1^1, I_2^0\cup I_2^1)$. 
Then $(\mathbb{L}_{I^0}\cup\mathbb{L}_{I^1})\subseteq \mathbb{L}_{I'}$ and $\rho(I')\le
\rho(I^0)+\rho(I^1)$, which entails Condition \eqref{(A3)}.

As consequence of our general results, we obtain the local results of Corollary \ref{meta_theorem}
for this case with the local rate $r^2(\theta)=\min_{ I \in \mathcal{I}} \big\{\|\theta-\mathrm{P}_I \theta\|^2 +
\sigma^2\rho(I)\big\}$. 
In turn, by virtue of Corollary \ref{minimax_results} the local results will  imply global minimax 
adaptive results at once over all scales $\{\Theta_\beta,\, \beta\in\mathcal{B}\}$ covered by 
the oracle rate $r^2(\theta)$ (i.e., for which \eqref{oracle_minimax} holds).
Below we present the example of scales $\{\Theta_\beta,\, \beta\in\mathcal{B}\}$ covered 
by the oracle rate $r^2(\theta)$.

\subsection{Minimax results for $\mathcal{F}(k_1, k_2, n_1, n_2)$}  
Let $\mathcal{F}(k_1, k_2, n_1, n_2)$ be the collection of matrices $\theta=(\theta_{ij}) \in \mathbb{R}^{n_1\times n_2}$ with at most $k_1$ nonzero rows and $k_2$ nonzero columns, which are not necessarily consecutive. 
Classes $\mathcal{F}(k_1, k_2, n_1, n_2)$ were introduced in \cite{Ma&Wu:2015}. In our notation, 
$\mathcal{F}(k_1, k_2, n_1, n_2) = \cup_{I\in \mathcal{I}: |I_1|\le k_1,|I_2|\le k_2}\mathbb{L}_{I}$. 
As is shown in \cite{Ma&Wu:2015},
the minimax rate over  $\mathcal{F}(k_1, k_2, n_1, n_2)$ is 
\[
r^2(\mathcal{F}(k_1, k_2, n_1, n_2))\asymp 
\sigma^2\big(k_1k_2+k_1\log(\tfrac{en_1}{k_1})+k_2\log(\tfrac{en_2}{k_2}) \big).
\] 

On the other hand, for each 
$\theta \in\mathcal{F}(k_1, k_2, n_1, n_2)$ there exists $I_*\in\mathcal{I}$ such that 
$\theta\in\mathbb{L}_{I_*}$ and $|I_{*1}|\le k_1$ and $|I_{*2}|\le k_2$. 
Hence, $\mathrm{P}_{I_*}\theta =\theta$
and 
\begin{align*}
r^2(\theta)&\le r^2(I_*,\theta)=\sigma^2\rho(I_*)
=\sigma^2\big(|I_{*1}||I_{*2}|+|I_{*1}|\log(\tfrac{en_1}{|I_{*1}|})+|I_{*2}|\log(\tfrac{en_2}{|I_{*2}|})\big)\\
&\le \sigma^2\big(k_1k_2+k_1\log(\tfrac{en_1}{k_1})+k_2\log(\tfrac{en_2}{k_2})\big)
\asymp r^2(\mathcal{F}(k_1, k_2, n_1, n_2)).
\end{align*} 
We thus established the relation 
\eqref{oracle_minimax} for this scale,
and Corollary \ref{minimax_results} follows with the minimax 
rate $r^2(\mathcal{F}(k_1, k_2, n_1, n_2))$ defined above.

\section{Matrix+noise with clustering structure: biclustering model} 
\label{sec_biclustering} %\label{sec_biclustering_suppl}
Suppose we observe a matrix $Y=(Y_{ij})\in \mathbb{R}^{n_1\times n_2}$:
\begin{align*}
Y_{ij}=\theta_{ij}+\sigma\xi_{ij}, \quad i=1,\ldots, n_1, \quad  j=1,\ldots, n_2,
\end{align*}
where $\theta=(\theta_{ij}) \in \mathbb{R}^{n_1\times n_2}$ is an unknown high-dimensional parameter 
of interest with \emph{biclustering} structure (to be specified later), $\sigma>0$ is the known noise intensity, 
$\xi=(\xi_{ij}) \in \mathbb{R}^{n_1\times n_2}$ is a random matrix 
with $\mathbb{E}_\theta \xi_{ij} =0$. 

The essence of biclustering structure is to reduce dimensionality of a large matrix of parameters by 
simultaneous grouping of the rows and columns. For example, if the rows of $\theta$ correspond to 
objects and the columns to features, a biclustering structure means that only a few features are relevant 
for identifying a few groups of similar objects. 
There is a large literature on the biclustering model (some relevant references can be found in 
\cite{Belitser&Nurushev:2018}), especially on its particular case, 
the so called \emph{stochastic block model} (briefly discussed below) which is rather popular 
in the literature on networks as this model is widely used to model undirected network graphs. 
This case of model/structure was studied at length in 
\cite{Belitser&Nurushev:2018}, here we demonstrate that the results obtained in 
\cite{Belitser&Nurushev:2018} also follow from our general framework results.

Biclustering structure means that the rows and columns of the matrix $\theta=(\theta_{ij} )\in
\mathbb{R}^{n_1\times n_2}$ are  split into $k_1$ and $k_2$ clusters, respectively,  and the 
values $\theta_{ij}$ are the same for $i,j$  from the same clusters. 
Let us give the mathematical formalization of this idea. 
%is given by the following linear spaces.
For $(k_1, k_2) \in [n_1]\times[n_2]$, consider a mapping $z=(z_1,z_2):\, 
[n_1]\times [n_2] \mapsto[k_1]\times [k_2]$, where $z_1:\, [n_1] \mapsto [k_1]$ and 
$z_2:\, [n_2]\mapsto [k_2]$. Each mapping $z \in [k_1]^{[n_1]}\times [k_2]^{[n_2]}$ 
determines the pertinent partition $I=I(z)$ of  
the rows and columns of any matrix 
$(M_{ij})\in\mathbb{R}^{n_1\times n_2}$  into $k_1\times k_2$ blocks:
\begin{align*}
[n_1]\times[n_2]=z^{-1}([k_1]\times [k_2])=z_1^{-1}([k_1]) \times z_2^{-1}([k_2])
=\cup_{(I_i^1,I_j^2)\in I} (I_i^1,I_j^2),
\end{align*}
where $I^1_i=z_1^{-1}(i)$ and $I^2_j=z_2^{-1}(j)$.
The \emph{biclustering structure} is nothing else but just %corresponding to mapping $z$ is 
this partition $I=I(z)=(I^1,I^2)$, where $I^1=I^1(z_1)=(I^1_i: i\in [k_1])$ 
is the row partition and $I^2=I^2(z_2)=(I^2_j: j\in [k_2])$ is the column partition.
So, the collection of all mappings $\mathcal{Z}= \mathcal{Z}(n_1,n_2)
=\{(z_1,z_2) \in [k_1]^{[n_1]}\times [k_2]^{[n_2]}, \, (k_1,k_2) \in [n_1]\times [n_2]\}$ yields
the collection of all biclustering structures 
(which are all \emph{biclustered} partitions of $[n_1]\times [n_2]$): 
\[
\mathcal{I}=\mathcal{I}(n_1,n_2)=\big\{I(z),\, z \in 
[k_1]^{[n_1]}\times [k_2]^{[n_2]}, \, 
(k_1,k_2)\in[n_1]\times [n_2]\big\}.
\] 

A biclustering structure $I\in\mathcal{I}$ in terms of parameter $\theta$ is expressed by
imposing $\theta\in\mathbb{L}_I\subseteq \mathbb{R}^{n_1n_2}$, where the linear subspace 
$\mathbb{L}_I$ is defined as
\begin{align}
\label{L_I}
\mathbb{L}_{I}=\big\{x\in \mathbb{R}^{n_1 n_2}: x_{ij}=
x_{i'j'} \;\forall\; (i,j),  (i',j') \in (I_1,I_2),\; \forall\; (I_1,I_2) \in I \big\}.
\end{align}
Assume that $\mathcal{I}$ is ``cleaned up'' in the sense that 
$\mathbb{L}_{I} \not= \mathbb{L}_{I'}$ for all $I\not= I'$ (see Remark \ref{remark1}). 

The structural slicing mapping $s:\mathcal{I} \mapsto \mathcal{S}$ is defined as $s(I)=(s_1(I),s_2(I))
\in [n_1]\times [n_2]\triangleq\mathcal{S}$, where $(s_1(I),s_2(I))$ denotes the numbers of nonempty row 
and column blocks in the structure $I \in \mathcal{I}$. Then $d_I=\dim(\mathbb{L}_{I})=s_1(I)s_2(I)$.
%and $D_{s(I)} = \max_{I \in \mathcal{I}_s} d_I = d_I$.
%and $\dim(\mathbb{L}_{I}^\perp)=\text{dim}(\mathbb{R}^{n_1 n_2})-\dim(\mathbb{L}_{I})=n_1n_2-k_1k_2$. 

Let us propose a majorant $\rho(I)$ for the layer complexity 
$d_I+\log|\mathcal{I}_{s(I)}|=s_1(I)s_2(I)+\log|\mathcal{I}_{s(I)}|$. 
Clearly, $|\mathcal{I}_{s}|\le N(n_1,s_1)N(n_2,s_2)$, 
where $N(n,k)$ is the number of ways to put $n$ different objects into $k$ different 
boxes so that each box contains at least one object.
Notice that $S(n,k)=N(n,k)/k!=\frac{1}{k!}\sum_{j=0}^k(-1)^{k-j}\binom{k}{j}j^n$
is a Stirling number of the second kind.
To have a simple closed form expression for a majorant of the complexity, instead of 
$N(n_1,s_1)N(n_2,s_2)$ we can use its upper bound $s_1^{n_1}s_2^{n_2}$ (all the partitions of 
$[n_1]\times[n_2]$ into $s_1\times  s_2$ blocks, some of which are possibly empty).
However,  the bound $|\mathcal{I}_{s}|\le s_1^{n_1}s_2^{n_2}$ 
becomes too crude for some $s\in\mathcal{S}$.
In particular, this bound is too crude for the cases 
(i) $(s_1,s_2)\in\mathcal{S}_1=\{(s_1,s_2) \in [n_1]\times[n_2]: s_1<n_1, s_2=n_2\}$, 
(ii) $(s_1,s_2)\in\mathcal{S}_2=\{(s_1,s_2) \in [n_1]\times[n_2]: s_1=n_1, s_2<n_2\}$,
and (iii) $(s_1,s_2)\in\mathcal{S}_3=\{(n_1,n_2)\}$. 
Indeed, let $\mathrm{id}_m: [m] \mapsto [m]$ with $\mathrm{id}_m(s)=s$, $s\in[m]$, 
the identity mapping  of $[m]$.
Then it is easy to see that  $\mathbb{L}_{I(z_1,z_2)}=
\mathbb{L}_{I(z_1,\mathrm{id}_{n_2})} $
for all $z_2\in [n_2]^{[n_2]}$ and all $z_1\in [s_1]^{[n_1]}$, $s_1\in[n_1]$.
Similarly, $\mathbb{L}_{I(z_1,z_2)}=
\mathbb{L}_{I(\mathrm{id}_{n_1},z_2)} $
for all $z_1\in [n_1]^{[n_1]}$, $z_2\in [s_2]^{[n_2]}$, $s_2\in[n_2]$; and 
$\mathbb{L}_{I(z_1,z_2)}= \mathbb{L}_{I(\mathrm{id}_{n_1},\mathrm{id}_{n_2})}$
for all $z_1\in [n_1]^{n_1}, z_2\in [n_2]^{n_2}$.
Hence, $|\mathcal{I}_{s}| \le \big| [s_1]^{[n_2]}\big|\le s_1^{n_1}$ for $(s_1,s_2)\in\mathcal{S}_1$,
$|\mathcal{I}_{s}| \le s_2^{n_2}$ for $(s_1,s_2)\in\mathcal{S}_2$, 
and $|\mathcal{I}_{s}| \le 1$ for $(s_1,s_2)\in\mathcal{S}_3$.
Thus, we improve the bound $d_I+\log|\mathcal{I}_{s(I)}|\le s_1(I)s_2(I)+\log[s_1^{n_1}(I)s_2^{n_2}(I)]$ 
by proposing the following majorant $\rho(I)\ge d_I+\log|\mathcal{I}_{s(I)}| $ for the complexity $ d_I+\log|\mathcal{I}_{s(I)}|$ of the layer $\mathcal{I}_s$:
\begin{align}
\label{rho(k)}
 \rho(I) \triangleq  
\begin{cases}  
s_1(I)s_2(I)+ n_1\log s_1(I)+n_2\log s_2(I), \quad s_1(I)<n_1, \; s_2(I) <n_2, \\
 s_1(I)n_2+n_1 \log s_1(I),\quad s_1(I) < n_1, \; s_2(I)=n_2,\\
n_1s_2(I)+ n_2 \log s_2(I), \quad s_1(I)=n_1, \; s_2(I)<n_2,\\
n_1n_2, \quad s_1(I)=n_1,\; s_2(I)=n_2.
\end{cases}
\end{align}
This is an example of the so called \emph{elbow effect} mentioned in Remark \ref{family_cover}.

In case $\xi_i \overset{\rm ind} {\sim}\mathrm{N}(0,1)$, Conditions \eqref{cond_nonnormal} 
and \eqref{cond_A4} hold with $d_I =\dim(\mathbb{L}_I)$ 
in view of Remarks \ref{rem_cond_A1} and \ref{rem15}.  
Let us show that Condition \eqref{cond_nonnormal} is also fulfilled 
in case $Y_{ij} \overset{\rm ind}{\sim} \text{Bernoulli}(\theta_{ij})$, 
which is typically used for modeling indirect network graphs. 
Indeed, we have 
$\xi_{ij}\in\{1-\theta_{ij}, -\theta_{ij}\}\subseteq [-1,1]$ and 
$\mathbb{E}_\theta\xi_{ij}=0$, $(i,j)\in[n_1]\times[n_2]$. Note that in this case 
the error distribution depends on $\theta$.  
We represent the projection $\mathrm{P}_I=\mathrm{B}\mathrm{B}^T$, where 
$\mathrm{B}$ is the $(n_1n_2 \times k_1k_2)$-matrix whose columns $(b_{I_1I_2},(I_1,I_2) \in I)$ 
form an orthonormal basis of $\mathbb{L}_I$. Then
$\|\mathrm{P}_I \xi\|^2=\|\mathrm{B}^T \xi\|^2=\|\eta\|^2$, with $\eta=(\eta_{I_1I_2},\, (I_1,I_2)\in I)$, 
$\eta_{I_1I_2}=b_{I_1I_2}^T \xi$. We choose the following orthogonal basis of $\mathbb{L}_I$:
$b_{I_1I_2}=\big((|I_1||I_2|)^{-1/2}1\{(i,j)\in(I_1,I_2)\}, \, (i,j)\in [n_1]\times [n_2]\big)$, $(I_1,I_2) \in I$,
so that $\eta_{I_1I_2}=\frac{1}{\sqrt{|I_1||I_2|}}\sum_{(i,j)\in(I_1,I_2)}\xi_{ij}$.
Hoeffding's inequality implies that for any $t\ge 0$ 
\begin{align*}
\mathbb{P}_\theta(|\eta_{I_1I_2}|\ge t)
\le 2e^{-t^2/2}, \quad \text{for all} \;\;  (I_1,I_2)\in I.  
\end{align*}
Using this, we obtain for any $0<b<1/2$
\begin{align*}
\mathbb{E}_\theta e^{b\eta_{I_1I_2}^2}&=1+\int_{1}^\infty \mathbb{P}_\theta
\big(e^{b\eta_{I_1I_2}^2}\ge t\big)dt
\le 1+ 2 \int_{1}^\infty e^{-(\log t)/(2b)}dt=1+\tfrac{4b}{1-2b}.
\end{align*}
By taking  $b_0=\frac{e-1}{2(1+e)}$, we derive 
\[
\mathbb{E}_\theta\exp\{b_0\|\mathrm{P}_I \xi\|^2\}
=\mathbb{E}_\theta\exp\{b_0\|\eta\|^2\}\le
\big(1+\tfrac{4b_0}{1-2b_0}\big)^{|I_1||I_2|}=e^{|I_1||I_2|}=e^{d_I},
\] 
which is Condition \eqref{cond_nonnormal} with the constant $\alpha=b_0=\frac{e-1}{2(1+e)}$. 
Of course, the above argument applies (with minor adjustments) 
to any independent zero mean bounded errors $\xi_{ij}\in [-c,c]$ for some $c>0$.

Let us  verify Condition \eqref{(A2)}: for any $\nu\ge 1$,
\begin{align*}
\sum_{I\in\mathcal{I}}e^{-\nu\rho(I)}&\le \sum_{(s_1,s_2)\in[n_1]\times[n_2]}
e^{-\nu s_1s_2}=(e^{\nu}+e^{-\nu}-2)^{-1}= C_\nu.
\end{align*}

Thus, the properties (i)-(iv) of Corollary \ref{meta_theorem} follow 
for the biclustering model with the $\xi_i$'s that are independent and 
either normal or binomial, in fact, for any $\xi$ satisfying Condition \eqref{cond_nonnormal}.

One can also check Condition \eqref{(A3)}, so that the coverage 
property (v) of Corollary \ref{meta_theorem} holds under EBR as well. 
However, the peculiarity of the biclustering structure is that the size and coverage 
claims (vi)--(vii) for the confidence ball 
$B(\hat{\theta},\tilde{R}_M)$ are stronger and more useful in this case
than the corresponding claims (iv)--(v) for the confidence ball 
$B(\hat{\theta},\hat{R}_M)$. 

Indeed, the coverage property (v) holds uniformly only under the EBR, whereas the coverage 
property (vii) is uniform over the entire space $\Theta=\mathbb{R}^{n_1 \times n_2}$.
So, basically the deceptiveness issue is not present in the coverage property (vii) for the confidence ball $B(\hat{\theta},\tilde{R}_M)$, it appears only marginally in the size relation (vi) of 
Corollary \ref{meta_theorem}. Indeed, the size $\tilde{R}_M$ of the ball is of the oracle rate order 
uniformly in $\theta\in\Theta\backslash \tilde{\Theta}=\mathbb{R}^{n_1\times n_2}\backslash\tilde{\Theta}$,
where $\tilde{\Theta}$ is defined by \eqref{tilde_Theta}. By the definition of  $\tilde{\Theta}$, 
$r^2(\theta)\ge c\sigma^2\sqrt{n_1n_2}$ for $\theta\in\Theta\backslash\tilde{\Theta}$. For the biclustering 
model, we can take $c=\log 2$ and $\tilde{\Theta}$ can be written as 
$\tilde{\Theta}=\{\theta\in\mathbb{R}^{n_1\times n_2}:\min\{s_{o1}(\theta),s_{o2}(\theta)\}=1\}$ with 
$(s_{o1}(\theta), s_{o2}(\theta))=s(I_o(\theta))$, where the oracle $I_o(\theta)$ is defined by \eqref{oracle}. 
Hence, for the biclustering model,  $\tilde{\Theta}$ is indeed a ``thin'' subset of 
$\mathbb{R}^{n_1\times n_2}$ consisting of \emph{highly structured parameters}, whose oracle 
number of either row or block columns is 1. As we have already discussed at the end of Section \ref{sec_without_EBR}, this means that, modulo highly structured parameters, there is no deceptiveness phenomenon  in the biclustering model.

%But in some other examples the set $\tilde{\Theta}$ is 
%a very ``thin'' subset of  $\mathbb{R}^N$ consisting of ``highly structured parameters'';
%for instance, in biclustering and stochastic block models (see Section \ref{sec_biclustering}).
%In those cases, this minor defect of the radial rate $R(\theta)$ is the only sacrifice in the size 
%relation of the optimality framework \eqref{defconfball}.
%We can also see this as the optimality framework \eqref{defconfball} with 
%$\Theta_0=\mathcal{Y}$, $\Theta_1 =\mathcal{Y}\backslash \tilde{\Theta}$ and 
%the effective radial rate $r(\theta)$, then a minor sacrifice occurs in the set 
%$\Theta_1 =\mathcal{Y}\backslash \tilde{\Theta}$ which is 
%``smaller'' than the entire space $\mathcal{Y}$. 

%This is not a problem for the ``majority'' of $\theta$'s 
%as this extra term does not increase the order of the radial rate:
%$N^{1/4} \le c r(\theta)$ for all $\theta \in \mathbb{R}^N\backslash\tilde{\Theta}$ 
%for some ``thin'' set $\tilde{\Theta}$. The set $\tilde{\Theta}$ can be informally described as a set 
%of ``highly structured'' parameters: namely, when either the (oracle) number of row or column 
%blocks of the underlying parameter $\theta$ is 1.

Consider an example of scale $\{\Theta_\beta,\, \beta\in\mathcal{B}\}$ 
covered by the local rate $r^2(\theta)$. %in the sense of \eqref{oracle_minimax}.

\subsection{Minimax results for the biclustering model}
In \cite{Gao&Lu&Zhou:2015}, classes $\Theta^{\rm asym}_{k_1k_2}$ are introduced 
(and classes $\Theta_{k_1k_2}(M)$ 
from \cite{Gao&Lu&Ma&Zhou:2016}). In  our notation, 
$\Theta^{\rm asym}_{s_1s_2}=\cup_{I \in \mathcal{I}_{s}} \Theta_{I}$, where 
$s=(s_1, s_2)\in [n_1]\times [n_2]\triangleq \mathcal{S}$,  
$\Theta_{I}\triangleq\mathbb{L}_{I} \cap[0,1]^{n_1\times n_2}$ and  
$\mathbb{L}_{I}$ is defined by \eqref{L_I}.
%(in \cite{Gao&Lu&Ma&Zhou:2016}, the class $\Theta_{k_1k_2}(M)=\mathbb{L}_{I} \cap[-M,M]^{n_1\times n_2}$), 
So, the family of classes $\Theta^{\rm asym}_{s_1s_2}$ is nothing else but 
the scale $\{\Theta_{s}, \, s\in\mathcal{S}\}$. 
The minimax rate $r^2(\Theta_s)\triangleq s_1s_2 +n_1\log s_1+n_2\log s_2$ 
over $\Theta_{s}$ is derived in \cite{Gao&Lu&Zhou:2015}, 
under the assumption $\log s_1 \asymp \log s_2$. 
It is easy to see that the oracle rate $r^2(\theta)$ covers the scale 
 $\{\Theta_{s}, \, s\in\mathcal{S}\}$ in the sense of \eqref{oracle_minimax}.
%for each $s\in\mathcal{S}$ the minimax rate over $\Theta_s$ is bigger than 
%the oracle rate $r^2(\theta)$ for any $\theta \in \Theta_s$. 
Indeed, if 
$\theta \in\Theta_s$, then $\theta \in \mathbb{L}_{I'}$ for some 
$I'\in \mathcal{I}_{s}$, so that $\mathrm{P}_{I'} \theta =\theta$ 
and hence 
\begin{align}
\label{oracle_less_minimax}
r^2(\theta) \le r^2(I',\theta)=\rho(I')=\rho(I) \le 
r^2(\Theta_s), \quad \theta \in \Theta_s.
\end{align}
Corollary \ref{minimax_results} follows for this case with the minimax rate $r^2(\Theta_s)$ 
defined above.
%This and Corollary \ref{meta_theorem} imply Corollary  \ref{minimax_results}.

\begin{remark}
From \eqref{oracle_less_minimax}, we have that $r^2(\theta) \le s_1s_2+n_1\log s_1+n_2\log s_2$ 
for each $\theta \in \Theta_s$.
Next, for any $I \in\mathcal{I}_{s}$ with $s=(s_1,s_2)$
and any $\mathbb{L}_{I}$, there exist $I'=I'(I)$ and 
$\mathbb{L}_{I'}$ such that 
$\mathbb{L}_{I} \subseteq \mathbb{L}_{I'}$ where 
$I' \in \mathcal{I}_{s'}$ with $s'=(s_1,n_2)$.
Then for any $\theta \in\Theta_s$, $\theta \in \mathbb{L}_{I}
\subseteq \mathbb{L}_{I'(I)}$ for some 
$I' \in \mathcal{I}_{s'}$ with $s'=(s_1,n_2)$, implying
$\mathrm{P}_{I'} \theta =\theta$.  In view of \eqref{rho(k)}, we obtain 
that $r^2(\theta) \le r^2(I',\theta)=\rho(I')
=s_1n_2+n_1\log s_1$  for all $\theta \in \Theta_s$. 
Similarly, we derive that $r^2(\theta) \le n_1 s_2 + n_2 \log s_2$ and  
$r^2(\theta) \le n_1 n_2$ for all $\theta \in \Theta_s$.
	Thus,  instead of \eqref{oracle_less_minimax}, we established
the following stronger bound  for  any  $\theta \in \Theta_s$
\[
r^2(\theta) \le  
\min\{r^2(\Theta_s), s_1n_2 + n_1 \log s_1, 
n_1 s_2 + n_2 \log s_2, n_1n_2\}\triangleq \bar{r}^2(\Theta_s).
\]
Notice that for some $s\in\mathcal{S}$, the quantity $\bar{r}^2(\Theta_s)$ can be less 
than the minimax rate $r^2(\Theta_s)=s_1s_2 +n_1\log s_1+n_2\log s_2$.
Recall however that the minimax rate $r^2(\Theta_s)$ is claimed in \cite{Gao&Lu&Zhou:2015} 
only under the assumption $\log s_1 \asymp \log s_2$, and, in this case, indeed 
$r^2(\Theta_s)\asymp\bar{r}^2(\Theta_s)$. 
In general, the minimax rate over $\Theta_s$ for arbitrary $s\in\mathcal{S}$ cannot be bigger 
than $\bar{r}^2(\Theta_s)$, we conjecture that it is $\bar{r}^2(\Theta_s)$ 
for all $s\in\mathcal{S}$.
%This  upper bound can be less than the minimax rate $r^2(\Theta_s)=
%s_1s_2 +n_1\log s_1+n_2\log s_2$ over $\Theta_{s}$ for some $s\in\mathcal{S}$. 
\end{remark}

\begin{remark}
\label{rem_binomial_case}
In view of Remark \ref{rem15}, 
Condition \eqref{cond_A4} is always fulfilled whenever $\xi_i \overset{\rm ind}{\sim}\mathrm{N}(0,1)$. 
However, for the biclustering model (and the stochastic block model, described below), a more appropriate distribution for the observations is binomial, i.e., $Y_{ij}\overset{\rm ind}{\sim}\text{Bernoulli}(\theta_{ij})$ 
and  $Y'_{ij}\overset{\rm ind}{\sim}\text{Bernoulli}(\theta_{ij})$.
This case is important in relation to network modeling.
Also in this case, Condition \eqref{cond_A4} holds in view of Remark \ref{rem_bin_model} 
if we have a second sample $Y'$.
%Unfortunately, we were unable to establish Condition \eqref{cond_A4} for the 
%binomial case (hence, unable to establish claims (vi)--(vii) of Corollary \ref{meta_theorem}). 
%Verification of Condition \eqref{cond_A4} for the binomial observations
%essentially boils down to the problem of estimating the functional 
%$F(\theta) = \sum_{i,j} \theta_{ij}^2$ with the rate $\sqrt{N}=\sqrt{n_1n_2}$. 
%It is not known to us whether it is possible to construct such an estimator (possibly exploiting 
%the biclustering structure of $\theta$), which is an interesting and challenging problem on its own.
\end{remark}

\subsection{Stochastic block model}

%A particular case of biclustering structure is the \emph{stochastic block model} (SBM)
%which is a popular model for the network analysis. In SBM $n_1=n_2=n$, row clusters coincide with 
%the column clusters, indices $i,j$ denote the vertices in an undirected network graph of $n$ vertices, 
%clusters have the meaning of communities, and the observations $X_{ij}\overset{\rm ind}{\sim} 
%\text{Bernoulli}(\theta_{ij})$ for $i>j$, independent Bernoulli random variables with 
%success probabilities $\theta_{ij}\in [0,1]$. This is the model \eqref{model} 
%with $\xi_{ij}\in\{1-\theta_{ij},-\theta_{ij}\}$ and $\mathbb{P}(\xi_{ij}=1-\theta_{ij})=\theta_{ij}$.
%Here $X_{ij}$ stands for the presence or absence of an edge between vertices $i$ and $j$. 
%Typically, to model undirected network graphs, additional structure is imposed: $X_{ii}=\theta_{ii}=0$, 
%$X_{ij}=X_{ji}$ and $\theta_{ij}=\theta_{ji}$. Thus, we observe the adjacency matrix $X= (X_{ij})$
%and want to infer on the symmetric block constant probability matrix $\theta=(\theta_{ij})$. 
%The SBM is discussed in more detail below.

Here we briefly discuss a particular case of biclustering model, the 
\emph{stochastic block model} (SBM) which is used in the literature on networks to model undirected 
network graphs. Oracle estimation and posterior contraction rate results for stochastic block model 
were recently derived in \cite{Gao&vanderVaart&Zhou:2015, Klopp&Tsybakov&Verzelen&:2017}.
Precisely, to get the SBM from the biclustering model, we assume additionally
$s_1=s_2=s$, $n_1=n_2=n$, $z_1=z_2=z$.
For a mapping $z\in [s]^{[n]}$, the pertinent row partition in the SBM is
$I =I(z)=(z^{-1}(i), i \in [s])$, which is the same as the column partition.
%This in turn determines the linear spaces $\mathbb{L}_I$ in \eqref{L_I}.

In the binomial case $Y_{ij}\overset{\rm ind}{\sim}\text{Bernoulli}(\theta_{ij})$, 
the observations $Y_{ij}$ can be associated with network data. In this case $Y_{ij}$ stands for 
the presence or absence of an edge between vertices $i$ and $j$ in the network interpretation. 
To model undirected network graphs, some conditions (called \emph{network conditions}) 
are then additionally assumed: the ``no self-loop'' condition $Y_{ii}=\theta_{ii}=0$ 
and symmetry condition $Y_{ij}=Y_{ji}$ and $\theta_{ij}=\theta_{ji}$.
Denote by $\Theta_{\rm net}$ the parameters $\theta\in\mathbb{R}^{n_1\times n_2}$ 
satisfying  these additional network conditions.

All the quantities, conditions and claims specialize to the SBM by setting 
$s_1=s_2=s$, $n_1=n_2=n$, $z_1=z_2=z$ in all the above formulas for the biclustering model.
The linear subspaces $\mathbb{L}_{I}$ defined by \eqref{L_I} will get adjusted since 
$z_1=z_2$, the family $\mathcal{I}_s$ can be associated with the collection of 
all possible partitions of $[n]$ into $s$ blocks, parametrized by mappings $z\in [s]^{[n]}$. 
$|\mathcal{I}_s|\le s^n$, $s\in \mathcal{S}\triangleq[n]$. 
%For each $I \in\mathcal{I}_s$, $\dim(\mathbb{L}_I)=s^2$.
The structural slicing mapping $s(I)$ is the number of blocks in the partition $I$. 
Notice that under additional network conditions $cs^2(I) \le \dim(\mathbb{L}_I)\le s^2(I)$, 
so that we can use $s^2(I)$ (instead of the true 
$d_I=\dim(\mathbb{L}_I)$) in the complexity part of the local rate as it is still of the same order, 
although some constants can be improved because of this extra network structure. 
We have $d_I=\dim(\mathbb{L}_I)\le s^2(I)$,  $\log|\mathcal{I}_s|\le n\log s$, 
and we take $\rho(I)=s^2(I)+n\log s(I)$. 
Conditions \eqref{cond_nonnormal}--\eqref{cond_A4} are fulfilled in the same way as for the 
biclustering  model, leading to Corollary \ref{meta_theorem}.
As to the binomial case, see Remark \ref{rem_binomial_case}.

Consider a couple of examples of scales $\{\Theta_\beta,\, \beta\in\mathcal{B}\}$ covered 
by the local rate $r^2(\theta)$.

\subsection{Minimax results for the stochastic block model}
We consider the SBM.
In \cite{Gao&Lu&Zhou:2015} (cf.\ \cite{Klopp&Tsybakov&Verzelen&:2017}), 
classes $\Theta_k$ were introduced for the SBM. In our notation, 
$\Theta_s = \cup_{I\in \mathcal{I}_s} \Theta_{I} $, where
$s=k$, 
$\Theta_{I} = \mathbb{L}_{I} \cap \Theta_{\rm net} \cap [0,1]^{n^2}$, $I\in  \mathcal{I}_s$, 
$s\in \mathcal{S}$. So,  we have the scale $\{\Theta_s, \,s\in \mathcal{S}\}$ and the adaptive minimax
results over this scale follow from the local results given by Corollary \ref{meta_theorem}. 
Indeed, as is shown in \cite{Gao&Lu&Zhou:2015},
the minimax rate over $\Theta_s$ in the SBM is 
$r^2(\Theta_s)=\inf_{\hat{\theta}}\sup_{\theta \in \Theta_s} \mathbb{E}_\theta\|\hat{\theta} 
- \theta\|^2 \asymp s^2 +n\log s=\rho(I)$, $I\in  \mathcal{I}_s$. On the other hand, for each 
$\theta \in\Theta_s$  there exists $I\in\mathcal{I}_s$ such that 
$\theta\in\mathbb{L}_{I}$. Hence, $\mathrm{P}_{I}\theta =\theta$
and $r^2(\theta)\le r^2(I,\theta)=\rho(I)\asymp r^2(\Theta_s)$. 
This implies  Corollary  \ref{minimax_results} for this scale.

\begin{remark}
\label{rem_dec_sbm}
As to the deceptiveness phenomenon in the SBM, for the confidence ball 
$B(\hat{\theta},\tilde{R}_M)$ we again have 
the coverage property uniformly over the whole scale $\{\Theta_s,s \in\mathcal{S}\}$, 
whereas the size property with the optimal radial rate holds 
over all classes $\{\Theta_s, s=2,\ldots,n\}$, but one: $\Theta_1$.
Indeed, the class $\Theta_1$ consists of highly structured parameters 
$\theta\in\mathbb{R}^{n^2}$, whose coordinates are all equal. The case 
$\theta\in\Theta_1$ reduces to just one-dimensional signal+noise model with $N=n^2$ 
observations. Since the effective radial rate $g_M(\theta,N)$ for the confidence ball 
$B(\hat{\theta},\tilde{R}_M)$ is always at least of the order 
$n\sigma^2 \gg \sigma^2=r^2(\Theta_1)$, we could not attain the optimal rate $r^2(\Theta_1)$ 
in the size relation only for the highly structured parameters $\theta\in \Theta_1$. 
%We conjecture that no method van do this.
\end{remark}

\subsection{Minimax results for the graphon classes} 
Consider the SBM. It is also possible to derive the global minimax results  
for the function class of graphons as consequence of our local results. We use the same notation 
as in \cite{Gao&Lu&Zhou:2015}. Consider a random graph with adjacency matrix 
$\{Y_{ij}\}\in\{0,1\}^{n\times n}$. Assume again the network conditions: $Y_{ii}=\theta_{ii}=0$, 
$Y_{ij}=Y_{ji}$, $\theta_{ij}=\theta_{ji}$. For any $i>j$, $Y_{ij}$ is sampled as follows: 
\begin{align*}
(\xi_1,\ldots, \xi_n)\sim\mathrm{P}_{\xi}, \quad Y_{ij}|(\xi_i,\xi_j)
\overset{\rm ind}{\sim}\text{Bernoulli}(\theta_{ij}), \quad \theta_{ij}=f(\xi_i,\xi_j).
\end{align*}
The function $f$ on $[0,1]^2$, which is assumed to be symmetric, is called \emph{graphon}. 
Introduce the derivative operator $\nabla_{jk}f(x,y) =\frac{\partial^{j+k}}{\partial x^j \partial y^k}f(x, y)$, 
with the convention $\nabla_{00}f(x,y) = f(x,y)$. For $\beta>0$, the H\"older norm is defined by
\begin{align*}
\|f\|_{\mathcal{H}_\beta}=\max_{j+k\le\lfloor\beta\rfloor }
\sup_{x,y}|\nabla_{jk}f(x,y)|+\max_{j+k=\lfloor\beta\rfloor }\sup_{(x,y)
\neq (x',y')}\tfrac{|\nabla_{jk}f(x,y)-\nabla_{jk}f(x',y')|}{\|(x-x',y-y')\|^{\beta-\lfloor\beta\rfloor}}.
\end{align*}
For $\beta,Q>0$, the H\"older graphon class is  
\begin{align*}
\mathcal{F}_\beta=\mathcal{F}_\beta(Q)=\{f: \|f\|_{\mathcal{H}_\beta}\le 
Q,\, f(x,y)=f(y,x),\;0\le f(x,y)\le 1\; \text{for}\; x\ge y\}.
\end{align*}
Recall that $\theta_{ij}=f(\xi_i,\xi_j)$. Slightly abusing notation, we will 
write $\theta \in \mathcal{F}_\beta(Q)$ if $f \in \mathcal{F}_\beta(Q)$.

The next proposition  is Lemma 2.1 from \cite{Gao&Lu&Zhou:2015}, which we give here 
(in our notation) for completeness. The proof can be found in \cite{Gao&Lu&Zhou:2015}.
\begin{proposition} 
\label{prop_graphon}
For any $\theta\in\mathcal{F}_\beta(Q)$, $s_0\in\mathcal{S}$, there exists a partition 
$I_0=I_0(\theta,s_0)\in\mathcal{I}_{s_0}$ such that, for some universal constant $\bar{C}_1>0$,
\begin{align*}
\|\theta-\mathrm{P}_{I_0}\theta\|^2 
\le \bar{C}_1Q^2n^2 s_0^{-2\min\{\beta,1\}}.
\end{align*}
\end{proposition}

By taking $s_0=\lfloor n^{1/(\min\{\beta,1\}+1)}\rfloor+1$ 
and $I_0\in\mathcal{I}_{s_0}$ from Proposition \ref{prop_graphon}, 
we obtain 
\begin{align*}
&\sup_{\theta\in\mathcal{F}_\beta(Q)}r^2(\theta)=\sup_{\theta\in\mathcal{F}_\alpha(Q)}
\big\{\|\theta-\mathrm{P}_{I_{o}}\theta\|^2+s^2(I_o)+n\log s(I_o)\big\}
\\
&\le \sup_{\theta\in\mathcal{F}_\beta(Q)}\|\theta-\mathrm{P}_{I_{0}}\theta\|^2+s^2(I_0)+
n \log s(I_0)\\
&\le\bar{C}_1Q^2n^2 s_0^{-2\min\{\beta,1\}}+s_0^2+n\log s_0
\lesssim n^{2-2\beta/(\beta+1)}+n\log n \\
&\asymp n^{2/(\beta+1)}+n\log n.
\end{align*} 
Corollary  \ref{minimax_results} follows for the scale $\{\mathcal{F}_\beta, \, \beta>0\}$ 
with the minimax rate $r^2(\mathcal{F}_\beta)\asymp 
n^{2/(\beta+1)}+n\log n$. The second claim of Corollary  \ref{minimax_results} 
recovers the same minimax estimation rate  as in \cite{Gao&Lu&Zhou:2015}
and \cite{Klopp&Tsybakov&Verzelen&:2017}.

\section{Matrix linear regression with group sparsity}
\label{sec_reg_group_sparsity}

Assume now that the unknown regression vectors $\beta^1,\ldots , \beta^m \in \mathbb{R}^p$ 
in the general regression model \eqref{matrix_regression} share the same support. 
Note that the model considered in Section \ref{sec_lin_regr_sparsity} is a special case of linear 
regression with group sparsity with $m=1$. Local results  for linear regression with group sparsity 
were derived in \cite{Lounici&Pontil&vandeGeer&Tsybakov:2011}, %\cite{Ma&Wu:2015}, 
and posterior contraction rate results in \cite{Gao&vanderVaart&Zhou:2015}.  
The group sparsity structure is modeled by the linear spaces 
\begin{align*}
\mathbb{L}_I=
\big\{ \text{vec}(\mathrm{X}_I^1 x^1_I,\ldots ,\mathrm{X}_I^m x^m_I)\in \mathbb{R}^{nm}:\,
x^j_I\in\mathbb{R}^{|I|},\,j\in[m]\big\},
%\Big\{(\sum_{i\in I} X_i^1\beta_i^1,\ldots,\sum_{i\in I}X_i^m\beta_i^m):\,\beta_i^j\in\mathbb{R},\; i\in I,\; j\in[m] \Big\},
\;\; I\in\mathcal{I},
\end{align*}
where $\mathcal{I}=\mathcal{I}_1 \cup\{I_p\}$, with $I_p=[p]$, $\mathcal{I}_1=
\{I\subseteq [p]: m|I|+|I|\log(ep/|I|)\le r\}$, $r=\sum_{i=1}^m\rank(\mathrm{X}^i)$. 
Clearly, $|\mathcal{I}|\le 2^p$ and $d_I=\dim(\mathbb{L}_I)\le m|I|$ for $I\in\mathcal{I}_1$ 
and $d_{I_p}=r$. In this case, $\theta=\mathrm{X}\beta$ with $\beta=
(\beta^1,\ldots, \beta^m)\in\mathbb{R}^{mp}$, the structural slicing mapping is 
$s(I)=|I|\in\mathcal{S}\triangleq [p]_0$. Further, we have  $|\mathcal{I}_{s(I)}|=\binom{p}{|I|}$ for 
$I \in \mathcal{I}_1$ and $|\mathcal{I}_{s(I_p)}|=1$, 
hence $\log |\mathcal{I}_{s(I)}|=\log \binom {p}{|I|}\le |I|\log(ep/|I|)$ for $I \in \mathcal{I}_1$ and 
$\log |\mathcal{I}_{s(I_p)}|=0$. 
Since $d_I +\log |\mathcal{I}_{s(I)}|\le m|I|+|I|\log(ep/|I|)$ for $I \in \mathcal{I}_1$ and
$d_{I_p} +\log |\mathcal{I}_{s(I_p)}| =r$, we take the majorant 
\[
\rho(I)=\big(m|I|+|I|\log(ep/|I|)\big) 1\{I\in\mathcal{I}_1\} + r 1\{I=I_p\}, \quad I \in \mathcal{I}.
\]
Notice the elbow effect in the majorant that emerges here for the same reason as in 
Section \ref{sec_lin_regr_sparsity}.

Conditions \eqref{(A2)} and \eqref{(A3)} are fulfilled in the same way as for the model 
in Section \ref{sec_lin_regr_sparsity}. As consequence of our general results, we obtain Corollary \ref{meta_theorem} for this case with the local rate 
\[
r^2(\beta)=\min_{ I \in \mathcal{I}} r^2(I,\beta)=\min_{ I \in \mathcal{I}}
\big\{\|(\mathrm{I}-\mathrm{P}_I)\mathrm{X}\beta\|^2+\sigma^2\rho(I)\big\}.
\]
\begin{remark}
We can redefine the structural slicing mapping as $s(I)=\dim(\mathbb{L}_I)$, and 
the bound $\log |\mathcal{I}_{s}|=\log \binom {p}{s}\le s\log(ep/s)$ would still be valid.
Notice further that we can slightly improve the above oracle rate by
using the exact quantity $d_I=\dim(\mathbb{L}_I)$ instead of its upper bound $m|I|$ 
in the expression for the complexity $\rho(I)$, which would make the oracle rate 
$r^2(\beta)$ slightly smaller. 
\end{remark}

\subsection{Minimax results for group sparsity}
One can formulate minimax results for appropriate scales. For example, introduce 
the scale of classes
\[
\ell^m_0[s] =\big\{\text{vec}(\beta^1,\ldots , \beta^m) \in \mathbb{R}^{pm}:
\, I^*(\beta^i)= I^*(\beta^j),\, |I^*(\beta^i)|\le s\; \forall i,j \in [m]\big\},
\]
where $I^*(\beta)=\{i\in[p]: \beta_i \not= 0\}$. 
%We can easily establish that 
%the minimax rate over the class is bounded from above by the multiple of 
%$\rho(I) = ms+s\log(ep/s)$. We conjecture (we did not find the proof of the lower 
%bound in the literature) that 
The minimax rate over this class is established in \cite{Lounici&Pontil&vandeGeer&Tsybakov:2011}
(under some conditions):
\[
r^2(\ell^m_0[s] )\triangleq 
\inf_{\hat{\beta}} \sup_{\beta \in \ell^m_0[s]} \mathbb{E}_\beta
\|\mathrm{X}\hat{\beta}-\mathrm{X}\beta\|^2 \asymp  \sigma^2 \big[ms+s\log(ep/s)\big].
\]
Then we can easily show that the oracle rate implies this global rate since 
\begin{align*}
r^2(\beta) &\le r^2(I^*(\beta),\beta) \wedge r^2(I_p,\beta)\le
\sigma^2 \big[\big(m|I^*(\beta)|+ |I^*(\beta)|\log(\tfrac{ep}{|I^*(\beta)|}) \big) \wedge r\big] \\
&\le \sigma^2 \big[ms+s\log(ep/s)\big]\asymp r^2(\ell^m_0[s])
\qquad \text{for all} \quad \beta \in \ell^m_0[s].
\end{align*}

\section{Matrix linear regression with group clustering (multi-task learning)}  
\label{sec_group_clustering}

Assume now a clustering structure shared by $m$ unknown regression vectors 
$\beta^1,\ldots , \beta^m\in\mathbb{R}^p$. That is, there is some  mapping 
$z:\, [m]\mapsto[k]$ such that $\beta^j=\beta^{z(j)}$, $j\in[m]$. 
Let the design matrix $\mathrm{X}=\text{diag}\{\mathrm{X}^1,\ldots , \mathrm{X}^m\}$ in \eqref{matrix_regression} be such that 
$\mathrm{X}^1=\ldots=\mathrm{X}^m=\bar{\mathrm{X}}$, with
$\det({\bar{\mathrm{X}}^T\bar{\mathrm{X}}})>0$. Full column rankness of the  
$(n\times p)$-matrix $\bar{\mathrm{X}}$ implies $p\le n$. 
%When the design matrix $\mathrm{X}=\mathrm{I}$, this leads to the ordinary clustering problem. 
Each mapping $z \in [k]^{[m]}$ determines (uniquely) the pertinent partition 
$I=I(z)=(I_i, i\in[k])$ of the vectors $\beta^1,\ldots , \beta^m$ into $k$ groups 
$I_i=I_i(z)=z^{-1}(i)\subseteq[m]$, $i\in[k]$, such that $\cup_{i\in [k]} I_i=[m]=z^{-1}([k])$. 
Thus, the collection of all mappings $\mathcal{Z}= \mathcal{Z}(m)=\{z \in [k]^{[m]}, \, k \in [m]\}$ 
yields the collection of all clustering partitions of $[m]$: $\bar{\mathcal{I}}=\bar{\mathcal{I}}(m)
=\big\{I(z),\, z \in [k]^{[m]}, \, k\in[m]\big\}$. Some local posterior contraction rate results for this model 
are claimed in \cite{Gao&vanderVaart&Zhou:2015}, where this model is called by \emph{multi-task learning}.
We will call this model rather by \emph{linear regression with group clustering}. 
To the best of our knowledge, there are no adaptive minimax results on estimation and  uncertainty 
quantification problems for this model.

In this model, the structures $I$ are going to be certain partitions from $\bar{\mathcal{I}}$.
Let $\bar{I}=(\{1\}, \ldots, \{m\})$ be the finest partition of $[m]$ into $m$ one-point clusters
and the structural slicing mapping $s(I)$ be the number of blocks in the partition $I$, so that 
$\mathcal{S} =[m]$. The group clustering structure is modeled by the following linear spaces
\begin{align*}
\mathbb{L}_{I}=\big\{\text{vec}(\bar{\mathrm{X}}x^1,\ldots ,\bar{\mathrm{X}}x^m)
\in\mathbb{R}^{nm}:\, &
x^j \in \mathbb{R}^p,\,j\in[m], \text{ such that }\\
&
x^{j}=x^{j'}\; \forall\, j,j'\in I_i,\,I_i \in I,\,i\in[s(I)] \big\},
\end{align*}
where $I\in \mathcal{I}\triangleq\mathcal{I}_1\cup\{\bar{I}\}$ with 
$\mathcal{I}_1=\{I \in\bar{\mathcal{I}}: ps(I)+m\log s(I)\le pm\}$.
In this case, $\theta=\mathrm{X}\beta$, $d_I=\dim(\mathbb{L}_{I})=ps(I)$ and
$|\mathcal{I}_{s(I)}|=N(m,s(I))$ for $I\in\mathcal{I}_1$, where $N(m,s)$ is the number of ways to put $m$ 
different objects into $s$ different boxes so that each box contains at least one object.
Then $\log |\mathcal{I}_{s(I)}|\le \log s^m(I)=m\log s(I)$ for $I\in\mathcal{I}_1$.
Besides, we have $d_{\bar{I}}=\dim(\mathbb{L}_{\bar{I}})=pm$ and 
$|\mathcal{I}_{s(\bar{I})}|=1$.
Since $d_I +\log |\mathcal{I}_{s(I)}|\le ps(I)+m\log s(I)$ for $I \in \mathcal{I}_1$ and
$d_{\bar{I}} +\log |\mathcal{I}_{s(\bar{I})}| =pm$, we take the majorant 
\[
\rho(I)=\big(ps(I)+m\log s(I)\big) 1\{I\in\mathcal{I}_1\} + pm1\{I=\bar{I}\}.
\]

\begin{remark}
\label{rem_second_regime}
As before, we have an elbow effect, again for the same reason.
The idea of the elbow in the majorant should be clear now: there is no point (although possible) 
to model the structures $I \in \bar{\mathcal{I}}\backslash \mathcal{I}$,
because all these structures are dominated by the structure $\bar{I} \in \mathcal{I}$.
Indeed, for each $I \in \bar{\mathcal{I}}\backslash \mathcal{I}$,
$r^2(I,\beta)=\|(\mathrm{I}-\mathrm{P}_I)\mathrm{X} \beta\|^2 +\sigma^2 \rho(I)
=\|(\mathrm{I}-\mathrm{P}_I)\mathrm{X} \beta\|^2 +
\sigma^2\big(ps(I)+m\log s(I)\big)\ge\sigma^2 pm
=\|(\mathrm{I}-\mathrm{P}_{\bar{I}})\mathrm{X} \beta\|^2 +\sigma^2pm
=r^2(\bar{I},\beta)$, because $\mathrm{P}_{\bar{I}}\mathrm{X} \beta =\mathrm{X} \beta$. 
\end{remark}

Condition \eqref{(A2)} is fulfilled, since, according to Remark \ref{rem5}, for any $\nu\ge 1$ 
\begin{align*}
\sum_{I\in\mathcal{I}}e^{-\nu\rho(I)}&\le %\sum_{s\in\mathcal{S}} e^{-\nu d_s}\le
\sum_{I\in \mathcal{I}_1}e^{-\nu \rho(I)}+e^{-\nu pm}
%\\&
\le\sum_{s\in[m]} e^{-\nu p s}+e^{-\nu pm}
\le (e^{\nu p}-1)^{-1}+1 %(e\nu p)^{-1}
= C_\nu.
\end{align*}

\begin{remark}
\label{rem_redundancy}
Notice that we could consider the full family of structures $\bar{\mathcal{I}}$
under some mild condition. Namely,
we could allow redundancy by associating the same space $\mathbb{L}_{\bar{I}}$ 
to each $I\in\bar{\mathcal{I}}\backslash \mathcal{I}$. 
The majorant becomes 
$\bar{\rho}(I)=\big(ps(I)+m\log s(I)\big) 1\{I\in\mathcal{I}_1\}
+pm1\{I\in\bar{\mathcal{I}}\backslash \mathcal{I}_1\}$, defined now for all $I\in\bar{\mathcal{I}}$.
Then, if $p \gtrsim \log m$, 
Condition \eqref{(A2)} is fulfilled for sufficiently large $\nu$:
\begin{align*}
\sum_{I\in\bar{\mathcal{I}}}e^{-\nu\rho(I)}&\le %\sum_{s\in\mathcal{S}} e^{-\nu d_s}\le
\sum_{I\in \mathcal{I}_1}e^{-\nu \rho(I)}+
\sum_{\in\bar{\mathcal{I}}\backslash \mathcal{I}_1}e^{-\nu \rho(I)}
\\&
\le\sum_{s\in[m]} e^{-\nu p s}+\sum_{s\in[m]} s^m e^{-\nu pm}
\le (e^{\nu p}-1)^{-1}+C%(e\nu p)^{-1}
= C_\nu.
\end{align*}
Thus, this structure redundancy  
$\bar{\mathcal{I}}\backslash \mathcal{I}_1$ does not affect the final local rate, only constant 
$C_\nu$ becomes slightly larger (and the condition $p \gtrsim \log m$ has to hold).
\end{remark}

Condition \eqref{(A3)} is also fulfilled. Indeed, for any  $I^0,I^1\in\mathcal{I}$  define the partition refinement 
\[
I'=I'(I^0,I^1)=I^0\vee I^1=\big(I_i\cap J_j, \, I_i\in I^0, J_j \in I^1\big).
\] 
Clearly, $\mathbb{L}_{I^0}\cup\mathbb{L}_{I^1}
\subseteq \mathbb{L}_{I'}\subseteq\mathbb{L}_{I^0} + \mathbb{L}_{I^1}$ 
and $ \max\{s(I^0),s(I^1)\} \le s(I') \le s(I^0)+s(I^1)$,
implying $ \rho(I')\le\rho(I^0)+\rho(I^1)$, which entails Condition \eqref{(A3)}. 

As consequence of our general results, we obtain the local results of Corollary \ref{meta_theorem}
for these model and structure  with the local rate 
\[
r^2(\beta)= \min_{ I \in \mathcal{I}} \big\{\|(\mathrm{I}-\mathrm{P}_I)\mathrm{X} \beta\|^2 +
\sigma^2\rho(I)\big\}.
\] 

\subsection{A conjectured minimax result for group clustering}
In turn, by virtue of Corollary \ref{minimax_results}, the local results will  imply global minimax 
adaptive results at once over all scales $\{\Theta_\gamma,\, \gamma\in\Gamma\}$ covered by 
the oracle rate $r^2(\beta)$ (i.e., for which \eqref{oracle_minimax} holds).
For example, let  $\Theta_{GC}(s) = \cup_{I\in\bar{\mathcal{I}}: s(I)\le s}\mathbb{L}_I$. To the best of our knowledge, there are no minimax results over $\Theta_{GC}(s)$. 
We conjecture that the minimax rate over 
$\Theta_{GC}(s)$ is 
\[
r^2( \Theta_{GC}(s))\triangleq 
\inf_{\hat{\beta}} \sup_{\beta: \mathrm{X}\beta \in \Theta_{GC}(s)} \mathbb{E}_\beta
\|\mathrm{X}\hat{\beta}-\mathrm{X}\beta\|^2 
\asymp\sigma^2 \min\{ps+m\log s,pm\}.
\]
It is not difficult to show that the local rate $r^2(\beta)$ covers this scale. Indeed, for each 
$\theta=\mathrm{X}\beta \in\Theta_{GS}(s)$ there exists $I_*=I_*(\theta)\in\bar{\mathcal{I}}$
such that $\theta=\mathrm{X}\beta\in\mathbb{L}_{I_*}$ and $s(I_*)\le s$. 
If $ps+m\log s\le pm$, then $I_*\in\mathcal{I}_1$. Hence, $r^2(\beta)\le r^2(I_*,\beta)=\sigma^2\rho(I_*)
=\sigma^2\big(ps(I_*)+m\log s(I_*)\big)\le\sigma^2(ps+m\log s)$ because 
$\mathrm{P}_{I_*}\mathrm{X}\beta=\mathrm{X}\beta$ and $s(I_*)\le s$. 
If $ps+m\log s>pm$, then $r^2(\beta)\le r^2(\bar{I},\beta)=\sigma^2\rho(\bar{I})=\sigma^2 pm$ because 
$\mathrm{P}_{\bar{I}}\mathrm{X}\beta=\mathrm{X}\beta$.

Summarizing, $r^2(\beta)\le\sigma^2 \min\{ps+m\log s,pm\}$.
We thus established the relation 
\eqref{oracle_minimax} for this scale, and Corollary \ref{minimax_results} follows 
with the minimax rate $r^2(\Theta_{GS}(s))$ defined above.

\section{Matrix linear regression with mixture structure}
\label{sec_mixture_model}

Consider the regression model \eqref{matrix_regression} with $p\in[n]$ such that 
$\mathrm{X}^1=\ldots= \mathrm{X}^m=\bar{\mathrm{X}}$, $\bar{\mathrm{X}}=
(\bar{X}_{ij})\in\{0,1\}^{n\times p}$, and $\sum_{j\in [p]} \bar{X}_{ij}=1$ for all $i\in[n]$, i.e., 
each row of the matrix $\bar{\mathrm{X}}$ has $n-1$ zeros and only one entry equals to 1. 
Recently, some estimation results for this model were derived in
\cite{Klopp&Lu&Tsybakov&Zhou&:2017}. To the best of our knowledge, there are no local results 
on posterior contraction rate and   uncertainty quantification problems for mixture model.

In this case, $\theta=\mathrm{X}\beta$ and $\dim(\beta^j)=p\in [n]$ is now not fixed 
but rather a varying ingredient of the structure. Another ingredient of the structure are
the locations $I_i$ of $1$'s in the $i$th $p$-dimensional row of the matrix $\bar{\mathrm{X}}$, $i\in[n]$.
Putting these together, we encode the whole structure as $I=[p, (I_i, i\in[n])]$ where $I_i\in[p]$, $p\in[n]$.
Thus, the full family of all structures is 
\[
\bar{\mathcal{I}}=\{[p, (I_i, i\in[n])]:\, I_i\in[p],\, p\in[n]\}.
\]
Let $\mathrm{X}_I=\text{diag}\{\bar{\mathrm{X}}_I,\ldots ,\bar{\mathrm{X}}_I\}$,
$\bar{\mathrm{X}}_I=(\bar{X}_{ij})$ be the $(n\times p(I))$-matrix corresponding 
to the structure $I \in\bar{\mathcal{I}}$, that is, $\bar{X}_{iI_i}=1$ for $i\in[n]$ 
and all the other entries of this matrix are zeros. By $p(I)$ we denote the first ingredient 
of the structure $I$, the number of columns in the matrix $\bar{\mathrm{X}}_I$.
The structural slicing mapping is $s(I)=r(I)$, where $r(I)=\rank(\bar{\mathrm{X}}_I)$, 
the number of linearly independent columns in the matrix $\bar{\mathrm{X}}_I$. 
So, $\mathcal{S}=[n]$ and notice that $r(I)\le p(I)$.

The structures in this model are modeled by the linear spaces
\begin{align*}
\mathbb{L}_{I}
%=\big\{\mathrm{X}_I\beta \in\mathbb{R}^{mn}:\beta\in \mathbb{R}^{mp(I)}\big\}
=\big\{\text{vec}(\bar{\mathrm{X}}_Ix^1,\ldots ,\bar{\mathrm{X}}_Ix^m)
\in\mathbb{R}^{nm}:\,x^{j}\in \mathbb{R}^{p(I)}, j\in[m] \big\},
\end{align*}
where $I\in\mathcal{I}\triangleq\mathcal{I}_1\cup\{\bar{I}\}$ with 
$\mathcal{I}_1=\{I \in\bar{\mathcal{I}}: mr(I)+n\log p(I)\le nm\}$ and
$\bar{I}=\big[n,[n]\big]$ (so that $\bar{\mathrm{X}}_{\bar{I}}=\mathrm{I}$ is the 
$n$-dimensional identity matrix).
In this case, $\theta=\mathrm{X}\beta$,  $d_I=\dim(\mathbb{L}_{I})=
m \rank(\bar{\mathrm{X}}_I)=mr(I)\le mp(I)$ 
and $|\mathcal{I}_{s(I)}|\le p^n(I)$ for $I\in\mathcal{I}_1$, because 
$p^n(I)$ is the number of possibilities to choose locations of 1's in the 
$n$ $p(I)$-dimensional rows of the design matrix $\bar{\mathrm{X}}_I$.  
Further, $d_{\bar{I}}=\dim(\mathbb{L}_{\bar{I}})=nm$ 
(as $\bar{\mathrm{X}}_{\bar{I}}=\mathrm{I}$)
and $|\mathcal{I}_{s(\bar{I})}|=1$.
Since $d_I +\log |\mathcal{I}_{s(I)}|\le mr(I)+n\log p(I)$ for $I \in \mathcal{I}_1$ and
$d_{\bar{I}} +\log |\mathcal{I}_{s(\bar{I})}| =nm$, we take the majorant 
\begin{align}
\label{majorant_mixture}
\rho(I)=\big(mr(I)+n\log p(I)\big) 1\{I\in\mathcal{I}_1\} +nm1\{I=\bar{I}\}.
\end{align}
The reason for considering the restricted family of structures $\mathcal{I}$ instead of the full family 
$\bar{\mathcal{I}}$ in this model is the same as for the model from 
Section \ref{sec_group_clustering} and is explained in Remark \ref{rem_second_regime}.
%Besides, there is the same sort of redundancy as the one described 
%in Remark \ref{rem_redundancy}: we could have simply restricted the subset of structures
%$\mathcal{I}\backslash \mathcal{I}_1$ to just one representative structure 
%$\bar{I}=[n,(1,\ldots, n)]$ making the whole family $\mathcal{I} = \mathcal{I}_1\cup\{\bar{I}\}$.
%However, this does not affect the resulting local rate, only the above 
%constant $C_\nu$ is slightly bigger that it could have been.
%\begin{align*}
%\sum_{I\in\mathcal{I}}e^{-\nu\rho(I)}
%&\le\sum_{s\in\mathcal{S}} e^{-\nu d_s}\le
%\sum_{s\in \mathcal{S}_1}e^{-\nu m s}
%+|\mathcal{S}_1^c|e^{-\nu nm}
%\\&
%\le\sum_{s=1}^n e^{-\nu m s}+ne^{-\nu nm}
%\le (e^{\nu m}-1)^{-1}+ (e\nu m)^{-1}\le C_\nu.
%\end{align*}

Condition \eqref{(A2)} is  fulfilled since, according to Remark \ref{rem5}, for any $\nu\ge 1 $
\begin{align*}
\sum_{I\in\mathcal{I}}e^{-\nu\rho(I)}
&\le\sum_{s\in\mathcal{S}} e^{-\nu d_s}
%\le\sum_{s\in \mathcal{S}_1}e^{-\nu m s}+e^{-\nu nm}\\&
\le\sum_{s\in[n]} e^{-\nu m s}+e^{-\nu nm}
\le (e^{\nu m}-1)^{-1}+1=C_\nu.
\end{align*}
%where $\mathcal{S}_1=\{s(I): I\in\mathcal{I}_1\}$.

Condition \eqref{(A3)} can also be verified, %is fulfilled as well. Indeed, take any $I,I' \in\mathcal{I}$,
%$I=[s,(I_i,i\in[n])]$ and $I'=[s',(I'_i,i\in[n])]$. Then $I''$ is constructed as follows.
%If either $I$ or $I'$ is from $\mathcal{I}\backslash \mathcal{I}_1$ we let 
%Let $s''=\max\{s,s'\}$,
which would ensure the coverage property (v) of Corollary \ref{meta_theorem} under EBR as well. 
However, there is no point in verifying Condition \eqref{(A3)} because for this linear regression model with mixture structure we have the same peculiar situation as for the biclustering model from Section 
\ref{sec_biclustering}: the size and coverage claims (vi)--(vii) for the confidence ball 
$B(\hat{\theta},\tilde{R}_M)$ are stronger and more useful than the corresponding claims (iv)--(v) for 
the confidence ball $B(\hat{\theta},\hat{R}_M)$. Let us demonstrate that the linear regression with 
mixture structure does not suffer from the deceptiveness phenomenon, modulo 
the so called highly structured parameters. 

Indeed, as consequence of our general results, we obtain the local results (i)--(iv) and (vi)--(vii)
of Corollary \ref{meta_theorem} for this case with the local rate $r^2(\beta)= \min_{ I \in \mathcal{I}} \big\{\|(\mathrm{I}-\mathrm{P}_I)\mathrm{X} \beta\|^2 +\sigma^2\rho(I)\big\}$,
with $\mathrm{P}_I$ as projection onto $\mathbb{L}_I$ defined above and the majorant
$\rho(I)$ defined by \eqref{majorant_mixture}.
The coverage property (v) for the confidence ball $B(\hat{\theta},\hat{R}_M)$ 
can be shown to hold also, but uniformly only under the EBR, whereas the coverage 
property (vii) for the confidence ball $B(\hat{\theta},\tilde{R}_M)$
is uniform over the entire space $\Theta=\mathbb{R}^{n \times m}$. 
%Since the size 
%$\tilde{R}_M$ turns out to be of the optimal order $\hat{R}_M$ except for 
%the so called highly structured parameters, this means that basically the deceptiveness issue is not present in the coverage property (vii) for the confidence ball $B(\hat{\theta},\tilde{R}_M)$. 
%%It appears only marginally in the size relation (vi): 
The size $\tilde{R}_M$ is of the oracle rate  order (as the radius $\hat{R}_M$) 
uniformly in $\theta\in\Theta\backslash \tilde{\Theta}=\mathbb{R}^{n\times m}\backslash\tilde{\Theta}$
%because $r^2(\theta)\ge c\sigma^2\sqrt{n_1n_2}$ for $\theta\in\Theta\backslash\tilde{\Theta}$, 
where $\tilde{\Theta}$ is defined by \eqref{tilde_Theta}. Since in 
this model  the total number of observations is $N=nm$, 
it is easy to see that $\tilde{\Theta}\subseteq \{\theta\in\mathbb{R}^{n\times m }: \, p(I_o(\theta))=1\}$
(i.e., $\bar{\mathrm{X}}_{I_o}=\mathrm{1}_n$, where $\mathrm{1}_n$ is the $n$-dimensional column of 1's)
where the oracle structure $I_o(\theta)$ is defined by \eqref{oracle}. 
Clearly, the $m$-dimensional $\tilde{\Theta}$ is a ``thin'' subset of  
$\mathbb{R}^{n\times m}$ consisting of \emph{highly structured parameters} $\theta$ 
whose oracle number of columns in the design matrix $\bar{\mathrm{X}}_{I_o}$ is $p(I_o(\theta))=1$. 
As we have already discussed at the end of Section \ref{sec_without_EBR}, this means that, 
modulo these highly structured parameters, there is no deceptiveness phenomenon
in this  model.

\begin{remark}
Notice that our local results for the linear regression model with mixture structure 
actually improve upon the results of \cite{Klopp&Lu&Tsybakov&Zhou&:2017} as we have $mr(I) \le 
m p(I)$ instead of $m p(I)$ (as in \cite{Klopp&Lu&Tsybakov&Zhou&:2017}) 
in the expression of the the local rate $r^2(\beta)$.
%We could have taken the upper bound $mp(I)$ of $m r(I)$ instead of $mr(I)$ 
%in the majorant $\rho(I)$ leading to the local result of \cite{Klopp&Lu&Tsybakov&Zhou&:2017}
This means that this oracle rate $r^2(\beta)$ defined above is smaller than the one from 
\cite{Klopp&Lu&Tsybakov&Zhou&:2017}. Notice that the below global minimax results over 
the considered class cannot be improved as the worst case of the both local rates is the same. 
\end{remark}

Finally, by virtue of Corollary \ref{minimax_results} the local results will  imply global minimax 
adaptive results at once over all scales $\{\Theta_\gamma,\, \gamma\in\Gamma\}$ covered by 
the oracle rate $r^2(\beta)$ (i.e., for which \eqref{oracle_minimax} holds). Below we present 
one such scale, covered by the oracle rate $r^2(\beta)$.

\subsection{Minimax results for the mixture model}
Define the class 
\[
\Theta_{M}(p)= \cup_{I\in\bar{\mathcal{I}}: p(I)\le p}\mathbb{L}_I.
\] 
As is shown in \cite{Klopp&Lu&Tsybakov&Zhou&:2017}, the minimax rate over 
$\Theta_{M}(p)$ is 
\[
r^2( \Theta_{M}(p))\triangleq 
\inf_{\hat{\beta}} \sup_{\beta: \mathrm{X}\beta \in \Theta_{M}(p)} \mathbb{E}_\beta
\|\mathrm{X}\hat{\beta}-\mathrm{X}\beta\|^2 
\asymp \sigma^2 \min\{mp+n\log p,nm\}. 
\]
For each $\theta=\mathrm{X}\beta \in\Theta_{M}(p)$ there exists $I_*=I_*(\theta)\in\bar{\mathcal{I}}$
such that $\theta\in\mathbb{L}_{I_*}$ and $p(I_*)\le p$. 
If $mp + n\log p \le nm$, then $r(I_*)\le p(I_*)\le p$, hence
$I_*\in\mathcal{I}_1$, so that $r^2(\beta)\le r^2(I_*,\beta)=\sigma^2 \rho(I_*) \le 
\sigma^2 \big(mp+n\log p\big)$ because 
$\mathrm{P}_{I_*}\mathrm{X}\beta=\mathrm{X}\beta$.
If $mp + n\log p >nm$, then $r^2(\beta)\le r^2(\bar{I},\beta)=\sigma^2\rho(\bar{I})=\sigma^2nm$ because 
$\mathrm{P}_{\bar{I}}\mathrm{X}\beta =\mathrm{X}\beta$. 

Piecing these together, we obtain that $r^2(\beta)\le\sigma^2 \min\{mp+n\log p,nm\}$ 
for all $\beta$ such that $\theta=\mathrm{X}\beta \in\Theta_{M}(p)$. 
We thus established the relation 
\eqref{oracle_minimax} for this scale, and Corollary \ref{minimax_results} follows 
with the minimax rate $r^2(\Theta_{M}(p))$ defined above.

\section{Matrix linear regression with unknown design: dictionary learning}
\label{sec_dict_learning}
Dictionary learning can be considered as a linear regression problem when 
the design matrix and (sparse) vector of regressors are both unknown. 
The data $Y=(Y_i)_{i\in[mn]}$ are observed according to the model:
\begin{align*}
Y=\bar{\mathrm{D}}r+\sigma\xi, 
\end{align*}
where  $\xi=(\xi_i, i \in[mn]), \xi_i \overset{\rm ind}{\sim}\mathrm{N}(0,1)$, 
$\bar{\mathrm{D}}=\text{diag}\{\mathrm{D},\ldots , \mathrm{D}\}\in\mathbb{R}^{mn\times mp}$ 
is an $m$ block diagonal matrix with $p\in \mathbb{N}$, whose block $\mathrm{D}
=(D_1,\ldots,D_p)\in\mathbb{R}^{n\times p}$ is an unknown dictionary matrix,
$p\le n$ without loss of generality, 
$\sigma>0$ is the known noise intensity, $r=(r^1,\ldots , r^m)\in\mathbb{R}^{mp}$ 
is a concatenation of unknown representations %(linearly independent regression vectors) 
$r^1, \ldots , r^m\in\mathbb{R}^p$ such that each entry $r^j_i$ of each $r^j$ comes from 
a (known) finite set  of numbers: $r^j_i\in \mathcal{R}_K=\{\bar{r}_1,\ldots,\bar{r}_K\}$ 
(for instance, $\mathcal{R}_3=\{-1,0,1\}$), for some $\bar{r}_k\in\mathbb{R}$, $k\in[K]$.
Recently, posterior contraction rate and  oracle estimation 
results for this model were derived by \cite{Gao&vanderVaart&Zhou:2015} and \cite{Klopp&Lu&Tsybakov&Zhou&:2017}, respectively. To the best of our knowledge, 
there are no local results on uncertainty quantification problem for dictionary learning.

In this model, we have $\theta=\bar{\mathrm{D}}r$. The structure $I$ consists of 
two parts: $m$ sparsity patterns $I^m\triangleq(I_1,\ldots, I_m) \subseteq [p]^m$ 
($I_j$ determines which columns are taken in the $j$-th diagonal block $\mathrm{D}$ of 
$\bar{\mathrm{D}}$) and $m$ sparse versions of representation vectors $R_{I^m}\triangleq
(r^1_{I_1},\ldots, r^m_{I_m})$ according to the sparsity patterns $I^m$, 
where $r^j_{I_j}=(r^j_i, i \in I_j)$ with $r^j_i \in  \mathcal{R}_K$, $i\in I_j$, $j\in[m]$.
We encode the structure $I$ as 
$I=(I^m, R_{I^m})$, and the whole family of structures is
\[
\bar{\mathcal{I}}=\{(I^m, R_{I^m}): \, 
r^j_i \in \mathcal{R}_K,\, i\in I_j,\, j\in[m];\,  I_k \subseteq[p],\, k\in[m]\}. 
\]
The structural slicing mapping is defined as $s(I)=(|I_k|,k\in[m])\in\mathcal{S}\triangleq [p]_0^m$.
Further, introduce the subfamily $\mathcal{I}_1$ of $\bar{\mathcal{I}}$:
\[
\mathcal{I}_1=\{I \in\bar{\mathcal{I}}:\,np+l_K(I)\le nm\},
\]
where the quantity $l_K(I)$ is defined as
\begin{align}
\label{l_K(I)}
l_K(I)\triangleq \sum_{j\in[m]} |I_j|\log(\tfrac{ep}{|I_j|})
+ (\log K) \sum_{j\in[m]} |I_j|.
\end{align}
This quantity has the meaning of the log of the cardinality of the structural layer $\mathcal{I}_{s(I)}$
and its motivation to appear here will become clear later.

%Following \cite{Klopp&Lu&Tsybakov&Zhou&:2017}, 
%we assume that $r_i^j=0$, $i\in I_j^c$,   $|I_j|\le \bar{s}$ and  $I_j\subseteq[p]$, $j\in[m]$, 
%for some $\bar{s}\ge0$.  In dictionary learning  we consider two regimes: 
%1) $np+m\bar{s}\log(\frac{ep}{\bar{s}})<nm$; 2) $np+m\bar{s}\log(\frac{ep}{\bar{s}})\ge nm$. 

The structures in this model are modeled by the linear spaces 
%1) For $I \in \mathcal{I}_1$, we set 
\begin{align*}
\mathbb{L}_{I}
%=\big\{\mathrm{X}_I\beta \in\mathbb{R}^{mn}:\beta\in \mathbb{R}^{mp(I)}\big\}
=\big\{\text{vec}(\mathrm{D}_{I_1} r^1_{I_1},\ldots,\mathrm{D}_{I_m} r^m_{I_m})
\in\mathbb{R}^{nm}:\,\mathrm{D}_{I_k} \in \mathbb{R}^{n\times |I_k|},\, k\in[m] \big\},
\end{align*}
where $I\in \mathcal{I}\triangleq\mathcal{I}_1\cup\{\bar{I}\}$ and 
$\bar{I}$ is one special structure (the finest possible) such that  
$s(\bar{I})=(p,\ldots,p)$ ($m$-dimensional vector of $p$'s) and the associated linear 
space is $\mathbb{L}_{\bar{I}}=\big\{\text{vec}(x^1,\ldots ,x^m)\in\mathbb{R}^{nm}: 
\, x^j \in \mathbb{R}^{n},\,j\in[m]\big\}$.
If some $I_j=\varnothing$, then the corresponding column $\mathrm{D}_{I_j}r^j_{I_j}$ is the 
zero column.

In this case, $\theta=\bar{\mathrm{D}}r\in\mathbb{R}^{n\times m}$ (recall that whenever 
appropriate we treat $\theta$ as vector: $\theta\in\mathbb{R}^{nm}$), 
$d_I=\dim(\mathbb{L}_{I})= n\, |\cup_{k\in[m]}I_k| \le np$ for $I\in \mathcal{I}_1$
and $d_{\bar{I}} =\dim(\mathbb{L}_{\bar{I}})=nm$. 
The layer $\mathcal{I}_{s(I)}$ consists of all the structures $I$ which 
have the same $s(I)=(|I_k|, k\in[m])$. Clearly, $|\mathcal{I}_{s(\bar{I})}|=1$ because there 
is only one structure $\bar{I}$ in the layer $\mathcal{I}_{s(\bar{I})}$.
To count the number of structures in $\mathcal{I}_{s(I)}$ for $I\in \mathcal{I}_1$, 
notice that there are $\prod_{j\in [m]}\binom{p}{|I_j|}$ 
possible choices of the sparsity patterns $I^m$  
and there are $K^{\sum_{k\in[m]} |I_k|}$ possible choices of sparse representation vectors 
$R_{I^m}$, yielding the cardinality 
$|\mathcal{I}_{s(I)}|=\prod_{j\in [m]}\binom{p}{|I_j|} \times K^{\sum_{k\in[m]} |I_k|}$. 
Hence, 
\begin{align*}
\log |\mathcal{I}_{s(I)}|&\le \sum_{j\in[m]} |I_j|\log(\tfrac{ep}{|I_j|})
+ (\log K) \sum_{j\in[m]} |I_j|=l_K(I)\quad\text{for} \;\; I\in \mathcal{I}_1,
\end{align*}
where $l_K(I)$ is introduced by \eqref{l_K(I)}. The last relation explains the origin of 
the quantity $l_K(I)$. Since $d_I +\log |\mathcal{I}_{s(I)}|\le np+l_K(I)$ for $I \in \mathcal{I}_1$ 
and $d_{\bar{I}} +\log |\mathcal{I}_{s(\bar{I})}| =nm$, we take the majorant 
\begin{align}
\label{majorant_dict_learning}
\rho(I)=\big(np+l_K(I)\big) 1\{I \in \mathcal{I}_1\} + nm 1\{I=\bar{I}\}.
\end{align}
As for some previous cases of model/structure, we have an elbow effect expressed 
by the quantity $l_K(I)$ in the majorant, and there is no need to consider the structures 
$I \in\bar{\mathcal{I}}\backslash \mathcal{I}$, because these %structures 
are dominated by the structure $\bar{I}$, by the same reasoning as 
in Remark \ref{rem_second_regime}.

Conditions \eqref{cond_nonnormal} and \eqref{cond_A4} hold with $d_I=\dim(\mathbb{L}_I)$
in view of Remarks \ref{rem_cond_A1} and \ref{rem15}. 
Denote $\mathcal{S}_1= \{s(I): \, I\in \mathcal{I}_1\}$.
Condition \eqref{(A2)} is fulfilled, since, according to Remark \ref{rem5}, for a sufficiently large $\nu>1$ 
\begin{align*}
\sum_{I\in\mathcal {I}} e^{-\nu \rho(I)} &\le \sum_{I\in\mathcal{I}_1}e^{-\nu \rho(I)}
+e^{-\nu nm}\le e^{-\nu np} \sum_{s\in\mathcal{S}_1}e^{- (\nu-1)l_K(I)}+e^{-\nu nm}\\
&\le  e^{-\nu np}  \sum_{|I_1|=0}^p \ldots \sum_{|I_m|=0}^p e^{- (\nu-1)l_K(I)}+e^{-\nu nm} \\&
\le e^{-\nu np} \Big(\sum_{l=0}^p e^{-(\nu-1)l}\Big)^m + 1\le\tfrac{e^{-\nu np}}{(1-e^{1-\nu})^m}+1\le C_\nu,
%&=e^{-\nu np}\prod_{j\in [m]}\Big[\sum_{|I_j|=0}^{\bar{s}}
%\tbinom{p}{|I_j|}(\tfrac{ep}{\bar{s}})^{-\nu \bar{s}}  \Big]\le 
%e^{-\nu np}\prod_{j\in [m]}\Big[\sum_{|I_j|=0}^{\bar{s}}
%e^{-(\nu-1)\bar{s}\log \big(\tfrac{ep}{\bar{s}}\big)}\Big]\\
%&\le e^{-\nu np-m(\nu-2)\bar{s}\log \big(\tfrac{ep}{\bar{s}}\big)}=C_\nu.
\end{align*}
under the assumption that $m\lesssim np$. 
\begin{remark}
Notice the emerging condition $m\lesssim np$. This is not completely surprising:
$m$ should not be too big in order not to have too many structures in the layers. 
Alternatively, instead of imposing this condition, we can make the majorant slightly
bigger by setting $np\vee m$ instead of just $np$ in \eqref{majorant_dict_learning}.
Yet another fix would be to remove those structures $I$ from $\mathcal{I}_1$ 
for which $\sum_{j\in[m]} |I_j| <m$. One can show that in this case the above sum will be 
uniformly bounded.
\end{remark}

As for  the previous model (linear regression with mixture structure), there is no point in verifying 
Condition \eqref{(A3)} because the size and coverage claims (vi)--(vii) for the confidence ball 
$B(\hat{\theta},\tilde{R}_M)$ are stronger and more useful for this model and this structure
than the corresponding claims (iv)--(v) for the confidence ball 
$B(\hat{\theta},\hat{R}_M)$. 
Let us demonstrate  that this model in essence does not suffer from 
the deceptiveness phenomenon, modulo the so called highly structured parameters. 

Indeed, as consequence of our general results, we obtain the local results (i)--(iv) and (vi)--(vii)
of Corollary \ref{meta_theorem} for this case with the local rate $r^2(\theta)=
\min_{ I \in \mathcal{I}} \big\{\|(\mathrm{I}-\mathrm{P}_I)\theta\|^2+\sigma^2\rho(I)\big\}$, 
where $\theta=\bar{\mathrm{D}}r$, with majorant $\rho(I)$ defined above and $\mathrm{P}_I$, 
the projection onto $\mathbb{L}_I$ defined above.
The coverage property (v) for the confidence ball $B(\hat{\theta},\hat{R}_M)$ 
can be shown to hold also, but uniformly only under the EBR, whereas the coverage 
property (vii) for the confidence ball $B(\hat{\theta},\tilde{R}_M)$
is uniform over the entire space $\theta=\bar{\mathrm{D}}r\in\Theta
=\mathbb{R}^{n \times m}$. 
The size $\tilde{R}_M$ is of the oracle rate  order (as the radius $\hat{R}_M$) 
uniformly in $\theta\in\Theta\backslash \tilde{\Theta}=\mathbb{R}^{n\times m}\backslash\tilde{\Theta}$,
%because $r^2(\theta)\ge c\sigma^2\sqrt{n_1n_2}$ for $\theta\in\Theta\backslash\tilde{\Theta}$, 
where $\tilde{\Theta}$ is defined by \eqref{tilde_Theta}. 
In this model the total number of observations is $N=nm$ and 
$\tilde{\Theta}=\{\theta\in\mathbb{R}^{n\times m }: \, 
\sigma^{-2}\|(\mathrm{I}-\mathrm{P}_{I_o(\theta)})\theta\|^2 +np+
l_K(I_o(\theta)) \lesssim \sqrt{N} = \sqrt{nm}\}$,
where $I_o(\theta)$ is the oracle structure defined by \eqref{oracle}. 
Clearly, $\tilde{\Theta}$ is a ``thin'' subset of  
$\mathbb{R}^{n\times m}$ consisting of \emph{highly structured parameters} $\theta$, in this case
\emph{ultra-sparse} parameters as their oracle structure must be very sparse: 
$l_K(I_o(\theta))\le C\sqrt{nm} - np$. Actually, $\tilde{\Theta}=\varnothing$ if $m\lesssim p^2n$ which 
is a very mild assumption on the dimensions $n,p,m$ only. 
To summarize, under the assumption $m\lesssim p^2n$, 
in the dictionary learning model there is no deceptiveness issue at all. 

\begin{remark}
Notice that we actually established stronger local results: the local rate $r^2(\theta)=
\min_{ I \in \mathcal{I}} \big\{\|(\mathrm{I}-\mathrm{P}_I)\theta\|^2+\sigma^2\bar{\rho}(I)\big\}$ 
is with a smaller majorant $\bar{\rho}(I)=\min\big\{n|\cup_{i\in[m]}I_i|+
\sum_{i\in[m]} |I_i|\log(\tfrac{ep}{|I_i|})+ (\log K) \sum_{i\in[m]} |I_i|, nm \big\}$, 
under the assumption $m\lesssim np$. If we want to avoid the assumption $m\lesssim np$, then 
we should put $\big(n|\cup_{i\in[m]}I_i|\big)\vee m$ instead of $n|\cup_{i\in[m]}I_i|$ 
in the expression of the majorant $\bar{\rho}(I)$.
\end{remark}

Finally, by virtue of Corollary \ref{minimax_results} the local results will  imply global minimax 
adaptive results at once over all scales $\{\Theta_\beta,\, \beta\in\mathcal{B}\}$ covered by 
the oracle rate $r^2(\theta)$ (i.e., for which \eqref{oracle_minimax} holds).
Below we present one example of scale covered by the oracle rate $r^2(\theta)$.

\subsection{Minimax results for the sparse dictionary learning} 
Define the sparsity class for the dictionary learning model:  for $\bar{s}\in[p]_0$,
$\Theta_{SDL}(\bar{s})=\cup \{\mathbb{L}_I:\, 
I\in \bar{\mathcal{I}}, \, |I_i|\le \bar{s}, \, i\in[m] \}$. 
As is shown in \cite{Klopp&Lu&Tsybakov&Zhou&:2017},
the minimax rate over $\Theta_{SDL}(\bar{s})$ is \[
r^2(\Theta_{SDL}(\bar{s}))\asymp
\sigma^2 \min\big\{np+m\bar{s}\log(\tfrac{ep}{\bar{s}}),nm\big\}.
\]

%Assume that $np+m\bar{s}\log(\frac{ep}{\bar{s}})\le nm$, the other case is even simpler to treat. 
For each $\theta=\bar{\mathrm{D}}r\in\Theta_{SDL}(\bar{s})$ there exists  $I_*\in\bar{\mathcal{I}}$ 
such that $\theta\in\mathbb{L}_{I_*}$, hence $\mathrm{P}_{I_*}\theta=\theta$ and 
$r^2(I_*(\theta),\theta)=\sigma^2\rho(I_*)$. Further, since $|I_{*i}(\theta) |\le \bar{s}$, $i\in[m]$, we have
$l_K(I_*(\theta))=\sum_{i\in[m]} |I_{*i}(\theta)|\log(\tfrac{ep}{|I_{*i}(\theta)|})
+ (\log K) \sum_{i\in[m]} |I_{*i}(\theta)|
\le(1+\log K)m\bar{s}\log(\tfrac{ep}{\bar{s}})$.
%We thus obtain $r^2(I_*(\theta),\theta)=\rho(I_*)=np+l(I_*(\theta))\le np
%+m\bar{s}\log(\tfrac{ep}{\bar{s}})$for any $\theta \in\Theta_{SDL}(\bar{s})$. 
Therefore, if $np+(1+\log K)m\bar{s}\log(\frac{ep}{\bar{s}})\le nm$, 
then $np+l_K(I_*(\theta))\le np +(1+\log K)m\bar{s}\log(\frac{ep}{\bar{s}})\le nm$, 
hence $I_*(\theta)\in\mathcal{I}_1$ and $r^2(\theta)\le r^2(I_*(\theta),\theta)=\sigma^2\rho(I_*(\theta))=
\sigma^2\big(np + l_K(I_*(\theta))\big)\le\sigma^2\big(np +(1+\log K)m\bar{s}\log(\frac{ep}{\bar{s}})\big)$ in this case.
Besides, recall that $\mathrm{P}_{\bar{I}}\theta=\theta$, so that $r^2(\theta) \le 
r^2(\bar{I},\theta)=\sigma^2\rho(\bar{I})=\sigma^2nm$. 
Piecing these together, we obtain that 
\[
r^2(\theta)\lesssim\sigma^2 \min\big\{np +m\bar{s}\log(\tfrac{ep}{\bar{s}}),nm\big\}
\asymp r^2(\Theta_{SDL}(\bar{s}))
\;\text{ for all} \;\theta\in\Theta_{SDL}(\bar{s}).
\] 
We thus established the relation \eqref{oracle_minimax} for this scale, and Corollary \ref{minimax_results} follows with the minimax rate $r^2(\Theta_{SDL}(\bar{s}))$ defined above.

%\thesection=numbsection
\end{document}